\documentclass[11pt]{article}

\usepackage{amsmath, amsfonts, amssymb, amsthm,  graphicx, mathtools, enumerate,mathtools,dsfont}

\usepackage[blocks, affil-it]{authblk}

\usepackage{mathrsfs}

\usepackage[numbers, square, sort]{natbib}
\usepackage[CJKbookmarks=true,
            bookmarksnumbered=true,
			bookmarksopen=true,
			colorlinks=true,
			citecolor=red,
			linkcolor=blue,
			anchorcolor=red,
			urlcolor=blue]{hyperref}
\usepackage[usenames]{color}

\usepackage[letterpaper, left=1.2truein, right=1.2truein, top = 1.2truein, bottom = 1.2truein]{geometry}
\usepackage[ruled, vlined, lined, commentsnumbered]{algorithm2e}

\usepackage{prettyref,soul}

\newtheorem{lemma}{Lemma}[section]
\newtheorem{proposition}{Proposition}[section]
\newtheorem{thm}{Theorem}[section]

\newtheorem{corollary}{Corollary}[section]
\newtheorem{remark}{Remark}[section]

\newrefformat{eq}{(\ref{#1})}
\newrefformat{chap}{Chapter~\ref{#1}}
\newrefformat{sec}{Section~\ref{#1}}
\newrefformat{algo}{Algorithm~\ref{#1}}
\newrefformat{fig}{Fig.~\ref{#1}}
\newrefformat{tab}{Table~\ref{#1}}
\newrefformat{rmk}{Remark~\ref{#1}}
\newrefformat{clm}{Claim~\ref{#1}}
\newrefformat{def}{Definition~\ref{#1}}
\newrefformat{cor}{Corollary~\ref{#1}}
\newrefformat{lmm}{Lemma~\ref{#1}}
\newrefformat{lemma}{Lemma~\ref{#1}}
\newrefformat{prop}{Proposition~\ref{#1}}
\newrefformat{app}{Appendix~\ref{#1}}
\newrefformat{ex}{Example~\ref{#1}}
\newrefformat{exer}{Exercise~\ref{#1}}
\newrefformat{soln}{Solution~\ref{#1}}
\newrefformat{cond}{Condition~\ref{#1}}

\def\Var{\textsf{Var}} 
\def\misc{\textsf{Misclust}}
\def\snr{\textsf{SNR}}
\def\h{\textsf{H}}

\def\text#1{\mbox{\rm #1}}

\def\X{\mathscr{X}}
\def\Z{\mathcal{Z}}

\def\sgn{\text{sign}}

\DeclarePairedDelimiter{\ceil}{\lceil}{\rceil}

\newcommand{\argmin}{\mathop{\rm argmin}}
\newcommand{\argmax}{\mathop{\rm argmax}}

\newcommand{\indc}[1]{{\mathbf{1}_{\left\{{#1}\right\}}}}
\newcommand{\norm}[1]{\left\|{#1} \right\|}

\newcommand{\wh}{\widehat}
\newcommand{\wt}{\widetilde}

\newcommand{\fnorm}[1]{\|#1\|_{\rm F}}
\newcommand{\opnorm}[1]{\|#1\|}

\newcommand{\rank}{\mathop{\sf rank}}

\newcommand{\iprod}[2]{\left \langle #1, #2 \right\rangle}

\newtheorem*{conda}{Condition A}
\newtheorem*{condb}{Condition B}
\newtheorem*{condb'}{Condition B'}
\newtheorem*{condc}{Condition C}
\newtheorem*{condd}{Condition D}
\newtheorem*{conde}{Condition E}

\newcommand{\br}[1]{\left( #1 \right)}
\newcommand{\sbr}[1]{\left[ #1 \right]}
\newcommand{\cbr}[1]{\left\{ #1 \right\}}
\newcommand{\ebr}[1]{\exp\left( #1 \right)}
\newcommand{\pbr}[1]{\p\left( #1 \right)}
\newcommand{\mathr}{\mathbb{R}}
\newcommand{\mathn}{\mathcal{N}}
\newcommand{\abs}[1]{\left| #1 \right|}

\newcommand{\p}{\mathbb{P}}
\newcommand{\E}{\mathbb{E}}
\newcommand{\dist}{\stackrel{d}{=}}

\newcommand{\Zk}{\mathbb{Z}/k\mathbb{Z}}

\renewcommand{\complement}{\mathsf{c}}

\title{Iterative Algorithm for Discrete Structure Recovery
}
\author[1]{Chao Gao}
\author[2]{Anderson Y. Zhang}
\affil[1]{
University of Chicago
}
\affil[2]{
University of Pennsylvania
}

\begin{document}
\maketitle

\begin{abstract}
We propose a general modeling and algorithmic framework for discrete structure recovery that can be applied to a wide range of problems. Under this framework, we are able to study the recovery of clustering labels, ranks of players, signs of regression coefficients, cyclic shifts, and even group elements from a unified perspective. A simple iterative algorithm is proposed for discrete structure recovery, which generalizes methods including Lloyd's algorithm and the power method. A linear convergence result for the proposed algorithm is established in this paper under appropriate abstract conditions on stochastic errors and initialization. We illustrate our general theory by applying it on several representative problems: (1) clustering in Gaussian mixture model, (2) approximate ranking, (3) sign recovery in compressed sensing, (4) multireference alignment,  and (5) group synchronization,  and show that minimax rate is achieved in each case.

\smallskip

\end{abstract}



\section{Introduction}

Discrete structure is commonly seen in modern statistics and machine learning, and various problems can be formulated into tasks of recovering the underlying discrete structure.
A leading example is clustering analysis \citep{hartigan1975clustering}, where the discrete structure of the data is parametrized by a vector of clustering labels. Theoretical and algorithmic understandings of clustering analysis have received much attention in the recent literature especially due to the interest in community detection of network data \citep{girvan2002community,leskovec2010empirical,zhang2016minimax,mossel2018proof}. Other important examples of discrete structure recovery include ranking \citep{braverman2008noisy,mao2017minimax}, variable selection \citep{ji2012ups,butucea2018variable}, crowdsourcing \citep{dawid1979maximum,gao2016exact}, estimation of unknown permutation \citep{collier2016minimax,pananjady2017linear}, graph matching \citep{conte2004thirty,ding2018efficient}, and recovery of hidden Hamiltonian cycle \citep{broder1994finding,bagaria2018hidden}.

Despite the the progress of understanding discrete structures in various specific problems, a general theoretical investigation has been lacking in the literature. This is partly due to the fact that theory of discrete structure recovery can be quite different from traditional statistical estimation of continuous parameters. In fact, it has been argued that the nature of discrete structure recovery is closely related to hypothesis testing theory \citep{gao2018minimax}. In addition, the existing literature on the statistical guarantees of discrete structure recovery mostly focuses on characterizing the condition of exact recovery \citep{zhao2006model,meinshausen2006high,lounici2008sup, wasserman2009high,mossel2014consistency,abbe2015exact,bagaria2018hidden}. Let $z^*=(z_1^*,z_2^*,\dotsc,z_p^*)$ represent a discrete structure of interest, where each $z_j^*$ parametrizes a discrete status of either the $j$th sample or the $j$th variable of the data set. The exact recovery is achieved by some estimator $\wh{z}$ if $\wh{z}_j=z_j^*$ for all $j\in[p]$. However, exact recovery of discrete structure usually requires a strong signal to noise ratio condition. A more interesting, more realistic, but harder problem is when only partial recovery \citep{yun2014accurate,zhang2016minimax,gao2016exact,gao2018community, butucea2018variable,ndaoud2018optimal,gao2017phase}  of $z^*$ is possible. Under this regime, a statistical guarantee can be established on the proportion of errors, and the result will naturally lead to the condition of exact recovery as a special case.

Discrete structure recovery is also challenging from a computational point of view. In spite of being optimal in many cases, maximum likelihood estimation of $z^*$ is often combinatorial and thus computationally infeasible. Though convex relaxations such as linear programming or semidefinite programming can be derived for many specific problems \citep{hajek2016achieving,hajek2016achieving2,montanari2016semidefinite,bagaria2018hidden,giraud2018partial,bandeira2014multireference,ling2020solving}, they may not be scalable to very large data sets and the analysis of partial recovery of convex relaxation is usually quite involved \citep{fei2018exponential,fei2020achieving,giraud2018partial}. Moreover, in many examples such as clustering and variable selection, the data generating process is parametrized both by a discrete structure and a continuous model parameter. The presence of the nuisance continuous parameter further complicates the design of efficient algorithms.

The goal of this paper is to develop a general modeling and algorithmic framework for partial recovery of discrete structures. We first propose a general structured linear model parametrized by a discrete structure $z^*$ and a global continuous parameter $B^*$, which unifies various problems of discrete structure recovery into the same framework. A simple iterative algorithm is then proposed for recovering $z^*$, which can be informally written in the following form
\begin{equation}
z^{(t)}=\argmin_z\sum_{j=1}^p\left\|T_j-\nu_j\left(\wh{B}(z^{(t-1)}),z_j\right)\right\|^2\quad \text{for all }t\geq 1. \label{eq:informal}
\end{equation}
Here, $T_j$ is some local statistic whose distribution depends both on the $j$th label $z_j^*$ and the global continuous parameter $B^*$ of the model. 
Because of the separability of the objective function across $j\in[p]$, each $z_j^{(t)}$ takes the value of $z_j$ such that $\nu_j(\wh{B}(z^{(t-1)}),z_j)$ is the closest to $T_j$, and therefore computation of (\ref{eq:informal}) is straightforward. The general iterative procedure (\ref{eq:informal}) recovers some interesting algorithms, among which perhaps the most important one is Lloyd's algorithm  \citep{lloyd1982least} for $k$-means clustering. In the clustering context, $T_j$ is the $j$th data point, and $\nu_j(\wh{B}(z^{(t-1)}),z_j)$ is the $z_j$th estimated clustering center computed based on the clustering labels $z^{(t-1)}$ from the previous step. In addition, (\ref{eq:informal}) also leads to algorithms in approximate ranking, sign recovery and many other problems that will be studied in details in this paper.

The main result of our paper characterizes conditions under which (\ref{eq:informal}) converges with respect to some loss function $\ell(\cdot,\cdot)$ to be defined later. An informal statement of the result is given below,
\begin{equation}
\ell(z^{(t)},z^*)\leq 2\xi_{\rm ideal}(\delta) + \frac{1}{2}\ell(z^{(t-1)},z^*)\quad \text{for all }t\geq 1, \label{eq:informal-converge}
\end{equation}
with high probability. That is, the value of $\ell(z^{(t)},z^*)$ converges at a linear rate to $4\xi_{\rm ideal}(\delta)$. Here, we use $\xi_{\rm ideal}(\delta)$ to characterize the error of an ideal procedure,
\begin{equation}
\wh{z}^{\rm ideal}=\argmin_z\sum_{j=1}^p\left\|T_j-\nu_j\left(\wh{B}(z^*),z_j\right)\right\|^2, \label{eq:ideal}
\end{equation}
and the definition of $\xi_{\rm ideal}(\delta)$ with a general $\delta>0$ will be given in Section \ref{sec:con}. The convergence result (\ref{eq:informal-converge}) is established with some $\delta>0$ arbitrarily close to $0$. We note that the ideal procedure (\ref{eq:ideal}) is not realizable because of its dependence on the true $z^*$, but (\ref{eq:informal-converge}) shows that the iterative algorithm (\ref{eq:informal}) achieves almost the same statistical performance of (\ref{eq:ideal}). The general abstract result is then applied to several concrete examples: clustering for Gaussian mixture model, approximate ranking, sign recovery in compressed sensing,  multireference alignment, and group synchronization, which represent different types of discrete structure recovery problems. 
Moreover, in each of the examples, we can relate $\xi_{\rm ideal}(\delta)$ to the minimax rate of the problem, and therefore claim that the simple algorithm (\ref{eq:informal}) is both computationally efficient and minimax optimal.

Another popular method that is suitable for discrete structure recovery is the EM algorithm \citep{dempster1977maximum}. The global convergence of EM algorithm has been established under the setting of unimodal likelihood \citep{wu1983convergence} and the setting of two-component Gaussian mixtures \citep{daskalakis2016ten,xu2016global,xu2018benefits,wu2019randomly}. Local convergence results for general settings are  obtained by \cite{balakrishnan2017statistical}. However, the most important difference between \cite{balakrishnan2017statistical} and our work, besides the obvious difference of algorithms, is that our convergence guarantee (\ref{eq:informal-converge}) is established for the estimation error of the discrete structure $z^*$, while the convergence result in \cite{balakrishnan2017statistical} for the EM algorithm is established for the estimation error of the continuous model parameter $B^*$. Results like (\ref{eq:informal-converge}) may be possibly established for the EM algorithm in the context of clustering using the techniques suggested by the paper \cite{zhang2017theoretical}\footnote{The paper \cite{zhang2017theoretical} established the convergence of mean-field coordinate ascent and Gibbs sampling in the sense of (\ref{eq:informal-converge}) for community detection in stochastic block models. Due to the connection and similarity between the EM algorithm and variational Bayes, we believe the techniques used in (\ref{eq:informal-converge}) can also be applied to the analysis of EM algorithms for clustering problems.}, but whether (\ref{eq:informal-converge}) can be proved for the EM algorithm in general settings is unknown.

The most related work to us in the literature is the analysis of Lloyd's algorithm in Gaussian mixture models by \cite{lu2016statistical}. Since Lloyd's algorithm is a special case of (\ref{eq:informal}), our convergence result (\ref{eq:informal-converge}) recovers the result in \cite{lu2016statistical} as a special case with even a slightly weaker condition on the number of clusters. We also mention the recent paper \cite{ndaoud2018sharp} that studies a variant of Lloyd's algorithm and improves the signal to noise ratio condition in \cite{lu2016statistical} for the two-component Gaussian mixtures.

~\\
\emph{Organization. }Our general modeling and algorithmic framework will be introduced in Section \ref{sec:framework}. In Section \ref{sec:con}, we formulate abstract conditions under which we can establish the convergence of the algorithm. Applications to specific examples will be discussed afterwards, including clustering in Gaussian mixture model (Section \ref{sec:clustering}), approximate ranking (Section \ref{sec:rank-results}),  sign recovery in compressed sensing (Section \ref{sec:regression}), multireference alignment (Section \ref{sec:MRA}), and group synchronization (Section \ref{sec:group}). Section \ref{sec:disc} discusses the potential limitations of our framework and possible open problems. Finally, all the technical proofs will be given in Section \ref{sec:pf}.

~\\
\emph{Notation.} For $d \in \mathbb{N}$, we write $[d] = \{1,\dotsc,d\}$.  Given $a,b\in\mathbb{R}$, we write $a\vee b=\max(a,b)$ and $a\wedge b=\min(a,b)$.  For two positive sequences $a_n$ and $b_n$, we write $a_n\lesssim b_n$ to mean that there exists a constant $C > 0$ independent of $n$ such that $a_n\leq Cb_n$ for all $n$; moreover, $a_n \asymp b_n$ means $a_n\lesssim b_n$ and $b_n\lesssim a_n$.  For a set $S$, we use $\indc{S}$ and $|S|$ to denote its indicator function and cardinality respectively. For a vector $v = (v_1,\ldots,v_d)^T \in\mathbb{R}^d$, we define $\norm{v}^2=\sum_{\ell=1}^d v_\ell^2$.  The trace inner product between two matrices $A,B\in\mathbb{R}^{d_1\times d_2}$ is defined as $\iprod{A}{B} =\sum_{\ell=1}^{d_1}\sum_{\ell'=1}^{d_2}A_{\ell \ell'}B_{\ell \ell'}$, while the Frobenius and operator norms of $A$ are given by $\fnorm{A}=\sqrt{\iprod{A}{A}}$ and $\opnorm{A}=s_{\max}(A)$ respectively, where $s_{\max}(\cdot)$ denotes the largest singular value.  The notation $\mathbb{P}$ and $\mathbb{E}$ are generic probability and expectation operators whose distribution is determined from the context.

\section{A General Framework of Models and Algorithms}\label{sec:framework}

We start with the introduction of structured linear model. Consider a pair of random vectors $Y\in\mathbb{R}^N$ and $X\in\mathbb{R}^D$. We impose the relation that
\begin{equation}
\mathbb{E}(Y|X)=\X_{z^*}(B^*). \label{eq:SLM}
\end{equation}
On the right hand side of (\ref{eq:SLM}), $z^*=(z_1^*,\dotsc,z_p^*)$ is a vector of discrete labels, and each $z_j^*$ is allowed to take its value from a label set of size $k$. For simplicity, we assume the label set to be $[k]$ without loss of generality. The vector $B^*$ is the model parameter that lives in a subspace indexed by $z^*$. We use the notation $\mathcal{B}_{z^*}$ for this subspace. Finally, $\X_{z^*}$ is a linear operator jointly determined by $X$ and $z^*$. It maps from $\mathcal{B}_{z^*}$ to $\mathbb{R}^N$.

The general structured linear model (\ref{eq:SLM}) can be viewed as a slight variation of the one introduced by \cite{gao2015general}. It is particularly suitable for the research of label recovery and includes some important examples that will be studied in this paper.

To estimate the labels $z_1^*,\dotsc,z_p^*$, one strategy is to first compute a local statistic $T_j=T_j(X,Y)\in\mathbb{R}^d$ and then infer $z_j^*$ from $T_j$ for each $j\in[p]$. We require that
\begin{equation}
\mathbb{E}(T_j|X)=\mu_j(B^*,z_j^*). \label{eq:local-T-E}
\end{equation}
Then, suppose the model parameter $B^*$ was known, a natural procedure to estimate $z_j^*$ would find an $a\in[k]$ such that $\|T_j-\mu_j(B^*,a)\|^2$ is the smallest. However, for some applications, the form of $\mu_j(B^*,z_j^*)$ may not be available, and thus we need to associate each $\mu_j(B^*,z_j^*)$ with a surrogate $\nu_j(B^*,z_j^*)$. An oracle procedure that uses the knowledge of $B^*$ is given by
\begin{equation}
\wh{z}_j^{\rm oracle} = \argmin_{a\in[k]}\|T_j-\nu_j(B^*,a)\|^2. \label{eq:oracle-procedure}
\end{equation}
On the other hand, since $B^*$ is unknown in practice, we need to replace the $B^*$ in (\ref{eq:oracle-procedure}) by an estimator. A natural procedure is the least-squares estimator $\wh{B}(z^*)$, where for a given $z$, $\wh{B}(z)$ is defined by
\begin{equation}
\wh{B}(z)=\argmin_{B\in\mathcal{B}_{z}}\|Y-\X_{z}(B)\|^2. \label{eq:B-oracle}
\end{equation}
This time we need to know $z$ in (\ref{eq:B-oracle}) to compute $\wh{B}(z)$. Therefore, we shall combine (\ref{eq:oracle-procedure}) and (\ref{eq:B-oracle}) and obtain the following iterative algorithm.

\begin{algorithm}[H]
\DontPrintSemicolon
\SetKwInOut{Input}{Input}\SetKwInOut{Output}{Output}
\Input{The data $Y$, $X$ and the number of iterations $t_{\max}$.}
\Output{The estimator $\wh{z}={z}^{(t_{\max})}$.} 
\nl Compute the initializer ${z}^{(0)}$.

\nl For $t$ in $1:t_{\max}$, compute 
\begin{align}\label{eqn:local_B}
&{B}^{(t)}=\argmin_{B\in\mathcal{B}_{{z}^{(t-1)}}}\|Y-\X_{z^{(t-1)}}(B)\|^2,\\
\text{and }& {z}_j^{(t)}=\argmin_{a\in[k]}\|T_j(X,Y)-\nu_j({B}^{(t)},a)\|^2\quad \forall\; j\in[p].\label{eqn:local_testing}
\end{align}
%

 \caption{Iterative discrete structure recovery}\label{alg:general}
\end{algorithm}

Let us now discuss a few important examples. Though we regard $X$ and $Y$ to be vectors in our general framework, in some specific examples, it is often more convenient to arrange the data into matrices instead of vectors. Of course, the two representations are equivalent and the relation can be precisely described with the operations of vectorization and Kronecker product.

\paragraph{Clustering in Gaussian Mixture Model.} Consider $Y\in\mathbb{R}^{d\times p}$ with $Y_1,\ldots,Y_p$ standing for its columns. We assume that $Y_j\sim \mathn(\theta_{z_j^*}^*,I_d)$ independently for $j\in[p]$. Here, $z_1^*,\ldots, z_p^*\in[k]$ are $p$ clustering labels and $\theta_1^*,\ldots,\theta_k^*\in\mathbb{R}^d$ are $k$ clustering centers. In our general framework, we have $N=dp$, $B^*$ is the concatenation of the $k$ clustering centers, and $\mathcal{B}_{z^*}=\mathbb{R}^{d\times k}$. The linear operator $\X_{z^*}$ maps the matrix $\{\theta_{a}^*\}_{a \in [k]}\in\mathbb{R}^{d\times k}$ to the matrix $\{\theta_{z_j^*}^*\}_{j\in[p]}\in\mathbb{R}^{d\times p}$.
For the algorithm to recover the clustering labels, the obvious local statistic is $T_j=Y_j$ for $j\in[p]$. Moreover, we set $\nu_j(B^*,a)=\mu_j(B^*,a)=\theta_a^*$. Then, Algorithm \ref{alg:general} is specialized into the following iterative procedures:
\begin{eqnarray*}
\theta_a^{(t)} &=& \frac{\sum_{j=1}^p\indc{z_j^{(t-1)}=a}Y_j}{\sum_{j=1}^p\indc{z_j^{(t-1)}=a}}, \quad\quad a\in[k], \\
z_j^{(t)} &=& \argmin_{a\in[k]}\|Y_j-\theta_a^{(t)}\|^2, \quad\quad j\in[p].
\end{eqnarray*}
This is recognized as Lloyd's algorithm \citep{lloyd1982least}, the most popular way to solve $k$-means clustering.

\paragraph{Approximate Ranking.} In the task of ranking, we consider the observation of pairwise interaction data $Y_{ij}$ for $(i,j)\in[p]^2$ and $i\neq j$. The rank or the position of the $j$th player is specified by an integer $z_j^*\in[p]$. What is known as the pairwise comparison model assumes that $Y_{ij}\sim \mathn(\beta^*(z_i^*-z_j^*),1)$ for some signal strength parameter $\beta^*\in\mathbb{R}$. Our goal is to estimate the discrete position $z_j^*$ for each player $j\in[p]$. This is known as the approximate ranking problem \citep{gao2017phase}, which is different from exact ranking where $z^*$ corresponds to a permutation. 
It is easy to see that this approximate ranking model is a special case of our general structured linear model. To be specific, we have $N=p(p-1)$, $B^*$ is identified with $\beta^*$, and $\mathcal{B}_{z^*}=\mathbb{R}$. The linear operator $\X_{z^*}$ maps $\beta^*$ to $\{\beta^*(z_i^*-z_j^*)\}_{1\leq i\neq j\leq p}$. To recover $z_j^*$, it is natural to define
\begin{equation}
T_j=\frac{1}{\sqrt{2(p-1)}}\sum_{i\in[p]\backslash\{j\}}(Y_{ji}-Y_{ij}). \label{eq:T_j-rank}
\end{equation}
Thus, we have
\begin{align}\label{eqn:rank_nu_def}
\mu_j(B^*,a)=\frac{2p}{\sqrt{2(p-1)}}\beta^*\left(a-\frac{1}{p}\sum_{i=1}^pz_i^*\right).
\end{align}
Because of the dependence of $\mu_j(B^*,a)$ on the unknown $\frac{1}{p}\sum_{i=1}^pz_i^*$, we also introduce $\nu_j(B^*,a)$ that replaces $\frac{1}{p}\sum_{i=1}^pz_i^*$ with a fixed value $\frac{p+1}{2}$,
$$\nu_j(B^*,a)=\frac{2p}{\sqrt{2(p-1)}}\beta^*\left(a-\frac{p+1}{2}\right).$$
The choice of $\frac{p+1}{2}$ is due to the parameter space of $z^*$ that will be introduced in Section \ref{sec:rank-results}.
This leads to the following iterative algorithm:
\begin{eqnarray}
\nonumber \beta^{(t)} &=& \frac{\sum_{1\leq i\neq j\leq p}(z_i^{(t-1)}-z_j^{(t-1)})Y_{ij}}{\sum_{1\leq i\neq j\leq p}(z_i^{(t-1)}-z_j^{(t-1)})^2}, \\
\label{eq:f-matching} z_j^{(t)} &=& \argmin_{a\in[p]}\left|\sum_{i\in[p]\backslash\{j\}}(Y_{ji}-Y_{ij})-2p\beta^{(t)}\left(a-\frac{p+1}{2}\right)\right|^2, \quad\quad j\in[p].
\end{eqnarray}
Since (\ref{eq:f-matching}) is recognized as feature matching \citep{collier2016minimax}, this is the iterative feature matching algorithm suggested by \cite{gao2017phase} for approximate ranking.

\paragraph{Sign Recovery in Compressed Sensing.} In a standard regression problem, we assume $Y|X\sim \mathn(X\beta^*,I_n)$. Consider a random design setting, where $X_{ij}\stackrel{iid}{\sim}\mathn(0,1)$ for $(i,j)\in[n]\times[p]$. We study the sign recovery problem, which is equivalent to estimating $z_j^*\in\{-1,0,1\}$, where the three possible values of $z_j^*$ standing for $\beta_j^*$ being negative, zero, and positive. We also define the sparsity level $s=\sum_{j=1}^p|z_j^*|$. In order that sign recovery is information-theoretically possible, we assume that either $\beta_j^*=0$ or $|\beta_j^*|\geq \lambda$. The same setting has been considered by \cite{ndaoud2018optimal}. The sparse linear regression model is clearly a special case of our general framework with the choices $N=n$, $B^*=\beta^*$, and $\mathcal{B}_{z^*}=\{\beta\in\mathbb{R}^p:\beta_j=\beta_j|z_j^*|\}$. The linear operator $\X_{z^*}$ maps $\beta^*$ to $X\beta^*$. Following \cite{ndaoud2018optimal}, we use the local statistic 
\begin{align}\label{eqn:local_statistics_reg}
T_j=\|X_j\|^{-1}X_j^TY 
\end{align}
to recover $z_j^*$. Here, $X_j\in\mathbb{R}^n$ stands for the $j$th column of $X$. Computing its conditional expectation, we obtain
\begin{align}\label{eqn:reg_mu_a_def}
\mu_j(B^*,a)=a\|X_j\|\max\{|\beta_j^*|,\lambda\}+\|X_j\|^{-1}\sum_{l\in[p]\backslash\{j\}}\beta_l^*X_j^TX_l,
\end{align}
for $a\in\{-1,0,1\}$, because of the assumption that $\beta_j^*$ is either $0$ or larger than $\lambda$ in magnitude. Replacing $\max\{|\beta_j^*|,\lambda\}$ in the above formula by some threshold level $2t(X_j)$, we get
\begin{align}\label{eqn:reg_nu_a_def}
\nu_j(B^*,a)=2a\|X_j\|t(X_j)+\|X_j\|^{-1}\sum_{l\in[p]\backslash\{j\}}\beta_l^*X_j^TX_l,
\end{align}
for $a\in\{-1,0,1\}$. The threshold level is specified by
\begin{equation}
t(X_j)=\frac{\lambda}{2}+\frac{\log\frac{p-s}{s}}{\lambda\|X_j\|^2}, \label{eq:def-t(x)}
\end{equation}
which can be derived from a minimax analysis \citep{butucea2018variable,ndaoud2018optimal}. Specializing Algorithm \ref{alg:general} to the current context gives
\begin{eqnarray}
\label{eq:LS} \beta^{(t)} &=& \argmin_{\left\{\beta\in\mathbb{R}^p: \beta_j=\beta_j|z_j^{(t-1)}|\right\}}\|y-X\beta\|^2, \\
\label{eq:VS} z_j^{(t)} &=& \begin{cases}
1 & \frac{X_j^TY-\sum_{l\in[p]\backslash\{j\}}\beta_l^{(t)}X_j^TX_l}{\|X_j\|^2} > t(X_j) \\
0 & -t(X_j) \leq \frac{X_j^TY-\sum_{l\in[p]\backslash\{j\}}\beta_l^{(t)}X_j^TX_l}{\|X_j\|^2} \leq t(X_j) \\
-1 & \frac{X_j^TY-\sum_{l\in[p]\backslash\{j\}}\beta_l^{(t)}X_j^TX_l}{\|X_j\|^2} < -t(X_j).
\end{cases}
\end{eqnarray}
We note that (\ref{eq:VS}) is a slight modification of the variable selection procedure in \cite{ndaoud2018optimal}. The main difference is that \cite{ndaoud2018optimal} uses an estimator of $\beta^*$ computed with an independent data set, while we compute a least-squares procedure (\ref{eq:LS}) restricted on the support of $z^{(t-1)}$ obtained from the previous step using the same data set.

\paragraph{Multireference Alignment.} Consider independent data points $Y_j\sim \mathn(Z_j^*\theta^*,I_d)$ for $j\in[p]$. A common parameter $\theta^*\in\mathbb{R}^d$ is shared by the $p$ observations. The matrix $Z_j^*$ is a cyclic shift such that $(Z_j\theta^*)_i=\theta^*_{i+t_j(\text{mod }d)}$ for some integer $t_j$. In other words, for each $j\in[p]$, a noisy shifted version is observed. To put the problem into our general framework, we have $N=dp$, $B^*=\theta^*$, $z_j^*=Z_j^*$, and $\mathcal{B}_{z_j^*}=\mathbb{R}^d$. The linear operator $\X_{z^*}$ maps $\theta^*$ to $(Z_1^*\theta^*,\cdots,Z_p^*\theta^*)$. We are interested in the recovery of the cyclic shifts $Z_1^*,\cdots,Z_p^*$. For this purpose, consider the local statistic $T_j=Y_j$ for all $j\in[p]$. This results in $\nu_j(B^*,U)=\mu_j(B^*,U)=U\theta^*$ for any $U \in\mathcal{C}_d$ where $\mathcal{C}_d$ is  the class of cyclic shifts.
Then, Algorithm \ref{alg:general} is specialized into the following iterative procedures:
\begin{eqnarray*}
\theta^{(t)} &=& \frac{1}{p}\sum_{j=1}^p(Z_j^{(t-1)})^TY_j, \\
Z_j^{(t)} &=& \argmin_{U\in\mathcal{C}_d}\|Y_j-U\theta^{(t)}\|^2, \quad\quad j\in[p].
\end{eqnarray*}
It is clear that $|\mathcal{C}_d|=d$ and thus the update for $Z_j^{(t)}$ has a linear complexity. The statistical property of the above algorithm will be analyzed in Section \ref{sec:MRA}.

\paragraph{Group Synchronization.}

Consider a group $(\mathcal{G},\circ)$ and group elements $g_1,\cdots,g_p\in\mathcal{G}$. The group synchronization problem is the recovery of $g_1,\cdots,g_p$ from noisy versions of $g_i\circ g_j^{-1}$. It turns out a number of important instances of group synchronization can be regarded as special cases of our general framework, and thus can be provably solved by the iterative algorithm. Let us illustrate by the simplest example of $\mathbb{Z}_2$ synchronization. In this model, one observes $Y_{ij}\sim \mathn(\lambda^* z_i^*z_j^*,1)$ for all $1\leq i<j\leq p$ with $z_1^*,\cdots,z_p^*\in\{-1,1\}$. The parameter $\lambda^*\in\mathbb{R}$ plays the role of signal-to-noise ratio.

Though it is most natural to identify $\lambda^*$ with $\beta^*$ in the general framework, this treatment would result in an iterative algorithm with a computationally infeasible update of $z^{(t)}$ because of the quadratic dependence. A smart and much better way is to regard the vector $\lambda^*z^*\in\mathbb{R}^p$ as $\beta^*$ by taking advantage of the flexibility of $\mathcal{B}_{z^*}$. To be specific, we can organize the observations into a matrix $Y=z^*(\beta^*)^T+W\in\mathbb{R}^{p\times p}$ with $\beta^*\in\mathcal{B}_{z^*}=\{\beta=\lambda z^*: \lambda\in\mathbb{R}\}$. In this way, we have $N=\frac{p(p-1)}{2}$, $B^*=\beta^*$, and $\mathcal{B}_{z^*}=\{\beta=\lambda z^*: \lambda\in\mathbb{R}\}$. Thus, given $\beta^*$, the mean of $Y$ is linear with respect to $z^*$. To derive an iterative algorithm, we let $T_j=Y_j$ be the $j$th column of $Y$. We then have $\nu_j(B^*,a)=\mu_j(B^*,a)=a\beta^*$ for $a\in\{-1,1\}$. The iterative algorithm is
\begin{eqnarray}
\label{eq:dead-cells} \beta^{(t)} &=& \argmin_{\beta=\lambda z^{(t-1)}:\lambda\in\mathbb{R}}\fnorm{Y-z^{(t-1)}\beta^T}^2, \\
\nonumber z_j^{(t)} &=& \argmin_{a\in\{-1,1\}}\|Y_j-a\beta^{(t)}\|^2.
\end{eqnarray}
It is easy to see that (\ref{eq:dead-cells}) has a closed form $\beta^{(t)}=\frac{(z^{(t-1)})^TYz^{(t-1)}}{p^2}z^{(t-1)}$. This leads to the following equivalent form of the algorithm
\begin{equation}
z_j^{(t)} = \begin{cases}
\sgn(Y_j^Tz^{(t-1)}), & (z^{(t-1)})^TYz^{(t-1)} \geq 0, \\
-\sgn(Y_j^Tz^{(t-1)}), & (z^{(t-1)})^TYz^{(t-1)} < 0,
\end{cases} \label{eq:power-z2}
\end{equation}
which is a variation of the power method. Despite the simplicity of (\ref{eq:power-z2}), its statistical property is unknown in the literature and will be analyzed in Section \ref{sec:z2}.

In addition to $\mathbb{Z}_2$ synchronization, the above idea can also be applied to other group synchronization problems. Examples of $\Zk$ synchronization and permutation synchronization will be analyzed in Section \ref{sec:zk} and \ref{sec:perm} respectively.

\section{Convergence Analysis}\label{sec:con}

In this section, we formulate abstract conditions under which we can derive the statistical and computational guarantees of Algorithm \ref{alg:general}. 

\paragraph{A General Loss Function}
Our goal is to establish a bound for every $t\geq 1$ with respect to the loss $\ell(z^{(t)},z^*)$. The loss function is defined by
\begin{equation}
\ell(z,z^*) = \sum_{j=1}^p\|\mu_j(B^*,z_j)-\mu_j(B^*,z_j^*)\|^2. \label{eq:loss}
\end{equation}
It has a close relation to the Hamming loss $h(z,z^*)=\sum_{j=1}^p\indc{z_j\neq z_j^*}$.
Define $$\Delta_{\min}^2=\min_{j\in[p]}\min_{1\leq a\neq b\leq k}\|\mu_j(B^*,a)-\mu_j(B^*,b)\|^2,$$ and then we immediately have
\begin{equation}
\ell(z,z^*)\geq \Delta_{\min}^2h(z,z^*). \label{eq:l-to-hamming}
\end{equation}

\paragraph{Error Decomposition} By (\ref{eq:local-T-E}), we can decompose each local statistic as
\begin{equation}
T_j=\mu_j(B^*,z_j^*)+\epsilon_j.\label{eq:Tj-model}
\end{equation}
We usually have $\epsilon_j\sim \mathn(0,I_d)$, but this is not required, and we shall also note that the $\epsilon_j$'s may not even be independent across $j\in[p]$. By (\ref{eqn:local_testing}), if we start the algrotihm from any $z$, then $z^*_j$ will be incorrectly estimated after one iteration if $z_j^* \neq \argmin_{a\in[k]}\|T_j-\nu_j({\wh B}(z),a)\|^2$.
Consequently, assume $z_j^*=a$, and it is important to analyze the event
\begin{equation}
\|T_j-\nu_j(\wh{B}(z),b)\|^2 \leq \|T_j-\nu_j(\wh{B}(z),a)\|^2, \label{eq:error:a-b}
\end{equation}
for any $b\in[k]\backslash\{a\}$. Recall the definition of $\wh{B}(z)$ in (\ref{eq:B-oracle}). We plug (\ref{eq:Tj-model}) into (\ref{eq:error:a-b}), and then after some rearrangement, we can see that the event (\ref{eq:error:a-b}) is equivalent to
\begin{equation}
\iprod{\epsilon_j}{\nu_j(\wh{B}(z^*),a)-\nu_j(\wh{B}(z^*),b)} \leq -\frac{1}{2}\Delta_j(a,b)^2 + F_j(a,b;z) + G_j(a,b;z) + H_j(a,b). \label{eq:decomposition}
\end{equation}
On the right hand side of (\ref{eq:decomposition}), $\Delta_j(a,b)^2$ is the main term that characterizes the difference between the two labels $a$ and $b$. It is defined as
$$\Delta_j(a,b)^2=\|\mu_j(B^*,a)-\nu_j(B^*,b)\|^2-\|\mu_j(B^*,a)-\nu_j(B^*,a)\|^2.$$
Note that with the notation $\Delta_j(a,b)^2$, we have implicitly assume that $\Delta_j(a,b)^2\geq 0$ throughout the paper. This assumption is easily satisfied in all the examples considered in the paper.
The other three terms in (\ref{eq:decomposition}) are the error terms that we need to control. Their definitions are given by
\begin{eqnarray*}
F_j(a,b;z) &=& \iprod{\epsilon_j}{\br{\nu_j(\wh{B}(z^*),a)-\nu_j(\wh{B}(z),a)}-\br{\nu_j(\wh{B}(z^*),b)-\nu_j(\wh{B}(z),b)}}, \\
G_j(a,b;z) &=& \frac{1}{2}\left(\|\mu_j(B^*,a)-\nu_j(\wh{B}(z),a)\|^2-\|\mu_j(B^*,a)-\nu_j(\wh{B}(z^*),a)\|^2\right) \\
&& -\frac{1}{2}\left(\|\mu_j(B^*,a)-\nu_j(\wh{B}(z),b)\|^2-\|\mu_j(B^*,a)-\nu_j(\wh{B}(z^*),b)\|^2\right), \\
H_j(a,b) &=& \frac{1}{2}\left(\|\mu_j(B^*,a)-\nu_j(\wh{B}(z^*),a)\|^2-\|\mu_j(B^*,a)-\nu_j(B^*,a)\|^2\right) \\
&& -\frac{1}{2}\left(\|\mu_j(B^*,a)-\nu_j(\wh{B}(z^*),b)\|^2-\|\mu_j(B^*,a)-\nu_j(B^*,b)\|^2\right).
\end{eqnarray*}
With these quantities defined as above, we can check that (\ref{eq:decomposition}) is indeed equivalent to (\ref{eq:error:a-b}). To help readers understand the meaning of these error terms, we work out the formulas in the context of $\mathbb{Z}_2$ synchronization. By specializing the definitions of the error terms in $\mathbb{Z}_2$ synchronization, we have for any $a\neq b$,
\begin{eqnarray}
\label{eq:aspid111} F_j(a,b;z) &=& (a-b)\iprod{\epsilon_j}{\wh{\beta}(z^*)-\wh{\beta}(z)}, \\
\label{eq:aspid112} G_j(a,b;z) &=& 2\iprod{\beta^*}{\wh{\beta}(z^*)-\wh{\beta}(z)}, \\
\label{eq:aspid113} H_j(a,b) &=& 2\iprod{\beta^*}{\beta^*-\wh{\beta}(z^*)}.
\end{eqnarray}

The reason to have such decomposition (\ref{eq:decomposition}) is as follows. 
\begin{itemize}
\item By igoring the three error terms, the event $\langle{\epsilon_j},{\nu_j(\wh{B}(z^*),a)-\nu_j(\wh{B}(z^*),b)}\rangle \leq -\frac{1}{2}\Delta_j(a,b)^2$ contributes to the ideal recovery error rate. That is, even if we were given the true $z^*$, applying one iteration in Algorithm \ref{alg:general}, i.e., (\ref{eqn:local_testing}) would still result in some error.
\item The error terms $F_j(a,b;z)$ and $G_j(a,b;z)$ can be controlled by the difference between $\wh{B}(z)$ and $\wh{B}(z^*)$, which further depends on $\ell(z,z^*)$. We will treat $F_j(a,b;z)$ and $G_j(a,b;z)$ differently because the former involves the additional randomness of $\epsilon_j$.
\item The error term $H_j(a,b)$ can be controlled by the difference between $\wh{B}(z^*)$ and $B^*$. In fact, unlike $F_j(a,b;z)$ or $G_j(a,b;z)$, $H_j(a,b)$ does not depend on $z$, and thus its value remains unchanged throughout the iterations.
\end{itemize}

\paragraph{Conditions for Algorithmic Convergence} Now we need to discuss how to analyze the error terms $F_j(a,b;z)$, $G_j(a,b;z)$ and $H_j(a,b)$. There are three types of conditions that we will impose.
\begin{conda}[$\ell_2$-type error control]
Assume that
$$
\max_{\{z:\ell(z,z^*)\leq\tau\}}\sum_{j=1}^p\max_{b\in[k]\backslash\{z_j^*\}}\frac{F_j(z_j^*,b;z)^2\|\mu_j(B^*,b)-\mu_j(B^*,z_j^*)\|^2}{\Delta_j(z_j^*,b)^4\ell(z,z^*)} \leq \frac{1}{256}\delta^2
$$
holds with probability at least $1-\eta_1$, for some $\tau,\delta,\eta_1>0$.
\end{conda}
\begin{condb}[restricted $\ell_2$-type error control]
Assume that
$$
\max_{\{z:\ell(z,z^*)\leq\tau\}}\max_{T\subset[p]}\frac{\tau}{4\Delta_{\min}^2|T|+\tau}\sum_{j\in T}\max_{b\in[k]\backslash\{z_j^*\}}\frac{G_j(z_j^*,b;z)^2\|\mu_j(B^*,b)-\mu_j(B^*,z_j^*)\|^2}{\Delta_j(z_j^*,b)^4\ell(z,z^*)} \leq \frac{1}{256}\delta^2$$
holds with probability at least $1-\eta_2$, for some $\tau,\delta,\eta_2>0$.
\end{condb}
\begin{condc}[$\ell_{\infty}$-type error control]
Assume that
$$
\max_{j\in[p]}\max_{b\in[k]\backslash\{z_j^*\}}\frac{|H_j(z_j^*,b)|}{\Delta_j(z_j^*,b)^2} \leq \frac{1}{4}\delta
$$
holds with probability at least $1-\eta_3$, for some $\tau,\delta,\eta_3>0$.
\end{condc}

Conditions A, B and C are for the error terms $F_j(a,b;z)$, $G_j(a,b;z)$ and $H_j(a,b)$, respectively. Because of the difference of the three terms that we have mentioned earlier, they are controlled in different ways. Both Conditions A and B impose $\ell_2$-type controls and relate $F_j(a,b;z)$ and  $G_j(a,b;z)$ to the loss function $\ell(z,z^*)$. On the other hand, $H_j(a,b)$ is controlled by an $\ell_{\infty}$-type bound in Condition C.

Next, we define a quantity referred to as the ideal error,
\begin{equation}\label{eqn:xi_ideal_def}
\xi_{\rm ideal}(\delta) = \sum_{j=1}^p\sum_{b\in[k]\backslash\{z_j^*\}}\|\mu_j(B^*,b)-\mu_j(B^*,z_j^*)\|^2\indc{\iprod{\epsilon_j}{\nu_j(\wh{B}(z^*),z_j^*)-\nu_j(\wh{B}(z^*),b)} \leq -\frac{1-\delta}{2}\Delta_j(z_j^*,b)^2}.
\end{equation}
We note that $\xi_{\rm ideal}(\delta)$ is a quantity that does not change with $t$. In fact, with some $\delta>0$, $\xi_{\rm ideal}(\delta)$ can be shown to be an error bound for the ideal procedure $\wh{z}_j^{\rm ideal}$ defined in (\ref{eq:ideal}). We therefore choose $\xi_{\rm ideal}(\delta)$ with a small $\delta>0$ as the target error that $z^{(t)}$ converges to. In specific examples studied later in Sections \ref{sec:clustering}-\ref{sec:regression}, we will show $\xi_{\rm ideal}(\delta)$ can be bounded by the minimax rate of each problem.

\begin{condd}[ideal error]
Assume that
\begin{equation}
\xi_{\rm ideal}(\delta) \leq \frac{1}{4}\tau,
\end{equation}
with probability at least $1-\eta_4$, for some $\tau,\delta,\eta_4>0$.
\end{condd}

Finally, we need a condition on $z^{(0)}$, the initialization of Algorithm \ref{alg:general}.
\begin{conde}[initialization]
Assume that
$$\ell(z^{(0)},z^*) \leq \tau,$$
with probability at least $1-\eta_5$, for some $\tau,\eta_5>0$.
\end{conde}

\paragraph{Convergence Guarantee} With all the conditions specified, we establish the convergence guarantee for Algorithm \ref{alg:general}.

\begin{thm}\label{thm:main}
Assume Conditions A,B,C, D and E hold for some $\tau,\delta,\eta_1,\eta_2,\eta_3,\eta_4, \eta_5>0$. We then have
$$\ell(z^{(t)},z^*)\leq 2\xi_{\rm ideal}(\delta) + \frac{1}{2}\ell(z^{(t-1)},z^*)\quad \text{for all }t\geq 1,$$
with probability at least $1-\eta$, where $\eta=\sum_{i=1}^5\eta_i$.
\end{thm}

The theorem shows that the error of $z^{(t)}$ converges to $4\xi_{\rm ideal}(\delta)$ at a linear rate. Among all the conditions, Conditions A, B and C are the most important ones. The largest $\tau$ that makes Conditions A, B and C hold simultaneously will be the required error bound for the initialization in Condition E. With (\ref{eq:l-to-hamming}), Theorem \ref{thm:main} also implies that the iterative algorithm achieves an error of $4\xi_{\rm ideal}(\delta)/\Delta^2_{\min}$ in terms of Hamming distance.

In Sections \ref{sec:clustering}-\ref{sec:group}, we will apply Theorem \ref{thm:main} to the examples mentioned in Section \ref{sec:framework}, covering different categories of discrete structures: clustering label, rank,  variable sign, cyclic shift, and group element.
The clustering labels are discrete objects without order or any topological structure. This is in contrast to the ranks that are ordered objects in the space of natural numbers. The variable signs are similar to the clustering labels except two differences. The first difference is the prior knowledge that most variables are zero in the context of sparse linear regression. The second difference is that a nonzero sign only implies a range of a variable instead of its specific value.  Group elements have their own unique properties that depend on the specific settings. Despite all the differences between these discrete structures, we are able to analyze them in a unified framework with the same algorithm.


\section{Clustering in Gaussian Mixture Model}\label{sec:clustering}

We assume the data matrix $Y\in\mathbb{R}^{d\times p}$ is generated from a Gaussian mixture model. This means we have $Y_j=\theta_{z_j^*}+\epsilon_j\sim \mathn(\theta_{z_j^*},I_d)$ independently for $j\in[p]$, where $z^*\in[k]^p$ is the vector of clustering labels that we aim to recover. Specializing Algorithm \ref{alg:general} to the clustering problem, we obtain the well-known Lloyd's algorithm, which can be summarized as
$$z_j^{(t)}=\argmin_{a\in[k]}\|Y_j-\wh{\theta}_a(z^{(t-1)})\|^2,\quad j\in[p],$$
where for each $z\in[k]^p$, we use the notation
$$\wh{\theta}_a(z)=\frac{\sum_{j=1}^p\indc{z_j=a}Y_j}{\sum_{j=1}^p\indc{z_j=a}}, \quad  a\in[k].$$
Even though general $k$-means clustering is known to be NP-hard \citep{dasgupta2008hardness,aloise2009np,mahajan2012planar}, local convergence of the Lloyd's iteration can be established under certain data-generating mechanism \citep{kumar2010clustering,awasthi2012improved}. In particular, the recent work \cite{lu2016statistical} shows that under the Gaussian mixture model, the misclustering error of $z^{(t)}$ in the Lloyd's iteration linearly converges to the minimax optimal rate. In this section, we show that our theoretical framework developed in Section \ref{sec:con} leads to a result that is comparable to the one in \cite{lu2016statistical}.

\subsection{Conditions}
To analyze the algorithmic convergence, we note that $\mu_j(B^*,a) = \nu_j(B^*,a) = \theta_a^*$, $\Delta_j(a,b)^2=\|\theta_a^*-\theta_b^*\|^2$, $\ell(z,z^*)=\sum_{j=1}^p\|\theta_{z_j}^*-\theta_{z_j^*}^*\|^2$, and $\Delta_{\min}=\min_{1\leq a\neq b\leq k}\|\theta_a^*-\theta_b^*\|$ in the current setting. The error terms that we need to control are
\begin{eqnarray*}
F_j(a,b;z) &=& \iprod{\epsilon_j}{\wh{\theta}_a(z^*)-\wh{\theta}_a(z)-\wh{\theta}_b(z^*)+\wh{\theta}_b(z)}, \\
G_j(a,b;z) &=& \frac{1}{2}\left(\|\theta_a^*-\wh{\theta}_a(z)\|^2-\|\theta_a^*-\wh{\theta}_a(z^*)\|^2-\|\theta_a^*-\wh{\theta}_b(z)\|^2+\|\theta_a^*-\wh{\theta}_b(z^*)\|^2\right), \\
H_j(a,b) &=& \frac{1}{2}\left(\|\theta_a^*-\wh{\theta}_a(z^*)\|^2-\|\theta_a^*-\wh{\theta}_b(z^*)\|^2+\|\theta_a^*-\theta_b^*\|^2\right).
\end{eqnarray*}
The following lemma controls the error terms $F_j(a,b;z)$, $G_j(a,b;z)$ and $H_j(a,b)$.
\begin{lemma}\label{lem:kmeans-error}
Assume that $\min_{a\in[k]}\sum_{j=1}^p\indc{z_j^*=a}\geq\frac{\alpha p}{k}$ and $\tau\leq \frac{\Delta_{\min}^2\alpha p}{2k}$ for some constant $\alpha>0$. Then, for any $C'>0$, there exists a constant $C>0$ only depending on $\alpha$ and $C'$ such that
\begin{eqnarray}
\nonumber && \max_{\{z:\ell(z,z^*)\leq\tau\}}\sum_{j=1}^p\max_{b\in[k]\backslash\{z_j^*\}}\frac{F_j(z_j^*,b;z)^2\|\mu_j(B^*,b)-\mu_j(B^*,z_j^*)\|^2}{\Delta_j(z_j^*,b)^4\ell(z,z^*)} \\
\label{eq:kmeans-error1} &\leq& C\frac{k^2(kd/p+1)}{\Delta_{\min}^2}\left(1+\frac{k(d/p+1)}{\Delta_{\min}^2}\right), \\
\nonumber && \max_{\{z:\ell(z,z^*)\leq\tau\}}\max_{T\subset[p]}\frac{\tau}{4\Delta_{\min}^2|T|+\tau}\sum_{j\in T}\max_{b\in[k]\backslash\{z_j^*\}}\frac{G_j(z_j^*,b;z)^2\|\mu_j(B^*,b)-\mu_j(B^*,z_j^*)\|^2}{\Delta_j(z_j^*,b)^4\ell(z,z^*)} \\
\label{eq:kmeans-error2} &\leq& C\left(\frac{k\tau}{p\Delta_{\min}^2} + \frac{k(d+p)}{p\Delta_{\min}^2} + \frac{k^2(d+p)^2}{p^2\Delta_{\min}^4}\right),
\end{eqnarray}
and
\begin{eqnarray}
\label{eq:kmeans-error3} \max_{j\in[p]}\max_{b\in[k]\backslash\{z_j^*\}}\frac{|H_j(z_j^*,b)|}{\Delta_j(z_j^*,b)^2} &\leq& C\left(\frac{k(d+\log p)}{p\Delta_{\min}^2} + \sqrt{\frac{k(d+\log p)}{p\Delta_{\min}^2} }\right),
\end{eqnarray}
with probability at least $1-p^{-C'}$.
\end{lemma}
From the bounds (\ref{eq:kmeans-error1})-(\ref{eq:kmeans-error3}), we can see that a sufficient condition that Conditions A, B and C hold is $\frac{\tau}{p\Delta_{\min}^2/k} \rightarrow 0$
and
\begin{equation}
\frac{\Delta_{\min}^2}{k^2(kd/p+1)} \rightarrow \infty. \label{eq:kmeans-signal}
\end{equation}
In fact, under this sufficient condition, we can set $\delta=\delta_p$ to be some sequence $\delta_p$ converging to $0$ in Conditions A, B and C.

Next, we need to control $\xi_{\rm ideal}(\delta)$ in Condition D. This is given by the following lemma.

\begin{lemma}\label{lem:kmeans-ideal}
Assume $\frac{\Delta_{\min}^2}{\log k+ kd/p}\rightarrow\infty$, $p/k\rightarrow\infty$, and $\min_{a\in[k]}\sum_{j=1}^p\indc{z_j^*=a}\geq\frac{\alpha p}{k}$ for some constant $\alpha>0$. Then, for any sequence $\delta_p = o(1)$, we have
$$\xi_{\rm ideal}(\delta_p) \leq p\exp\left(-(1+o(1))\frac{\Delta_{\min}^2}{8}\right),$$
with probability at least $1-\exp\left(-\Delta_{\min}\right)$.
\end{lemma}

We note that the signal condition $\frac{\Delta_{\min}^2}{\log k+ kd/p}\rightarrow\infty$ required by Lemma \ref{lem:kmeans-ideal} is implied by the stronger condition (\ref{eq:kmeans-signal}). Therefore, we need to require (\ref{eq:kmeans-signal}) for the Conditions A, B, C and D to hold simultaneously.

\subsection{Convergence}
With the help of Lemma \ref{lem:kmeans-error} and Lemma \ref{lem:kmeans-ideal}, we can specialize Theorem \ref{thm:main} into the following result.

\begin{thm}\label{thm:main-kmeans}
Assume (\ref{eq:kmeans-signal}) holds, $p/k\rightarrow\infty$, and $\min_{a\in[k]}\sum_{j=1}^p\indc{z_j^*=a}\geq\frac{\alpha p}{k}$ for some constant $\alpha>0$. Suppose $z^{(0)}$ satisfies
\begin{equation}
\ell(z^{(0)},z^*) = o\left(\frac{p\Delta_{\min}^2}{k}\right), \label{eq:kmeans-ini}
\end{equation}
with probability at least $1-\eta$. Then, we have
$$\ell(z^{(t)},z^*)\leq p\exp\left(-(1+o(1))\frac{\Delta_{\min}^2}{8}\right) + \frac{1}{2}\ell(z^{(t-1)},z^*)\quad \text{for all }t\geq 1,$$
with probability at least $1-\eta-\exp\left(-\Delta_{\min}\right)-p^{-1}$.
\end{thm}

\begin{remark}
Our result is comparable to the main result of \cite{lu2016statistical}. The main difference is that the convergence analysis in \cite{lu2016statistical} is for the misclustering error, defined by
\begin{equation}
\misc(z,z^*)=\frac{1}{p}\sum_{j=1}^p\indc{z_j\neq z_j^*}, \label{eq:misc-loss}
\end{equation}
while Theorem \ref{thm:main-kmeans} is established for an $\ell_2$ type loss function, which is  more natural  in our general framework. The main condition of Theorem \ref{thm:main-kmeans} is the signal requirement (\ref{eq:kmeans-signal}). Interestingly, this is exactly the same condition used in \cite{lu2016statistical}. On the other hand, we only require $k=o(p)$ for the number of clusters allowed, whereas \cite{lu2016statistical} assumes a slightly stronger condition $k=o(p/(\log p)^{1/3})$.
\end{remark}

In the context of clustering, the loss function (\ref{eq:misc-loss}) may be more natural than $\ell(z,z^*)$. Given the relation that
$$\misc(z,z^*)\leq \frac{\ell(z,z^*)}{p\Delta_{\min}^2},$$
we immediately obtain the following corollary on the misclustering error.

\begin{corollary}\label{cor:main-kmeans}
Assume (\ref{eq:kmeans-signal}) holds, $p/k\rightarrow\infty$, and $\min_{a\in[k]}\sum_{j=1}^p\indc{z_j^*=a}\geq\frac{\alpha p}{k}$ for some constant $\alpha>0$. Suppose $z^{(0)}$ satisfies (\ref{eq:kmeans-ini}) with probability at least $1-\eta$. Then, we have
\begin{equation}
\misc(z^{(t)},z^*) \leq \exp\left(-(1+o(1))\frac{\Delta_{\min}^2}{8}\right) + 2^{-t}\quad \text{for all }t\geq 1, \label{eq:misc-bound-ite}
\end{equation}
with probability at least $1-\eta-\exp\left(-\Delta_{\min}\right)-p^{-1}$.
\end{corollary}

According to a lower bound result in \cite{lu2016statistical}, the quantity $\exp\left(-(1+o(1))\frac{\Delta_{\min}^2}{8}\right)$ is the minimax rate of recovering $z^*$ with respect to the loss function $\misc(z,z^*)$ under the Gaussian mixture model. Since $\misc(z,z^*)$ takes value in the set $\{j/p: j\in[p]\cup\{0\}\}$, the term $2^{-t}$ in (\ref{eq:misc-bound-ite}) is negligible as long as $2^{-t}=o(p^{-1})$. We therefore can claim
$$\misc(z^{(t)},z^*) \leq \exp\left(-(1+o(1))\frac{\Delta_{\min}^2}{8}\right) \quad \text{for all }t\geq 3\log p.$$
In other words, the minimax rate is achieved after at most $\ceil{3\log p}$ iterations.

\subsection{Initialization}
To close this section, we discuss how to initialize Lloyd's algorithm. In the literature, this is usually done by spectral methods \citep{kumar2010clustering,awasthi2012improved,lu2016statistical}. We consider the following variation that is particularly suitable for Gaussian mixture models. Our initialization procedure has two steps:
\begin{enumerate}
\item Perform a singular value decomposition on $Y$, and obtain $Y = \sum_{l=1}^{p\wedge n} \wh d_l \wh u_l\wh v_l^T$ with $\wh d_1 \geq \ldots \geq \wh d_{p\wedge n} \geq 0$, $\{\wh u_l\}_{l\in[p\wedge n]}\in\mathr^d$ and $\{\wh v_l\}_{l\in[p\wedge n]}\in\mathr^p$. With $\wh U = (\wh u_1,\ldots, \wh u_k)\in\mathr^{d\times k}$, we define
\begin{equation}
\wh{\mu}=\wh U^T Y \in\mathr^{k\times p}. \label{eq:low-rank-spectral}
\end{equation}
\item Find some $\beta_1^{(0)},\dotsc,\beta_k^{(0)}\in\mathbb{R}^k$ and $z^{(0)}\in[k]^p$ that satisfy
\begin{equation}
\sum_{j=1}^p\|\wh{\mu}_j-\beta^{(0)}_{z_j^{(0)}}\|^2 \leq M\min_{\substack{\beta_1,\dotsc,\beta_k\in\mathbb{R}^k\\ z\in[k]^p}}\sum_{j=1}^p\|\wh{\mu}_j-\beta_{z_j}\|^2, \label{eq:k-means-relax-approx}
\end{equation}
where $\wh{\mu}_j$ is the $j$th column of $\wh{\mu}$.
\end{enumerate}
The first step (\ref{eq:low-rank-spectral}) serves as a dimensionality reduction procedure, which reduces the dimension of data from $d$ to $k$. Then, the columns of $\wh{\mu}$ are collected to compute the $M$-approximation of the $k$-means objective in (\ref{eq:k-means-relax-approx}). We note that approximation of the $k$-means objective can be computed efficiently in polynomial time \citep{kumar2004simple,kanungo2004local,arthur2007k}. For example, the $k$-means++ algorithm \citep{arthur2007k} can efficiently solve (\ref{eq:k-means-relax-approx}) with $M=O(\log k)$. However, we shall treat $M$ flexible here, and its value will be reflected in the error bound of $z^{(0)}$. The second step (\ref{eq:k-means-relax-approx}) can also be replaced by a greedy clustering algorithm used in \citep{gao2017achieving}. The theoretical guarantee of $z^{(0)}$ is given in the following proposition.

\begin{proposition}\label{prop:ini-clust}
Assume $\min_{a\in[k]}\sum_{j=1}^p\indc{z_j^*=a}\geq\frac{\alpha p}{k}$ for some constant $\alpha>0$ and $\Delta_{\min}^2/((M+1)k^2(1+d/p))\rightarrow\infty$.
For any $C'>0$, there exists a constant $C>0$ only depending on $\alpha$ and $C'$ such that
\begin{equation}
\min_{\pi\in\Pi_k} \ell(\pi\circ z^{(0)}, z^*) \leq  C(M+1) k(p+d), \label{eq:da-niu-bi}
\end{equation}
with probability at least $1-e^{-C'(p+d)}$, where $\Pi_k$ denotes the set of permutations on $[k]$.
\end{proposition}

We remark that 
a signal to noise ratio condition that is sufficient for both the conclusions of Proposition \ref{prop:ini-clust} and Theorem \ref{thm:main-kmeans} is given by
\begin{equation}
\frac{\Delta_{\min}^2}{(M +1)k^2(kd/p+1)} \rightarrow \infty, \label{eq:ini-clustering-signal-con}
\end{equation}
which is almost identical to (\ref{eq:kmeans-signal}).
Note that the clustering structure is only identifiable up to a label permutation, and this explains the necessity of the minimum over $\Pi_k$ in (\ref{eq:da-niu-bi}). In other words, (\ref{eq:da-niu-bi}) implies that there exists some $\pi\in\Pi_k$, such that $\ell(z^{(0)},\pi^{-1}\circ z^*)\leq C(M+1) k^2(p+d)$. Then, under the condition (\ref{eq:ini-clustering-signal-con}), (\ref{eq:kmeans-ini}) is satisfied with $z^*$ replaced by $\pi^{-1}\circ z^*$. Therefore, Theorem \ref{thm:main-kmeans} implies that $\ell(z^{(t)},\pi^{-1}\circ z^*)$ converges to the minimax error with a linear rate.

\section{Approximate Ranking}\label{sec:rank-results}

In this section, we study the estimation of $z^*\in[p]^p$ using the pairwise interaction data generated according to $Y_{ij}\sim \mathn(\beta^*(z_i^*-z_j^*),1)$ independently for all $1\leq i\neq j\leq p$. This model can be viewed as a special case of the more general pairwise comparison model $Y_{ij}\sim \mathn(\theta_{z_i^*}^*-\theta_{z_j^*}^*,1)$, where $\theta_i^*$ parametrizes the ability of the $i$th player, and the choice $\theta_i^*=\alpha^* +\beta^*i$ leads to $Y_{ij}\sim \mathn(\beta^*(z_i^*-z_j^*),1)$ that will be studied in this section. Let $\Pi_p$ be the set of all possible permutations of $[p]$. We assume the rank vector $z^*$ belongs to the following class,
\begin{equation}
\mathcal{R}=\left\{z\in[p]^p: \min_{\wt{z}\in\Pi_p}\|z-\wt{z}\|^2\leq c_p\right\}, \label{eq:def-R}
\end{equation}
for some sequence $1\leq c_p=o(p)$. In other words, $\mathcal{R}$ is a set of approximate permutations. A rank vector $z^*\in\mathcal{R}$ is allowed to have ties and not necessarily to start from $1$. To be more precise, a $z^*\in\mathcal{R}$ should be interpreted as discrete positions of the $p$ players in the latent space of their abilities. This is in contrast to the exact ranking problem, also known as ``noisy sorting" in the literature, where $z^*$ is assumed to be a permutation \citep{braverman2008noisy,shah2017simple,mao2017minimax}.

For the loss function
\begin{equation}
L_2(z,z^*)=\frac{1}{p}\sum_{j=1}^p(z_j-z_j^*)^2, \label{eq:emp-l2}
\end{equation}
the minimax rate of estimating $z^*$ takes the following formula,
\begin{equation}
\inf_{\wh{z}}\sup_{z^*\in\mathcal{R}}\mathbb{E}L_2(\wh{z},z^*)\asymp \begin{cases}
\exp\left(-(1+o(1))\frac{p(\beta^*)^2}{4}\right), & p(\beta^*)^2 > 1, \\
\frac{1}{p(\beta^*)^2}\wedge p^2, & p(\beta^*)^2 \leq 1.
\end{cases}\label{eq:minimax-rank-rate}
\end{equation}
See Theorems 2.2 and 2.3 in \cite{gao2017phase}\footnote{The paper \cite{gao2017phase} considers a parameter space that is slightly different from $\mathcal{R}$. However, the proof of \cite{gao2017phase} can be modified so that the same minimax rate also applies to $\mathcal{R}$}. Interestingly, the minimax rate either takes a polynomial form or an exponential form, depending on the signal strength parametrized by $p(\beta^*)^2$. In the paper \cite{gao2017phase}, a combinatorial procedure is constructed to achieve the optimal rate (\ref{eq:minimax-rank-rate}), and whether (\ref{eq:minimax-rank-rate}) can be achieved by a polynomial-time algorithm is unknown. This is where our proposed iterative algorithm comes. We will particularly focus on the regime of $p(\beta^*)^2\rightarrow\infty$, where the minimax rate takes an exponential form.

Specializing Algorithm \ref{alg:general} to the approximate ranking problem, we can write the iterative feature matching algorithm as
$$z_j^{(t)} = \argmin_{a\in[p]}\left|\sum_{i\in[p]\backslash\{j\}}(Y_{ji}-Y_{ij})-2p\wh{\beta}(z^{(t-1)})\left(a-\frac{p+1}{2}\right)\right|^2,\quad j\in[p],$$
where for each $z\in[p]^p$, we use the notation
\begin{equation}
\wh{\beta}(z)=\frac{\sum_{1\leq i\neq j\leq p}(z_i-z_j)Y_{ij}}{\sum_{1\leq i\neq j\leq p}(z_i-z_j)^2}. \label{eq:beta-z-rank}
\end{equation}

\subsection{Conditions}
From (\ref{eqn:rank_nu_def}), we have
\begin{align}\label{eqn:rank_delta_def}
\Delta_j(a,b)^2=\frac{2p^2(\beta^*)^2}{p-1}\left[(a-b)^2-2(a-b)\left(\frac{1}{p}\sum_{j=1}^pz_j^*-\frac{p+1}{2}\right)\right],
\end{align}
and 
\begin{align}\label{eqn:rank_loss_l_def}
\ell(z,z^*)=\frac{2p^2(\beta^*)^2}{p-1}\sum_{j=1}^p(z_j-z_j^*)^2
\end{align} 
 in the current setting. It is easy to check that $\Delta_j(a,b)^2>0$ for all $a\neq b$ as long as $z^*\in\mathcal{R}$. From (\ref{eq:T_j-rank}) and  (\ref{eqn:rank_nu_def}), we have $T_j = \mu_j(B^*,z^*_j) + \epsilon_j$ where $\epsilon_j\sim \mathn(0,1)$ for all $j\in[p]$. One can also write down the formulas of $F_j(a,b;z)$, $G_j(a,b;z)$ and $H_j(a,b)$, which are included in Section \ref{sec:proof_ranking} due to the limit of space.
 
\begin{lemma}\label{lem:rank-error}
Assume $z^*\in\mathcal{R}$, $\tau= o(p^2(\beta^*)^2)$, and $p(\beta^*)^2\geq 1$. Then, for any $C'>0$, there exists a constant $C>0$ only depending on $C'$ such that
\begin{equation}
\max_{\{z:\ell(z,z^*)\leq\tau\}}\sum_{j=1}^p\max_{b\in[p]\backslash\{z_j^*\}}\frac{F_j(z_j^*,b;z)^2\|\mu_j(B^*,b)-\mu_j(B^*,z_j^*)\|^2}{\Delta_j(z_j^*,b)^4\ell(z,z^*)} \leq Cp^{-2}, \label{eq:rank-error1}
\end{equation}
\begin{eqnarray}
\nonumber && \max_{\{z:\ell(z,z^*)\leq\tau\}}\max_{T\subset[p]}\frac{\tau}{4\Delta_{\min}^2|T|+\tau}\sum_{j\in T}\max_{b\in[p]\backslash\{z_j^*\}}\frac{G_j(z_j^*,b;z)^2\|\mu_j(B^*,b)-\mu_j(B^*,z_j^*)\|^2}{\Delta_j(z_j^*,b)^4\ell(z,z^*)} \\
\label{eq:rank-error2} &\leq& C\left(\frac{\tau}{p^2|\beta^*|^2} + \frac{1}{p|\beta^*|^2}\right),
\end{eqnarray}
and
\begin{equation}
\max_{j\in[p]}\max_{b\in[p]\backslash\{z_j^*\}}\frac{|H_j(z_j^*,b)|}{\Delta_j(z_j^*,b)^2} \leq C \frac{1}{\sqrt{p}|\beta^*|}, \label{eq:rank-error3} 
\end{equation}
with probability at least $1-(C'p)^{-1}$ for a sufficiently large $p$.
\end{lemma}
Lemma \ref{lem:rank-error} implies that Conditions A , B and C hold with some sequence $\delta=\delta_p=o(1)$ as long as $\tau= o(p^2(\beta^*)^2)$ and $p(\beta^*)^2\rightarrow\infty$.

Next, we need to control $\xi_{\rm ideal}(\delta)$ in Condition D. This is given by the following lemma.

\begin{lemma}\label{lem:ranking-ideal}
Assume $p(\beta^*)^2\rightarrow\infty$. Then, for any sequence $\delta_p = o(1)$, we have
$$\xi_{\rm ideal}(\delta_p) \leq p\exp\left(-(1+o(1))\frac{p(\beta^*)^2}{4}\right),$$
with probability at least $1-\exp\left(-\sqrt{p(\beta^*)^2}\right)-p^{-1}$.
\end{lemma}

We note that the signal condition $p(\beta^*)^2\rightarrow\infty$ implies that Conditions A, B, C and D  hold simultaneously. 

\subsection{Convergence}
With the help of Lemma \ref{lem:rank-error} and Lemma \ref{lem:ranking-ideal}, we can specialize Theorem \ref{thm:main} into the following result.

\begin{thm}\label{thm:rank-main}
Assume $p(\beta^*)^2\rightarrow\infty$ and $z^*\in \mathcal{R}$. Suppose $z^{(0)}$ satisfies $\ell(z^{(0)},z^*)=o(p^2(\beta^*)^2)$ with probability at least $1-\eta$. Then, we have
$$\ell(z^{(t)},z^*)\leq p\exp\left(-(1+o(1))\frac{p(\beta^*)^2}{4}\right) + \frac{1}{2}\ell(z^{(t-1)},z^*)\quad \text{for all }t\geq 1,$$
with probability at least $1-\eta-\exp\left(-\sqrt{p(\beta^*)^2}\right)-2p^{-1}$.
\end{thm}

Using the relation from (\ref{eq:emp-l2}) and (\ref{eqn:rank_loss_l_def}) that
\begin{equation}
L_2(z,z^*)=\frac{p-1}{2p^3(\beta^*)^2}\ell(z,z^*), \label{eq:relation-loss-rank}
\end{equation}
we immediately obtain the following result on the loss $L_2(z,z^*)$.
\begin{corollary}\label{cor:rank-main}
Assume $p(\beta^*)^2\rightarrow\infty$ and $z^*\in \mathcal{R}$. Suppose $z^{(0)}$ satisfies $\ell(z^{(0)},z^*)=o(p^2(\beta^*)^2)$ with probability at least $1-\eta$. Then, we have
\begin{equation}
L_2(z^{(t)},z^*) \leq \exp\left(-(1+o(1))\frac{p(\beta^*)^2}{4}\right) + 2^{-t}\quad \text{for all }t\geq 1, \label{eq:L2-bound-ite}
\end{equation}
with probability at least $1-\eta-\exp\left(-\sqrt{p(\beta^*)^2}\right)-2p^{-1}$.
\end{corollary}

We observe that $L_2(z,z^*)$ takes value in the set $\{j/p: j\in\mathbb{N}\cup\{0\}\}$, the term $2^{-t}$ in (\ref{eq:L2-bound-ite}) is negligible as long as $2^{-t}=o(p^{-1})$. We therefore can claim
$$L_2(z^{(t)},z^*) \leq \exp\left(-(1+o(1))\frac{p(\beta^*)^2}{4}\right) \quad \text{for all }t\geq 3\log p.$$
Hence, by (\ref{eq:minimax-rank-rate}) the iterative feature matching algorithm achieves the minimax rate of approximate ranking in the regime of $p(\beta^*)^2\rightarrow\infty$ after at most $\ceil{3\log p}$ iterations.

\subsection{Initialization}

To initialize the iterative feature matching algorithm, we consider a simple ranking procedure based on the statistics $\{T_j\}_{j\in[p]}$. That is, letting $T_{(1)}\leq \cdots \leq T_{(p)}$ be the order statistics of $\{T_j\}_{j\in[p]}$, we define $z^{(0)}$ to be a permutation vector that satisfies $T_{z_j^{(0)}}=T_{(j)}$ for all $j\in[p]$.

\begin{proposition} \label{prop:initial-rank}
Assume $z^*\in\mathcal{R}$ and $\beta^*>0$. Then, we have
$$L_2(z^{(0)},z^*)\lesssim \begin{cases}
o(1), & p(\beta^*)^2\rightarrow\infty, \\
\frac{1}{p(\beta^*)^2}\wedge p^2, & p(\beta^*)^2=O(1),
\end{cases}$$
with probability at least $1-p^{-1}$.
\end{proposition}

Note that the additional condition $\beta^*>0$ guarantees that $z^{(0)}$ estimates $z^*$ instead of its reverse order.
In the regime of $p(\beta^*)^2\rightarrow\infty$, the initialization procedure achieves $L_2(z^{(0)},z^*)=o(1)$ with high probability. Given the relation (\ref{eq:relation-loss-rank}), this implies that $\ell(z^{(0)},z^*)=o(p^2(\beta^*)^2)$, and thus the initialization condition of Theorem \ref{thm:rank-main} is satisfied. In the regime of $p(\beta^*)^2=O(1)$, the initialization procedure achieves the rate $\frac{1}{p(\beta^*)^2}\wedge p^2$, which is already minimax optimal according to (\ref{eq:minimax-rank-rate}), and there is no need for the improvement via the iterative algorithm.

\section{Sign Recovery in Compressed Sensing}\label{sec:regression}

We consider a regression model $Y=X\beta^*+w\in\mathbb{R}^n$, where $X\in\mathbb{R}^{n\times p}$ is a random design matrix with i.i.d. entries $X_{ij}\sim \mathn(0,1)$, and $w$ is an independent noise vector with i.i.d. entries $w_i\sim \mathn(0,1)$. Our goal is to recover the signs of the regression coefficients $\beta_j^*$'s. Formally speaking, we assume
$$z^*\in\mathcal{Z}_s=\left\{z\in\{-1,0,1\}^p:\sum_{j=1}^p|z_j|=s\right\},$$
and $\beta^*\in\mathcal{B}_{z^*,\lambda}$, where for some $z\in\{-1,0,1\}^p$ and some $\lambda>0$, the space $\mathcal{B}_{z,\lambda}$ is defined by
$$\mathcal{B}_{z,\lambda}=\left\{\beta\in\mathbb{R}^p: \beta_j=z_j|\beta_j|, \min_{\{j\in[p]:z_j\neq 0\}}|\beta_j|\geq\lambda\right\}.$$
The problem is to estimate the sign vector $z^*$. A closely related problem is support recovery, which is equivalent to estimating the vector $\{|z_j^*|\}_{j\in[p]}$. This problem has received much attention in the literature of compressed sensing, where one usually has control over the distribution of the design matrix. Necessary and sufficient conditions on $(n,p,s,\lambda)$ for exact support recovery have been derived in \cite{fletcher2009necessary,wainwright2009information,aeron2010information,wang2010information,saligrama2011thresholded,rad2011nearly} and references therein.
Recently, the minimax rate of partial support recovery with respect to the Hamming loss has been derived in \cite{ndaoud2018optimal}. Their results can be easily modified to the estimation of the sign vector $z^*$ as well. We will state the lower bound result in \cite{ndaoud2018optimal} as our benchmark. To do that, we need to introduce the normalized Hamming loss
$$\h_{(s)}(z,z^*)=\frac{1}{s}h(z,z^*)=\frac{1}{s}\sum_{j=1}^p\indc{z_j\neq z_j^*}.$$
We also define the signal-to-noise ratio of the problem by
\begin{align}\label{eqn:lambda_def}
\snr=\frac{\lambda\sqrt{n}}{2}-\frac{\log\frac{p-s}{s}}{\lambda\sqrt{n}}.
\end{align}

\begin{thm}[Ndaoud and Tsybakov \citep{ndaoud2018optimal}]\label{thm:lower-regression}
Assume $\limsup {s}/{p}<\frac{1}{2}$ and $s\log p\leq n$. If $\snr\rightarrow\infty$, we have
$$\inf_{\wh{z}}\sup_{z^*\in\mathcal{Z}_s}\sup_{\beta^*\in\mathcal{B}_{z^*,\lambda}}\mathbb{E}\h_{(s)}(\wh{z},z^*)\geq \exp\left(-\frac{(1+o(1))\snr^2}{2}\right)-4e^{-s/8}.$$
Otherwise if $\snr=O(1)$, we then have
$$\inf_{\wh{z}}\sup_{z^*\in\mathcal{Z}_s}\sup_{\beta^*\in\mathcal{B}_{z^*,\lambda}}\mathbb{E}\h_{(s)}(\wh{z},z^*)\geq c,$$
for some constant $c>0$.
\end{thm}
We remark that the lower bound result in \cite{ndaoud2018optimal} is stated in a more general non-asymptotic form. Here, we choose to work out its asymptotic formula (by Lemma \ref{prop:reg_rate_simplify}) so that we can better compare the lower bound with the upper bound rate achieved by our algorithm. In \cite{ndaoud2018optimal}, the minimax rate is achieved by a thresholding procedure that requires sample splitting. Though theoretically sound, the requirement of splitting the data into two halves may not be appealing in practice. This is where our general Algorithm \ref{alg:general} comes. We will show that Algorithm \ref{alg:general} can achieve the minimax rate without sample splitting.

Our analysis is focused in the regime where $\snr\rightarrow\infty$, which is necessary for consistency under the loss $\h_{(s)}(\wh{z},z^*)$ according to Theorem \ref{thm:lower-regression}. 
Specializing Algorithm \ref{alg:general} to the current setting, we obtain the following iterative procedure
\begin{equation}
z_j^{(t)} = \begin{cases}
1 & \frac{X_j^TY-\sum_{l\in[p]\backslash\{j\}}\wh{\beta}_l(z^{(t-1)})X_j^TX_l}{\|X_j\|^2} > t(X_j) \\
0 & -t(X_j) \leq \frac{X_j^TY-\sum_{l\in[p]\backslash\{j\}}\wh{\beta}_l(z^{(t-1)})X_j^TX_l}{\|X_j\|^2} \leq t(X_j) \\
-1 & \frac{X_j^TY-\sum_{l\in[p]\backslash\{j\}}\wh{\beta}_l(z^{(t-1)})X_j^TX_l}{\|X_j\|^2} < -t(X_j)
\end{cases},\quad\quad j\in[p], \label{eq:iterative-sign-recovery}
\end{equation}
where $t(X_j)$ is defined by (\ref{eq:def-t(x)}).
Here, for some $z\in\{-1,0,1\}^p$, we use the notation
$$\wh{\beta}(z)=\argmin_{\{\beta\in\mathbb{R}^p:\beta_j=\beta_j|z_j|\}}\|y-X\beta\|^2.$$
In other words, $\wh{\beta}(z)$ is the least-squares solution on the support of $z$. The formula (\ref{eq:iterative-sign-recovery}) resembles the thresholding procedure proposed in \cite{ndaoud2018optimal}. In \cite{ndaoud2018optimal}, $\wh{\beta}_l(z^{(t-1)})$ is replaced by some estimator $\wh{\beta}_l$ computed from an independent data set. In comparison, we use $\wh{\beta}_l(z^{(t-1)})$ and thus avoid sample splitting. The iteration (\ref{eq:iterative-sign-recovery}) is also different from existing algorithms in the literature for support/sign recovery in compressed sensing. For example, the popular iterative hard thresholding algorithm \citep{blumensath2009iterative} updates the regression coefficients with a gradient step instead of a full least-squares step. The hard thresholding pursuit algorithm \citep{foucart2011hard} has a full least-squares steps, but updates the support by choosing the $s$ variables with the largest absolute values.

\subsection{Conditions}
For any $j\in[p]$, $T_j$ is the local statistic defined in (\ref{eqn:local_statistics_reg}) and it can be decomposed as
$T_j  = \mu_j(B^*,z^*_j) + \epsilon_j$,
with $\epsilon_j = \|X_j\|^{-1} X_j^Tw \sim \mathn(0,1)$. To analyze the algorithmic convergence, we need to specialize the abstract objects $\|\mu_j(B^*,z^*_j)-\mu_j(B^*,b)\|^2$, $\Delta_j(z^*_j,b)^2$, and $\ell(z,z^*)$ into the current setting. With the formulas (\ref{eqn:reg_mu_a_def}) and (\ref{eqn:reg_nu_a_def}), we have
\begin{align}
\norm{\mu_j(B^*,z^*_j)  -\mu_j(B^*,b) }^2 = \begin{cases}
\lambda^2\|X_j\|^2 &  z^*_j=0 \text{ and } b \neq 0\\
|\beta_j^*|^2\|X_j\|^2 & z^*_j\neq 0 \text{ and } b=0\\
4|\beta_j^*|^2\|X_j\|^2 &  z^*_jb=-1,
\end{cases}\label{eqn:reg_mu_diff}
\end{align}
which leads to the formula of the loss function
\begin{eqnarray*}
\ell(z,z^*) &=& \sum_{j=1}^p\left(\lambda^2\|X_j\|^2\indc{z_j^*=0,z_j\neq 0}+|\beta_j^*|^2\|X_j\|^2\indc{z_j^*\neq 0,z_j=0}+4|\beta_j^*|^2\|X_j\|^2\indc{z_jz_j^*=-1}\right).
\end{eqnarray*}
By (\ref{eq:l-to-hamming}), we have the relation
\begin{equation}
\h_{(s)}(z,z^*) \leq \frac{\ell(z,z^*)}{s\Delta_{\min}^2},\label{eq:relation-H_s-l}
\end{equation}
where $\Delta_{\min}^2=\lambda^2\min_{j\in[p]}\|X_j\|^2$ in the current setting. Lastly, the formula of $\Delta_j(z^*_j,b)^2$ is given by
\begin{align}
\Delta_j(z^*_j,b)^2 = 
\begin{cases}
 4t(X_j)^2 \|X_j\|^2  & z^*_j=0 \text{ and } b \neq 0\\
4t(X_j) (|\beta_j^*| - t(X_j))\|X_j\|^2 & z^*_j\neq 0 \text{ and } b=0  \\
 8 t(X_j)|\beta_j^*|\|X_j\|^2 & z^*_jb=-1. 
\end{cases}\label{eqn:reg_delta}
\end{align}
One may question whether we always have $\Delta_j(z^*_j,b)^2>0$ for all $b\neq z^*_j$ and $j\in[p]$. We note that this property is guaranteed by Lemma \ref{prop:useful-reg} with high probability.

Next, we analyze the error terms. In the current setting, they are
\begin{eqnarray*}
F_j(a,b;z) &=& 0, \\
G_j(a,b;z) &=& 2(a-b)t(X_j)\sum_{l\in[p]\backslash\{j\}}\left(\wh{\beta}_l(z)-\wh{\beta}_l(z^*)\right)X_j^TX_l, \\
H_j(a,b) &=& 2(a-b)t(X_j)\sum_{l\in[p]\backslash\{j\}}\left(\wh{\beta}_l(z^*)-\beta^*\right)X_j^TX_l.
\end{eqnarray*}

\begin{lemma}\label{lem:regression-error}
Assume $s\log p\leq n$ and $\tau\leq C_0sn\lambda^2$ for some constant $C_0>0$. Then, for any $C'>0$, there exists a cosntant $C>0$ only depending on $C_0$ and $C'$ such that
\begin{eqnarray}
\nonumber && \max_{\{z:\ell(z,z^*)\leq\tau\}}\max_{T\subset[p]}\frac{\tau}{4\Delta_{\min}^2|T|+\tau}\sum_{j\in T}\max_{b\in\{-1,0,1\}\backslash\{z_j^*\}}\frac{G_j(z_j^*,b;z)^2\|\mu_j(B^*,b)-\mu_j(B^*,z_j^*)\|^2}{\Delta_j(z_j^*,b)^4\ell(z,z^*)} \\
\label{eq:regression-error-1} &\leq& C\frac{s(\log p)^2}{n}\left(1+\frac{1}{n\lambda^2}\right)\max_{j\in[p]}\left[\frac{|\beta_j^*|^2}{(|\beta_j^*|-t(X_j))^2}\vee\frac{\lambda^2}{t(X_j)^2}\right],
\end{eqnarray}
and
\begin{eqnarray}
\label{eq:regression-error-2} \max_{j\in[p]}\max_{b\in\{-1,0,1\}\backslash\{z_j^*\}}\frac{|H_j(z_j^*,b)|}{\Delta_j(z_j^*,b)^2} \leq C\sqrt{\frac{s(\log p)^2}{n}}\frac{1}{\min_{j\in[p]}\sqrt{n}||\beta_j^*|-t(X_j)|},
\end{eqnarray}
with probability at least $1-p^{-C'}$.
\end{lemma}
The two error bounds (\ref{eq:regression-error-1}) and (\ref{eq:regression-error-2}) are complicated. However, by Lemma \ref{prop:useful-reg}, if we additionally assume $\limsup {s}/{p}<\frac{1}{2}$, $\snr\rightarrow\infty$, and $s(\log p)^4=o(n)$, the right hand sides of (\ref{eq:regression-error-1}) and (\ref{eq:regression-error-2}) can be shown to be of order $o((\log p)^{-1})$. Therefore, Conditions A, B and C hold with some $\delta=\delta_p=o((\log p)^{-1})$.

The following lemma controls $\xi_{\rm ideal}(\delta)$ in Condition D.
\begin{lemma}\label{lem:regression-ideal}
Assume $\limsup {s}/{p}<\frac{1}{2}$, $s\log p\leq n$, and $\snr\rightarrow\infty$. Then for any sequence $\delta_p=o((\log p)^{-1})$, we have
$$\xi_{\rm ideal}(\delta_p)\leq sn\lambda^2\exp\left(-\frac{(1+o(1))\snr^2}{2}\right),$$
with probability at least $1-\exp(-\snr)-p^{-1}$.
\end{lemma}

\subsection{Convergence}
With Lemma \ref{lem:regression-error} and Lemma \ref{lem:regression-ideal}, we then can specialize Theorem \ref{thm:main} into the following result.
\begin{thm}\label{thm:regression-iter}
Assume $\limsup {s}/{p}<\frac{1}{2}$, $s(\log p)^4= o(n)$, and $\snr\rightarrow\infty$. Suppose $\ell(z^{(0)},z^*)\leq C_0sn\lambda^2$ with probability at least $1-\eta$ for some constant $C_0>0$. Then, we have
$$\ell(z^{(t)},z^*)\leq sn\lambda^2\exp\left(-\frac{(1+o(1))\snr^2}{2}\right) + \frac{1}{2}\ell(z^{(t-1)},z^*)\quad \text{for all }t\geq 1,$$
with probability at least $1-\eta-\exp(-\snr)-2p^{-1}$.
\end{thm}
The relation (\ref{eq:relation-H_s-l}) and a simple concentration result for $\min_{j\in[p]}\|X_j\|^2$ immediately implies a convergence result for the loss $\h_{(s)}(z,z^*)$.
\begin{corollary}\label{cor:regression-iter}
Assume $\limsup {s}/{p}<\frac{1}{2}$, $s(\log p)^4= o(n)$, and $\snr\rightarrow\infty$. Suppose $\ell(z^{(0)},z^*)\leq C_0sn\lambda^2$ with probability at least $1-\eta$ for some constant $C_0>0$. Then, we have
\begin{equation}
\h_{(s)}(z^{(t)},z^*)\leq \exp\left(-\frac{(1+o(1))\snr^2}{2}\right) + 2^{-t}\quad \text{for all }t\geq 1,\label{eq:reg-iter-bound-h}
\end{equation}
with probability at least $1-\eta-\exp(-\snr)-2p^{-1}$.
\end{corollary}

Since the loss function $\h_{(s)}(z,z^*)$ takes value in the set $\{j/s: j\in[p]\cap\{0\}\}$, the term $2^{-t}$ in (\ref{eq:reg-iter-bound-h}) is negligible as long as $2^{-t}=o(s^{-1})$. We therefore can claim
$$\h_{(s)}(z^{(t)},z^*)\leq \exp\left(-\frac{(1+o(1))\snr^2}{2}\right)\quad \text{for all }t\geq 3\log s,$$
when $s\rightarrow\infty$. If instead we have $s=O(1)$, then any $t\rightarrow\infty$ will do. This implies after at most $\ceil{3\log p}$ iterations, Algorithm \ref{alg:general} achieves the minimax rate.

\begin{remark}
The leading term of the non-asymptotic minimax lower bound in \cite{ndaoud2018optimal} with respect to the loss $\h_{(s)}(z,z^*)$ takes the form of $\psi(n,p,s,\lambda,0)/s$, where
\begin{align}\label{eqn:psi_def}
\psi(n,p,s,\lambda,\delta)=s\mathbb{P}\left(\epsilon>(1-\delta)\|\zeta\|(\lambda-t(\zeta))\right)+(p-s)\mathbb{P}\left(\epsilon>(1-\delta)\|\zeta\|t(\zeta)\right)
\end{align}
with $\epsilon\sim \mathn(0,1)$ and $\zeta\sim \mathn(0,I_n)$ independent of each other. By scrutinizing the proof of Lemma \ref{lem:regression-ideal}, we can also write (\ref{eq:reg-iter-bound-h}) as
$$\h_{(s)}(z^{(t)},z^*)\lesssim  \psi(n,p,s,\lambda,\delta_p)/s + 2^{-t}\quad \text{for all }t\geq 1,$$
with high probability with some $\delta_p=o((\log p)^{-1})$.
\end{remark}

\subsection{Initialization}
Our final task in this section is to provide an initialization procedure that satisfies the bound $\ell(z^{(0)},z^*)\leq C_0sn\lambda^2$ with high probability. We consider a simple procedure that thresholds the solution of the square-root SLOPE \citep{bogdan2015slope,su2016slope,derumigny2018improved,bellec2018slope}. It has the following two steps:
\begin{enumerate}
\item Compute
\begin{equation}
\wt{\beta}=\argmin_{\beta\in\mathbb{R}^p}\left(\|Y-X\beta\|+A\|\beta\|_{\rm SLOPE}\right), \label{eq:sqrt-slope}
\end{equation}
where the penalty takes the form $\|\beta\|_{\rm SLOPE}=\sum_{j=1}^p\sqrt{{\log\left({2p}/{j}\right)}}|\beta|_{(j)}$. Here $|\beta|_{(1)}\geq |\beta|_{(2)}\geq\cdots\geq |\beta|_{(p)}$ is a non-increasing ordering of $|\beta_1|, |\beta_2|,\cdots, |\beta_p|$.
\item For any $j\in[p]$, compute $z^{(0)}_j=\sgn(\wt{\beta}_j)\indc{|\wt{\beta}_j|\geq \lambda/2}$.
\end{enumerate}
The theoretical guarantee of $z^{(0)}$ is given by the following proposition.
\begin{proposition}\label{prop:reg_init}
Assume $\limsup {s}/{p}<\frac{1}{2}$, $s\log p \leq n$, and $\snr\rightarrow\infty$. For some sufficiently large constant $A>0$ in (\ref{eq:sqrt-slope}) and any constant $C'>0$, there exist some $C_0$ and $C_1$ only depending on $A$ and $C'$, such that
$$\ell(z^{(0)},z^*)\leq C_0sn\lambda^2,$$
with probability at least $1-e^{-C_1s\log(ep/s)}-p^{-C'}$.
\end{proposition}

\section{Multireference Alignment}\label{sec:MRA}

Multireference alignment \citep{bandeira2014multireference,aguerrebere2016fundamental} is an important problem in mathematical chemistry and captures fundamental aspects of applications such as cryogenic electron microscopy (cryo-EM) \citep{sigworth1998maximum,frank2006three}. In this problem, we have observations $Y_j\sim \mathn(Z_j^*\theta^*,I_d)$ for $j\in[p]$ with $\theta^*\in\mathbb{R}^d$ and $Z_j^*\in\mathcal{C}_d$. Here, $\mathcal{C}_d \subset\cbr{0,1}^{d\times d}$ is the set of index cyclic shifts. For any $U\in\mathcal{C}_d$, there exists some integer $t$ such that $(Uv)_i=v_{i+t(\text{mod } d)}$ for any vector $v\in\mathbb{R}^d$. The literature for this problem has been focused on the recovery of the common signal parameter $\theta^*\in\mathbb{R}^d$ shared by the $p$ observations \citep{bandeira2014multireference,perry2019sample,bandeira2017optimal,monteiller2019alleviating,abbe2018multireference,bendory2017bispectrum}. To the best of our knowledge, optimal estimation of the cyclic shifts $Z_1^*,\cdots, Z_p^*$ still remains an open problem, and this is the focus of the current section. As is already discussed in Section \ref{sec:framework}, the cyclic shifts can be recovered by the general iterative algorithm. Before giving the statistical guarantee of the algorithm, we first present a minimax lower bound of the problem. We define
$$\Delta_{\min}^2=\min_{U\in\mathcal{C}_d}\|(I_d-U)\theta^*\|^2.$$
This quantity plays the role of the minimal signal strength of the problem, and is very different from the signal strength required for estimating $\theta^*$ in \cite{perry2019sample}. Note that $\Delta_{\min}^2$ is a function of $\theta^*$ and thus captures the difficulty of the problem for each instance of $\theta^*\in\mathbb{R}^d$. For example, if each coordinate of $\theta^*$ takes the same value, the corresponding $\Delta_{\min}^2=0$ and thus it is impossible to recover the shifts. The quantity $\Delta_{\min}^2$ is thus a characterization of the diversity of the sequence $\{\theta_i^*\}$.

\begin{thm}\label{thm:mra-lower}
If $\Delta_{\min}^2\rightarrow\infty$, we have
$$\inf_{\wh{Z}}\sup_{Z^*}\mathbb{E}\min_{U\in\mathcal{C}_d}\frac{1}{p}\sum_{j=1}^p\indc{\wh{Z}_jU\neq Z_j^*}\geq \exp\left(-\frac{(1+o(1))\Delta_{\min}^2}{8}\right).$$
Otherwise if $\Delta_{\min}^2=O(1)$, we then have
$$\inf_{\wh{Z}}\sup_{Z^*}\mathbb{E}\min_{U\in\mathcal{C}_d}\frac{1}{p}\sum_{j=1}^p\indc{\wh{Z}_jU\neq Z_j^*}\geq c,$$
for some constant $c>0$.
\end{thm}

Our main result in this section shows that the minimax lower bound in Theorem \ref{thm:mra-lower} can be achieved by an efficient algorithm adaptively over all $\theta^*\in\mathbb{R}^d$ under the signal-to-noise ratio condition $\frac{\Delta_{\min}^2}{d/p+\sqrt{d\log d}}\rightarrow\infty$. Specializing Algorithm \ref{alg:general} to the current problem, the iterative procedure is given by the following formula, 
\begin{equation}
Z_j^{(t)}=\argmin_{U\in\mathcal{C}_d}\|Y_j-U\wh{\theta}(Z^{(t-1)})\|^2, \label{eq:white-palace-difficult}
\end{equation}
where
$$\wh{\theta}(Z)=\frac{1}{p}\sum_{j=1}^pZ_j^TY_j.$$
The computation of (\ref{eq:white-palace-difficult}) is straightforward given that $|\mathcal{C}_d|=d$ and one can simply evaluate $\|Y_j-U\wh{\theta}(Z^{(t-1)})\|^2$ for each $U\in\mathcal{C}_d$.

\subsection{Conditions}

To analyze the algorithmic convergence, we note that $\mu_j(B^*,U)=\nu_j(B^*,U)=U\theta^*$, $\Delta_j(U,V)^2=\|(U-V)\theta^*\|^2$, and $\ell(Z,Z^*)=\sum_{j=1}^p\|(Z_j-Z_j^*)\theta^*\|^2$ under the current setting. The error terms that we need to control are
\begin{eqnarray*}
F_j(U,V;Z) &=& \iprod{\epsilon_j}{(U-V)(\wh{\theta}(Z^*)-\wh{\theta}(Z))}, \\
G_j(U,V;Z) &=& \iprod{\wh{\theta}(Z)-\wh{\theta}(Z^*)}{(V^TU-I_d)\theta^*}, \\
H_j(U,V) &=& \iprod{\wh{\theta}(Z^*)-\theta^*}{(V^TU-I_d)\theta^*}.
\end{eqnarray*}
Here, the noise vector is given by $\epsilon_j=Y_j-Z_j^*\theta^* \sim\mathn(0,I_d)$. The error terms are controlled by the following lemma.

\begin{lemma}\label{lem:error-MRA}
For any $C'>0$, there exists a constant $C>0$ only depending on $C'$ such that
\begin{eqnarray}
\nonumber && \max_{\{Z:\ell(Z,Z^*)\leq \tau\}}\sum_{j=1}^p\max_{U\in\mathcal{C}_d\backslash\{Z_j^*\}} \frac{F_j(Z_j^*,U;Z)^2\|\mu_j(B^*,U) - \mu_j(B^*,Z_j^*)\|^2}{\Delta_j(Z_j^*,U)^4\ell(Z,Z^*)}  \\
\label{eq:error-MRA1}  &\leq&  C \br{ \frac{(\log d+d/p)\log d}{\Delta_{\min}^4} +  \frac{\tau (\log d + d/p) }{p\Delta_{\min}^4}}, \\
\nonumber && \max_{\{Z:\ell(Z,Z^*)\leq \tau\}}\max_{T\subset[p]}\frac{\tau}{4\Delta_{\min}^2|T|+\tau}\sum_{j\in T}\max_{U\in\mathcal{C}_d\backslash\{Z_j^*\}} \frac{G_j(Z_j^*,U;Z)^2\|\mu_j(B^*,U) - \mu_j(B^*,Z_j^*)\|^2}{\Delta_j(Z_j^*,U)^4\ell(Z,Z^*)}   \\
\label{eq:error-MRA2} &\leq& C\br{ \frac{\tau \log d}{p\Delta_{\min}^4} + \frac{\tau^2}{p^2\Delta_{\min}^4}},
\end{eqnarray}
and
\begin{equation}
\max_{j\in[p]}\max_{U\in\mathcal{C}_d\backslash\{Z_j^*\}}\frac{|H_j(Z_j^*,U)|}{\Delta_j(Z_j^*,U)^2} \leq C\sqrt{\frac{d}{p\Delta_{\min}^2}}, \label{eq:error-MRA3} 
\end{equation}
with probability at least $1-e^{-C'd}$.
\end{lemma}
From the bounds (\ref{eq:error-MRA1})-(\ref{eq:error-MRA3}), we can see that a sufficient condition that Conditions A, B and C hold is $\tau = o(p\Delta_{\min}^2)$ and
\begin{equation}
\frac{\Delta_{\min}^2}{\log d + d/p}\rightarrow\infty. \label{eq:SNR-MRA}
\end{equation}
In fact, when $d=O(1)$, the above condition is reduced to $\Delta_{\min}^2\rightarrow\infty$, which is the necessary condition for consistency by Theorem \ref{thm:mra-lower}.

Next, we need to bound $\xi_{\rm ideal}(\delta)$ in Condition D. This is given by the following lemma.

\begin{lemma}\label{lem:ideal-MRA}
Assume $\frac{\Delta_{\min}^2}{\log d + d/p}\rightarrow\infty$. Then, for any sequence $\delta_p=o(1)$, we have
$$\xi_{\rm ideal}(\xi_p)\leq p\exp\left(-(1+o(1))\frac{\Delta_{\min}^2}{8}\right),$$
with probability at least $1-\exp(-\Delta_{\min})$.
\end{lemma}


To summarize, under the signal-to-noise ratio condition (\ref{eq:SNR-MRA}) and the initialization condition $\tau = o(p\Delta_{\min}^2)$, Conditions A, B, C and D hold simultaneously.

\subsection{Convergence}

With the help of Lemma \ref{lem:error-MRA} and Lemma \ref{lem:ideal-MRA}, we can specialize Theorem \ref{thm:main} into the following result.

\begin{thm}\label{thm:final-MRA}
Assume (\ref{eq:SNR-MRA}) holds. Suppose $Z^{(0)}$ satisfies
\begin{equation}
\ell(Z^{(0)},Z^*) = o\left(p\Delta_{\min}^2\right), \label{eq:ini-cond-MRA}
\end{equation}
with probability at least $1-\eta$. Then, we have
$$\ell(Z^{(t)},Z^*)\leq p\exp\left(-(1+o(1))\frac{\Delta_{\min}^2}{8}\right) + \frac{1}{2}\ell(Z^{(t-1)},Z^*)\quad\text{ for all }t\geq 1,$$
with probability at least $1-\eta-\exp(-\Delta_{\min})-e^{-d}$.
\end{thm}

By the inequality $\frac{1}{p}\sum_{j=1}^p\indc{Z_j\neq Z_j^*}\leq \frac{\ell(Z,Z^*)}{p\Delta_{\min}^2}$, we immediately obtain the following corollary for the Hamming loss.

\begin{corollary}\label{cor:final-MRA}
Assume (\ref{eq:SNR-MRA}) holds. Suppose $Z^{(0)}$ satisfies (\ref{eq:ini-cond-MRA}) with probability at least $1-\eta$. Then, we have
\begin{equation}
\frac{1}{p}\sum_{j=1}^p\indc{Z_j^{(t)}\neq Z_j^*} \leq \exp\left(-(1+o(1))\frac{\Delta_{\min}^2}{8}\right) + 2^{-t}\quad\text{ for all }t\geq 1,\label{eq:path-of-pain}
\end{equation}
with probability at least $1-\eta-\exp(-\Delta_{\min})-e^{-d}$.
\end{corollary}

According to our lower bound result given by Theorem \ref{thm:mra-lower}, the quantity $\exp\left(-(1+o(1))\frac{\Delta_{\min}^2}{8}\right)$ is the minimax rate. Moreover, since the loss function $\frac{1}{p}\sum_{j=1}^p\indc{Z_j\neq Z_j^*}$ takes value in the set $\{j/p:j\in[p]\cup\{0\}\}$, the term $2^{-t}$ in (\ref{eq:path-of-pain}) is negligible as long as $2^{-t}=o(p^{-1})$. We therefore can claim
$$\frac{1}{p}\sum_{j=1}^p\indc{Z_j^{(t)}\neq Z_j^*} \leq \exp\left(-(1+o(1))\frac{\Delta_{\min}^2}{8}\right) \quad\text{ for all }t\geq 3\log p.$$
In other words, the minimax rate is achieved after at most $\ceil{3\log p}$ iterations.

\subsection{Initialization}

To close this section, we discuss a simple initialization procedure that achieves the condition (\ref{eq:ini-cond-MRA}). The idea is to find the cyclic shifts for all $j\geq 2$ by using that of $j=1$ as a reference. To be specific, we define $Z_1^{(0)}=I_d$. For each $j\geq 2$, compute
\begin{equation}
Z_j^{(0)} = \argmin_{U\in\mathcal{C}_d}\|Y_1-Z_j^TY_j\|^2. \label{eq:very-naive-MRA}
\end{equation}
The estimator (\ref{eq:very-naive-MRA}) has also been discussed by \cite{bandeira2014multireference}. It is known that (\ref{eq:very-naive-MRA}) does not have optimal statistical error. However, the performance of (\ref{eq:very-naive-MRA}) is sufficient for the purpose of initializing the iterative algorithm (\ref{eq:white-palace-difficult}).

\begin{proposition}\label{prop:ini-MRA}
There exists some $C>0$, such that for any $\eta>0$,  we have
$$\min_{U\in\mathcal{C}_d}\sum_{j=1}^p\|(Z_j^{(0)}U-Z_j^*)\theta^*\|^2\leq C\frac{p\sqrt{d\log d}}{\eta},$$
with probability at least $1-\eta$.
\end{proposition}

Proposition \ref{prop:ini-MRA} shows that $Z^{(0)}$ achieves the rate $O(p\sqrt{d\log d})$ for estimating $Z^*$ up to a global cyclic shift $U\in\mathcal{C}_d$. Since $Y_j=Z_j^*\theta^*+\epsilon_j=Z_j^*U^TU\theta^*+\epsilon_j$, the ambiguity of this global cyclic shift cannot be avoided.

In order that the initialization condition (\ref{eq:ini-cond-MRA}) is satisfied, we shall consider the signal-to-noise ratio condition
\begin{equation}
\frac{\Delta_{\min}^2}{\sqrt{d\log d} + d/p}\rightarrow\infty. \label{eq:true-snr-cond-mra}
\end{equation}
Note that (\ref{eq:true-snr-cond-mra}) implies (\ref{eq:SNR-MRA}) and thus the algorithmic convergence holds.
Given the condition (\ref{eq:true-snr-cond-mra}), we can take $\eta=\frac{\sqrt{d\log d}}{\Delta_{\min}^2}$ in Proposition \ref{prop:ini-MRA}. Then, under (\ref{eq:true-snr-cond-mra}), we have
$$\min_{U\in\mathcal{C}_d}\sum_{j=1}^p\|(Z_j^{(0)}U-Z_j^*)\theta^*\|^2= o\left(p\Delta_{\min}^2\right),$$
with probability at least $1-\frac{\sqrt{d\log d}}{\Delta_{\min}^2}$. This means there exists some $U\in\mathcal{C}_d$, such that the initial estimator $\{Z_j^{(0)}U\}$ recovers $\{Z_j^*\}$ after a global shift with an error that satisfies the condition (\ref{eq:ini-cond-MRA}). Theorefore, Corollary \ref{cor:final-MRA} implies that $\min_{U\in\mathcal{C}_d}\frac{1}{p}\sum_{j=1}^p\indc{Z_j^{(t)}U\neq Z_j^*}$ converges to the minimax error with a linear rate under the signal-to-noise ratio condition (\ref{eq:true-snr-cond-mra}).

\section{Group Synchronization}\label{sec:group}

In this section, we study a general class of problems called group synchronization. Given a group $(\mathcal{G},\circ)$ and group elements $g_1,\cdots,g_p\in\mathcal{G}$, we observe noisy versions of $g_i\circ g_j^{-1}$, and the goal is to recover the group elements $g_1,\cdots,g_p$. It turns out our general framework is particularly suitable to solve group synchronization, at least for discrete groups. We will consider the following three representative examples:
\begin{enumerate}
\item $\mathbb{Z}_2$ synchronization. This is the simplest example of group synchronization, and it is closely related to the more general phase/angular synchronization problem \citep{bandeira2017tightness}. The group only consists of two elemnts $\{-1,1\}$, and the group operation is the ordinary product.
\item $\Zk$ synchronization. Also known as joint alignment from pairwise differences, $\Zk$ synchronization was first considered by \cite{chen2018projected}. The group consists of elements $\{0,1,2,\cdots,k-1\}$ with group operation $g\circ h=g+h(\text{mod }k)$.
\item Permutation synchronization. As one of the most popular methods for multiple image alignment, permutation synchronization was first proposed by \cite{pachauri2013solving}. In this example, the group contains all permutations of $[d]$, and the group operation is the composition of two permutations.
\end{enumerate}
Given its importance in applied mathematics and engineering, group synchronization has been extensively studied in the literature \citep{fei2020achieving,perry2016optimality,bandeira2017tightness,abbe2020entrywise,zhou2015multi,yan2015multi,chen2014near,ling2020near,ling2020solving,lerman2019robust,abbe2017group}. Most approaches in the literature are based on semi-definite programming (SDP) and other forms of convex relaxations. In terms of statistical guarantees, the literature is mainly focused on conditions of exact recovery. In fact, for the three examples that we list above, the minimax rates are unknown with $\mathbb{Z}_2$ synchronization being the only exception. In this section, we will show Algorithm \ref{alg:general} can be specialized to the three models and can achieve the minimax rate of each one. Even for $\mathbb{Z}_2$ synchronization, our result offers some new insight of the problem. The minimax rate of $\mathbb{Z}_2$ synchronization is achieved by an SDP procedure \citep{fei2020achieving} in the literature. In comparison, Algorithm \ref{alg:general} leads to a much simpler power method, which is easier to implement in practice.

There are several different options of noise models in the literature. The most standard and popular choice is $Y_{ij}=g_i\circ g_j^{-1}+\sigma W_{ij}$ with $W_{ij}$ being a Gaussian element. However, it is also natural to restrict the noisy observation of $g_i\circ g_j^{-1}$ to be an element of the group $\mathcal{G}$. One way to achieve this is the noise model \citep{singer2011angular},
\begin{equation}
Y_{ij} = \begin{cases}
g_i\circ g_j^{-1} \quad \text{with probability }q, \\
\text{Uniform}(\mathcal{G})\quad \text{with probability }1-q.
\end{cases}
\end{equation}
Another way is through projection \citep{abbe2017group}. Namely, $Y_{ij}=\mathcal{P}_{\mathcal{G}}(g_i\circ g_j^{-1}+\sigma W_{ij})$, where $\mathcal{P}_{\mathcal{G}}$ is a projection onto $\mathcal{G}$ with respect to the $\ell_2$ norm. In addition, one can also consider partial observations on a random graph \citep{chen2014near}. As argued in \cite{abbe2017group}, these noise models are all equivalent to
\begin{equation}
Y_{ij}=\lambda g_i\circ g_j^{-1} + W_{ij},\label{eq:g-sync-canonical}
\end{equation}
with some $\lambda\in\mathbb{R}$ depending on $\sigma$ or $q$ and some additive noise $W_{ij}$ that is sub-Gaussian\footnote{The equivalence between the projection noise and (\ref{eq:g-sync-canonical}) is only true for some special groups}. For simplicity, we thus consider the noise model (\ref{eq:g-sync-canonical}) with $W_{ij}$ being a standard Gaussian element. The results we obtain in this section can all be extended with a sub-Gaussian $W_{ij}$ to include more general noise settings.

\subsection{$\mathbb{Z}_2$ Synchronization} \label{sec:z2}

Consider the observations $Y_{ij}\sim\mathn(\lambda^*z_i^*z_j^*,1)$ independently for all $1\leq i<j\leq p$ with $z_i^*\in\{-1,1\}$ and $\lambda^*\in\mathbb{R}$. Using matrix notation, we can write $Y=\lambda^*z^*z^{*T}+W$, where $W$ is a symmetric matrix such that $W_{ij}=W_{ji}\sim\mathn(0,1)$ for all $1\leq i<j\leq p$ and $W_{ii}=0$ for all $i\in[p]$. This is the simplest group synchronization problem, and is closely related to the problem of angular synchronization \citep{bandeira2017tightness}. The minimax lower bound of this problem has been recently obtained by \cite{fei2020achieving}.

\begin{thm}[Fei and Chen \citep{fei2020achieving}]\label{thm:lower-z2}
If $p\lambda^{*2}\rightarrow\infty$, we have
$$\inf_{\wh{z}}\sup_{z^*}\mathbb{E}\left(\frac{1}{p}\sum_{j=1}^p\indc{\wh{z}_j\neq z_j^*}\wedge \frac{1}{p}\sum_{j=1}^p\indc{\wh{z}_j\neq -z_j^*}\right)\geq \exp\left(-\frac{(1+o(1))p\lambda^{*2}}{2}\right).$$
Otherwise if $p\lambda^{*2}=O(1)$, we then have
$$\inf_{\wh{z}}\sup_{z^*}\mathbb{E}\left(\frac{1}{p}\sum_{j=1}^p\indc{\wh{z}_j\neq z_j^*}\wedge \frac{1}{p}\sum_{j=1}^p\indc{\wh{z}_j\neq -z_j^*}\right)\geq c,$$
for some constant $c>0$.
\end{thm}

The result is coherent with the necessary condition of weak recovery (i.e. to find a $\wh{z}$ that is correlated with $z^*$) $p\lambda^{*2}\rightarrow\infty$ and strong/exact recover (i.e. to find a $\wh{z}$ that equals $z^*$ up to a sign) $p\lambda^{*2}>2\log p$ in the literature \citep{perry2016optimality,bandeira2017tightness}. It was proved by \cite{fei2020achieving} that the minimax rate can be achieved by a semi-definite programming (SDP). In this section, we show that a simpler power iteration method, a special case of our general iterative algorithm, also achieves this minimax rate.

As is discussed in Section \ref{sec:framework}, the $\mathbb{Z}_2$ synchronization model can be equivalently represented as $Y=z^*(\beta^*)^T+W$ with $\beta^*=\lambda^*z^*\in\mathcal{B}_{z^*}=\{\beta=\lambda z^*:\lambda\in\mathbb{R}\}$. The resulting iterative algorithm can be summarized as
\begin{equation}
z_j^{(t)}=\argmin_{a\in\{-1,1\}}\|Y_j-a\wh{\beta}(z^{(t-1)})\|^2, \label{eq:algo-comp-z2}
\end{equation}
where for any $z\in\{-1,1\}^p$, we use the notation
$$\wh{\beta}(z)=\frac{z^TYz}{p^2}z.$$
The iterative procedure (\ref{eq:algo-comp-z2}) is equivalent to the power method (\ref{eq:power-z2}). The power method enjoys good theoretical properties in the setting of angular synchronization \citep{zhong2018near}, but whether it can achieve the minimax rate of $\mathbb{Z}_2$ synchronization is unknown in the literature.

\subsubsection{Conditions}

To analyze the algorithmic convergence of (\ref{eq:algo-comp-z2}), we note that $\norm{\beta^*}^2 =\abs{\lambda^{*}}^2 p$. Then  $\mu_j(B^*,a)=\nu_j(B^*,a)=a\beta^*$, $\Delta_j(a,b)^2=(a-b)^2\|\beta^*\|^2$, $\Delta_{\min}^2 = \min_{a \neq b}\Delta_j(a,b)^2 = 4\|\beta^*\|^2$, and
\begin{align}\label{eq:z2_loss_identity}
\ell(z,z^*)=\sum_{j=1}^p(z_j-z_j^*)^2\|\beta^*\|^2=p|\lambda^*|^2\sum_{j=1}^p(z_j-z_j^*)^2,
\end{align}
under the current setting. 
The error terms that we need to control are given by the formulas (\ref{eq:aspid111})-(\ref{eq:aspid113}).
Here, the noise vector is given by $\epsilon_j=W_j$, the $j$th column of the error matrix $W$. The error terms are controlled by the following lemma.

\begin{lemma}\label{lem:error-z2}
For any $C'>0$, there exists a constant $C>0$ only depending on $C'$ such that
\begin{eqnarray}
\nonumber && \max_{\{z:\ell(z,z^*)\leq\tau\}}\sum_{j=1}^p  \max_{b\in\{-1,1\}\setminus\{z_j^*\}} \frac{F_j(z_j^*,b;z)^2 \norm{\mu_j(B^*,b)-\mu_j(B^*,z_j^*)}^2}{\Delta_j(z_j^*,b)^4 \ell(z,z^*)} \\
 \label{eq:error-z21} &\leq& C\left(\frac{1}{p\lambda^{*2}} + \frac{1}{p^2\lambda^{*4}}\right),\\
\nonumber && \max_{\{z:\ell(z,z^*)\leq\tau\}}\max_{T\subset[p]}\frac{\tau}{4\Delta_{\min}^2|T|+\tau}\sum_{j\in T}  \max_{b\in\{-1,1\}\setminus\{z_j^*\}} \frac{G_j(z_j^*,b;z)^2 \norm{\mu_j(B^*,b)-\mu_j(B^*,z_j^*)}^2}{\Delta_j(z_j^*,b)^4 \ell(z,z^*)}\\
\label{eq:error-z22} &\leq& C\left(\frac{\tau}{\lambda^{*2}p^2} + \frac{\tau}{\lambda^{*4}p^3}\right),
\end{eqnarray}
and
\begin{equation}
\max_{j\in[p]}\max_{b\neq z_j^*} \frac{\abs{H_j(z_j^*,b)}}{\Delta_j(z_j^*,b)^2}   \leq C\frac{1}{\sqrt{p\lambda^{*2}}}, \label{eq:error-z23} 
\end{equation}
with probability at least $1-e^{-C'p}$.
\end{lemma}

From the bounds (\ref{eq:error-z21})-(\ref{eq:error-z23}), we can see that a sufficient condition that Conditions A, B, and C hold is $p\lambda^{*2}\rightarrow\infty$ and $\tau=o(p^2\lambda^{*2})$. Note that $p\lambda^{*2}\rightarrow\infty$ is also the necessary condition for consistency according to Theorem \ref{thm:lower-z2}.

Next, we need to bound $\xi_{\rm ideal}(\delta)$ in Condition D. This is given by the following lemma.

\begin{lemma}\label{lem:ideal-z2}
Assume $p\lambda^{*2}\rightarrow\infty$. Then, for any sequence $\delta_p=o(1)$, we have
$$\xi_{\rm ideal}(\delta_p)\leq p\exp\left(-(1+o(1))\frac{p\lambda^{*2}}{2}\right),$$
with probability at least $1-\exp(-\sqrt{p\lambda^{*2}})$.
\end{lemma}

Again, the same condition $p\lambda^{*2}\rightarrow\infty$ is required for Lemma \ref{lem:ideal-z2}. Thus, under the condition $p\lambda^{*2}\rightarrow\infty$, Conditions A, B, C and D hold simultaneously.

\subsubsection{Convergence}

With the help of Lemma \ref{lem:error-z2} and Lemma \ref{lem:error-z2}, we can specialize Theorem \ref{thm:main} into the following result.

\begin{thm}\label{thm:z2-alg}
Assume $p\lambda^{*2}\rightarrow\infty$. Suppose $z^{(0)}$ satsifies
\begin{equation}
\ell(z^{(0)},z^*)=o(p^2\lambda^{*2}), \label{eq:init-cond-z2}
\end{equation}
with probability at least $1-\eta$. Then, we have
$$\ell(z^{(t)},z^*)\leq p\exp\left(-(1+o(1))\frac{p\lambda^{*2}}{2}\right) + \frac{1}{2}\ell(z^{(t-1)},z^*)\quad\text{ for all }t\geq 1,$$
with probability at least $1-\eta-\exp(-\sqrt{p\lambda^{*2}})-e^{-p}$.
\end{thm}

By the inequality, $\frac{1}{p}\sum_{j=1}^p\indc{z_j\neq z_j^*}\leq \frac{\ell(z,z^*)}{p^2\lambda^{*2}}$, we immediately obtain the following corollary for the Hamming loss.

\begin{corollary}\label{cor:alg-z2}
Assume $p\lambda^{*2}\rightarrow\infty$. Suppose $z^{(0)}$ satisfies (\ref{eq:init-cond-z2}) with probability at least $1-\eta$. Then,
\begin{equation}
\frac{1}{p}\sum_{j=1}^p\indc{z_j^{(t)}\neq z_j^*} \leq \exp\left(-(1+o(1))\frac{p\lambda^{*2}}{2}\right) + 2^{-t}\quad\text{ for all }t\geq 1,
\end{equation}
with probability at least $1-\eta-\exp(-\sqrt{p\lambda^{*2}})-e^{-p}$.
\end{corollary}
By the property of the Hamming loss, the algorithmic error $2^{-t}$ is negligible after $\ceil{3\log p}$ iterations, and we have
$$\frac{1}{p}\sum_{j=1}^p\indc{z_j^{(t)}\neq z_j^*} \leq \exp\left(-(1+o(1))\frac{p\lambda^{*2}}{2}\right)\quad\text{ for all }t\geq 3\log p.$$
Thus, the minimax rate is achieved given that the initialization condition (\ref{eq:init-cond-z2}) is satisfied.

\subsubsection{Initialization}

Observe that the expectation of $Y$ has a rank one structure. Thus, a natural initialization procedure is to extract the information of $z^*$ by computing the leading eigenvector of $Y$. Let $\wh{u}=\argmax_{\|u\|=1}u^TYu$, and we define $z_j^{(0)}=\sgn(\wh{u}_j)$ for all $j\in[p]$. The behavior of $\wh{z}^{(0)}$ has been analyzed by \cite{abbe2020entrywise} when $p\lambda^{*2}>2\log p$. Without this condition, we show $\wh{z}^{(0)}$ can be used as a good initialization for the iterative algorithm.

\begin{proposition}\label{prop:ini-z2}
For any $C'>0$, there exists a constant $C>0$ only depending on $C'$ such that
$$\ell(z^{(0)},z^*)\wedge\ell(z^{(0)},-z^*)\leq Cp,$$
with probability at least $1-e^{-C'p}$.
\end{proposition}

Proposition \ref{prop:ini-z2} shows that $z^{(0)}$ achieves the rate $O(p)$ for estimating $z^*$ up to a change of sign. Interestingly, the initialization condition (\ref{eq:init-cond-z2}) is satisfied as long as $p\lambda^{*2}\rightarrow\infty$, the same condition that we use for both the lower bound (Theorem \ref{thm:lower-z2}) and the algorithmic convergence (\ref{thm:z2-alg}). Therefore, by Corollary \ref{cor:alg-z2}, $\frac{1}{p}\sum_{j=1}^p\indc{{z}_j^{(t)}\neq z_j^*}\wedge \frac{1}{p}\sum_{j=1}^p\indc{{z}^{(t)}_j\neq -z_j^*}$ converges to the minimax rate with a linear rate under the condition $p\lambda^{*2}\rightarrow\infty$.

\subsection{$\mathbb{Z}/k\mathbb{Z}$ Synchronization} \label{sec:zk}

$\Zk$ synchronization is known as the famous problem of joint alignment from pairwise differences \citep{chen2018projected}. It has some interesting applications in the haplotype assembly problem \citep{si2014haplotype} and and Dixon imaging \citep{zhang2017resolving}. The group $\mathbb{Z}/k\mathbb{Z}$ consists of elements $\{0,1,2,\cdots,k-1\}$. For any $g,h\in\Zk$, the group operation is defined by $g\circ h=(g+h)\;\text{mod }k$. Since it implies the inverse of an element $g$ is $g^{-1} = (k-g) \; \text{mode }k$, which is equal to $k-g$ if $g\neq 0$ and $0$ if $g=0$. Then we have $g \circ h^{-1} = (g-h)\text{ mod }k$.

In $\mathbb{Z}/k\mathbb{Z}$ synchronization, we have independent observations $Y_{ij}\sim\mathn(\lambda^*(z_i^*\circ z_j^{*^{-1}}),1)$ for all $1\leq i\neq j\leq p$ with $\lambda^*\in\mathbb{R}$ and $z_1^*,\cdots,z_p^*\in\Zk$, and our goal is to recover $z_1^*,\cdots,z_p^*$. When $k=2$, the problem is very similar but not identical to $\mathbb{Z}_2$ synchronization. Though $\mathbb{Z}/2\mathbb{Z}=\{0,1\}$ is isomorphic to $\mathbb{Z}_2=\{-1,1\}$, the current model assumption $Y_{ij}\sim\mathn(\lambda^*(z_i^*\circ z_j^{*^{-1}}),1)$ does not lead to the symmetric property $\mathbb{E}Y_{ij}=\mathbb{E}Y_{ji}$ in $\mathbb{Z}_2$ synchronization. This is because $(z_i^*-z_j^*)\text{ mod }k\neq (z_j^*-z_i^*)\text{ mod }k$. As a consequence, it is required that both $Y_{ij}$ and $Y_{ji}$ are observed in the more general setting of $\mathbb{Z}/k\mathbb{Z}$ synchronization.

Minimax rate for estimating $z_1^*,\cdots,z_p^*$ is unknown in the literature.
We first present a minimax lower bound for the problem. 

\begin{thm}\label{thm:lower-zk}
If $p\lambda^{*2}\rightarrow\infty$, we have
$$\inf_{\wh{z}}\sup_{z^*}\mathbb{E}\min_{a\in \Zk}\left(\frac{1}{p}\sum_{j=1}^p\indc{\wh{z}_j\neq z_j^*\circ a^{-1}}\right)\geq \exp\left(-\frac{(1+o(1))p\lambda^{*2}}{4}\right).$$
Otherwise if $p\lambda^{*2}=O(1)$, we then have
$$\inf_{\wh{z}}\sup_{z^*}\mathbb{E}\min_{a\in \Zk}\left(\frac{1}{p}\sum_{j=1}^p\indc{\wh{z}_j\neq z_j^*\circ a^{-1}}\right)\geq c,$$
for some constant $c>0$.
\end{thm}

Next, we will show the above minimax lower bound can be achieved via an computationally efficient algorithm under the signal-to-noise ratio condition $\frac{p\lambda^{*2}}{k^7\log k}\rightarrow\infty$.
We first explain how to view the $\mathbb{Z}/k\mathbb{Z}$ synchronization problem from the perspective of our general framework. Though it is not obvious because of the nonlinear group operation in the model, we can still adopt a similar idea used in $\mathbb{Z}_2$ synchronization and view one of the $z_j^*$ in $\lambda^*(z_i^*\circ z_j^{*^{-1}})$ and $\lambda^*$ jointly as the continuous model parameter. Write $Y=\mathbb{E}Y+W\in\mathbb{R}^{p\times p}$ with $\mathbb{E}Y_{ij}=\lambda^*(z_i^*\circ z_j^{*-1})$ and $W_{ij}\sim \mathn(0,1)$ for all $i\neq j$ and $\mathbb{E}Y_{ii}=W_{ii}=0$ for all $i$. Define $\beta^*=(\lambda^*,z^*)$. It is clear that $\beta^*\in\mathcal{B}_{z^*}=\left\{(\lambda,z^*):\lambda\in\mathbb{R}\right\}=\mathbb{R}\times \{z^*\}$. We can then write $\mathbb{E}Y=\X_{z^*}(\beta^*)$, where the operator $\X_{z^*}(\cdot)$ is determined by $[\X_{z^*}(\beta^*)]_{ij}=[\X_{z^*}((\lambda^*,z^*))]_{ij}=\lambda^*(z_i^*\circ z_j^{*-1})$.

To derive an iterative algorithm, we define the local statistic $T_j=Y_j\in\mathbb{R}^p$, the $j$th column of the matrix $Y$. We then have $\mathbb{E}T_j=\mu_j(\beta^*,z_j^*)=\nu_j(\beta^*,z_j^*)\in\mathbb{R}^p$. For any $i\in[p]$ and any $a\in\Zk$, the $i$th entry of $\mu_j(\beta^*,a)=\nu_j(\beta^*,a)$ is
$$[\mu_j(\beta^*,a)]_i=[\nu_j(\beta^*,a)]_i=\lambda^* (z_i^*\circ a^{-1}).$$
Then, we can specialize Algorithm \ref{alg:general} into the following iterative procedure,
\begin{eqnarray}
\label{eq:zk-alg1}\beta^{(t)} &=& \argmin_{\beta=(\lambda,z^{(t-1)}):\lambda\in\mathbb{R}}\fnorm{Y-\X_{z^{(t-1)}}(\beta)}^2, \\
\label{eq:zk-alg2} z_j^{(t)} &=& \argmin_{a\in\Zk}\|Y_j-\mu_j(\beta^{(t)},a)\|^2.
\end{eqnarray}
The second step (\ref{eq:zk-alg2}) is straightforward, since one can easily evaluate $\|Y_j-\mu_j(\beta^{(t-1)},a)\|^2$ for all $a\in\Zk$. The first step (\ref{eq:zk-alg1}) looks complicated, but thanks to the constraint $\beta=(\lambda,z^{(t-1)})$, it is a simple one-dimensional optimization problem. In fact, there is a closed form solution to (\ref{eq:zk-alg1}), and we can use that to write down an equivalent form of the algorithm  (\ref{eq:zk-alg1})-(\ref{eq:zk-alg2}) as follows,
\begin{eqnarray}
\label{eq:zk-s-alg1} \lambda^{(t)} &=& \frac{\sum_{1\leq i\neq j\leq p}Y_{ij}\left(z_i^{(t-1)}\circ (z_j^{(t-1)})^{-1}\right)}{\sum_{1\leq i\neq j\leq p}\left(z_i^{(t-1)}\circ (z_j^{(t-1)})^{-1}\right)^2}, \\
\label{eq:zk-s-alg2} z_j^{(t)} &=& \argmin_{a\in\Zk}\sum_{i=1}^p\left(Y_{ij}-\lambda^{(t)}(z_i^{(t-1)}\circ a^{-1})\right)^2.
\end{eqnarray}
This example really demonstrates the flexibility of our general framework, especially the possible dependence of the space $\mathcal{B}_z$ on the discrete structure $z$. It allows us to separate $z_i$ from $z_j$ in $z_i\circ z_j^{-1}$, and derive the simple iterative algorithm (\ref{eq:zk-s-alg1})-(\ref{eq:zk-s-alg2}).

\subsubsection{Conditions}

To analyze the algorithmic convergence of (\ref{eq:zk-s-alg1})-(\ref{eq:zk-s-alg2}), we note that
$$\Delta_j(a,b)^2=\lambda^{*2}\sum_{i=1}^p\left((z_i^*\circ a^{-1})-(z_i^*\circ b^{-1})\right)^2,$$
under the current setting. Thus, the loss function is
\begin{equation}
\ell(z,z^*)=\lambda^{*2}\sum_{j=1}^p\sum_{i=1}^p\left((z_i^*\circ z_j^{-1})-(z_i^*\circ z_j^{*-1})\right)^2.\label{eq:loss-zk-sync}
\end{equation}
One can also write down the formulas of $F_j(a,b;z)$, $G_j(a,b;z)$ and $H_j(a,b)$, which are included in Section \ref{sec:proof_zk} due to the limit of space. The error terms are controlled by the following lemma.

\begin{lemma}\label{lem:error-zk}
Assume $\frac{p\lambda^{*2}}{k^4}\rightarrow\infty$ and $\max_{a\in\Zk}\sum_{j=1}^p\indc{z_j^*=a}\leq (1-\alpha)p$ for some constant $\alpha>0$. Then, for any constant $C'>0$, there exists a constant $C>0$ only depending on $C'$ such that
\begin{eqnarray}
\nonumber && \max_{z:\ell(z,z^*)\leq \tau}\sum_{j=1}^p\max_{b\in\Zk\backslash\{z_j^*\}}\frac{F_j(z_j^*,b;z)^2 \norm{\mu_j(B^*,b) - \mu_j(B^*,z_j^*)}^2}{\Delta_j(z_j^*,b)^4\ell(z,z^*)}  \\
\label{eq:error-zk1} &\leq& C\frac{k^4}{p\lambda^{*2}}, \\
\nonumber && \max_{\{z:\ell(z,z^*)\leq\tau\}}\max_{T\subset[p]}\frac{\tau}{4\Delta_{\min}^2|T|+\tau}\sum_{j\in T}\max_{b\in\Zk\backslash\{z_j^*\}}\frac{G_j(z_j^*,b;z)^2 \norm{\mu_j(B^*,b) - \mu_j(B^*,z_j^*)}^2}{\Delta_j(z_j^*,b)^2\ell(z,z^*)}  \\
\label{eq:error-zk2} &\leq& C\frac{\tau k^6}{p^2\lambda^{*2}},
\end{eqnarray}
and
\begin{equation}
\max_{j\in[p]}\max_{a\in\Zk\backslash\{z_j^*\}}\frac{|H_j(z_j^*,a)|}{\Delta_j(z_j^*,a)^2} \leq C\left(\frac{k^2}{p\lambda^{*2}} + \sqrt{\frac{k^2}{p\lambda^{*2}}}\right), \label{eq:error-zk3}
\end{equation}
with probability at least $1-e^{-C'p}$.
\end{lemma}

From the bounds (\ref{eq:error-zk1})-(\ref{eq:error-zk3}), we can see that a sufficient condition under which Conditions A, B and C hold is $\tau=o\left(\frac{p^2\lambda^{*2}}{k^6}\right)$ and
\begin{equation}
\frac{p\lambda^{*2}}{k^4}\rightarrow\infty. \label{eq:snr-zk}
\end{equation}
The signal-to-noise ratio condition (\ref{eq:snr-zk}) extends the condition $p\lambda^{*2}\rightarrow\infty$ required by $\mathbb{Z}_2$ synchronization.

Next, we need to bound $\xi_{\rm ideal}$ in Condition D. This is given by the following lemma.
\begin{lemma}\label{lem:ideal-zk}
Assume $p\lambda^{*2}\rightarrow\infty$, $p/k^2\rightarrow\infty$  and $\max_{a\in\Zk}\sum_{j=1}^p\indc{z_j^*=a}\leq (1-\alpha)p$ for some constant $\alpha>0$. Then, for any sequence $\delta_p=o(1)$, we have
$$\xi_{\rm ideal}(\delta_p)\leq p\exp\left(-(1+o(1))\frac{p\lambda^{*2}}{8}\right),$$
with probability at least $1-\exp(-\sqrt{p\lambda^{*2}})$.
\end{lemma}

Thus, under the conditions $p/k^2\rightarrow\infty$ and (\ref{eq:snr-zk}), Conditions A, B, C and D hold simultaneously.

\subsubsection{Convergence}

With the help of Lemma \ref{lem:error-zk} and Lemma \ref{lem:ideal-zk}, we can specialize Theorem \ref{thm:main} into the following result.
\begin{thm}\label{thm:zk-con-alg}
Assume $p/k^2\rightarrow\infty$, $\frac{p\lambda^{*2}}{k^4}\rightarrow\infty$, and $\max_{a\in\Zk}\sum_{j=1}^p\indc{z_j^*=a}\leq (1-\alpha)p$ for some constant $\alpha>0$. Suppose $z^{(0)}$ satisfies
\begin{equation}
\ell(z^{(0)},z^*) = o\left(\frac{p^2\lambda^{*2}}{k^6}\right),\label{eq:ini-cond-zk}
\end{equation}
with probability at least $1-\eta$. Then, we have
$$\ell(z^{(t)},z^*)\leq p\exp\left(-(1+o(1))\frac{p\lambda^{*2}}{8}\right) + \frac{1}{2}\ell(z^{(t-1)},z^*)\quad\text{ for all }t\geq 1,$$
with probability at least $1-\eta-\exp(-\sqrt{p\lambda^{*2}})-e^{-p}$.
\end{thm}

According to the definition of the loss in (\ref{eq:loss-zk-sync}), we have the inequality
\begin{align}\label{eqn:zk_h_l_connection}
\ell(z,z^*)\geq p^2\lambda^{*2}\frac{1}{p}\sum_{j=1}^p\indc{z_j\neq z_j^*}.
\end{align}
This immediately implies the following corollary for the Hamming loss.
\begin{corollary}\label{cor:zk-con-alg}
Assume $p/k^2\rightarrow\infty$, $\frac{p\lambda^{*2}}{k^4}\rightarrow\infty$ and $\max_{a\in\Zk}\sum_{j=1}^p\indc{z_j^*=a}\leq (1-\alpha)p$ for some constant $\alpha>0$. Suppose $z^{(0)}$ satisfies (\ref{eq:ini-cond-zk}) with probability at least $1-\eta$. Then, we have
\begin{equation}
\frac{1}{p}\sum_{j=1}^p\indc{z_j^{(t)}\neq z_j^*} \leq \exp\left(-(1+o(1))\frac{p\lambda^{*2}}{8}\right) + 2^{-t}\quad\text{ for all }t\geq 1,
\end{equation}
with probability at least $1-\eta-\exp(-\sqrt{p\lambda^{*2}})-e^{-p}$.
\end{corollary}
By the property of the Hamming loss, the algorithmic error $2^{-t}$ is negligible after $\ceil{3\log p}$ iterations, and we have
$$\frac{1}{p}\sum_{j=1}^p\indc{z_j^{(t)}\neq z_j^*} \leq \exp\left(-(1+o(1))\frac{p\lambda^{*2}}{8}\right)\quad\text{ for all }t\geq 3\log p.$$
Thus, the minimax rate is achieved given that the initialization condition (\ref{eq:ini-cond-zk}) is satisfied.

\subsubsection{Initialization}

Unlike the $\mathbb{Z}_2$ synchronization setting, the vector $z^*\in\mathbb{R}^p$ does not correspond to any eigenvector of $\mathbb{E}Y$. However, we observe that the columns of $\mathbb{E}Y$ has only $k$ possibilities, and the $k$ different vectors that each column can take are well separated. We thus propose the following initialization algorithm based on a spectral clustering step.

\begin{enumerate}
\item Apply the spectral clustering algorithm (\ref{eq:low-rank-spectral})-(\ref{eq:k-means-relax-approx}) to $Y$ and obtain the column clustering label vector $\bar{z}\in\{0,1,.\cdots, k-1\}^p$.
\item For any $l\in\{1,2,\cdots, k-1\}$, compute $\bar{Y}_l=\frac{1}{|\bar{\mathcal{Z}}_{l0}|}\sum_{(i,j)\in\bar{\Z}_{l0}}Y_{ij}$, where $\bar{\Z}_{ab}=\{(i,j)\in[p]\times[p]:\bar{z}_i=a,\bar{z}_j=b\}$ for any $a,b\in[k]$.
\item Sort $|\bar{Y}_1|,\cdots,|\bar{Y}_{k-1}|$ into the order statistics $|\bar{Y}|_{(1)}\leq \cdots\leq |\bar{Y}|_{(k-1)}$. Let $\wh{\pi}$ be a permutation of the labels $l\in\{1,2,\cdots, k-1\}$ so that $|\bar{Y}|_{(\wh{\pi}(l))}=|\bar{Y}_l|$ for all $l\in\{1,2,\cdots,k-1\}$.
\item Output the estimator $z_j^{(0)}=\wh{\pi}(\bar{z}_j)$ for $j\in[p]$.
\end{enumerate}

Let us give some intuitions why the above algorithm works. By Proposition \ref{prop:ini-clust}, it is clear that there exists some label permutation $\pi$ such that $\pi(\bar{z}_j)$ recovers the underlying true label $z_j^*$. This is the purpose of Step 1. We then use Steps 2-4 to recover this unknown label permutation $\pi$. Under an appropriate signal-to-noise ratio condition, we can show that $\max_{l\in\{1,\cdots,k-1\}}|\bar{Y}_l-\lambda^*(\pi(l)\circ \pi(0)^{-1})|=o(|\lambda^*|)$. This error bound immediately implies
$$
\max_{l\in\{1,\cdots,k-1\}}||\bar{Y}_l|-|\lambda^*|(\pi(l)\circ \pi(0)^{-1})| =o(|\lambda^*|).
$$
By the fact that $\min_{a\neq b}||\lambda^*|(\pi(a)\circ \pi(0)^{-1})-|\lambda^*|(\pi(b)\circ \pi(0)^{-1})|\geq |\lambda^*|$, we can deduce the fact that the order of $\{|\bar{Y}_l|\}_{1\leq l\leq k-1}$ perfectly preserves that of $\{\pi(l)\circ \pi(0)^{-1}\}_{1\leq l\leq k-1}$. Therefore, the unknown label permutation $\pi$ can be recovered via sorting $\{|\bar{Y}_l|\}_{1\leq l\leq k-1}$.

\begin{proposition}\label{prop:ini}
Assume $\min_{a\in\Zk}\sum_{j=1}^p\indc{z_j^*=a}\geq\frac{\alpha p}{k}$ for some constant $\alpha>0$ and $\frac{p\lambda^{*2}}{(M+1)k^3}\rightarrow\infty$. For any constant $C'>0$, there exists a constant $C>0$ only depending on $\alpha$ and $C'$ such that
$$\min_{a\in\Zk}\ell(z^{(0)},z^*\circ a^{-1})\leq C(M+1)kp,$$
with probability at least $1-e^{-C'p}$. We have used the notation $z^*\circ a^{-1}$ for the vector $\{z_i^*\circ a^{-1}\}_{i\in[p]}$.
\end{proposition}

Proposition \ref{prop:ini} shows that $z^{(0)}$ achieves the rate $O((M+1)kp)$ for estimating $z^*$ up to a group multiplication. This uncertainty cannot be avoided since $z_i\circ z_j^{-1}=(z_i\circ a^{-1})\circ (z_j\circ a^{-1})^{-1}$. The factor $M$ in the bound comes from the computation of the $M$-approximation of the $k$-means objective, and one can take $M=O(\log k)$ when the $k$-means++ algorithm is used for the approximation.  In order that (\ref{eq:ini-cond-zk}) is satisfied, we thus require
\begin{equation}
\frac{p\lambda^{*2}}{(M+1)k^7}\rightarrow \infty,\label{eq:stronger-snr-zk}
\end{equation}
a condition that implies (\ref{eq:snr-zk}). Hence, according to Corollary \ref{cor:zk-con-alg}, the iterative algorithm initialized by spectral clustering converges to the minimax error with a linear rate under the condition (\ref{eq:stronger-snr-zk}).

\subsection{Permutation Synchronization} \label{sec:perm}

Permutation synchronization is a problem first proposed by \cite{pachauri2013solving} in computer vision as an approach to align multiple images from observations of pairwise similarities. Polynomial algorithms for this problem are mostly based on convex relaxation \citep{zhou2015multi,yan2015multi,chen2014near,ling2020near} with theoretical guarantees in the form of exact recovery \citep{chen2014near} and polynomial convergence rate in Frobenius norm \citep{ling2020near}. In this section, we will derive the minimax rate of the problem with exponential convergence and show the optimal rate can be achieved by Algorithm \ref{alg:general} under some signal-to-noise ratio condition.

Let $\Pi_d$ be the set of permutations on the set $[d]=\{1,\cdots, d\}$. The set of permutation matrices is defined by $\mathcal{P}_d=\{(e_{\pi(1)},\cdots, e_{\pi(d)}):\pi\in\Pi_d\}$, where $e_i\in\mathbb{R}^d$ is the $i$th canonical vector of $\mathbb{R}^d$. In permutation synchronization, we observe $Y_{ij}=\lambda^*Z_i^*Z_j^{*T}+W_{ij}\in\mathbb{R}^{d\times d}$ for $1\leq i<j\leq p$ with $\lambda^*\in\mathbb{R}$, $Z_1^*,\cdots,Z_p^*\in\mathcal{P}_d$, and $W_{ij}$'s are independent error matrices with i.i.d. entries following $\mathn(0,1)$.

We first present the minimax lower bound for the problem.

\begin{thm}\label{thm:lower-per}
If $p\lambda^{*2}\rightarrow\infty$, we have
$$\inf_{\wh{Z}}\sup_{Z^*}\mathbb{E}\min_{U\in\mathcal{P}_d}\frac{1}{p}\sum_{j=1}^p\indc{\wh{Z}_j\neq Z_j^*U^T}\geq \exp\left(-\frac{(1+o(1))p\lambda^{*2}}{2}\right).$$
Otherwise if $p\lambda^{*2}=O(1)$, we then have
$$\inf_{\wh{Z}}\sup_{Z^*}\mathbb{E}\min_{U\in\mathcal{P}_d}\frac{1}{p}\sum_{j=1}^p\indc{\wh{Z}_j\neq Z_j^*U^T}\geq c,$$
for some constant $c>0$.
\end{thm}

Computationally efficient algorithms to recover the permutation matrices $Z_1^*,\cdots,Z_p^*$ with minimax error are unknown in the literature. The problem is hard even when the dimension of the permutation $d$ is a constant. We will show our general iterative algorithm leads to a solution of this open problem whenever the signal-to-noise ratio conditoin $\frac{p\lambda^{*2}}{d^2}\rightarrow\infty$ is satisfied. We first put the problem into our general framework by organizing all the observations $\{Y_{ij}\}$ into a single matrix $Y=\lambda^*Z^*Z^{*T}+W\in\mathbb{R}^{pd\times pd}$. Here, $Z^{*T}=(Z_1^{*T},\cdots,Z_p^{*T})\in\mathbb{R}^{d\times pd}$ is a matrix by concatenating the $p$ permutation matrices together. For each $W_{ij}$, it can be viewed as the $(i,j)$th block of $W$. We have $W_{ij}=W_{ji}^T$ to be independent standard Gaussian matrices for all $1\leq i<j\leq p$ and $W_{ii}=0$ for all $i\in[p]$. We identify $z^*$ with $Z^*$ and define $\mathcal{B}_{Z^*}=\{\lambda Z^*:\lambda\in\mathbb{R}\}$ as the space of model parameter. Then, we can write $Y=Z^*(B^*)^T+W$ with $B^*\in\mathcal{B}_{Z^*}$. This is the same strategy that has been used for $\mathbb{Z}_2$ synchronization. To derive an iterative algorithm, let $T_j=Y_j$ and the definition of the matrix $Y_j\in\mathbb{R}^{pd\times d}$ is given by $Y_j^T=(Y_{1j}^T,\cdots,Y_{pj}^T)$. We then have $\nu_j(B^*,U)=\mu_j(B^*,U)=B^*U^T$ for any $U\in\mathcal{P}_d$. The iterative algorithm is
\begin{eqnarray}
\label{eq:alg-per1} B^{(t)} &=& \argmin_{B=\lambda Z^{(t-1)}:\lambda\in\mathbb{R}}\fnorm{Y-Z^{(t-1)}B^T}^2, \\
\label{eq:alg-per2} Z_j^{(t)} &=& \argmin_{U\in\mathcal{P}_d}\fnorm{Y_j-B^{(t)}U^T}^2.
\end{eqnarray}
The computation of (\ref{eq:alg-per1}) is a one-dimensional optimization problem, and its solution is given by $B^{(t)}=\wh{\lambda}(Z^{(t-1)})Z^{(t-1)}$, with $\wh{\lambda}(Z)=\frac{\iprod{Y}{ZZ^T}}{p^2d^2}$. For (\ref{eq:alg-per2}), we have the equivalent form
$$Z_j^{(t)}=\argmax_{U\in\mathcal{P}_d}\iprod{Y_j^TB^{(t)}}{U},$$
which can be solved by the Kuhn-Munkres algorithm \citep{kuhn1955hungarian} with $O(d^3)$ complexity.

\subsubsection{Conditions}

To analyze the algorithmic convergence of (\ref{eq:alg-per1})-(\ref{eq:alg-per2}), we note that
$$\Delta_j(U,V)^2=\fnorm{B^*(U-V)^T}^2=p\lambda^{*2}\fnorm{U-V}^2,$$
for any $U,V\in\mathcal{P}_d$.
Therefore, the natural loss function of the problem is
$$\ell(Z,Z^*)=p\lambda^{*2}\sum_{j=1}^p\fnorm{Z_j-Z_j^*}^2.$$
To write down the error terms, we introduce the notation $\wh{B}(Z)=\wh{\lambda}(Z)Z$.
The error terms are
\begin{eqnarray*}
F_j(U,V;Z) &=& \iprod{\epsilon_j}{(\wh{B}(Z^*)-\wh{B}(Z))(U-V)^T}, \\
G_j(U,V;Z) &=& \iprod{\wh{B}(Z^*)-\wh{B}(Z)}{B^*(I_d-U^TV)}, \\
H_j(U,V) &=& \iprod{B^*-\wh{B}(Z^*)}{B^*(I_d-U^TV)}.
\end{eqnarray*}
Here $\epsilon_j\in\mathbb{R}^{pd\times d}$ is an error matrix defined by $\epsilon_j^T=(W_{1j}^T,\cdots,W_{pj}^T)$. The error terms are controlled by the following lemma.

\begin{lemma}\label{lem:error-per}
For any $C'>0$, there exists a constant $C>0$ only depending on $C'$ such that
\begin{eqnarray}
\nonumber && \max_{\{Z:\ell(Z,Z^*)\leq\tau\}} \sum_{j=1}^p\max_{U\neq Z_j^*}\frac{F_j(Z_j^*,U;Z)^2\norm{\mu_j(B^*,U) - \mu_j(B^*,Z_j^*)}^2}{\Delta_j(Z_j^*,U)^4\ell(Z,Z^*)} \\
\label{eq:error-per1} &\leq& C\left(\frac{d}{p\lambda^{*2}} + \frac{1}{p^2\lambda^{*4}}\right), \\
\nonumber && \max_{\{Z:\ell(Z,Z^*)\leq\tau\}}\frac{\tau}{4\Delta_{\min}^2|T|+\tau}\sum_{j\in T}\max_{U\neq Z_j}\frac{G_j(Z_j^*,U;Z)^2\norm{\mu_j(B^*,U) - \mu_j(B^*,Z_j^*)}^2}{\Delta_j(Z_j^*,U)^4\ell(Z,Z^*)}  \\
\label{eq:error-per2} &\leq& C\frac{\tau\left(\lambda^{*2}+\frac{1}{pd}\right)}{(p\lambda^{*2})^2},
\end{eqnarray}
and
\begin{equation}
\max_{j\in[p]}\max_{U\neq Z_j^*} \frac{\abs{H_j(Z_j^*,U)}}{\Delta_j(Z_j^*,U)^2}\leq C\sqrt{\frac{d}{p\lambda^{*2}}},\label{eq:error-per3}
\end{equation}
with probability at least $1-e^{-C'pd}$.
\end{lemma}

From the bounds (\ref{eq:error-per1})-(\ref{eq:error-per3}), we can see that a sufficient condition under which Conditions A, B and C hold is $\tau=o(p^2\lambda^{*2})$ and $\frac{p\lambda^{*2}}{d}\rightarrow\infty$.

Next, we need to bound $\xi_{\rm ideal}$ in Condition D. This is given by the following lemma.
\begin{lemma}\label{lem:ideal-per}
Assume $\frac{p\lambda^{*2}}{d\log d}\rightarrow\infty$. Then, for any sequence $\delta_p=o(1)$, we have
\begin{eqnarray*}
\xi_{\rm ideal}(\delta_p)\leq p\exp\left(-\frac{1+o(1)}{2}p\lambda^{*2}\right),
\end{eqnarray*}
with probability at least $1-\exp(-\sqrt{p\lambda^{*2}})$.
\end{lemma}
Thus, under the condition $\frac{p\lambda^{*2}}{d\log d}\rightarrow\infty$, Conditions A, B, C and D hold simultaneously.

\subsubsection{Convergence}

With the help of Lemma \ref{lem:error-per} and Lemma \ref{lem:ideal-per}, we can specialize Theorem \ref{thm:main} into the following result.
\begin{thm}\label{thm:per-con-alg}
Assume $\frac{p\lambda^{*2}}{d\log d}\rightarrow\infty$. Suppose $Z^{(0)}$ satisfies
\begin{equation}
\ell(Z^{(0)},Z^*)=o(p^2\lambda^{*2}),\label{eq:ini-con-per}
\end{equation}
with probability at least $1-\eta$. Then, we have
$$\ell(Z^{(t)},Z^*)\leq p\exp\left(-(1+o(1))\frac{p\lambda^{*2}}{2}\right) + \frac{1}{2}\ell(Z^{(t-1)},Z^*)\quad\text{ for all }t\geq 1,$$
with probability at least $1-\eta-\exp(-\sqrt{p\lambda^{*2}})-e^{-pd}$.
\end{thm}

According to the definition of the loss $\ell(Z,Z^*)$, we have the inequality
$$\ell(Z,Z^*)\geq 4p^2\lambda^{*2}\frac{1}{p}\sum_{j=1}^p\indc{Z_j\neq Z_j^*}.$$
This immediately implies the following corollary for the Hamming loss.
\begin{corollary}\label{cor:per-con-alg}
Assume $\frac{p\lambda^{*2}}{d\log d}\rightarrow\infty$. Suppose $z^{(0)}$ satisfies (\ref{eq:ini-con-per}) with probability at least $1-\eta$. Then, we have
\begin{equation}
\frac{1}{p}\sum_{j=1}^p\indc{Z_j^{(t)}\neq Z_j^*} \leq \exp\left(-(1+o(1))\frac{p\lambda^{*2}}{2}\right) + 2^{-t}\quad\text{ for all }t\geq 1,
\end{equation}
with probability at least $1-\eta-\exp(-\sqrt{p\lambda^{*2}})-e^{-pd}$.
\end{corollary}
By the property of the Hamming loss, the algorithmic error $2^{-t}$ is negligible after $\ceil{3\log p}$ iterations, and we have
$$\frac{1}{p}\sum_{j=1}^p\indc{Z_j^{(t)}\neq Z_j^*} \leq \exp\left(-(1+o(1))\frac{p\lambda^{*2}}{2}\right)\quad\text{ for all }t\geq 3\log p.$$
Thus, the minimax rate is achieved given that the initialization condition (\ref{eq:ini-con-per}) is satisfied.

\subsubsection{Initialization}

Since $Y=\lambda^*Z^*Z^{*T}+W$, the matrix $\mathbb{E}Y\in\mathbb{R}^{pd\times pd}$ has a low rank structure, and thus we can recover the information of $Z^*$ via eigenvalue decomposition. Let $\wh{U}\in\mathbb{R}^{pd\times d}$ collect the leading eigenvectors of $Y$. We use the notation $\wh{U}_j\in\mathbb{R}^{d\times d}$ for the $j$th block of $\wh{U}$ so that $\wh{U}^T=(\wh{U}_1,\cdots,\wh{U}_p)$. For each $j\in[p]$, find $Z_j^{(0)}=\argmin_{V\in\mathcal{P}_d}\fnorm{\sqrt{p}\wh{U}_j-V}^2$. Again, this optimization is equivalent to
$$Z_j^{(0)}=\argmax_{V\in\mathcal{P}_d}\iprod{\wh{U}_j}{V},$$
and can be solved by the Kuhn-Munkres algorithm \citep{kuhn1955hungarian} with $O(d^3)$ complexity.
The statistical guarantee of $Z^{(0)}$ is given by the following proposition.

\begin{proposition}\label{prop:ini-per}
For any constant $C'>0$, there exists a constant $C>0$ only depending on $C$ such that
$$\min_{U\in\mathcal{P}_d}\ell(Z^{(0)},Z^*U^T)\leq Cpd^2,$$
with probability at least $1-e^{-C'pd}$.
\end{proposition}

Proposition \ref{prop:ini-per} show that $Z^{(0)}$ achieves the rate $O(pd^2)$ for estimating $Z^*$ up to a global permutation. This uncertainty cannot be avoided since $Z_iZ_j^T=Z_iU^TUZ_j^T$. In order that the initialization condition (\ref{eq:ini-con-per}) is satisfied, we thus require $\frac{p\lambda^{*2}}{d^2}\rightarrow\infty$, which implies the condition for algorithmic convergence $\frac{p\lambda^{*2}}{d\log d}\rightarrow\infty$. By Corollary \ref{cor:per-con-alg}, we can thus conclude that the iterative algorithm initialized by the spectral method converges to the minimax error with a linear rate under the condition $\frac{p\lambda^{*2}}{d^2}\rightarrow\infty$.

\section{Discussion}\label{sec:disc}

In this paper, we show that a number of different discrete structure recovery problems can be unified into a single framework, and a general iterative algorithm is proved to achieve the optimal statistical error rate for each problem. In addition to all the examples covered by the paper, we expect our framework will lead to applications in many other statistical models, thanks to the flexibility of the general structured linear models (\ref{eq:SLM}). However, it is also worthwhile to note some important limitations of our proposed framework in the end of the paper. We compile four major points, listed below.

\begin{enumerate}

\item \emph{Parameter estimation.} While our iterative algorithm is designed for estimating the discrete structure $z^*$, it also outputs an estimator for the continuous model parameter. A natural question is whether this estimator enjoys any statistical optimality for estimating $B^*$. The answer to this question is a very clear no. As a concrete example, it is known that the $k$-means algorithm leads to  sub-optimal parameter estimation due to the bias resulted from the clustering error \citep{lu2016statistical}. Instead, for parameter estimation, one should use the EM algorithm for the Gaussian mixture model \citep{dwivedi2018singularity,wu2019randomly}. This phenomenon is also known as the Neyman-Scott paradox in linear mixed models. Basically, for optimal global parameter estimation, one should integrate out the local latent variables instead of optimizing over them. In general, the iterative algorithm proposed in the paper is only statistically optimal for recovering the discrete structure $z^*$.

\item \emph{Models with extremely weak SNR.} The examples that we analyze in the paper all exhibit reasonable signal-to-noise ratio behaviors. There are other examples that fit perfectly in the framework of structured linear models, but do not lead to convergent iterative algorithms. Consider a shuffled regression problem with independent observations $x_i\sim\mathn(0,I_d)$ and $y_i|x_{z_i^*}\sim\mathn(x_{z_i^*}^T\beta^*, 1)$ for $i=1,\cdots,n$. In this model, the vector $z^*$ is a label permutation that links the covariates to the response. To recover $z^*$, the quantity $\|\beta^*\|^2$ serves as the signal strength of the problem. It was proved by \cite{pananjady2017linear} that the recovery of $z^*$ is only possible under the signal-to-noise ratio condition $\|\beta^*\|^2\geq e^{O(n)}$. The problem has such a weak signal strength, and our analysis of the error terms $F_j(a,b;z)$, $G_j(a,b;z)$ and $H_j(a,b)$ simply breaks down. In fact, statistical recovery of $z^*$ in polynomial time under the condition $\|\beta^*\|^2\geq e^{O(n)}$ still remains an open problem in the literature.

\item \emph{Models with non-local discrete structure.} Consider a changepoint problem $Y\sim\mathn(\theta^*,I_n)$ with $\theta^*$ having a piecewise constant structure. In other words, there exists $z_2^*,\cdots,z_n^*\in\{0,1\}$ such that $\theta_i^*=\theta_{i-1}^*$ for all $z_i^*=1$. This model can be easily written as a structured linear model, but there is no suitable iterative algorithm to recover the changepoint structure encoded in $z_2^*,\cdots,z_n^*$. The reason is the lack of local statistic $T_i$ that is sufficient for $z_i^*$. The changepoint structure is a discrete but non-local structure. The amount of information for $z_i^*$ is dependent on the locations of the previous and the next changepoints, which are further determined by other $z_i$'s. Our iterative algorithm is not suitable for recovering such non-local discrete structures. An appropriate algorithm for this problem is dynamic programming \citep{friedrich2008complexity}.

\item \emph{Link function.} Though the Gaussianity assumption in the paper can all be relaxed to sub-Gaussian errors, we still require both the structured linear model (\ref{eq:SLM}) and the local statistic (\ref{eq:local-T-E}) to have additive error structures. This requirement is coherent with the iterative algorithm, since both iteration steps (\ref{eqn:local_B}) and (\ref{eqn:local_testing}) are least-squares optimization. Problems such as variable selection in generalized linear models and ranking in Bradley-Terry-Luce model \citep{bradley1952rank,luce2012individual} involve link functions that are not identity. This poses new challenges in the error analysis in addition to the modification of the iterative algorithm.

\end{enumerate}

While some of the listed points may be addressed by appropriate extensions of our framework, others may require a fundamentally different approach to the problem.
We hope the above discussion not only highlights the critical features of our proposed framework, but also leads to potential future research projects in discrete structure recovery.

\section{Proofs}\label{sec:pf}

\subsection{Proof of Theorem \ref{thm:main}}

Suppose $\ell(z^{(t-1)},z^*)\leq\tau$, and we will show $\ell(z^{(t)},z^*)\leq 2\xi_{\rm ideal}(\delta) + \frac{1}{2}\ell(z^{(t-1)},z^*)$. By the definition of the loss (\ref{eq:loss}), we have
\begin{eqnarray}
\nonumber \ell(z^{(t)},z^*) &=& \sum_{j=1}^p\|\mu_j(B^*,z_j^{(t)})-\mu_j(B^*,z_j^*)\|^2 \\
\label{eq:pf-er-Sc} &=& \sum_{j=1}^p\sum_{b\in[k]\backslash\{z_j^*\}}\|\mu_j(B^*,b)-\mu_j(B^*,z_j^*)\|^2\indc{z_j^{(t)}=b}.
\end{eqnarray}
To bound (\ref{eq:pf-er-Sc}), we have
\begin{eqnarray}
\nonumber && \indc{z_j^{(t)}=b} \\
\label{eq:basic1} &\leq& \indc{\|T_j-\nu_j(\wh{B}(z^{(t-1)}),b)\|^2 \leq \|T_j-\nu_j(\wh{B}(z^{(t-1)}),z_j^*)\|^2} \\
\label{eq:basic2} &=& \indc{\iprod{\epsilon_j}{\nu_j(\wh{B}(z^*),z_j^*)-\nu_j(\wh{B}(z^*),b)} \leq -\frac{1}{2}\Delta_j(z_j^*,b)^2 + F_j(z_j^*,b;z^{(t-1)})+G_j(z_j^*,b;z^{(t-1)}) + H_j(z_j^*,b)} \\
\label{eq:basic3} &\leq& \indc{\iprod{\epsilon_j}{\nu_j(\wh{B}(z^*),z_j^*)-\nu_j(\wh{B}(z^*),b)} \leq -\frac{1-\delta}{2}\Delta_j(z_j^*,b)^2} \\
\nonumber && + \indc{\frac{\delta}{2}\Delta_j(z_j^*,b)^2 \leq  F_j(z_j^*,b;z^{(t-1)})+G_j(z_j^*,b;z^{(t-1)}) + H_j(z_j^*,b)}  \\
\label{eq:basic4} &\leq& \indc{\iprod{\epsilon_j}{\nu_j(\wh{B}(z^*),z_j^*)-\nu_j(\wh{B}(z^*),b)} \leq -\frac{1-\delta}{2}\Delta_j(z_j^*,b)^2}  \\
\nonumber && + \indc{\frac{\delta}{4}\Delta_j(z_j^*,b)^2 \leq   F_j(z_j^*,b;z^{(t-1)})+G_j(z_j^*,b;z^{(t-1)})} \\
\label{eq:basic5} &\leq& \indc{\iprod{\epsilon_j}{\nu_j(\wh{B}(z^*),z_j^*)-\nu_j(\wh{B}(z^*),b)} \leq -\frac{1-\delta}{2}\Delta_j(z_j^*,b)^2}  \\
\nonumber && + \frac{32F_j(z_j^*,b;z^{(t-1)})^2}{\delta^2\Delta_j(z_j^*,b)^4} + \frac{32G_j(z_j^*,b;z^{(t-1)})^2}{\delta^2\Delta_j(z_j^*,b)^4}.
\end{eqnarray}
The inequality (\ref{eq:basic1}) is due to the definition that $z_j^{(t)}=\argmin_{a\in[k]}\|T_j-\nu_j(\wh{B}(z^{(t-1)}),a)\|^2$. Then, the equality (\ref{eq:basic2}) uses the equivalence between (\ref{eq:error:a-b}) and (\ref{eq:decomposition}). The inequality (\ref{eq:basic3}) uses a union bound, and (\ref{eq:basic4}) applies Condition C. Finally, (\ref{eq:basic5}) follows Markov's inequality.

Apply the bound (\ref{eq:basic5}) to (\ref{eq:pf-er-Sc}), and then $\ell(z^{(t)},z^*)$ can be bounded by
\begin{eqnarray}
\nonumber && \sum_{j=1}^p\sum_{b\in[k]\backslash\{z_j^*\}}\|\mu_j(B^*,b)-\mu_j(B^*,z_j^*)\|^2\indc{\iprod{\epsilon_j}{\nu_j(\wh{B}(z^*),z_j^*)-\nu_j(\wh{B}(z^*),b)} \leq -\frac{1-\delta}{2}\Delta_j(z_j^*,b)^2} \\
\nonumber && + \sum_{j=1}^p\sum_{b\in[k]\backslash\{z_j^*\}}\|\mu_j(B^*,b)-\mu_j(B^*,z_j^*)\|^2\indc{z_j^{(t)}=b}\frac{32F_j(z_j^*,b;z^{(t-1)})^2}{\delta^2\Delta_j(z_j^*,b)^4} \\
\nonumber && + \sum_{j=1}^p\sum_{b\in[k]\backslash\{z_j^*\}}\|\mu_j(B^*,b)-\mu_j(B^*,z_j^*)\|^2\indc{z_j^{(t)}=b}\frac{32G_j(z_j^*,b;z^{(t-1)})^2}{\delta^2\Delta_j(z_j^*,b)^4} \\
\nonumber &\leq& \xi_{\rm ideal}(\delta)  + \sum_{j=1}^p\max_{b\in[k]\backslash\{z_j^*\}}\|\mu_j(B^*,b)-\mu_j(B^*,z_j^*)\|^2\frac{32F_j(z_j^*,b;z^{(t-1)})^2}{\delta^2\Delta_j(z_j^*,b)^4} \\
&& + \sum_{j=1}^p\indc{z_j^{(t)}\neq z_j^*}\max_{b\in[k]\backslash\{z_j^*\}}\|\mu_j(B^*,b)-\mu_j(B^*,z_j^*)\|^2\frac{32G_j(z_j^*,b;z^{(t-1)})^2}{\delta^2\Delta_j(z_j^*,b)^4} \\
\label{eq:basic6} &\leq& \xi_{\rm ideal}(\delta) + \frac{1}{8}\ell(z^{(t-1)},z^*) + \frac{4\Delta_{\min}^2h(z^{(t)},z^*)+\tau}{8\tau}\ell(z^{(t-1)},z^*) \\
\label{eq:basic7} &\leq& \xi_{\rm ideal}(\delta) + \frac{1}{2}\Delta_{\min}^2h(z^{(t)},z^*) + \frac{1}{4}\ell(z^{(t-1)},z^*) \\
\label{eq:basic8} &\leq& \xi_{\rm ideal}(\delta) + \frac{1}{2}\ell(z^{(t)},z^*) + \frac{1}{4}\ell(z^{(t-1)},z^*),
\end{eqnarray}
where we have used Conditions A and B in (\ref{eq:basic6}). The inequality (\ref{eq:basic7}) uses the condition $\ell(z^{(t-1)},z^*)\leq\tau$, and (\ref{eq:basic8}) is by (\ref{eq:l-to-hamming}). To summarize, we have obtained
$$\ell(z^{(t)},z^*)\leq \xi_{\rm ideal}(\delta) + \frac{1}{2}\ell(z^{(t)},z^*) + \frac{1}{4}\ell(z^{(t-1)},z^*),$$
which can be rearranged into
$$\ell(z^{(t)},z^*)\leq 2\xi_{\rm ideal}(\delta) + \frac{1}{2}\ell(z^{(t-1)},z^*).$$

To prove the conclusion of Theorem \ref{thm:main}, we use a mathematical induction argument. First, Condition D asserts that $\ell(z^{(0)},z^*)\leq \tau$. This leads to $\ell(z^{(1)},z^*)\leq 2\xi_{\rm ideal}(\delta) + \frac{1}{2}\ell(z^{(0)},z^*)\leq \tau$, together with Condition C that $\xi_{\rm ideal}(\delta)\leq \frac{1}{4}\tau$. Suppose $\ell(z^{(t-1)},z^*)\leq\tau$, we then have $\ell(z^{(t)},z^*)\leq 2\xi_{\rm ideal}(\delta) + \frac{1}{2}\ell(z^{(t-1)},z^*)\leq \tau$. Hence, $\ell(z^{(t-1)},z^*)\leq\tau$ for all $t\geq 1$, which implies that $\ell(z^{(t)},z^*)\leq 2\xi_{\rm ideal}(\delta) + \frac{1}{2}\ell(z^{(t-1)},z^*)$ for all $t\geq 1$, and the proof is complete.

\subsection{Proofs in Section \ref{sec:clustering}}

In this section, we present the proofs of Lemma \ref{lem:kmeans-error}, Lemma \ref{lem:kmeans-ideal} and Proposition \ref{prop:ini-clust}. The conclusions of Theorem \ref{thm:main-kmeans} and Corollary \ref{cor:main-kmeans} are direct consequences of Theorem \ref{thm:main}, and thus their proofs are omitted. We first list some technical lemmas.  The following $\chi^2$ tail probability is Lemma 1 of \cite{laurent2000adaptive}.

\begin{lemma}\label{prop:chisq}
For any $x>0$, we have
\begin{eqnarray*}
\mathbb{P}\left(\chi_d^2\geq d+2\sqrt{dx}+2x\right) &\leq& e^{-x}, \\
\mathbb{P}\left(\chi_d^2\leq d-2\sqrt{dx}\right) &\leq& e^{-x}.
\end{eqnarray*}
\end{lemma}

\begin{lemma}
Consider i.i.d. random vectors $\epsilon_1,\dotsc,\epsilon_p\sim \mathn(0,I_d)$ and some $z^*\in[k]^p$ and $k\in[p]$. Then, for any constant $C'>0$, there exists some constant $C>0$ only depending on $C'$ such that
\begin{eqnarray}
\label{eq:esp-km1} \max_{a\in[k]}\left\|\frac{\sum_{j=1}^p\indc{z_j^*=a}\epsilon_j}{\sqrt{\sum_{j=1}^p\indc{z_j^*=a}}}\right\| &\leq& C\sqrt{d+\log p}, \\
\label{eq:esp-km2} \max_{T\subset[p]}\left\|\frac{1}{\sqrt{|T|}}\sum_{j\in T}\epsilon_j\right\| &\leq& C\sqrt{d+p}, \\
\label{eq:esp-km3} \max_{a\in[k]}\frac{1}{d+\sum_{j=1}^p\indc{z_j^*=a}}\left\|\sum_{j=1}^p\indc{z_j^*=a}\epsilon_j\epsilon_j^T\right\|&\leq& C,
\end{eqnarray}
with probability at least $1-p^{-C'}$. We have used the convention that $0/0=0$.
\end{lemma}
\begin{proof}
By 
Lemma \ref{prop:chisq}, we have $\mathbb{P}(\chi_d^2\geq d+2\sqrt{xd}+2x)\leq e^{-x}$. Then, a union bound argument leads to (\ref{eq:esp-km1}). The inequalities (\ref{eq:esp-km2}) and (\ref{eq:esp-km3}) are Lemmas A.1 and A.2 in \cite{lu2016statistical}. We need to slightly extend Lemma A.2 in \cite{lu2016statistical}, but this can be done by a standard union bound argument.
\end{proof}

With the two lemmas above, we are ready to state the proofs of Lemma \ref{lem:kmeans-error} and Lemma \ref{lem:kmeans-ideal}.
\begin{proof}[Proof of Lemma \ref{lem:kmeans-error}]
We write $\epsilon_j=Y_j-\theta_{z_j^*}$ and consider the event that the three inequalities (\ref{eq:esp-km1})-(\ref{eq:esp-km3}) hold.
For any $z\in[k]^p$ such that $\ell(z,z^*)\leq\tau\leq \frac{\Delta_{\min}^2\alpha p}{2k}$, we have
\begin{eqnarray*}
\sum_{j=1}^p\indc{z_j=a} &\geq& \sum_{j=1}^p\indc{z_j^*=a} - \sum_{j=1}^p\indc{z_j\neq z_j^*} \\
&\geq& \sum_{j=1}^p\indc{z_j^*=a} - \frac{\ell(z,z^*)}{\Delta_{\min}^2} \\
&\geq& \frac{\alpha p}{k} - \frac{\alpha p}{2k} \\
&=& \frac{\alpha p}{2k},
\end{eqnarray*}
which implies
\begin{equation}
\min_{a\in[k]}\sum_{j=1}^p\indc{z_j=a} \geq \frac{\alpha p}{2k}. \label{eq:small-size-z}
\end{equation}
We then introduce more notation. We write $\theta_a(z)=\mathbb{E}\wh{\theta}_a(z)$ and
$$\bar{\epsilon}_a(z)=\frac{\sum_{j=1}^p\indc{z_j=a}\epsilon_j}{\sum_{j=1}^p\indc{z_j=a}}.$$
We first derive bounds for $\max_{a\in[k]}\|\wh{\theta}_a(z^*)-\theta_a^*\|$, $\max_{a\in[k]}\|\theta_a(z)-\theta_a(z^*)\|$ and $\max_{a\in[k]}\|\bar{\epsilon}_a(z)-\bar{\epsilon}_a(z^*)\|$. By (\ref{eq:esp-km1}) and (\ref{eq:small-size-z}), we have
\begin{eqnarray}
\nonumber \max_{a\in[k]}\|\wh{\theta}_a(z^*)-\theta_a^*\| &=& \max_{a\in[k]}\left\|\frac{\sum_{j=1}^p\indc{z_j^*=a}\epsilon_j}{\sum_{j=1}^p\indc{z_j^*=a}}\right\| \\
\nonumber &\leq& \sqrt{\frac{k}{\alpha p}}\max_{a\in[k]}\left\|\frac{\sum_{j=1}^p\indc{z_j^*=a}\epsilon_j}{\sqrt{\sum_{j=1}^p\indc{z_j^*=a}}}\right\| \\
\label{eq:useful-1} &\lesssim& \sqrt{\frac{k(d+\log p)}{p}}.
\end{eqnarray}
By (\ref{eq:small-size-z}), we have
\begin{eqnarray}
\nonumber \max_{a\in[k]}\|\theta_a(z)-\theta_a(z^*)\| &=& \left\|\frac{1}{\sum_{j=1}^p\indc{z_j=a}}\sum_{j=1}^p\sum_{b\in[k]\backslash\{a\}}\indc{z_j=a,z_j^*=b}(\theta_b^*-\theta_a^*)\right\| \\
\nonumber &\leq& \frac{2k}{\alpha p}\sum_{j=1}^p\sum_{b\in[k]\backslash\{a\}}\|\theta_b^*-\theta_a^*\|\indc{z_j=a,z_j^*=b} \\
\label{eq:useful-2} &\leq& \frac{2k}{\alpha p\Delta_{\min}}\ell(z,z^*).
\end{eqnarray}
By (\ref{eq:small-size-z}), we have
\begin{eqnarray*}
&& \max_{a\in[k]}\|\bar{\epsilon}_a(z)-\bar{\epsilon}_a(z^*)\| \\
&=& \max_{a\in[k]}\left\|\frac{\sum_{j=1}^p\indc{z_j=a}\epsilon_j}{\sum_{j=1}^p\indc{z_j=a}}-\frac{\sum_{j=1}^p\indc{z_j^*=a}\epsilon_j}{\sum_{j=1}^p\indc{z_j^*=a}}\right\| \\
&\leq& \max_{a\in[k]}\left\|\frac{\sum_{j=1}^p\indc{z_j=a}\epsilon_j}{\sum_{j=1}^p\indc{z_j=a}}-\frac{\sum_{j=1}^p\indc{z_j^*=a}\epsilon_j}{\sum_{j=1}^p\indc{z_j=a}}\right\| + \max_{a\in[k]}\left\|\frac{\sum_{j=1}^p\indc{z_j^*=a}\epsilon_j}{\sum_{j=1}^p\indc{z_j=a}}-\frac{\sum_{j=1}^p\indc{z_j^*=a}\epsilon_j}{\sum_{j=1}^p\indc{z_j^*=a}}\right\| \\
&\leq& \frac{2k}{\alpha p}\max_{a\in[k]}\left\|\sum_{j=1}^p(\indc{z_j=a}-\indc{z_j^*=a})\epsilon_j\right\| \\
&& + \frac{2k}{\alpha p}\sqrt{\frac{k}{\alpha p}}\left|\sum_{j=1}^p\indc{z_j=a}-\sum_{j=1}^p\indc{z_j^*=a}\right|\max_{a\in[k]}\left\|\frac{\sum_{j=1}^p\indc{z_j^*=a}\epsilon_j}{\sqrt{\sum_{j=1}^p\indc{z_j^*=a}}}\right\|,
\end{eqnarray*}
where the first term in the above bound can be bounded by
\begin{eqnarray*}
&& \frac{2k}{\alpha p}\max_{a\in[k]}\left\|\sum_{j=1}^p\indc{z_j=a,z_j^*\neq a}\epsilon_j\right\| + \frac{2k}{\alpha p}\max_{a\in[k]}\left\|\sum_{j=1}^p\indc{z_j^*=a,z_j\neq a}\epsilon_j\right\| \\
&\lesssim& \frac{k\sqrt{d+p}}{p}\sqrt{\frac{\ell(z,z^*)}{\Delta^2_{\min}}},
\end{eqnarray*}
because of the facts that $\max_{a\in[k]}\sum_{j=1}^p\indc{z_j=a,z_j^*\neq a}\leq \frac{\ell(z,z^*)}{\Delta^2_{\min}}$, $\max_{a\in[k]}\sum_{j=1}^p\indc{z_j^*=a,z_j\neq a}\leq \frac{\ell(z,z^*)}{\Delta^2_{\min}}$, and the inequality (\ref{eq:esp-km2}), and the second term can be bounded by
\begin{eqnarray*}
&& \frac{2k}{\alpha p}\sqrt{\frac{k}{\alpha p}}\max_{a\in[k]}\left\|\frac{\sum_{j=1}^p\indc{z_j^*=a}\epsilon_j}{\sqrt{\sum_{j=1}^p\indc{z_j^*=a}}}\right\|\left(\max_{a\in[k]}\sum_{j=1}^p\indc{z_j=a,z_j^*\neq a}+\max_{a\in[k]}\sum_{j=1}^p\indc{z_j^*=a,z_j\neq a}\right) \\
&\leq& \frac{4k}{\alpha p}\sqrt{\frac{k}{\alpha p}}\frac{\ell(z,z^*)}{\Delta^2_{\min}}\max_{a\in[k]}\left\|\frac{\sum_{j=1}^p\indc{z_j^*=a}\epsilon_j}{\sqrt{\sum_{j=1}^p\indc{z_j^*=a}}}\right\| \\
&\lesssim& \frac{k\sqrt{k}\ell(z,z^*)\sqrt{d+\log p}}{p\sqrt{p}\Delta^2_{\min}}.
\end{eqnarray*}
Under the condition that $\ell(z,z^*)\leq\tau\leq \frac{\Delta_{\min}^2\alpha p}{2k}$, we have
\begin{eqnarray}
\nonumber && \max_{a\in[k]}\|\bar{\epsilon}_a(z)-\bar{\epsilon}_a(z^*)\| \\
\nonumber &\lesssim& \frac{k\sqrt{d+p}}{p}\sqrt{\frac{\ell(z,z^*)}{\Delta^2_{\min}}} + \frac{k\sqrt{k}\ell(z,z^*)\sqrt{d+\log p}}{p\sqrt{p}\Delta^2_{\min}} \\
\label{eq:useful-3} &\lesssim& \frac{k\sqrt{d+p}}{p}\sqrt{\frac{\ell(z,z^*)}{\Delta^2_{\min}}}.
\end{eqnarray}
Combining the two bounds (\ref{eq:useful-2}) and (\ref{eq:useful-3}) and using triangle inequality, we also have
\begin{eqnarray}
\nonumber && \max_{a\in[k]}\|\wh{\theta}_a(z)-\wh{\theta}_a(z^*)\| \\
\nonumber &\leq& \max_{a\in[k]}\|\theta_a(z)-\theta_a(z^*)\| + \max_{a\in[k]}\|\bar{\epsilon}_a(z)-\bar{\epsilon}_a(z^*)\| \\
\label{eq:useful-4} &\lesssim& \frac{k}{p\Delta_{\min}}\ell(z,z^*) + \frac{k\sqrt{d+p}}{p\Delta_{\min}}\sqrt{\ell(z,z^*)}.
\end{eqnarray}

Now we proceed to prove (\ref{eq:kmeans-error1})-(\ref{eq:kmeans-error3}). For (\ref{eq:kmeans-error1}), we have
\begin{eqnarray*}
&& \sum_{j=1}^p\max_{b\in[k]\backslash\{z_j^*\}}\frac{F_j(z_j^*,b;z)^2\|\mu_j(B^*,b)-\mu_j(B^*,z_j^*)\|^2}{\Delta_j(z_j^*,b)^4\ell(z,z^*)} \\
&\leq& \sum_{j=1}^p\sum_{b=1}^k\frac{\left|\iprod{\epsilon_j}{\wh{\theta}_{z_j^*}(z^*)-\wh{\theta}_{z_j^*}(z)-\wh{\theta}_b(z^*)+\wh{\theta}_b(z)}\right|^2}{\|\theta_{z_j^*}^*-\theta_b^*\|^2\ell(z,z^*)} \\
&\leq& \sum_{b=1}^k\sum_{a\in[k]\backslash\{b\}}\sum_{j=1}^p\indc{z_j^*=a}\frac{\left|\iprod{\epsilon_j}{\wh{\theta}_{a}(z^*)-\wh{\theta}_{a}(z)-\wh{\theta}_b(z^*)+\wh{\theta}_b(z)}\right|^2}{\|\theta_{a}^*-\theta_b^*\|^2\ell(z,z^*)} \\
&\leq& \sum_{b=1}^k\sum_{a\in[k]\backslash\{b\}}\frac{\left\|\wh{\theta}_{a}(z^*)-\wh{\theta}_{a}(z)-\wh{\theta}_b(z^*)+\wh{\theta}_b(z)\right\|^2}{\|\theta_{a}^*-\theta_b^*\|^2\ell(z,z^*)}\left\|\sum_{j=1}^p\indc{z_j^*=a}\epsilon_j\epsilon_j^T\right\| \\
&\lesssim& \frac{k^2(kd/p+1)}{\Delta_{\min}^2}\left(1+\frac{k(d/p+1)}{\Delta_{\min}^2}\right)
\end{eqnarray*}
where we have used (\ref{eq:esp-km3}), (\ref{eq:useful-4}) and the condition that $\ell(z,z^*)\leq\tau\leq \frac{\Delta_{\min}^2\alpha p}{2k}$. Next, for (\ref{eq:kmeans-error2}), we have
\begin{eqnarray*}
|G_j(a,b;z)| &\leq&  \frac{1}{2}\|\wh{\theta}_a(z)-\wh{\theta}_a(z^*)\|^2 + \frac{1}{2}\|\wh{\theta}_b(z)-\wh{\theta}_b(z^*)\|^2 \\
&& + \|\wh{\theta}_a(z^*)-\theta_a^*\|\|\wh{\theta}_a(z)-\wh{\theta}_a(z^*)\| + \|\wh{\theta}_a(z^*)-\theta_b^*\|\|\wh{\theta}_b(z)-\wh{\theta}_b(z^*)\| \\
&\leq& \max_{a\in[k]}\|\wh{\theta}_a(z)-\wh{\theta}_a(z^*)\|^2 + 2\left(\max_{a\in[k]}\|\wh{\theta}_a(z^*)-\theta_a^*\|\right)\left(\max_{a\in[k]}\|\wh{\theta}_a(z)-\wh{\theta}_a(z^*)\|\right) \\
&& + \|\theta_a^*-\theta_b^*\|\left(\max_{a\in[k]}\|\wh{\theta}_a(z)-\wh{\theta}_a(z^*)\|\right).
\end{eqnarray*}
This implies for any subset $T\subset[p]$, we have
\begin{eqnarray*}
&& \frac{\tau}{4\Delta_{\min}^2|T|+\tau}\sum_{j\in T}\max_{b\in[k]\backslash\{z_j^*\}}\frac{G_j(z_j^*,b;z)^2\|\mu_j(B^*,b)-\mu_j(B^*,z_j^*)\|^2}{\Delta_j(z_j^*,b)^4\ell(z,z^*)} \\
&\leq& \frac{\tau}{4\Delta_{\min}^2|T|}\sum_{j\in T}\frac{3\max_{a\in[k]}\|\wh{\theta}_a(z)-\wh{\theta}_a(z^*)\|^4}{\Delta_{\min}^2\ell(z,z^*)} \\
&& + \frac{\tau}{4\Delta_{\min}^2|T|}\sum_{j\in T}\frac{12\left(\max_{a\in[k]}\|\wh{\theta}_a(z^*)-\theta_a^*\|^2\right)\left(\max_{a\in[k]}\|\wh{\theta}_a(z)-\wh{\theta}_a(z^*)\|^2\right)}{\Delta_{\min}^2\ell(z,z^*)} \\
&& + \frac{\tau}{4\Delta_{\min}^2|T|}\sum_{j\in T}\frac{3\max_{a\in[k]}\|\wh{\theta}_a(z)-\wh{\theta}_a(z^*)\|^2}{\ell(z,z^*)} \\
&=& \frac{3\tau\max_{a\in[k]}\|\wh{\theta}_a(z)-\wh{\theta}_a(z^*)\|^4}{4\Delta_{\min}^4\ell(z,z^*)} + \frac{3\tau\max_{a\in[k]}\|\wh{\theta}_a(z)-\wh{\theta}_a(z^*)\|^2}{4\Delta_{\min}^2\ell(z,z^*)} \\
&& + \frac{3\tau\left(\max_{a\in[k]}\|\wh{\theta}_a(z^*)-\theta_a^*\|^2\right)\left(\max_{a\in[k]}\|\wh{\theta}_a(z)-\wh{\theta}_a(z^*)\|^2\right)}{\Delta_{\min}^4\ell(z,z^*)} \\
&\lesssim& \frac{k\tau}{p\Delta_{\min}^2} + \frac{k(d+p)}{p\Delta_{\min}^2} + \frac{k^2(d+p)^2}{p^2\Delta_{\min}^4},
\end{eqnarray*}
where we have used (\ref{eq:useful-1}), (\ref{eq:useful-4}), and the condition that $\ell(z,z^*)\leq\tau\leq \frac{\Delta_{\min}^2\alpha p}{2k}$.
Finally, for (\ref{eq:kmeans-error3}), the bound (\ref{eq:useful-1}) leads to
\begin{eqnarray*}
\frac{|H_j(a,b)|}{\Delta_j(a,b)^2} &\leq& \frac{\frac{1}{2}\|\wh{\theta}_a(z^*)-\theta_a^*\|^2+\frac{1}{2}\|\wh{\theta}_b(z^*)-\theta_b^*\|^2+\|\theta_a^*-\theta_b^*\|\|\wh{\theta}_b(z^*)-\theta_b^*\|}{\|\theta_a^*-\theta_b^*\|^2} \\
&\lesssim& \frac{k(d+\log p)}{p\Delta_{\min}^2} + \sqrt{\frac{k(d+\log p)}{p\Delta_{\min}^2} }.
\end{eqnarray*}
By taking maximum, we have obtained (\ref{eq:kmeans-error1})-(\ref{eq:kmeans-error3}).
The proof is complete.
\end{proof}

\begin{proof}[Proof of Lemma \ref{lem:kmeans-ideal}]
Note that
\begin{eqnarray*}
&& \mathbb{P}\left(\iprod{\epsilon_j}{\wh{\theta}_a(z^*)-\wh{\theta}_b(z^*)}\leq -\frac{1-\delta}{2}\|\theta_a^*-\theta_b^*\|^2\right) \\
&\leq& \mathbb{P}\left(\iprod{\epsilon_j}{\theta_a^*-\theta_b^*}\leq -\frac{1-\delta-\bar{\delta}}{2}\|\theta_a^*-\theta_b^*\|^2\right) \\
&& + \mathbb{P}\left(\iprod{\epsilon_j}{\wh{\theta}_a(z^*)-\theta_a^*}\leq -\frac{\bar{\delta}}{4}\|\theta_a^*-\theta_b^*\|^2\right) \\
&& + \mathbb{P}\left(-\iprod{\epsilon_j}{\wh{\theta}_b(z^*)-\theta_b^*}\leq -\frac{\bar{\delta}}{4}\|\theta_a^*-\theta_b^*\|^2\right),
\end{eqnarray*}
where $\bar{\delta}=\bar{\delta}_p$ is some sequence to be chosen later,
and we need to bound the three terms on the right hand side of the above inequality respectively. For the first term, a standard Gaussian tail bound gives
$$\mathbb{P}\left(\iprod{\epsilon_j}{\theta_a^*-\theta_b^*}\leq -\frac{1-\delta-\bar{\delta}}{2}\|\theta_a^*-\theta_b^*\|^2\right)\leq \exp\left(-\frac{(1-\delta-\bar{\delta})^2}{8}\|\theta_a^*-\theta_b^*\|^2\right).$$
To bound the second term, we note that
$$\iprod{\epsilon_j}{\wh{\theta}_a(z^*)-\theta_a^*} = \frac{\indc{z_j^*=a}\|\epsilon_j\|^2}{\sum_{l=1}^p\indc{z_l^*=a}} + \frac{\sum_{l\in[p]\backslash\{j\}}\indc{z_l^*=a}\epsilon_j^T\epsilon_l}{\sum_{l=1}^p\indc{z_l^*=a}}\geq \frac{\sum_{l\in[p]\backslash\{j\}}\indc{z_l^*=a}\epsilon_j^T\epsilon_l}{\sum_{l=1}^p\indc{z_l^*=a}}.$$
This implies
\begin{eqnarray*}
\nonumber && \mathbb{P}\left(\iprod{\epsilon_j}{\wh{\theta}_a(z^*)-\theta_a^*}\leq -\frac{\bar{\delta}}{4}\|\theta_a^*-\theta_b^*\|^2\right) \\
&\leq& \mathbb{P}\left(\frac{\sum_{l\in[p]\backslash\{j\}}\indc{z_l^*=a}\epsilon_j^T\epsilon_l}{\sum_{l=1}^p\indc{z_l^*=a}}\leq-\frac{\bar{\delta}}{4}\|\theta_a^*-\theta_b^*\|^2\right) \\
&\leq& \mathbb{P}\left(\frac{\sum_{l\in[p]\backslash\{j\}}\indc{z_l^*=a}\epsilon_j^T\epsilon_l}{\sum_{l=1}^p\indc{z_l^*=a}}\leq-\frac{\bar{\delta}}{4}\|\theta_a^*-\theta_b^*\|^2\Bigg|\|\epsilon_l\|^2<d+2\sqrt{xd}+2x\right) \\
&& + \mathbb{P}\left(\|\epsilon_l\|^2>d+2\sqrt{xd}+2x\right) \\
&\leq& \mathbb{E}\left(\exp\left(-\frac{\bar{\delta}^2\|\theta_a^*-\theta_b^*\|^4\sum_{l=1}^p\indc{z_l^*=a}}{32\|\epsilon_l\|^2}\right)\Bigg|\|\epsilon_l\|^2<d+2\sqrt{xd}+2x\right) \\
&& + \mathbb{P}\left(\|\epsilon_l\|^2>d+2\sqrt{xd}+2x\right) \\
&\leq& \exp\left(-\frac{\bar{\delta}^2\|\theta_a^*-\theta_b^*\|^4\alpha p}{32k\left(d+2\sqrt{xd}+2x\right)}\right) + \exp(-x).
\end{eqnarray*}
Choosing $x=\bar{\delta}\|\theta_a^*-\theta_b^*\|^2\sqrt{\alpha p/k}$, we have
\begin{eqnarray}
\nonumber && \mathbb{P}\left(\iprod{\epsilon_j}{\wh{\theta}_a(z^*)-\theta_a^*}\leq -\frac{\bar{\delta}}{4}\|\theta_a^*-\theta_b^*\|^2\right) \\
\label{eq:felix} &\leq& \exp\left(-C\frac{\bar{\delta}^2\|\theta_a^*-\theta_b^*\|^4p}{kd}\right) + \exp\left(-C\frac{\bar{\delta}\|\theta_a^*-\theta_b^*\|^2\sqrt{p}}{\sqrt{k}}\right).
\end{eqnarray}
To bound the third term, we note that
$$-\iprod{\epsilon_j}{\wh{\theta}_b(z^*)-\theta_b^*} = -\frac{\indc{z_j^*=b}\|\epsilon_j\|^2}{\sum_{l=1}^p\indc{z_l^*=b}} - \frac{\sum_{l\in[p]\backslash\{j\}}\indc{z_l^*=b}\epsilon_j^T\epsilon_l}{\sum_{l=1}^p\indc{z_l^*=b}},$$
and we then have
\begin{eqnarray*}
\nonumber && \mathbb{P}\left(-\iprod{\epsilon_j}{\wh{\theta}_b(z^*)-\theta_b^*}\leq -\frac{\bar{\delta}}{4}\|\theta_a^*-\theta_b^*\|^2\right) \\
&\leq& \mathbb{P}\left(-\frac{\indc{z_j^*=b}\|\epsilon_j\|^2}{\sum_{l=1}^p\indc{z_l^*=b}} \leq -\frac{\bar{\delta}}{8}\|\theta_a^*-\theta_b^*\|^2\right) \\
&& + \mathbb{P}\left(- \frac{\sum_{l\in[p]\backslash\{j\}}\indc{z_l^*=b}\epsilon_j^T\epsilon_l}{\sum_{l=1}^p\indc{z_l^*=b}} \leq -\frac{\bar{\delta}}{8}\|\theta_a^*-\theta_b^*\|^2\right).
\end{eqnarray*}
The second term on the right hand side of the above inequality can be bounded in the same way as (\ref{eq:felix}). For the first term, we have
\begin{eqnarray*}
&& \mathbb{P}\left(-\frac{\indc{z_j^*=b}\|\epsilon_j\|^2}{\sum_{l=1}^p\indc{z_l^*=b}} \leq -\frac{\bar{\delta}}{8}\|\theta_a^*-\theta_b^*\|^2\right) \\
&\leq& \mathbb{P}\left(\|\epsilon_j\|^2 > \frac{\bar{\delta}}{8}\|\theta_a^*-\theta_b^*\|^2\frac{\alpha p}{k}\right) \\
&\leq& \exp\left(-C\bar{\delta}\|\theta_a^*-\theta_b^*\|^2\frac{p}{k}\right),
\end{eqnarray*}
under the condition $\frac{\Delta_{\min}^2}{\log k+ kd/p}\rightarrow\infty$. Combining the bounds above, we have
\begin{eqnarray*}
&& \mathbb{P}\left(\iprod{\epsilon_j}{\wh{\theta}_a(z^*)-\wh{\theta}_b(z^*)}\leq -\frac{1-\delta}{2}\|\theta_a^*-\theta_b^*\|^2\right) \\
&\leq& \exp\left(-\frac{(1-\delta-\bar{\delta})^2}{8}\|\theta_a^*-\theta_b^*\|^2\right) + \exp\left(-C\bar{\delta}\|\theta_a^*-\theta_b^*\|^2\frac{p}{k}\right) \\
&& + 2\exp\left(-C\frac{\bar{\delta}^2\|\theta_a^*-\theta_b^*\|^4p}{kd}\right) + 2\exp\left(-C\frac{\bar{\delta}\|\theta_a^*-\theta_b^*\|^2\sqrt{p}}{\sqrt{k}}\right) \\
&\leq& 6\exp\left(-\frac{(1-\delta-\bar{\delta})^2}{8}\|\theta_a^*-\theta_b^*\|^2\right),
\end{eqnarray*}
where the last inequality above is obtained under the condition that $\frac{\Delta_{\min}^2}{\log k+ kd/p}\rightarrow\infty$ and $p/k\rightarrow\infty$, so that we can choose some $\bar{\delta}=\bar{\delta}_p=o(1)$ that is slowly diverging to zero.

Now we are ready to bound $\xi_{\rm ideal}(\delta)$. We first bound its expectation. We have
\begin{eqnarray*}
\mathbb{E}\xi_{\rm ideal}(\delta) &=& \sum_{j=1}^p\sum_{b\in[k]\backslash\{z_j^*\}}\|\theta_b^*-\theta_{z_j^*}^*\|^2\mathbb{P}\left(\iprod{\epsilon_j}{\wh{\theta}_{z_j^*}(z^*)-\wh{\theta}_b(z^*)}\leq -\frac{1-\delta}{2}\|\theta_{z_j^*}^*-\theta_b^*\|^2\right) \\
&\leq& 6\sum_{j=1}^p\sum_{b\in[k]\backslash\{z_j^*\}}\|\theta_b^*-\theta_{z_j^*}^*\|^2\exp\left(-\frac{(1-\delta-\bar{\delta})^2}{8}\|\theta_{z_j^*}^*-\theta_b^*\|^2\right).
\end{eqnarray*}
With $\delta=\delta_p=o(1)$, we then have
$$\mathbb{E}\xi_{\rm ideal}(\delta_p)\leq \sum_{j=1}^p\sum_{b\in[k]\backslash\{z_j^*\}}\exp\left(-(1+o(1))\frac{\|\theta_{z_j^*}^*-\theta_b^*\|^2}{8}\right)\leq p\exp\left(-(1+o(1))\frac{\Delta_{\min}^2}{8}\right),$$
under the condition that $\frac{\Delta_{\min}^2}{\log k+ kd/p}\rightarrow\infty$. Finally, by Markov's inequality, we have
$$\mathbb{P}\left(\xi_{\rm ideal}(\delta_p) > \mathbb{E}\xi_{\rm ideal}(\delta_p)\exp\left(\Delta_{\min}\right)\right) \leq \exp\left(-\Delta_{\min}\right).$$
In other words, with probability at least $1-\exp\left(-\Delta_{\min}\right)$, we have
$$\xi_{\rm ideal}(\delta_p)\leq \mathbb{E}\xi_{\rm ideal}(\delta_p)\exp\left(\Delta_{\min}\right).$$
By the fact that $\Delta_{\min}\rightarrow\infty$, we have
$$\mathbb{E}\xi_{\rm ideal}(\delta_p)\exp\left(\Delta_{\min}\right)\leq p\exp\left(-(1+o(1))\frac{\Delta_{\min}^2}{8}\right),$$
and thus the proof is complete.
\end{proof}

Finally, we prove Proposition \ref{prop:ini-clust}.
\begin{proof}[Proof of Proposition \ref{prop:ini-clust}] We divide the proof into three steps.

~\\
\emph{Step 1.}
Define $\wh{P} = \wh U\wh U^T Y \in\mathr^{d\times p}$ with $\wh P_j$ being the $j$th column of $\wh P$. Since $\wh P_j = \wh U \wh\mu_j$ for all $j\in[p]$, we  have $\|\wh P_j - \wh P_{j'}\| = \|\wh\mu_j - \wh\mu_{j'}\|$ for all $j,j'\in[p]$. This implies
\begin{align*}
\min_{\substack{\theta_1,\dotsc,\theta_k\in\mathbb{R}^k\\ z\in[k]^p}}\sum_{j=1}^p\|\wh{P}_j-\theta_{z_j}\|^2 = \min_{\substack{\beta_1,\dotsc,\beta_k\in\mathbb{R}^k\\ z\in[k]^p}}\sum_{j=1}^p\|\wh{\mu}_j-\beta_{z_j}\|^2.
\end{align*}
Similarly, define $\theta^{(0)}_a = \wh U \beta^{(0)}_a$ for all $a\in[k]$, we have 
\begin{align*}
\sum_{j=1}^p\|\wh{P}_j-\theta^{(0)}_{z_j^{(0)}}\|^2 = \sum_{j=1}^p\| \wh U \wh\mu_j-\wh U \beta^{(0)}_{z_j^{(0)}}\|^2  = \sum_{j=1}^p\|\wh{\mu}_j-\beta^{(0)}_{z_j^{(0)}}\|^2.
\end{align*}
Thus, (\ref{eq:k-means-relax-approx}) leads to
\begin{align}\label{eq:k-means-relax-approx-new}
\sum_{j=1}^p\|\wh{P}_j-\theta^{(0)}_{z_j^{(0)}}\|^2 \leq M\min_{\substack{\theta_1,\dotsc,\theta_k\in\mathbb{R}^k\\ z\in[k]^p}}\sum_{j=1}^p\|\wh{P}_j-\theta_{z_j}\|^2.
\end{align}
That is, any $z^{(0)}\in[k]^p$ that satisfies (\ref{eq:k-means-relax-approx}) with some $\beta^{(0)}_1,\ldots \beta^{(0)}_k$ also satisfies (\ref{eq:k-means-relax-approx-new}) with some $\theta^{(0)}_1,\ldots \theta^{(0)}_k$.

~\\
\emph{Step 2.}
It is sufficient to study any  $\theta_1^{(0)},\dotsc,\theta_k^{(0)}\in\mathbb{R}^d$ and $z^{(0)}\in[k]^p$ that satisfies  (\ref{eq:k-means-relax-approx-new}). 
Let us define $P^*=\mathbb{E}Y$, and we have $P _j^*=\theta^*_{z_j^*}$ according to the model assumption.
We first give an error bound for $\fnorm{\wh{P }-P ^*}^2$. Since $\wh P$ is the rank-$k$ approximation of $Y$, we have $\fnorm{Y-\wh{P }}^2\leq\fnorm{Y-P ^*}^2$, which implies that $\fnorm{\wh{P }-P ^*}^2\leq 4\max_{\{A\in\mathbb{R}^{d\times p}:\fnorm{A}\leq 1,\rank(A)\leq 2k\}}|\iprod{A}{Y-P ^*}|^2$. Use a standard random matrix theory result \citep{vershynin2010introduction}, we have $\opnorm{Y-P ^*}^2\lesssim p+d$ with probability at least $1-e^{-C'(d+p)}$. For any $A$ such that $\fnorm{A}\leq 1$ and $\rank(A)\leq 2k$, its singular value decomposition can be written as $A=\sum_{l=1}^{2k}d_l u_lv_l^T$, where $\sum_{l=1}^{2k}d_l^2\leq 1$. Thus, we have
$$|\iprod{A}{Y-P ^*}|^2 = \left|\sum_{l=1}^{2k}d_lu_l^T(Y-P ^*)v_l\right|^2\leq \sum_{l=1}^{2k}\left|u_l^T(Y-P ^*)v_l\right|^2\leq 2k\opnorm{Y-P ^*}^2\lesssim k(p+d).$$
Taking maximum over $A$, we have $\fnorm{\wh{P }-P ^*}^2\lesssim k(p+d)$ with probability at least $1-e^{-C'(d+p)}$.

By (\ref{eq:k-means-relax-approx-new}), we have
\begin{align*}
\sum_{j=1}^p\|\wh{P}_j-\theta^{(0)}_{z_j^{(0)}}\|^2 \leq M\fnorm{\wh{P }-P ^*}^2 \lesssim Mk(p+d),
\end{align*}
and as a consequence, 
\begin{align}\label{eqn:clustering_pre_1}
\sum_{j=1}^p\|\theta^*_{z^*_j}-\theta^{(0)}_{z_j^{(0)}}\|^2 \leq 2\sum_{j=1}^p \br{\|\wh{P}_j-\theta^{(0)}_{z_j^{(0)}}\|^2 + \|\theta^*_{z^*_j}-\wh{P}_j\|^2} \lesssim (M+1)k(p+d).
\end{align}
Define
\begin{align*}
S = \cbr{j\in[p] : \|\theta^*_{z^*_j}-\theta^{(0)}_{z_j^{(0)}}\| \geq \frac{\Delta_{\min}}{2}},
\end{align*}
and we have
\begin{align*}
\abs{S} \leq  \frac{\sum_{j=1}^p\|\theta^*_{z^*_j}-\theta^{(0)}_{z_j^{(0)}}\|^2}{\br{ \frac{\Delta_{\min}}{2}}^2} \lesssim \frac{(M+1)k(p+d)}{\Delta_{\min}^2}.
\end{align*}
We are now going to show that all the data points in $S^\complement$ are all correctly clustered. We define
$$
\mathcal{C}_a=\left\{j\in[p]:z^*_j=a,j\in S^\complement\right\},
$$
for all $a\in[k]$.
Under the assumption $\Delta_{\min}^2/((M+1)k^2(1+d/p))\rightarrow\infty$, we have 
\begin{align}\label{eqn:clustering_pre_4}
\abs{S} =o(p/k).
\end{align}
We have the following arguments:
\begin{itemize}
\item For each $a\in[k]$, $\mathcal{C}_a$ cannot be empty, as
\begin{align}\label{eqn:clustering_pre_2}
|\mathcal{C}_a|\geq |\{j\in[p]:z^*_j=a\}| - |S| \geq \frac{ |\{j\in[p]:z^*_j=a\}|}{2} \geq \frac{\alpha p}{2k}.
\end{align}
\item For each pair $a,b\in[k],a\neq b$, there cannot exist some $j\in \mathcal{C}_a,j'\in \mathcal{C}_b$ such that $z^{(0)}_j=z^{(0)}_{j'}$. Otherwise $ \theta^{(0)}_{ z^{(0)}_j} = \theta^{(0)}_{ z^{(0)}_{j'}}$ would imply
\begin{align*}
\norm{\theta^*_a - \theta^*_b} = \norm{\theta^*_{z^*_{j}} - \theta^*_{z^*_{j'}}} \leq \norm{\theta^*_{z^*_{j}} - \theta^{(0)}_{ z^{(0)}_j} } + \norm{\theta^{(0)}_{ z^{(0)}_j} -\theta^{(0)}_{ z^{(0)}_{j'}}} +  \norm{\theta^{(0)}_{ z^{(0)}_{j'}}- \theta^*_{z^*_{j'}}}< \Delta_{\min},
\end{align*}
contradicting the definition of $\Delta_{\min}$.
\end{itemize}
Since $z^{(0)}_j$ can only take values in $[k]$, we conclude that $\{z^{(0)}_j:j\in \mathcal{C}_a\}$ contains only one and different element for all $a\in[k]$. That is, there exists a permutation $\pi_0\in\Pi_k$, such that
\begin{align}\label{eqn:clustering_pre_3}
z^{(0)}_j = \pi_0(z_j^*),
\end{align}
for all $j\in S^\complement$.

~\\
\emph{Step 3.} The last step is to establish an upper bound for $\ell(\pi_0^{-1} \circ z^{(0)}, z^*) $. By (\ref{eqn:clustering_pre_1}), (\ref{eqn:clustering_pre_2}) and (\ref{eqn:clustering_pre_3}), we have
\begin{align*}
\norm{\theta^*_a - \theta^{(0)}_{\pi_0(a)}}^2 = \frac{\sum_{j\in \mathcal{C}_a} \norm{\theta^*_{z_j^*} - \theta^{(0)}_{z_j^{(0)}}}^2 }{\abs{\mathcal{C}_a}} \leq \frac{\sum_{j=1}^p\norm{\theta^*_{z^*_j}-\theta^{(0)}_{z_j^{(0)}}}^2 }{\abs{\mathcal{C}_a}} \lesssim (M+1)k^2\br{1+ \frac{d}{p}},
\end{align*}
for all $a\in[k]$. As a result, together with (\ref{eqn:clustering_pre_1}), (\ref{eqn:clustering_pre_4}) and (\ref{eqn:clustering_pre_3}), we have
\begin{align*}
\ell\br{\pi_0^{-1} \circ z^{(0)},z^*}  &= \sum_{j\in[p]} \norm{\theta^*_{z^*_j} - \theta^*_{\pi_0^{-1}\br{ z^{(0)}_j}}}^2  = \sum_{j\in[p]} \norm{\theta^*_{z^*_j} - \theta^*_{\pi_0^{-1}\br{ z^{(0)}_j}}}^2 \indc{z^*_j \neq \pi_0^{-1}\br{ z^{(0)}_j}}\\
&  \leq  2\sum_{j\in[p]}  \br{ \norm{\theta^*_{z^*_j} - \theta^{(0)}_{z^{(0)}_j}}^2 + \norm{\theta^{(0)}_{z^{(0)}_j}- \theta^*_{\pi_0^{-1}\br{z^{(0)}_j}}}^2} \indc{z^*_j \neq \pi_0^{-1}\br{ z^{(0)}_j}} \\
& \leq 2\sum_{j\in[p]}   \norm{\theta^*_{z^*_j} - \theta^{(0)}_{z^{(0)}_j}}^2 +  \max_{a\in[k]}\norm{  \theta^{(0)}_a - \theta^*_{\pi_0^{-1}(a)}}^2\sum_{j=1}^p \indc{z^*_j \neq \pi_0^{-1}\br{ z^{(0)}_j}}\\
& \leq 2\sum_{j\in[p]}   \norm{\theta^*_{z^*_j} - \theta^{(0)}_{z^{(0)}_j}}^2 + \abs{S} \max_{a\in[k]}\norm{  \theta^{(0)}_{\pi_0(a)} - \theta^*_{a}}^2 \\
& \lesssim (M+1)k\br{p+ d}.
\end{align*}
The proof is complete.
\end{proof}

\subsection{Proofs in Section \ref{sec:rank-results}}\label{sec:proof_ranking}

This section collects the proofs of Lemma \ref{lem:rank-error}, Lemma \ref{lem:ranking-ideal}, and Proposition \ref{prop:initial-rank}. The conclusions of Theorem \ref{thm:rank-main} and Corollary \ref{cor:rank-main} are direct consequences of Theorem \ref{thm:main}, and thus we omit their proofs. We first need the following technical lemma.

\begin{lemma}
Consider i.i.d. random variables $w_{ij}\sim \mathcal{N}(0,1)$ for $1\leq i\neq j\leq p$. Then, for any constant $C'>0$, there exists some constant $C>0$ only depending on $C'$ such that
\begin{eqnarray}
\label{eq:w-ranking2} \max_{a\in\mathbb{R}^p}\left|\frac{\sum_{1\leq i\neq j\leq p}(a_i-a_j)w_{ij}}{\sqrt{\sum_{1\leq i\neq j\leq p}(a_i-a_j)^2}}\right| &\leq& C\sqrt{p}, \\
\label{eq:w-ranking3} \sum_{j=1}^p\left(\frac{1}{\sqrt{2(p-1)}}\sum_{i\in[p]\backslash\{j\}}(w_{ji}-w_{ij})\right)^2 &\leq& Cp, \\
\label{eq:w-ranking4} \max_{j\in[p]}\left|\frac{1}{\sqrt{2(p-1)}}\sum_{i\in[p]\backslash\{j\}}(w_{ji}-w_{ij})\right| &\leq& C\sqrt{\log p},
\end{eqnarray}
with probability at least $1-(C'p)^{-1}$. We have used the convention that $0/0=0$.
\end{lemma}
\begin{proof}
To bound the first inequality, we define
$$\mathcal{A}=\left\{A=\{a_{ij}\}_{(i,j)\in[p]^2}:a_{ij}=a_i-a_j\text{ for some }a\in\mathbb{R}^p, \fnorm{A}\leq 1\right\},$$
and
$$\mathcal{B}=\left\{B=\{b_{ij}\}_{(i,j)\in[p]^2}: \rank(B)\leq 2, \fnorm{B}\leq 1\right\}.$$ Then, we have $\mathcal{A}\subset\mathcal{B}$, and
$$\max_{a\in\mathbb{R}^p}\left|\frac{\sum_{1\leq i\neq j\leq p}(a_i-a_j)w_{ij}}{\sqrt{\sum_{1\leq i\neq j\leq p}(a_i-a_j)^2}}\right| = \max_{A\in\mathcal{A}}|\iprod{A}{W}|.$$
By Lemma 3.1 of \cite{candes2011tight}, the covering number of the low-rank set $\mathcal{B}$ is bounded by $e^{O(p)}$, which further implies the same covering number bound for $\mathcal{A}$ by the fact that $\mathcal{A}\subset\mathcal{B}$. In other words, there exists $A_1,\dotsc,A_m\in\mathcal{A}$, such that $m\leq e^{C_1p}$, and for any $A\in\mathcal{A}$, $\min_{1\leq l\leq m}\fnorm{A_l-A}\leq 1/2$. Let us choose any $A\in\mathcal{A}$, and then let $A_l$ be the matrix in the covering set that satisfies $\fnorm{A_l-A}\leq 1/2$. We then have
$$|\iprod{A}{W}|\leq \fnorm{A-A_l}\left|\iprod{\frac{A-A_l}{\fnorm{A-A_l}}}{W}\right|+|\iprod{A_l}{W}|\leq \frac{1}{2}\max_{A\in\mathcal{A}}|\iprod{A}{W}| + |\iprod{A_l}{W}|,$$
which implies
$$\max_{A\in\mathcal{A}}|\iprod{A}{W}|\leq \frac{1}{2}\max_{A\in\mathcal{A}}|\iprod{A}{W}| + \max_{1\leq l\leq m}|\iprod{A_l}{W}|.$$
After rearrangement, we get $\max_{A\in\mathcal{A}}|\iprod{A}{W}|\leq 2\max_{1\leq l\leq m}|\iprod{A_l}{W}|$. Then, the conclusion follows by a standard union bound argument. 

For the second inequality, we use the notation $r_j=\frac{1}{\sqrt{2(p-1)}}\sum_{i\in[p]\backslash\{j\}}(w_{ji}-w_{ij})$. It is clear that $r_j \sim \mathn(0,1)$ for all $j\in[p]$, and thus we have $\mathbb{E}\left(\sum_{j=1}^pr_j^2\right)=p$. We then calculate the variance. We have
$$\Var\left(\sum_{j=1}^pr_j^2\right)=\sum_{j=1}^p\sum_{l=1}^p\mathbb{E}(r_j^2-1)(r_l^2-1).$$
For $j=l$, we get $\mathbb{E}(r_j^2-1)^2=2$. For $j\neq l$, we have $\mathbb{E}(r_j^2-1)(r_l^2-1)=\mathbb{E}r_j^2r_l^2-1$, and
$$\mathbb{E}r_j^2r_l^2=\frac{1}{4(p-1)^2}\mathbb{E}\left(\sum_{i\in[p]\backslash\{l\}}(w_{ji}-w_{ij})+(w_{jl}-w_{lj})\right)^2\left(\sum_{i\in[p]\backslash\{j\}}(w_{li}-w_{il})+(w_{lj}-w_{jl})\right)^2.$$
Since the three terms $\sum_{i\in[p]\backslash\{l\}}(w_{ji}-w_{ij})$, $\sum_{i\in[p]\backslash\{j\}}(w_{li}-w_{il})$ and $(w_{jl}-w_{lj})$ are independent, we can expand the above display and calculate the expectation of each term in the expansion, and we get
$$\mathbb{E}r_j^2r_l^2=\frac{4(p-2)^2+4+8(p-2)}{4(p-1)^2}=1.$$
Therefore, $\Var\left(\sum_{j=1}^pr_j^2\right)=2p$, and the desired conclusion is obtained by Chebyshev's inequality. Finally, the last inequality is a direct consequence of a union bound argument.
\end{proof}

Now we are ready to state the proofs of Lemma \ref{lem:rank-error} and Lemma \ref{lem:ranking-ideal}. Note that under the setting of approximate ranking, the error terms are
\begin{eqnarray*}
F_j(a,b;z) &=& \epsilon_j\frac{2p}{\sqrt{2(p-1)}}(\wh{\beta}(z^*)-\wh{\beta}(z))(a-b), \\
G_j(a,b;z) &=& \frac{p^2}{p-1}\left(\beta^*\left(a-\frac{1}{p}\sum_{j=1}^pz_j^*\right)-\wh{\beta}(z)\left(a-\frac{p+1}{2}\right)\right)^2 \\
&& - \frac{p^2}{p-1}\left(\beta^*\left(a-\frac{1}{p}\sum_{j=1}^pz_j^*\right)-\wh{\beta}(z^*)\left(a-\frac{p+1}{2}\right)\right)^2 \\
&& -\frac{p^2}{p-1}\left(\beta^*\left(a-\frac{1}{p}\sum_{j=1}^pz_j^*\right)-\wh{\beta}(z)\left(b-\frac{p+1}{2}\right)\right)^2 \\
&& +\frac{p^2}{p-1}\left(\beta^*\left(a-\frac{1}{p}\sum_{j=1}^pz_j^*\right)-\wh{\beta}(z^*)\left(b-\frac{p+1}{2}\right)\right)^2, \\
H_j(a,b) &=& \frac{p^2}{p-1}\left(\beta^*\left(a-\frac{1}{p}\sum_{j=1}^pz_j^*\right)-\wh{\beta}(z^*)\left(a-\frac{p+1}{2}\right)\right)^2 \\
&& - \frac{p^2}{p-1}\left(\beta^*\left(a-\frac{1}{p}\sum_{j=1}^pz_j^*\right)-\beta^*\left(a-\frac{p+1}{2}\right)\right)^2 \\
&& -\frac{p^2}{p-1}\left(\beta^*\left(a-\frac{1}{p}\sum_{j=1}^pz_j^*\right)-\wh{\beta}(z^*)\left(b-\frac{p+1}{2}\right)\right)^2 \\
&& +\frac{p^2}{p-1}\left(\beta^*\left(a-\frac{1}{p}\sum_{j=1}^pz_j^*\right)-\beta^*\left(b-\frac{p+1}{2}\right)\right)^2.
\end{eqnarray*}
\begin{proof}[Proof of Lemma \ref{lem:rank-error}]
For any $z\in[p]^p$ such that $\ell(z,z^*)\leq \tau= o(p^2(\beta^*)^2)$, we have
\begin{equation}
\sum_{1\leq i\neq j\leq p}(z_i-z_i^*-z_j+z_j^*)^2 \leq 4(p-1)\sum_{j=1}^p(z_j-z_j^*)^2 \leq \frac{2}{(\beta^*)^2}\ell(z,z^*)=o(p^2). \label{eq:harry}
\end{equation}
For any $z^*\in\mathcal{R}$, we have $\sum_{1\leq i\neq j\leq p}(z_i^*-z_j^*)^2 \leq p^4$. Moreover, by the definition of $\mathcal{R}$, there exists a $\wt{z}\in\Pi_p$ such that $\|z^*-\wt{z}\|^2\leq c_p$. This implies
\begin{equation}
\left(\frac{1}{p}\sum_{j=1}^pz_j^*-\frac{p-1}{2}\right)^2 = \left(\frac{1}{p}\sum_{j=1}^pz_j^*-\sum_{j=1}^p\wt{z}_j\right)^2 \leq \frac{1}{p}\|z^*-\wt{z}\|^2 \leq \frac{c_p}{p}=o(1), \label{eq:1st-m-z}
\end{equation}
and
$$
\left|\sum_{j=1}^p(z_j^*)^2-\sum_{j=1}^pj^2\right| = \left|\|z^*\|^2-\|\wt{z}\|^2\right| \leq \|z^*-\wt{z}\|(\|z^*\|+\|\wt{z}\|) \lesssim c_p^{1/2}p^{1.5} = o(p^2).
$$
Thus,
\begin{eqnarray}
\nonumber  \sum_{1\leq i\neq j\leq p}(z_i^*-z_j^*)^2  &=&  2p\sum_{j=1}^p(z_j^*)^2 - 2\left(\sum_{j=1}^pz_j^*\right)^2 \\
\nonumber &\geq& 2p\left(\sum_{j=1}^pj^2-o(p^2)\right) - (1+o(1))2\left(\sum_{j=1}^p j\right)^2 \\
\label{eq:p4/12} &\geq& \frac{p^4}{12}.
\end{eqnarray}
Therefore,
\begin{eqnarray}
\nonumber \sum_{1\leq i\neq j\leq p}(z_i-z_j)^2 &\geq& \frac{1}{2}\sum_{1\leq i\neq j\leq p}(z_i^*-z_j^*)^2 - \sum_{1\leq i\neq j\leq p}(z_i-z_i^*-z_j+z_j^*)^2 \\
\nonumber &\geq& \frac{p^4}{24} - o(p^2) \\
\label{eq:small-size-z2} &\geq& \frac{p^4}{25},
\end{eqnarray}
where the last inequality assumes $p$ is sufficiently large.
We then introduce more notations. We define
$$\beta(z)=\beta^*\frac{\sum_{1\leq i\neq j\leq p}(z_i-z_j)(z_i^*-z_j^*)}{\sum_{1\leq i\neq j\leq p}(z_i-z_j)^2},$$
and
$$\bar{w}(z)=\frac{\sum_{1\leq i\neq j\leq p}(z_i-z_j)w_{ij}}{\sum_{1\leq i\neq j\leq p}(z_i-z_j)^2}.$$
We write $w_{ij}=Y_{ij}-\beta^*(z_i^*-z_j^*)$ so that $\epsilon_j=\frac{1}{\sqrt{2(p-1)}}\sum_{i\in[p]\backslash\{j\}}(w_{ji}-w_{ij})$. We consider the event that the three inequalities (\ref{eq:w-ranking2})-(\ref{eq:w-ranking4}) hold.
We first derive bounds for $|\wh{\beta}(z^*)-\beta^*|$, $|\beta(z)-\beta^*|$ and $|\bar{w}(z)-\bar{w}(z^*)|$. By (\ref{eq:w-ranking2}), we have
\begin{equation}
|\wh{\beta}(z^*)-\beta^*|=\left|\frac{\sum_{1\leq i\neq j\leq p}(z_i^*-z_j^*)w_{ij}}{\sum_{1\leq i\neq j\leq p}(z_i^*-z_j^*)^2}\right|\lesssim p^{-1.5}. \label{eq:helpful-1}
\end{equation}
By (\ref{eq:harry}) and (\ref{eq:small-size-z2}), we have
\begin{eqnarray}
\nonumber |\beta(z)-\beta^*| &=& \left|\beta^*\frac{\sum_{1\leq i\neq j\leq p}(z_i-z_j)(z_i^*-z_j^*-z_i+z_j)}{\sum_{1\leq i\neq j\leq p}(z_i-z_j)^2}\right| \\
\nonumber &\leq& \frac{25|\beta^*|}{p^4}\sqrt{\sum_{1\leq i\neq j\leq p}(z_i-z_j)^2}\sqrt{\sum_{1\leq i\neq j\leq p}(z_i^*-z_j^*-z_i+z_j)^2} \\
\nonumber &\leq& \frac{25|\beta^*|}{p^2}\sqrt{\frac{2}{(\beta^*)^2}\ell(z,z^*)} \\
\label{eq:helpful-2} &\lesssim& \frac{\sqrt{\ell(z,z^*)}}{p^2}.
\end{eqnarray}
Next, we bound $|\bar{w}(z)-\bar{w}(z^*)|$. We have
\begin{eqnarray*}
|\bar{w}(z)-\bar{w}(z^*)| &\leq& \left|\frac{\sum_{1\leq i\neq j\leq p}(z_i-z_j-z_i^*+z_j^*)w_{ij}}{\sum_{1\leq i\neq j\leq p}(z_i-z_j)^2}\right| \\
&& + \left|\frac{\sum_{1\leq i\neq j\leq p}(z_i-z_j)^2-\sum_{1\leq i\neq j\leq p}(z^*_i-z^*_j)^2}{\sum_{1\leq i\neq j\leq p}(z_i-z_j)^2\sqrt{\sum_{1\leq i\neq j\leq p}(z_i^*-z_j^*)^2}}\right|\left|\frac{\sum_{1\leq i\neq j\leq p}(z^*_i-z^*_j)w_{ij}}{\sqrt{\sum_{1\leq i\neq j\leq p}(z^*_i-z^*_j)^2}}\right|.
\end{eqnarray*}
We bound the two terms on the right hand side of the above inequality separately. The first term can be bounded by
\begin{eqnarray*}
&& \frac{\sqrt{\sum_{1\leq i\neq j\leq p}(z_i-z_i^*-z_j+z_j^*)^2}}{\sum_{1\leq i\neq j\leq p}(z_i-z_j)^2}\left|\frac{\sum_{1\leq i\neq j\leq p}(z_i-z_j-z_i^*+z_j^*)w_{ij}}{\sqrt{\sum_{1\leq i\neq j\leq p}(z_i-z_i^*-z_j+z_j^*)^2}}\right| \\
&\leq& \frac{25}{p^4}\sqrt{\frac{2}{(\beta^*)^2}\ell(z,z^*)} \left|\frac{\sum_{1\leq i\neq j\leq p}(z_i-z_j-z_i^*+z_j^*)w_{ij}}{\sqrt{\sum_{1\leq i\neq j\leq p}(z_i-z_i^*-z_j+z_j^*)^2}}\right| \\
&\lesssim& \frac{\sqrt{p\ell(z,z^*)}}{|\beta^*|p^4},
\end{eqnarray*}
where we have used the inequalities (\ref{eq:w-ranking2}), (\ref{eq:harry}), and (\ref{eq:small-size-z2}). By (\ref{eq:w-ranking2}) and (\ref{eq:small-size-z2}), the second term can be bounded by
\begin{eqnarray*}
&& C_1p^{-5.5}\left|\sum_{1\leq i\neq j\leq p}(z_i-z_j)^2-\sum_{1\leq i\neq j\leq p}(z^*_i-z^*_j)^2\right| \\
&\leq& C_1p^{-5.5}\left|\sum_{1\leq i\neq j\leq p}(z_i-z_j)(z_i^*-z_j^*-z_i+z_j)\right| \\
&& + C_1p^{-5.5}\left|\sum_{1\leq i\neq j\leq p}(z_i^*-z_j^*)(z_i^*-z_j^*-z_i+z_j)\right| \\
&\leq& C_1p^{-5.5}\left(\sqrt{\sum_{1\leq i\neq j\leq p}(z_i-z_j)^2}+\sqrt{\sum_{1\leq i\neq j\leq p}(z_i^*-z_j^*)^2}\right)\sqrt{\sum_{1\leq i\neq j\leq p}(z_i-z_i^*-z_j+z_j^*)^2} \\
&\lesssim& \frac{\sqrt{p\ell(z,z^*)}}{|\beta^*|p^4},
\end{eqnarray*}
where we have used (\ref{eq:harry}) in the last inequality. Combining the two bounds, we obtain
\begin{equation}
|\bar{w}(z)-\bar{w}(z^*)| \lesssim  \frac{\sqrt{p\ell(z,z^*)}}{|\beta^*|p^4}. \label{eq:helpful-3} 
\end{equation}
From (\ref{eq:helpful-1}), (\ref{eq:helpful-2}) and (\ref{eq:helpful-3}), we can further derive
\begin{eqnarray}
\label{eq:helpful-4}  |\wh{\beta}(z)-\wh{\beta}(z^*)| &\leq& |\beta(z)-\beta^*| + |\bar{w}(z)-\bar{w}(z^*)| \lesssim \frac{\sqrt{\ell(z,z^*)}}{p^2},
\end{eqnarray}
under the condition $p(\beta^*)^2\geq 1$.

We are ready to prove (\ref{eq:rank-error1})-(\ref{eq:rank-error3}). Recall that $\epsilon_j=\frac{1}{\sqrt{2(p-1)}}\sum_{i\in[p]\backslash\{j\}}(w_{ji}-w_{ij})$, and we have
\begin{eqnarray}
\nonumber \sum_{j=1}^p\epsilon_j^2 &\leq& \sum_{j=1}^p\left(\frac{1}{\sqrt{2(p-1)}}\sum_{i\in[p]\backslash\{j\}}(w_{ji}-w_{ij})\right)^2 \\
\label{eq:ugly} &\lesssim& p,
\end{eqnarray}
by (\ref{eq:w-ranking3}). Moreover, from (\ref{eqn:rank_delta_def}), since $z^*\in\mathcal{R}$, we have
\begin{equation}
\Delta_j(a,b)^2=(1+o(1))\frac{2p^2(\beta^*)^2}{p-1}(a-b)^2.\label{eq:even-delta}
\end{equation}
Thus,
\begin{eqnarray*}
&& \sum_{j=1}^p\max_{b\in[k]\backslash\{z_j^*\}}\frac{F_j(z_j^*,b;z)^2\|\mu_j(B^*,b)-\mu_j(B^*,z_j^*)\|^2}{\Delta_j(z_j^*,b)^4\ell(z,z^*)} \\
&=& \frac{|\wh{\beta}(z)-\wh{\beta}(z^*)|^2}{|\beta^*|^2\ell(z,z^*)}\sum_{j=1}^p \epsilon_j^2 \\
&\lesssim& \frac{1}{p^2},
\end{eqnarray*}
where we have used (\ref{eq:helpful-4}), (\ref{eq:ugly}), (\ref{eq:even-delta}) and the condition $p(\beta^*)^2\geq 1$ in the last inequality. Taking maximum, we obtain (\ref{eq:rank-error1}). For (\ref{eq:rank-error2}), we note that
\begin{eqnarray*}
|G_j(a,b;z)| &\leq& \frac{p^2}{p-1}\left|\left(a-\frac{p+1}{2}\right)^2-\left(b-\frac{p+1}{2}\right)^2\right||\wh{\beta}(z)-\wh{\beta}(z^*)|^2 \\
&& + \frac{2p^2}{p-1}\left(a-\frac{p+1}{2}\right)^2|\wh{\beta}(z)-\wh{\beta}(z^*)||\wh{\beta}(z^*)-\beta^*| \\
&& + \frac{2p^2}{p-1}\left|b-\frac{p+1}{2}\right|\left|\wh{\beta}(z)-\wh{\beta}(z^*)\right|\left|\left(b-\frac{p+1}{2}\right)\wh{\beta}(z^*)-\left(a-\frac{p+1}{2}\right)\beta^*\right| \\
&& + \frac{2p^3}{p-1}|\beta^*||\wh{\beta}(z)-\wh{\beta}(z^*)|\left|\frac{1}{p}\sum_{j=1}^pz_j^*-\frac{p+1}{2}\right| \\
&\leq& \frac{p^4}{p-1}|\wh{\beta}(z)-\wh{\beta}(z^*)|^2 + \frac{4p^4}{p-1}|\wh{\beta}(z)-\wh{\beta}(z^*)||\wh{\beta}(z^*)-\beta^*| \\
&& + \frac{4p^3}{p-1}|a-b||\beta^*||\wh{\beta}(z)-\wh{\beta}(z^*)|.
\end{eqnarray*}
Therefore, for any subset $T\subset[p]$, we have
\begin{eqnarray*}
&& \frac{\tau}{4\Delta_{\min}^2|T|+\tau}\sum_{j\in T}\max_{b\in[k]\backslash\{z_j^*\}}\frac{G_j(z_j^*,b;z)^2\|\mu_j(B^*,b)-\mu_j(B^*,z_j^*)\|^2}{\Delta_j(z_j^*,b)^4\ell(z,z^*)} \\
&\lesssim& \frac{\tau p^4|\wh{\beta}(z)-\wh{\beta}(z^*)|^4}{|\beta^*|^4\ell(z,z^*)} + \frac{\tau p^4|\wh{\beta}(z)-\wh{\beta}(z^*)|^2|\wh{\beta}(z^*)-\beta^*|^2}{|\beta^*|^4\ell(z,z^*)} + \frac{\tau p^2|\wh{\beta}(z)-\wh{\beta}(z^*)|^2}{|\beta^*|^2\ell(z,z^*)} \\
&\lesssim& \frac{\tau}{p^2|\beta^*|^2} + \frac{1}{p|\beta^*|^2},
\end{eqnarray*}
where we have used (\ref{eq:helpful-1}), (\ref{eq:helpful-4}), and $\ell(z,z^*)\leq \tau =o(p^2(\beta^*)^2)$. Taking maximum, we thus obtain (\ref{eq:rank-error2}). Finally, for (\ref{eq:rank-error3}), we have
\begin{eqnarray*}
\frac{|H_j(a,b)|}{\Delta_j(a,b)^2} &\leq& \frac{1}{2|\beta^*|^2}(\wh{\beta}(z^*)-\beta^*)^2\left(a-\frac{p+1}{2}\right)^2 + \frac{1}{2|\beta^*|^2}(\wh{\beta}(z^*)-\beta^*)^2\left(b-\frac{p+1}{2}\right)^2 \\
&& +2p\frac{|\wh{\beta}(z^*)-\beta^*|}{|\beta^*|} \\
&\lesssim& \frac{1}{\sqrt{p}|\beta^*|},
\end{eqnarray*}
where we have used (\ref{eq:helpful-1}). We thus obtain (\ref{eq:rank-error3}) by taking maximum. The proof is complete.
\end{proof}

\begin{proof}[Proof of Lemma \ref{lem:ranking-ideal}]
By (\ref{eq:even-delta}), there exists some $\delta'=\delta_p'=o(1)$, such that
\begin{eqnarray*}
&& \indc{\iprod{\epsilon_j}{\nu_j(\wh{B}(z^*),z_j^*)-\nu_j(\wh{B}(z^*),b)} \leq -\frac{1-\delta}{2}\Delta_j(z_j^*,b)^2} \\
&\leq& \indc{\epsilon_j\frac{2p}{\sqrt{2(p-1)}}\wh{\beta}(z^*)(z_j^*-b)\leq -\frac{1-\delta-\delta'}{2}\frac{2p^2(\beta^*)^2}{p-1}(z_j^*-b)^2} \\
&\leq& \indc{\epsilon_j\frac{2p}{\sqrt{2(p-1)}}\beta^*(z_j^*-b)\leq -\frac{1-\delta-\delta'-\bar{\delta}}{2}\frac{2p^2(\beta^*)^2}{p-1}(z_j^*-b)^2} \\
&& + \indc{\epsilon_j\frac{2p}{\sqrt{2(p-1)}}(\wh{\beta}(z^*)-\beta^*)(z_j^*-b)\leq -\frac{\bar{\delta}}{2}\frac{2p^2(\beta^*)^2}{p-1}(z_j^*-b)^2}.
\end{eqnarray*}
By (\ref{eq:w-ranking4}), (\ref{eq:helpful-1}), and $p(\beta^*)^2\rightarrow\infty$, we have
\begin{eqnarray*}
\max_{j\in[p]}\frac{\left|\epsilon_j\frac{2p}{\sqrt{2(p-1)}}(\wh{\beta}(z^*)-\beta^*)(z_j^*-b)\right|}{\frac{2p^2(\beta^*)^2}{p-1}(z_j^*-b)^2} = o\left(\frac{\sqrt{\log p}}{p}\right),
\end{eqnarray*}
with probability at least $1-p^{-1}$. Therefore, we can set $\bar{\delta}=\bar{\delta}_p$ for some sequence $\bar{\delta}_p\rightarrow 0$ and $\bar{\delta}_p\gtrsim \frac{\sqrt{\log p}}{p}$, and then
$$\indc{\epsilon_j\frac{2p}{\sqrt{2(p-1)}}(\wh{\beta}(z^*)-\beta^*)(z_j^*-b)\leq -\frac{\bar{\delta}}{2}\frac{2p^2(\beta^*)^2}{p-1}(z_j^*-b)^2}=0,$$
for all $j\in[p]$ with probability at least $1-p^{-1}$. This immediately implies that $\xi_{\rm ideal}(\delta_p)\leq \wt{\xi}_{\rm ideal}(\delta_p+\delta_p'+\bar{\delta}_p)$ with high probability, where
$$\wt{\xi}_{\rm ideal}(\delta_p+\delta_p'+\bar{\delta}_p)= \frac{2p^2(\beta^*)^2}{p-1}\sum_{j=1}^p\sum_{b\in[p]\backslash\{z_j^*\}}(z_j^*-b)^2 \indc{\epsilon_j\frac{2p}{\sqrt{2(p-1)}}\beta^*(z_j^*-b)\leq -\frac{1-\delta_p-\delta_p'-\bar{\delta}_p}{2}\frac{2p^2(\beta^*)^2}{p-1}(z_j^*-b)^2}.$$
A standard Gaussian tail bound implies
\begin{eqnarray*}
&& \mathbb{E}\wt{\xi}_{\rm ideal}(\delta_p+\delta_p'+\bar{\delta}_p) \\
&=& \frac{2p^2(\beta^*)^2}{p-1}\sum_{j=1}^p\sum_{b\in[p]\backslash\{z_j^*\}}(z_j^*-b)^2\mathbb{P}\left(\mathn(0,1)\leq -\frac{1-\delta_p-\delta_p'-\bar{\delta}_p}{2}\sqrt{\frac{2p^2(\beta^*)^2}{p-1}(z_j^*-b)^2}\right) \\
&\leq& \sum_{j=1}^p\sum_{l=1}^{\infty}\frac{4p^2(\beta^*)^2}{p-1}l^2\exp\left(-\left(\frac{1-\delta_p-\delta_p'-\bar{\delta}_p}{2}\right)^2\frac{p^2(\beta^*)^2}{p-1}l^2\right) \\
&\leq& p\exp\left(-(1+o(1))\frac{p(\beta^*)^2}{4}\right),
\end{eqnarray*}
where we have used the conditions $p(\beta^*)^2\rightarrow\infty$ and $\delta_p+\delta_p'+\bar{\delta}_p=o(1)$ in the last inequality. Finally, by Markov's inequality, with probability at least $1-\exp\left(-\sqrt{p(\beta^*)^2}\right)$, we have
\begin{align*}
\wt{\xi}_{\rm ideal}(\delta_p+\delta_p'+\bar{\delta}_p)& \leq \mathbb{E}\wt{\xi}_{\rm ideal}(\delta_p+\delta_p'+\bar{\delta}_p)\exp\left(\sqrt{p(\beta^*)^2}\right)\\
& \leq p\exp\left(-(1+o(1))\frac{p(\beta^*)^2}{4}\right),
\end{align*}
as $p(\beta^*)^2\rightarrow\infty$. Since $\xi_{\rm ideal}(\delta_p)\leq \wt{\xi}_{\rm ideal}(\delta_p+\delta_p'+\bar{\delta}_p)$,  the proof is complete.
\end{proof}

Finally, we state the proof of Proposition \ref{prop:initial-rank}.

\begin{proof}[Proof of Proposition \ref{prop:initial-rank}]
Note that we have the following fact. Consider  any  $x=(x_1,x_2,\ldots, x_m)^T \in\mathr^m$.
Let $ y = \argmin_{z\in\Pi_p} \sum_{i=1}^p(x_i - z_i)^2$.  
Then for any pair $(i,j)$ such that $x_i <x_j$, we must have $y_i < y_j$. Otherwise if $y_i > y_j$,  since $(\br{  x_i - y_i}^2 + \br{  x_j - y_j}^2) - (\br{  x_i - y_j}^2 + \br{  x_j - y_i}^2) = -2(x_i -x_j)(y_i - y_j) >0$, we can always  swap $y_i$ and $y_j$ to make $\sum_{i=1}^p( x_i - y_i)^2$ strictly smaller. This indicates that $y$ preserves the order of $x$.

As a result,
since a linear transformation does not change the rank, we can write $z^{(0)}$ as
\begin{equation}
z^{(0)}=\argmin_{z\in\Pi_p}\sum_{j=1}^p\left(\frac{\sqrt{2(p-1)}}{2p\beta^*}T_j+\frac{p+1}{2}-z_j\right)^2.\label{eq:sort-basic}
\end{equation}
Since $z^*\in\mathcal{R}$, there exists some $\wt{z}\in\Pi_p$ such that $L_2(\wt{z},z^*)=o(1)$. 
By (\ref{eq:sort-basic}),
$$\sum_{j=1}^p\left(\frac{\sqrt{2(p-1)}}{2p\beta^*}T_j+\frac{p+1}{2}-z_j^{(0)}\right)^2\leq \sum_{j=1}^p\left(\frac{\sqrt{2(p-1)}}{2p\beta^*}T_j+\frac{p+1}{2}-\wt{z}_j\right)^2.$$
We then have
\begin{align*}
&  L_2(z^{(0)}, \wt{z}) = p^{-1}\sum_{j=1}^p(z_j^{(0)}-\wt{z}_j)^2 \\
&\leq 2p^{-1}\sum_{j=1}^p\left(\frac{\sqrt{2(p-1)}}{2p\beta^*}T_j+\frac{p+1}{2}-z_j^{(0)}\right)^2 + 2p^{-1}\sum_{j=1}^p\left(\frac{\sqrt{2(p-1)}}{2p\beta^*}T_j+\frac{p+1}{2}-\wt{z}_j\right)^2 \\
&\leq 4p^{-1}\sum_{j=1}^p\left(\frac{\sqrt{2(p-1)}}{2p\beta^*}T_j+\frac{p+1}{2}-\wt{z}_j\right)^2 \\
&\leq 12p^{-1}\sum_{j=1}^p\left(\frac{\sqrt{2(p-1)}}{2p\beta^*}T_j+\frac{1}{p}\sum_{j=1}^pz_j^*-z_j^*\right)^2 + 12p^{-1}\sum_{j=1}^p(\wt{z}_j-z_j^*)^2 + 12\left(\frac{1}{p}\sum_{j=1}^pz_j^*-\frac{p+1}{2}\right)^2\\
& = \frac{6\br{p-1}}{p^3\beta^{*2}} \sum_{j=1}^p \epsilon_j^2 + 12L_2(\wt{z},z^*) + 12\left(\frac{1}{p}\sum_{j=1}^pz_j^*-\frac{p+1}{2}\right)^2,
\end{align*}
where the last equation is due to the fact that $T_j = \mu_j(B^*,z^*_j) + \epsilon_j$.
By (\ref{eq:ugly}), (\ref{eq:1st-m-z}), and $L_2(\wt{z},z^*)=o(1)$, we have
$$L_2(z^{(0)},\wt{z})\lesssim \frac{1}{p(\beta^*)^2} + o(1),$$
with probability at least $1-p^{-1}$. By $L(z^{(0)},z^*)\leq 2L(z^{(0)},\wt{z})+2L_2(\wt{z},z^*)$, we have
$$L_2(z^{(0)},z^*)\lesssim \frac{1}{p(\beta^*)^2} + o(1),$$
with high probability. When $p(\beta^*)^2\rightarrow\infty$, we clearly have $L_2(z^{(0)},z^*)=o(1)$. When $p(\beta^*)^2=O(1)$, we have $L_2(z^{(0)},z^*)\lesssim \min\left(p^2, \frac{1}{p(\beta^*)^2}\right)$, where $L_2(z^{(0)},z^*)\lesssim p^2$ is by the definition of the loss.
\end{proof}

\subsection{Proofs in Section \ref{sec:regression}}

In this section, we will prove results presented in Section \ref{sec:regression}. Most efforts will be devoted to the proofs of Lemma \ref{lem:regression-error}, Lemma \ref{lem:regression-ideal}, and Proposition \ref{prop:reg_init}. With these results established, the conclusions of Theorem \ref{thm:lower-regression}, Theorem \ref{thm:regression-iter}, and Corollary \ref{cor:regression-iter} easily follow.

Let us introduce some more notation to facilitate the proofs. Given some vector $v\in\mathbb{R}^d$, some matrix $A\in\mathbb{R}^{d'\times d}$ and some set $S\subset[d]$, we use $v_S\in\mathbb{R}^{|S|}$ for the sub-vector $(v_j: j\in S)$ and $A_S\in\mathbb{R}^{d'\times |S|}$ for the sub-matrix $(A_{ij}:i\in[d'],j\in S)$. We denote $\text{span}(A)$ to be the space spanned by the columns of $A$. 
For any $j\in[d]$, we denote $[v]_j = v_j$ to be the $j$th coordinate of $v$.
We also write $\phi_S:[d]\rightarrow [|S|]$ for the map that satisfies $v_j=[v_S]_{\phi_S(j)}$. The domain of the map $\phi_S$ can also be extended to sets so that for any $S'\subset S$, we can write $v_{S'}=[v_S]_{\phi_S(S')}$.
For any $j\in S$, we write $S_{-j}=S\backslash\{j\}$.
We use $I_{d}$ for the $d\times d$ identity matrix, and sometimes just write $I$ for simplicity if the dimension is clear from the context. Given any square matrix $A\in\mathbb{R}^{d\times d}$, we use $\text{diag}\{A\}$ for the diagonal matrix whose diagonal entries are identical to those of $A$. For two random elements $X$ and $Y$,  we write $X \dist Y$ if their distributions are identical, and $X\perp Y$ if they are independent of each other. 

We first state and prove three technical lemmas.

\begin{lemma}\label{lem:random-events-reg}
Assume $s\log p \leq n$. Consider a random matrix $X\in\mathbb{R}^{n\times p}$ with i.i.d. entries $X_{ij}\sim\mathcal{N}(0,1)$, an independent $w\sim \mathcal{N}(0,I_n)$, some $S^*\subset [p]$ satisfying $|S^*| = s$, and some $\beta^*\in\mathr^p$. For any $S\subset[p]$, denote $P_S = X_S\br{X^T_S X_S}^{-1}X_S^T$ to be the projection matrix onto the subspace $\text{span}(X_S)$. We also use the notation $P_j=X_jX_j^T/\|X_j\|^2$, where $X_j$ represents the $j$th column of $X$. Then, for any constants $C_0,C'>0$, there exists some constant $C>0$ only depending on $C_0,C'$ such that
\begin{align}
& \max_{j \in [p]} \abs{\norm{X_j }^2 - n} \leq C\sqrt{n\log p}   \leq n/2 \label{eqn:re0},  \\
&  \max_{S\subset[p]:\abs{S}\leq 2C_0s} \norm{\br{X_S^TX_S}^{-1}} \leq \frac{C}{n}  \label{eqn:re0_1}, \\
& \max_{S,T\subset[p]:S\cap T = \emptyset, \abs{S},\abs{T}\leq 2C_0s } \norm{\br{X_S^TX_S}^{-1}X_S^TX_T}^2  \leq C \frac{s\log p}{n}, \label{eqn:re1}\\
&\max_{S,T\subset[p]: S\cap T = \emptyset,\abs{S},\abs{T} \leq  2C_0s } \frac{1}{\abs{T}}\norm{\br{X_T^T(I-P_S)X_T }^{-1}  X_T^T (I-P_S) w }^2 \leq \frac{C \log p}{n}, \label{eqn:re2}\\
&  \max_{S\subset[p]:\abs{S} \leq  2C_0s }\max_{j  \in S^*\cap S^\complement} \abs{X_j ^T P_S X_j } \leq C s\log p,  \label{eqn:re3_0} \\
&\max_{S,T\subset[p]: S\cap T = \emptyset,\abs{S},\abs{T} \leq  2C_0s }\norm{  X_{T}^T (I-P_S) X_{T}  - \text{diag}\cbr{X_{T}^T (I-P_S) X_{T}  }  }^2 \leq  Cns\log p, \label{eqn:re3}\\
& \max_{S\subset[p]:\abs{S}\leq  2C_0s } \frac{1}{\abs{S^*\cap S^\complement}+ \abs{S^{*\complement}\cap S}}\sum_{j \in S^*\cap S^\complement} \br{X_j ^TP_S w }^2 \leq C  s\log^2 p, \label{eqn:re4}\\
& \max_{j \notin S^*}\br{X_j ^T P_{S^*} w }^2 \leq C s\log p, \label{eqn:re7}\\
&\max_{S,T\subset[p]: S\cap T = \emptyset,\abs{S},\abs{T} \leq  2C_0s } \frac{1}{\abs{T}}\norm{\br{X_T^T(I-P_S)X_T }^{-1}  X_T^T P_S w }^2 \leq \frac{C s\log p}{n^2}, \label{eqn:re8_0}\\
& \max_{S,T\subset[p]:S\cap T = \emptyset, \abs{S},\abs{T}\leq  2C_0s  }\norm{X_{T }^T (I-P_S) X_{T }  - (n-\abs{S})I_{\abs{T}} }^2 \leq  Cns\log p,\label{eqn:re8}\\
&  \max_{S,T\subset[p]:S\cap T = \emptyset, \abs{S},\abs{T}\leq  2C_0s  } \norm{\br{X_{T }^T (I-P_S) X_{T }}^{-1}} \leq \frac{C}{n}, \label{eqn:re9}\\
&\max_{S\subset[p]:\abs{S}\leq  2C_0s } \frac{1}{\abs{S}} \norm{X_S^T w }^2 \leq Cn\log p, \label{eqn:re10} \\
& \max_{S,T\subset[p]:S\cap T = \emptyset, \abs{S}\leq  2C_0s } \frac{1}{ \norm{ \beta^*_{S^\complement \cap S^*}}^2} \frac{1}{\abs{T}\vee s} \norm{X_T^T(I -P_S)X_{S^\complement \cap S^*} \beta^*_{S^\complement \cap S^*}}^2 \leq C n\log p, \label{eqn:re11}\\
&\max_{S,T\subset[p]: T\cap \br{S\cup S^*} = \emptyset, \abs{S} \leq  2C_0s }\frac{1}{\abs{S^*\cap S^\complement}+\abs{S^{*\complement}\cap S}}\frac{1}{\abs{T}}\norm{X_T^T(P_{S^*}-P_S) w }^2 \leq C \log^2 p, \label{eqn:re12}\\
& \max_{j \in S^*}  \abs{\norm{X_j }^{-1} X_j ^T X_{S^*_{-j }} \br{X_{S^*_{-j }}^T \br{I - P_j }X_{S^*_{-j }}}^{-1} X_{S^*_{-j }}^{T} \br{I - P_j } w  } \leq \sqrt{\frac{Cs
\log^2 p}{n}}, \label{eqn:re13}
\end{align}
with probability at least $1-\ebr{-C'\log p}$. We have used the convention that 0/0 = 0.
\end{lemma}
\begin{proof}
We first present a fact that will be used repeatedly in the proof. For two independent $\xi_1,\xi_2\sim\mathn(0,I_d)$, we have $\xi_1^T \xi_2 = \norm{\xi_1} (\norm{\xi_1}^{-1}\xi_1 )^T\xi_2  \dist \norm{\xi_1} \zeta$, where $\zeta \sim \mathn(0,1)$ and $\zeta \perp \xi_1$. Throughout the proof, we will use $c,c',c_1,c_2,\ldots$ as generic constants whose values may change from place to place. We refer to Lemma \ref{prop:chisq} for the $\chi^2$ tail probability bound.

\emph{Equation (\ref{eqn:re0}):} We have $\norm{X_j }^2 \sim \chi^2_n$. Then the $\chi^2$ tail bound and a union bound argument over $j \in[p]$ lead to the desired bound.

\emph{Equation (\ref{eqn:re0_1}):} It is sufficient to study the smallest eigenvalue of $X_S^TX_S$. For a fixed $S$ and $\theta\in\mathr^{\abs{S}}$ such that $\norm{\theta}=1$, we have $\theta^T X_S^TX_S\theta \sim \chi^2_{n}$. Thus $\pbr{\abs{\theta^T X_S^TX_S\theta-n} \leq \frac{n}{2}} \leq 2\ebr{-n/16}$. By a standard $\epsilon$-net argument \citep{vershynin2010introduction}, we can obtain
\begin{align*}
\pbr{\min_{\theta\in\mathr^{\abs{S}}:\norm{\theta}=1} \theta^T X_S^TX_S\theta \leq c_1n} \leq c_2^{\abs{S}} \ebr{-n/16}.
\end{align*}
We then take a union bound over $S$ to obtain
\begin{align*}
\pbr{\min_{S\subset[p]: \abs{S}\leq  2C_0s }\min_{\theta\in\mathr^{\abs{S}}:\norm{\theta}=1} \theta^T X_S^TX_S\theta \leq c_1n} \leq 2\binom{p}{ 2C_0s }c_2^{\abs{S}} \ebr{-n/16}.
\end{align*}
Thus
\begin{align*}
\pbr{\max_{S\subset[p]: \abs{S}\leq  2C_0s } \norm{\br{ X_S^TX_S}^{-1}} \geq \frac{1}{c_1 n}}  \leq \ebr{-c_3 n},
\end{align*}
for some constant $c_3>0$.

\emph{Equation (\ref{eqn:re1}):} 
Conditioning on $X_S$, we have $\br{X_S^TX_S}^{-1}X_S^TX_T \dist \br{X_S^TX_S}^{-\frac{1}{2}} \zeta$, where $\zeta \sim \mathn(0,I_{\abs{S}})$. Thus $\norm{\br{X_S^TX_S}^{-1}X_S^TX_T}^2 \dist \zeta^T\br{X_S^TX_S}^{-1} \zeta \leq \norm{\br{X_S^TX_S}^{-1}} \norm{\zeta}^2$. We have 
\begin{align*}
\pbr{\norm{\br{X_S^TX_S}^{-1}X_S^TX_T}^2  \geq 4c\norm{\br{X_S^TX_S}^{-1}} s\log p \;\Bigg|\; X_S} \leq \ebr{- cs \log p}.
\end{align*}
A union bound over $T$ gives 
\begin{align*}
\pbr{\max_{T\subset[p]:\abs{T}\leq  2C_0s }\norm{\br{X_S^TX_S}^{-1}X_S^TX_T}^2  \geq 4c\norm{\br{X_S^TX_S}^{-1}} s\log p \;\Bigg|\; X_S} \leq  \binom{p}{2C_0s}\ebr{- c s \log p}.
\end{align*}
Consequently, for a fixed $S$,
\begin{align*}
& \pbr{\max_{T\subset[p]:\abs{T}\leq  2C_0s }\norm{\br{X_S^TX_S}^{-1}X_S^TX_T}^2  \geq \frac{4c s\log p }{c_1 n} }  \\
& \leq \binom{p}{2C_0s}\ebr{- c s \log p} + \pbr{\norm{\br{X_S^TX_S}^{-1}} \geq \frac{1}{c_1 n}}.
\end{align*}
Using the result established above when proving (\ref{eqn:re0_1}), together with a union bound over $S$, we have
\begin{align*}
&\pbr{\max_{S,T\subset[p]:S\cap T = \emptyset,\abs{S},\abs{T}\leq  2C_0s }\norm{\br{X_S^TX_S}^{-1}X_S^TX_T}^2  \geq \frac{4cs\log p }{c_1 n} }  \\
&\leq \binom{p}{ 2C_0s } \binom{p}{ 2C_0s }\ebr{- c s \log p} + \pbr{\max_{S\subset[p]:\abs{S}\leq  2C_0s }\norm{\br{X_S^TX_S}^{-1}} \geq \frac{1}{c_1 n}}\\
& \leq \binom{p}{ 2C_0s } \binom{p}{ 2C_0s }\ebr{- c s \log p} + \ebr{-c_2 n}\\
& \leq \ebr{-c_3s\log p}.
\end{align*}

\emph{Equation (\ref{eqn:re2}):} The proof of  (\ref{eqn:re2}) is very similar to that of  (\ref{eqn:re1}). Conditioning on $X_S,X_T$, we have $\br{X_T^T(I-P_S)X_T }^{-1}  X_T^T (I-P_S) w \dist \br{X_T^T(I-P_S)X_T }^{-1}\zeta$ where $\zeta \sim \mathn(0,I_{\abs{T}})$. Consequently, $\norm{\br{X_T^T(I-P_S)X_T }^{-1}  X_T^T (I-P_S) w }^2$ is stochastically dominated by $\norm{\br{X_T^T(I-P_S)X_T }^{-1}}\norm{\zeta}^2$. Similar to the proof of  (\ref{eqn:re1}), we have
\begin{align*}
&\pbr{\max_{S,T\subset[p]:S\cap T = \emptyset,\abs{S},\abs{T}\leq  2C_0s } \norm{ \frac{1}{\abs{T}}\br{X_T^T(I-P_S)X_T }^{-1}  X_T^T (I-P_S) w }^2 \geq \frac{4c\log p}{c_1 n}}\\
& \leq  \sum_{S\subset [p]:\abs{S}\leq  2C_0s } \sum_{m=0}^{ 2C_0s } \sum_{T\subset [p]:\abs{T}= m } \pbr{\chi^2_m \geq 4c m\log p} \\
&\quad  + \pbr{\max_{S,T\subset[p]:S\cap T = \emptyset,\abs{S},\abs{T}\leq  2C_0s } \norm{\br{X_T^T(I-P_S)X_T }^{-1}} \geq \frac{1}{c_1 n}}\\
&\leq \binom{p	}{m} \sum_{m=0}^{ 2C_0s } \binom{p}{ 2C_0s } \ebr{-cm\log p} + \ebr{-cn}\\
& \leq \ebr{-c'\log p},
\end{align*}
where in the second to the last inequality, we use  (\ref{eqn:re9}), which will be proved later.


\emph{Equation (\ref{eqn:re3_0}):} First we fix some $S$. Then $X_j ^T P_S X_j $ is stochastically dominated by a $\chi^2_{\abs{S}}$. We have $\pbr{X_j ^T P_S X_j  \geq 4cs\log p} \leq \ebr{-cs\log p}$. A union bound over $S$ and $j $ leads to the desired result.

\emph{Equation (\ref{eqn:re3}):} For any pair $\br{S,T}$, we have
\begin{align*}
&\norm{  X_{T}^T (I-P_S) X_{T}  - \text{diag}\cbr{X_{T}^T (I-P_S) X_{T}  }  }\\
&\leq \norm{  X_{T}^T (I-P_S) X_{T} - \br{n-\abs{S}}I } + \norm{ \br{n-\abs{S}} I   - \text{diag}\cbr{X_{T}^T (I-P_S) X_{T}  } }.
\end{align*}
The first term can be controlled by  (\ref{eqn:re8}), to be proved later. For the second term, we have
\begin{align*}
 \norm{ \br{n-\abs{S}} I   - \text{diag}\cbr{X_{T}^T (I-P_S) X_{T}  } } & =\max_{j \in T} \abs{X_j ^T \br{I-P_S}X_j  - \br{n-\abs{S}}} \\
 &\leq \max_{j \in T} \abs{\norm{X_j }^2 -n} + \max_{j \in T}  \abs{X_j ^T P_S X_j  - \abs{S}},
\end{align*}
which can be bounded by  (\ref{eqn:re0}) and (\ref{eqn:re3_0}). Combining the two terms together gives the desired result.

\emph{Equation (\ref{eqn:re4}):} For a fixed $S$ and any $j \in S^*\cap S^\complement$, using the fact we give at the beginning of the proof, we have $X_j ^TP_S w \dist \norm{P_S w} \xi_j $ where $\xi_j  \sim \mathn(0,1)$ and $\xi_j \perp \norm{P_S w}$. Since $\xi_j $ only depends on $ X_j^T( \norm{P_S w}^{-1} P_S w )$, we have the independence among $\{\xi_j \}_{j \in S^*\cap S^\complement}$. As a result, we have $\sum_{j \in S^*\cap S^\complement} \br{X_j ^TP_S w }^2 \dist \zeta\xi$, where $\zeta \sim \chi^2_{\abs{S}}$, $\xi \sim \chi^2_{\abs{S^*\cap S^\complement}}$ and $\zeta \perp \xi$. Similar arguments will also be used later to prove (\ref{eqn:re7})-(\ref{eqn:re8_0}) and (\ref{eqn:re10})-(\ref{eqn:re13}) and will be omitted there. Then
\begin{align*}
& \pbr{\sum_{j \in S^*\cap S^\complement} \br{X_j ^TP_S w }^2 \geq 16c^2s\log^2 p \br{\abs{S^*\cap S^\complement} + \abs{S^{*\complement}\cap S}}}\\
& \leq \pbr{\zeta \geq 4cs\log p} + \pbr{\xi \geq 4c\br{\abs{S^*\cap S^\complement} + \abs{S^{*\complement}\cap S}} \log p}\\
&\leq \ebr{-cs \log p} + \ebr{ - c\br{\abs{S^*\cap S^\complement} + \abs{S^{*\complement}\cap S}} \log p}.
\end{align*}
After applying union bound,  we get
\begin{align*}
&\pbr{\max_{S\subset [p]:\abs{S}\leq  2C_0s } \frac{1}{\abs{S^*\cap S^\complement} + \abs{S^{*\complement}\cap S}}\sum_{j \in S^*\cap S^\complement} \br{X_j ^TP_S w }^2 \geq 16c^2s\log^2 p } \\
&\leq \sum_{m=0}^{ 2C_0s } \binom{p}{m} \br{\ebr{-cs \log p} + \ebr{ - cm \log p}} \\
&\leq \ebr{-c'\log p}.
\end{align*}



\emph{Equation (\ref{eqn:re7}):} For each $j \notin S^*$, we have $({X_j ^TP_{S^*}  w })^2$ stochastically dominated by $\xi\zeta$ where $\xi\sim\chi^2_s$, $\zeta \sim\chi^2_1$ and $\xi\perp \zeta$. We get the desired result by the $\chi^2$ tail bound and a union bound over $j \notin S^*$.


\emph{Equation (\ref{eqn:re8_0}):} By  (\ref{eqn:re9}) to be proved later, it is sufficient to establish 
\begin{align*}
\pbr{\max_{S,T\subset[p]: S\cap T = \emptyset,\abs{S},\abs{T} \leq  2C_0s } \frac{1}{\abs{T}}\norm{ X_T^T P_S w }^2 \geq c s\log p} \leq \ebr{-c'\log p}.
\end{align*}
Note that for any fixed $S,T$, we have $\norm{ X_T^T P_S w }^2$ stochastically dominated by $\xi\zeta$ where $\xi \sim \chi^2_{ 2C_0s }$, $\zeta \sim \chi^2_{\abs{T}}$ and $\xi\perp \zeta$. Then we have $\pbr{\norm{ X_T^T P_S w }^2 \geq c^2s\abs{T}\log p} \leq \pbr{\zeta \geq c\abs{T}} + \pbr{\xi \geq cs\log p}$ which can be controlled by the $\chi^2$ tail bound. A union bound is then sufficient to complete the proof. 

\emph{Equation (\ref{eqn:re8}):} For any fixed $S,T$, and any $\theta \in\mathr^{\abs{T}}$ such that $\norm{\theta}=1$, we have 
\begin{align*}
\theta^T \br{X_{T }^T (I-P_S) X_{T }  } \theta \sim \chi^2_{n - \abs{S}},
\end{align*}
and $\theta^T(n-\abs{S})I_{\abs{T}} \theta = \br{n - \abs{S}}$.
By a standard $\epsilon$-net argument \citep{vershynin2010introduction}, the $\chi^2$ tail bound, and a union bound over $S,T$, we conclude its proof.

\emph{Equation (\ref{eqn:re9}):} Its proof is similar to that of  (\ref{eqn:re0_1}). We can show  $$\pbr{\max_{S,T\subset[p]:S\cap T = \emptyset, \abs{S},\abs{T}\leq  2C_0s  } \norm{\br{X_{T }^T (I-P_S) X_{T }}^{-1}} \geq \frac{1}{c n}} \leq \ebr{-c'n}$$ for some $c,c'$. Its proof is omitted here.

\emph{Equation (\ref{eqn:re10}):} We have $\norm{X_S^T w }^2 \dist \xi\zeta$ where $\xi\sim\chi^2_n$, $\zeta \sim\chi^2_{\abs{S}}$ and $\xi\perp \zeta$. Thus,
\begin{align*}
\pbr{\norm{X_S^T w }^2 \geq c^2 n\abs{S}\log p} \leq \pbr{\xi \geq cn} + \pbr{\zeta \geq c\abs{S}\log p}.
\end{align*}
A union bound over integers $0\leq m\leq  2C_0s $ and over all sets $\cbr{S\subset [p]:\abs{S}=m}$ leads to the desired result.

\emph{Equation (\ref{eqn:re11}):} For a fixed pair $S,T$, we have $ \norm{ \beta^*_{S^\complement \cap S^*}}^{-1}X_{S^\complement \cap S^*} \beta^*_{S^\complement \cap S^*} \sim \mathn(0,I_n)$, and consequently $ \norm{ \beta^*_{S^\complement \cap S^*}}^{-2}\norm{X_T^T(I -P_S)X_{S^\complement \cap S^*} \beta^*_{S^\complement \cap S^*}} $ is stochastically dominated by $\xi \zeta$ where $\xi \sim \chi^2_n$, $\zeta\sim \chi^2_{\abs{T}}$ and $\xi\perp \zeta$. Note that $\xi$ only depends on $S^\complement\cap S^*$ and $\zeta$ only depends on $T$. For a fixed $S$, in order to take a union bound over $T$, we add a subscript to $\zeta$ as in $\zeta_T$ to make the dependence explicit.
We have
\begin{align*}
&\pbr{\max_{T\subset[p]} \frac{1}{\abs{T}\vee s}\norm{ \beta^*_{S^\complement \cap S^*}}^{-2}  \norm{X_T^T(I -P_S)X_{S^\complement \cap S^*} \beta^*_{S^\complement \cap S^*}}^2 \geq 16c^2n\log p }\\
& \leq \pbr{\xi \geq 4cn } + \pbr{ \max_{T\subset[p]} \frac{1}{\abs{T}\vee s} \zeta_T \geq 4c\log p}\\
& \leq \ebr{-cn}  + \sum_{m=0}^{p} \sum_{T\subset[p]:\abs{T}=m}\ebr{-c\br{m\wedge s}\log p}\\
& \leq \ebr{-c's\log p}.
\end{align*}
The proof is completed by an additional union bound argument over $S$.

\emph{Equation (\ref{eqn:re12}):} Consider a fixed pair $S,T$. For any $x\in\mathr^n$,  we have  $\br{P_{S^*}-P_S}x = P_{S,1}x - P_{S,2}x$, where $ P_{S,1}$ is the projection matrix onto the space $\text{span}(X_{S^*}) \setminus \br{\text{span}(X_{S^*}) \cap \text{span}(X_{S})}$, and $P_{S,2}$  is the projection matrix onto the space $\text{span}(X_{S}) \setminus \br{\text{span}(X_{S^*}) \cap \text{span}(X_{S})}$. Then we have
\begin{align*}
\norm{X_T^T(P_{S^*}-P_S) w }^2 & = \norm{X_T^T P_{S,1} w  - X_T^T P_{S,2} w }^2 \leq 2 \br{\norm{X_T^T P_{S,1} w }^2 + \norm{X_T^T P_{S,2} w }^2 }.
\end{align*}
Note that $\text{span}(X_{S^*\cap S}) \subset \text{span}(X_{S^*}) \cap \text{span}(X_{S})$, and thus the rank of $P_{S,1}$ is bounded by $\abs{S^*\cap S^\complement}$. Hence, $ \norm{X_T^T P_{S,1} w }^2  $ is stochastically dominated by $\xi\zeta$ where $\xi \sim \chi^2_{\abs{S^*\cap S^\complement}+\abs{S^{*\complement}\cap S}}  $, $\zeta \sim\chi^2_{\abs{T}}$ and $\xi\perp\zeta$. Note that $\xi$ only depends on $S$ and $\zeta$ only depends on $T$. For a fixed $S$, in order to take a union bound over $T$, we add a subscript to $\zeta$ as in $\zeta_T$ to make the dependence explicit. We have
\begin{align*}
&\pbr{\max_{T\subset [p]:T\cap \br{S\cup S^*} =\emptyset}\frac{1}{\abs{T}\vee s}\norm{X_T^T P_{S,1} w }^2  \geq 16c^2 \br{\abs{S^*\cap S^\complement}+\abs{S^{*\complement}\cap S}}\log^2 p} \\
& \leq \pbr{ \xi \geq 4c \br{\abs{S^*\cap S^\complement}+\abs{S^{*\complement}\cap S}} \log p} + \pbr{\max_{T\subset [p]}\frac{1}{\abs{T}\vee s} \zeta_T \geq cs\log p}\\
& \leq \ebr{-c \br{\abs{S^*\cap S^\complement}+\abs{S^{*\complement}\cap S}} \log p}   + \sum_{m=0}^{p} \sum_{T\subset[p]:\abs{T}=m}\ebr{-c\br{m\wedge s}\log p}\\
& \leq  \ebr{-c \br{\abs{S^*\cap S^\complement}+\abs{S^{*\complement}\cap S}} \log p}  + \ebr{-c's\log p}.
\end{align*}
Then we take a union bound of $S$.
\begin{align*}
&\pbr{\max_{S\subset [p]:\abs{S}\leq  2C_0s }\max_{T\subset [p]:T\cap \br{S\cup S^*} =\emptyset} \frac{1}{\abs{S^*\cap S^\complement}+\abs{S^{*\complement}\cap S}}\frac{1}{\abs{T}\vee s}\norm{X_T^T P_{S,1} w }^2  \geq 16c^2 \log^2 p} \\
& \leq \sum_{m'=0}^{ 2C_0s }\sum_{S\subset [p] :\abs{S^*\cap S^\complement}+\abs{S^{*\complement}\cap S}=m' } \br{ \ebr{-c m' \log p}  + \ebr{-c's\log p}}\\
& \leq \ebr{-c''\log p}.
\end{align*}
A similar result holds for the term related to $P_{S,2}$. Putting them together, we complete the proof.

%

%

\emph{Equation (\ref{eqn:re13}):} Define $B_j  =  \norm{X_j }^{-1} X_j ^TX_{S^*_{-j }} \br{X_{S^*_{-j }}^T \br{I - P_j }X_{S^*_{-j }}}^{-1}  X_{S^*_{-j }}^T X_j  \norm{X_j }^{-1}$ for all $j \in S^*$.  Note that 
$\norm{X_j }^{-1} X_j ^T X_{S^*_{-j }} \br{X_{S^*_{-j }}^T \br{I - P_j }X_{S^*_{-j }}}^{-1} X_{S^*_{-j }}^{T} \br{I - P_j } w$ is identically distributed by  $\sqrt{B_j }\xi_j
$ with $\xi_j \sim\mathn(0,1)$ and $\xi_j  \perp B_j $. Here we have the subscript for both $\xi_j$ and $B_j $ to make their dependence on $j $ explicit. Then,
\begin{align*}
& \pbr{\norm{X_j }^{-1} X_j ^T X_{S^*_{-j }} \br{X_{S^*_{-j }}^T \br{I - P_j }X_{S^*_{-j }}}^{-1} X_{S^*_{-j }}^{T} \br{I - P_j } w  \geq \sqrt{\frac{c^2 s\log^2 p}{n}}  } \\
&\leq  \pbr{\xi_j   \geq  \sqrt{c\log p}} + \pbr{B_j  \geq \frac{cs\log p}{n}},
\end{align*}
and thus
\begin{align*}
& \pbr{\max_{j \in S^*}\norm{X_j }^{-1} X_j ^T X_{S^*_{-j }} \br{X_{S^*_{-j }}^T \br{I - P_j }X_{S^*_{-j }}}^{-1} X_{S^*_{-j }}^{T} \br{I - P_j } w  \geq \sqrt{\frac{c^2 s\log^2 p}{n}}  } \\
& \leq  \pbr{\max_{j \in S^*}\xi_j   \geq  \sqrt{c\log p}}  + \pbr{\max_{j \in S^*}B_j  \geq \frac{cs\log p}{n}}.
\end{align*}
The first term can be easily bounded by $s\ebr{-2^{-1}c\log p} \leq \ebr{-c'\log p}$. For the second term, we have
\begin{align*}
B_j  \leq \norm{\br{X_{S^*_{-j }}^T \br{I - P_j }X_{S^*_{-j }}}^{-1}}\norm{X^T_{S^*_{-j }} X_j \norm{X_j }^{-1}}^2,
\end{align*}
for all $j\in S^*$.
By a similar analysis as in (\ref{eqn:re9}), we can show $ \max_{j \in S^*}\norm{\br{X_{S^*_{-j }}^T \br{I - P_j }X_{S^*_{-j }}}^{-1}} \leq c_1/n$ with probability at least $1- \ebr{-c_2n}$. Note that $\norm{X^T_{S^*_{-j }} X_j \norm{X_j }^{-1}}^2 \sim \chi^2_{s-1}$. Easily we can show $\max_{j \in S^*} \norm{X^T_{S^*_{-j }} X_j \norm{X_j }^{-1}}^2 \leq 4c_3 s\log p$ with probability at least $1-\ebr{-c_4s\log p}$. As a result,
\begin{align*}
\pbr{\max_{j \in S^*}B_j  \geq \frac{cs\log p}{n}} \leq \ebr{-c_2n} + \ebr{-c_4s\log p},
\end{align*}
which completes the proof.
\end{proof}


\begin{lemma}\label{prop:reg_rate_simplify}
Define
\begin{align}\label{eqn:tilde_psi_def}
\wt \psi(n,p,s,\lambda,\delta,C)=&s\mathbb{P}\left(\epsilon>(1-\delta)\|\zeta\|(\lambda-t(\zeta)) \; \& \: \abs{\norm{\zeta}^2 -n} \leq C\sqrt{n\log p}\right)\nonumber \\
&\quad +(p-s)\mathbb{P}\left(\epsilon>(1-\delta)\|\zeta\|t(\zeta)    \; \& \: \abs{\norm{\zeta}^2 -n} \leq C\sqrt{n\log p}  \right),
\end{align}
where $\epsilon\sim \mathcal{N}(0,1)$, $\zeta\sim \mathcal{N}(0,I_n)$, and they are independent of each other.
Assume $s\log p \leq n$, $\limsup s/p <1/2$, and  $\snr\rightarrow\infty$. For any $\delta \leq 1/\log p$ and any constant $C>0$, we have
$$\wt{\psi}(n,p,s,\lambda,\delta,C) = s\ebr{- \frac{(1+o(1))\snr^2}{2}}.$$
\end{lemma}
\begin{proof}
Throughout the proof, we use $g(x)$ and $G(x)$ for the density and survival functions of $\mathcal{N}(0,1)$. A standard Gaussian tail analysis gives
\begin{align}\label{eqn:gaussian_tail_bound}
\frac{1}{2x}g(x) \leq G(x) \leq \frac{1}{x}g(x),
\end{align}
for all $x\geq 2$. With a slight abuse of notation, we also use the notation
\begin{align}\label{eqn:t_u_def}
t(u)= \frac{\lambda}{2} + \frac{\log \frac{p-s}{s}}{\lambda u^2},
\end{align}
for all $u>0$.
We first focus on deriving an upper bound for $\wt \psi(n,p,s,\lambda,\delta,C)$.
For any $u>0$, we define
\begin{align*}
m(u) = u \br{\lambda - \frac{\log \frac{p-s}{s}}{\lambda u^2}} = \lambda u - \frac{\log \frac{p-s}{s}}{\lambda u}.
\end{align*}
Recall the definition of $t(u)$ in   (\ref{eqn:t_u_def}). Define $u_{\min} = \sqrt{n- C\sqrt{n\log p}}$ and $u_{\max} = \sqrt{n+C\sqrt{n\log p}}$, and $U=[u_{\min}, u_{\max}]$.
Since $u\br{\lambda -t(u)}$ is an increasing function of $u>0$. We have
\begin{align*}
m(u_{\min}) \leq u\br{\lambda -t(u)} \leq m(u_{\max}),
\end{align*}
for all $u\in U$.
This gives
\begin{align*}
&s\pbr{\frac{\epsilon}{1-\delta} \geq \norm{\zeta} \br{\lambda - t(\zeta)}\;\&\; \abs{\norm{\zeta} -n}\leq C\sqrt{n\log p}}  \\
&\leq s\pbr{\frac{\epsilon}{1-\delta} \geq m(u_{\min}) } \\
&\leq  \frac{s}{\sqrt{2\pi} \br{1-\delta}m(u_{\min}) } \ebr{- \frac{1}{2} \br{1-\delta}^2m^2(u_{\min}) },
\end{align*}
where the last inequality is by   (\ref{eqn:gaussian_tail_bound}). In addition, we have the identity $\br{ut(u) }^2 = 2\log \frac{p-s}{s} + m^2(u)$. This leads to
\begin{align*}
2\log \frac{p-s}{s} + m^2(u_{\min}) \leq \br{ut(u) }^2 \leq 2\log \frac{p-s}{s} + m^2(u_{\max}),
\end{align*}
for all $u\in U$.
As a result, using   (\ref{eqn:gaussian_tail_bound}), 
we have 
\begin{align*}
& \br{p-s}\pbr{\frac{\epsilon}{1-\delta} \geq \norm{\zeta}t(\zeta)\;\&\; \abs{\norm{\zeta} -n}\leq C\sqrt{n\log p}} \\
& = \br{p-s} \E_{u^2\sim \chi^2_n} \left[ \pbr{\frac{\epsilon}{1-\delta} \geq |u| t(u)\Big | u} \indc{\abs{u^2 -n} \leq C\sqrt{n\log p}}\right]\\
& \leq \frac{p-s}{\sqrt{2\pi }}\E_{u^2\sim \chi^2_n} \left[ \frac{1}{ut(u)}\ebr{-\frac{1}{2}\br{1-\delta}^2 \br{ut(u) }^2 }\indc{\abs{u^2 -n} \leq C\sqrt{n\log p}}\right]\\
& \leq \frac{p-s}{\sqrt{2\pi }}\E_{u^2\sim \chi^2_n} \left[  \frac{1}{\min_{u \in U} u t(u)}\ebr{-\frac{1}{2}\br{1-\delta}^2 \br{2\log \frac{p-s}{s}  + m^2(u_{\min})} }\indc{\abs{u^2 -n} \leq C\sqrt{n\log p}}\right]\\
& \leq  \frac{p-s}{\sqrt{2\pi } \min_{u \in U} u t(u)} \ebr{-\br{1-\delta}^2 \log \frac{p-s}{s} -\frac{1}{2}\br{1-\delta}^2 m^2(u_{\min}) }\\
&  = \frac{s}{\sqrt{2\pi } \min_{u \in U} u t(u) } \ebr{\br{2\delta -\delta^2} \log \frac{p-s}{s} -\frac{1}{2}\br{1-\delta}^2 m^2(u_{\min}) }.
\end{align*}
Combining the above results together, we have
\begin{align*}
\wt \psi(n,p,s,\lambda,\delta,C) &\leq   \frac{s}{\sqrt{2\pi} \br{1-\delta}m(u_{\min}) } \ebr{- \frac{1}{2} \br{1-\delta}^2m^2(u_{\min}) } \\
& \quad + \frac{s}{\sqrt{2\pi } \min_{u \in U} u t(u) } \ebr{\br{2\delta -\delta^2} \log \frac{p-s}{s} -\frac{1}{2}\br{1-\delta}^2 m^2(u_{\max}) }.
\end{align*}


Now we derivative a lower bound for $\wt \psi(n,p,s,\lambda,\delta,C)$. Note that $\pbr{ \abs{\norm{\zeta} -n}\leq C\sqrt{n\log p}}\geq 1/2$. We therefore have
\begin{align*}
& s\pbr{\epsilon \geq \norm{\zeta} \br{\lambda - t(\zeta)}\;\&\; \abs{\norm{\zeta} -n}\leq C\sqrt{n\log p}} \\
& \geq \frac{s}{2} \pbr{\epsilon \geq m(u_{\max})} \\
& \geq  \frac{s}{4\sqrt{2\pi}  m(u_{\max})} \ebr{- \frac{1}{2} m^2(u_{\max})},
\end{align*}
and
\begin{align*}
& \br{p-s}\pbr{\epsilon \geq \norm{\zeta}t(\zeta)\;\&\; \abs{\norm{\zeta} -n}\leq C\sqrt{n\log p}} \\
& \geq  \frac{s}{4\sqrt{2\pi }\max_{u\in U} ut(u)} \ebr{ \br{2\delta -\delta^2} \log \frac{p-s}{s} -\frac{1}{2}\br{1-\delta}^2 m^2(u_{\max}) }.
\end{align*}
Consequently,
\begin{align*}
\wt \psi(n,p,s,\lambda,\delta,C) & \geq  \frac{s}{4\sqrt{2\pi}  m(u_{\max})} \ebr{- \frac{1}{2} m^2(u_{\max})} \\
&\quad +  \frac{s}{4\sqrt{2\pi }\max_{u\in U} ut(u)} \ebr{ \br{2\delta -\delta^2} \log \frac{p-s}{s} -\frac{1}{2}\br{1-\delta}^2 m^2(u_{\max}) }.
\end{align*}

Since $\delta \leq 1/\log p$ and $\snr\rightarrow\infty$, with the same arguments used in the proof of Lemma \ref{prop:useful-reg}, we can show for all $u\in U$, we have $\frac{\lambda u}{2} - \frac{\log \frac{p-s}{s}}{\lambda u} =\br{1+o(1)} \snr$ and  $\frac{\lambda u}{2} + \frac{\log \frac{p-s}{s}}{\lambda u} \rightarrow \infty$.
This leads to $\wt \psi(n,p,s,\lambda,\delta,C) =  s\ebr{- \frac{(1+o(1))\snr^2}{2}}$ as desired, which completes the proof.
\end{proof}

\begin{lemma}\label{prop:useful-reg}
Consider some $\beta^*\in\mathbb{R}^p$ that satisfies either $|\beta_j^*|\geq\lambda$ or $\beta_j^*=0$ for all $j\in[p]$.
Assume $\limsup {s}/{p}<\frac{1}{2}$ and $\snr\rightarrow\infty$. Then, for i.i.d. $X_1,\cdots,X_p\sim \mathcal{N}(0,I_n)$, we have
\begin{eqnarray*}
\min_{j\in[p]}\sqrt{n}||\beta_j^*|-t(X_j)| &>& 1, \\
\max_{j\in[p]}\frac{|\beta_j^*|}{||\beta_j^*|-t(X_j)|} &\leq& \sqrt{\log p},
\end{eqnarray*}
with probability at least $1-e^{-p}$.
\end{lemma}
\begin{proof}  
We first show that under the assumption that $\limsup s/p <1/2$,  the condition $\snr\rightarrow\infty$ is equivalent to
\begin{align}\label{eqn:lambda_def_2}
 \frac{n\lambda^2 -  2\log \frac{p-s}{s}}{ \sqrt{\log \frac{p-s}{s}}} \rightarrow\infty.
\end{align} 
A direct calculation gives
\begin{align*}
\br{\frac{\lambda \sqrt{n}}{2} - \frac{\log \frac{p-s}{s}}{\lambda \sqrt{n}} }^2 = \frac{\br{n\lambda^2 -2\log \frac{p-s}{s}}^2}{4n\lambda^2} = \br{\frac{n\lambda^2 -2\log \frac{p-s}{s}}{\sqrt{\log \frac{p-s}{s}}}}^2 \frac{\log \frac{p-s}{s}}{4n\lambda^2}.
\end{align*}
If  $\snr\rightarrow\infty$ holds, we have $n\lambda^2 \geq 2\log \frac{p-s}{s}$, which leads to   (\ref{eqn:lambda_def_2}).  For the other direction, if   (\ref{eqn:lambda_def_2}) holds, there exists some $A\rightarrow\infty$ such that $n\lambda^2 =  2\log \frac{p-s}{s} + A\sqrt{\log \frac{p-s}{s}}$. By the above identity, we have
\begin{align*}
\br{\frac{\lambda \sqrt{n}}{2} - \frac{\log \frac{p-s}{s}}{\lambda \sqrt{n}} }^2 = A^2 \frac{\log \frac{p-s}{s}}{2\br{ 2\log \frac{p-s}{s} + A\sqrt{\log \frac{p-s}{s}}}} \rightarrow\infty.
\end{align*}
Thus we have shown that $\snr\rightarrow\infty$ and (\ref{eqn:lambda_def_2}) are equivalent.

Now we are going to prove the proposition under the high-probability event (\ref{eqn:re0}). Note that for any $j\in[p]$ such that $\beta_j^* =0$, we have $\sqrt{n}|\beta_j^*|-t(X_j)| = \sqrt{n}t(X_j) \geq \sqrt{n}\lambda/2 \rightarrow\infty $ by (\ref{eqn:lambda_def_2}) and $|\beta_j^*|/||\beta_j^*|-t(X_j)| =0$. Thus, we only need to consider the remaining $j\in[p]$ such that $\beta^*_j \neq 0$. It is sufficient to prove $\min_{j\in[p]:z_j^*\neq 0} \sqrt{n}\br{\lambda -t(X_j)} >1$ and $\max_{j\in[p]:z_j^*\neq 0} \frac{\lambda}{\lambda -t(X_j)} \leq \sqrt{\log p}$.
Consider any $j\in[p]$ such that $z^*_j \neq 0$. We have
\begin{align*}
\sqrt{n} \br{\lambda - t(X_j)} & = \frac{\sqrt{n}}{2\lambda} \br{\lambda^2 - 2\frac{\log \frac{p-s}{s}}{\norm{X_j}^2}} \geq \frac{\sqrt{n}}{2\lambda} \br{\lambda^2 - 2\frac{\log \frac{p-s}{s}}{n-C\sqrt{n\log p}}} ,
\end{align*}
where the last inequality is by   (\ref{eqn:re0}). By  (\ref{eqn:lambda_def_2}), there exists an $A\rightarrow\infty$, such that
\begin{align*}
n\lambda^2 =  2\log \frac{p-s}{s} + A\sqrt{\log \frac{p-s}{s}}.
\end{align*}
Then, we have
\begin{align*}
\sqrt{n} \br{\lambda - t(X_j)} &\geq \frac{1}{2\sqrt{n}\lambda } \br{ 2\log \frac{p-s}{s} + A\sqrt{\log \frac{p-s}{s}}-\frac{2\log \frac{p-s}{s}}{1-C\sqrt{\frac{\log p}{n}}} } \\
& \geq  \frac{1}{2\sqrt{n}\lambda } \br{ A\sqrt{\log \frac{p-s}{s}} - C'\sqrt{\frac{\log p}{n}}\log \frac{p-s}{s}}\\
& \geq \frac{C''A\sqrt{\log \frac{p-s}{s}} }{\sqrt{n}\lambda},
\end{align*}
for some constants $C',C''>0$.  Starting from here, first we have 
\begin{align*}
\sqrt{n} \br{\lambda - t(X_j)}  \geq C''\sqrt{\frac{A^2 \log \frac{p-s}{s}}{n\lambda^2}} =  C'' \sqrt{\frac{A^2 \log \frac{p-s}{s}}{ 2\log \frac{p-s}{s} + A\sqrt{\log \frac{p-s}{s}}}}\rightarrow\infty,
\end{align*}
as $A\rightarrow\infty$. Second, we have 
\begin{align*}
\frac{\lambda}{\lambda - t(X_j)}  &= \frac{\sqrt{n}\lambda}{\sqrt{n}\br{\lambda - t(X_j)}} \leq   \frac{\sqrt{n}\lambda}{\frac{C''A\sqrt{\log \frac{p-s}{s}} }{\sqrt{n}\lambda}} = \frac{n\lambda^2}{C'' A \sqrt{\log \frac{p-s}{s}}} \\
& =  \frac{ 2\log \frac{p-s}{s} + A\sqrt{\log \frac{p-s}{s}}}{C'' A \sqrt{\log \frac{p-s}{s}}} \leq o\br{\sqrt{\log \frac{p-s}{s}}}.
\end{align*}
Hence, the proof is complete.
\end{proof}

Now we are ready to prove Theorem \ref{thm:lower-regression}, Theorem \ref{thm:regression-iter}, and Corollary \ref{cor:regression-iter}.
\begin{proof}[Proof of Theorem \ref{thm:lower-regression}]
By \cite{ndaoud2018optimal}, we have
$$\inf_{\wh{z}}\sup_{z^*\in\mathcal{Z}_s}\sup_{\beta^*\in\mathcal{B}_{z^*,\lambda}}\mathbb{E}\h_{(s)}(\wh{z},z^*)\geq \frac{1}{2s}\psi(n,p,s,\lambda,0)-4e^{-s/8},$$
where $\psi(n,p,s,\lambda,0)$ is defined in (\ref{eqn:psi_def}). By Lemma \ref{prop:reg_rate_simplify},
$$\frac{1}{2s}\psi(n,p,s,\lambda,0)\geq \frac{1}{2s}\wt{\psi}(n,p,s,\lambda,0,C)=\exp\left(-\frac{(1+o(1))\snr^2}{2}\right),$$
and we obtain the desired conclusion.
\end{proof}

\begin{proof}[Proof of Theorem \ref{thm:regression-iter}]
The condition of Theorem \ref{thm:regression-iter} allows us to apply Lemma \ref{prop:useful-reg} to the conclusion of Lemma \ref{lem:regression-error}. This implies that the right hand sides of (\ref{eq:regression-error-1}) and (\ref{eq:regression-error-2}) can be bounded by $o((\log p)^{-1})$, which then implies Conditions A-C hold with some $\delta=o((\log p)^{-1})$. Then, the desired conclusion is a special case of Theorem \ref{thm:main}.
\end{proof}

\begin{proof}[Proof of Corollary \ref{cor:regression-iter}]
By (\ref{eq:relation-H_s-l}) and (\ref{eqn:re0}), we have
$$\h_{(s)}(z,z^*)\leq \frac{\ell(z,z^*)}{s\lambda^2\min_{j\in[p]}\|X_j\|^2}\leq \frac{2\ell(z,z^*)}{sn\lambda^2},$$
with high probability. Then, the conclusion is a direct consequence of Theorem \ref{thm:regression-iter}.
\end{proof}

Finally, we present the proofs of Lemma \ref{lem:regression-error}, Lemma \ref{lem:regression-ideal}, and Proposition \ref{prop:reg_init}.

\begin{proof}[Proof of Lemma \ref{lem:regression-error}]
The proof will be established under the high-probability events (\ref{eqn:re0})-(\ref{eqn:re13}).
First we present  a few important quantities closely related to $\ell(z,z^*)$. By $h(z,z^*) \leq \frac{\ell(z,z^*)}{\lambda^2 \min_j \norm{X_j}^2}$ and (\ref{eqn:re0}), we have
\begin{align}
h(z,z^*) \leq \frac{2\ell(z,z^*)}{n\lambda^2}.\label{eqn:reg_h_explicit}
\end{align}
By the definition of $\ell(z,z^*)$, it is obvious
\begin{align}
\sum_{j=1}^p \beta^{*2}_j\indc{\abs{z_j} \neq \abs{z_j^*}}\leq  \frac{\ell(z,z^*)}{\min_j \norm{X_j}^2} \leq \frac{2\ell(z,z^*)}{n},\label{eqn:reg_beta_s}
\end{align}
where we have used  (\ref{eqn:re0}) again.
For any $z\in\cbr{0,1,-1}^p$ such that $\ell(z,z^*)\leq \tau \leq  C_0sn\lambda^2$, (\ref{eqn:reg_h_explicit}) implies 
\begin{align}\label{eqn:h_z_2s}
h(z,z^*)\leq 2C_0s.
\end{align}
We will first prove the easier conclusion (\ref{eq:regression-error-2}) and then prove (\ref{eq:regression-error-1}).

\paragraph{Proof of (\ref{eq:regression-error-2}).}
According to the definition, we have
\begin{align*}
\abs{\frac{1}{\norm{X_j}}\sum_{l\in[p]\backslash\{j\}}\left(\wh{\beta}_l(z^*)-\beta^*\right)X_j^TX_l} & = \abs{ \frac{1}{\norm{X_j} }X_j^T X \wh \beta(z^*)-  \frac{1}{\norm{X_j} }X_j^T  X \beta^* - \norm{X_j} \br{\wh \beta_j(z^*) - \beta_j^*}}\\
& =\abs{ \frac{1}{\norm{X_j} }X_j^T \br{X \wh \beta(z^*)-   X \beta^*} - \norm{X_j} \br{\wh \beta_j(z^*) - \beta_j^*}}.
\end{align*}
By the fact that $\wh \beta_{S^*}(z^*) = \beta_{S^*} + \br{X_{S^*}^TX_{S^*}}^{-1}X_{S^*}^T w $ and $\wh \beta_{S^{*\complement}}(z^*)=0$, we have 
\begin{align}\label{eqn:reg_54}
\abs{\frac{1}{\norm{X_j}}\sum_{l\in[p]\backslash\{j\}}\left(\wh{\beta}_l(z^*)-\beta^*\right)X_j^TX_l}=\begin{cases}
\frac{1}{\norm{X_j}} X_j^T  w  - \norm{X_j}\sbr{\br{X_{S^*}^T X_{S^*}}^{-1} X_{S^*}^T w }_{\phi_{S^*}(j)} & j\in S^*\\
\frac{1}{\norm{X_j}}X_j^T P_{S^*} w & j\notin S^*.
\end{cases}
\end{align}
\begin{itemize}
\item We first consider $j\notin S^{*}$. By (\ref{eqn:re0}),  (\ref{eqn:re7}) and (\ref{eqn:reg_54}), we have
\begin{align*}
\max_{j\notin S^*} \abs{\frac{1}{\norm{X_j}}\sum_{l\in[p]\backslash\{j\}}\left(\wh{\beta}_l(z^*)-\beta^*\right)X_j^TX_l}  \leq \sqrt{\frac{2Cs\log p}{n}}.
\end{align*}
\item Next, we consider $j\in S^{*}$. Writing $X_{S^*}$ into a block matrix form $X_{S^*}=(X_j,X_{S^*_{-j}})$ , we have a block matrix inverse formula
\begin{align}\label{eqn:block_matrix_B}
\br{X_{S^*}^T X_{S^*}}^{-1} = \begin{pmatrix}
\norm{X_j}^2 & X_j^T X_{S^*_{-j}}   \\
X_{S^*_{-j}}^T X_j & X_{S^*_{-j}}^TX_{S^*_{-j}}
\end{pmatrix}^{-1} = \begin{pmatrix}
B_{11} & B_{12}\\
B_{21} & B_{22}
\end{pmatrix},
\end{align}
with
\begin{align*}
&B_{11} = \norm{X_j}^{-2} + \norm{X_j}^{-2}  X_j^T X_{S^*_{-j}} \br{X_{S^*_{-j}}^T \br{I - P_j}X_{S^*_{-j}}}^{-1} X_{S^*_{-j}}^{T}X_j \norm{X_j}^{-2},\\
& B_{12}   = -\norm{X_j}^{-2}  X_j^T X_{S^*_{-j}} \br{X_{S^*_{-j}}^T \br{I - P_j}X_{S^*_{-j}}}^{-1},
\end{align*}
and the explicit expressions of $B_{21},B_{22}$ not displayed since they are irrelevant to our proof.
We have $[\br{X_{S^*}^T X_{S^*}}^{-1} X_{S^*}^T w ]_{\phi_{S^*}(j)} = B_{11} X_j^T w + B_{12}X_{S^*_{-j}}^Tw$. From (\ref{eqn:reg_54}),
some algebra leads to 
\begin{align}
&\max_{j\in S^*} \abs{\frac{1}{\norm{X_j}}\sum_{l\in[p]\backslash\{j\}}\left(\wh{\beta}_l(z^*)-\beta^*\right)X_j^TX_l} \nonumber\\
& =\max_{j\in S^*} \abs{\frac{1}{\norm{X_j}} X_j^T  w  - \norm{X_j}\sbr{\br{X_{S^*}^T X_{S^*}}^{-1} X_{S^*}^T w }_{\phi_{S^*}(j)}} \nonumber \\
 & = \max_{j\in S^*}  \abs{\norm{X_j}^{-1} X_j^T X_{S^*_{-j}} \br{X_{S^*_{-j}}^T \br{I - P_j}X_{S^*_{-j}}}^{-1} X_{S^*_{-j}}^{T} \br{I - P_j} w  }\nonumber \\
 &\leq \sqrt{\frac{Cs\log^2 p}{n}},\label{eqn:eqn_54}
\end{align}
where the last inequality is by   (\ref{eqn:re13}).
\end{itemize}

Combining the two cases, we have
\begin{align*}
\max_{j\in[p]} \abs{\frac{1}{\norm{X_j}}\sum_{l\in[p]\backslash\{j\}}\left(\wh{\beta}_l(z^*)-\beta^*\right)X_j^TX_l} \leq \sqrt{\frac{2Cs\log^2 p}{n}}.
\end{align*}
Using (\ref{eqn:reg_delta}),  we get $\abs{\Delta_j(z_j^*,b)^2} \geq 4t(X_j) \abs{\abs{\beta_j^*} -  t(X_j)}\norm{X_j}^2$ for all $j\in[p]$. Then,
\begin{align*}
\max_{j\in[p]}\abs{\frac{H_{j}(z_j^*,b)}{\Delta_j(z_j^*,b)^2}} &\leq \max_{j\in[p]}\frac{4 \norm{X_j} t(X_j) \abs{{\norm{X_j}}^{-1}\sum_{l\in[p]\backslash\{j\}}\left(\wh{\beta}_l(z^*)-\beta^*\right)X_j^TX_l}}{4t(X_j)\abs{\abs{\beta_j^*} -  t(X_j)}\norm{X_j}^2}  \\
& \leq  \max_{j\in[p]}\frac{ \abs{{\norm{X_j}}^{-1}\sum_{l\in[p]\backslash\{j\}}\left(\wh{\beta}_l(z^*)-\beta^*\right)X_j^TX_l}}{\abs{\abs{\beta_j^*} -  t(X_j)}\norm{X_j}} \\
& \leq  \sqrt{\frac{2Cs\log^2p}{n}}  \frac{1}{\min_{j\in[p]}\abs{\abs{\beta_j^*} -  t(X_j)} \norm{X_j}} \\
& \leq \sqrt{\frac{4Cs\log^2p}{n}}\frac{1}{\min_{j\in[p]}\sqrt{n}\abs{\abs{\beta_j^*} -  t(X_j)}},
\end{align*}
where in the last  inequality we use   (\ref{eqn:re0}).



\paragraph{Proof of (\ref{eq:regression-error-1}).}
By (\ref{eqn:reg_mu_diff}) and (\ref{eqn:reg_delta}), we have 
\begin{align*}
&\frac{G_j(z_j^*,b;z)^2\|\mu_j(B^*,b)-\mu_j(B^*,z_j^*)\|^2}{\Delta_j(z_j^*,b)^4\ell(z,z^*)} \\
& \leq  \frac{\br{4\|X_j \|t(X_j )\br{{\norm{X_j}}^{-1}\sum_{l\in[p]\backslash\{j\}}\left(\wh{\beta}_l(z)-\wh{\beta}_l(z^*)\right)X_j^TX_l}}^2  \br{4\abs{\beta^*_j}^2 \norm{X_j}^2 \indc{z^*_j = \pm 1} + \lambda^2\norm{X_j}^2 \indc{z^*_j = 0} }}{\br{\br{4t(X_j) \br{\abs{\beta^*_j} - t(X_j)}\norm{X_j}^2}^2\indc{z^*_j = \pm 1}    + \br{4t(X_j)^2\norm{X_j}^2}^2 \indc{z^*_j = 0}  } \ell(z,z^*)}\\
& = \frac{\br{{\norm{X_j}}^{-1}\sum_{l\in[p]\backslash\{j\}}\left(\wh{\beta}_l(z)-\wh{\beta}_l(z^*)\right)X_j^TX_l}^2}{\ell(z,z^*)} \br{ 4\br{\frac{\abs{\beta^*_j}}{\abs{\beta^*_j} - t(X_j)}}^2 \indc{z^*_j = \pm 1} + \br{\frac{\lambda}{t(X_j)}}^2\indc{z^*_j = 0}}\\
& \leq  \frac{\br{{\norm{X_j}}^{-1}\sum_{l\in[p]\backslash\{j\}}\left(\wh{\beta}_l(z)-\wh{\beta}_l(z^*)\right)X_j^TX_l}^2}{\ell(z,z^*)}  \max\cbr{4\br{\frac{\abs{\beta^*_j}}{\abs{\beta^*_j} - t(X_j)}}^2,\br{\frac{\lambda}{t(X_j)}}^2}.
\end{align*}
Define
$$\wt \beta_j (z) =   \|X_j \|^{-2} X_j ^T \br{Y - \sum_{l\in[p]\backslash\{j\}}X_l \wh \beta_l(z)}.$$
We then have
\begin{align*}
& {\norm{X_j}}^{-1}\sum_{l\in[p]\backslash\{j\}}\left(\wh{\beta}_l(z)-\wh{\beta}_l(z^*)\right)X_j^TX_l \\
 =&  -\|X_j \|  \br{ \|X_j \|^{-2} X_j ^T \br{Y - \sum_{l\in[p]\backslash\{j\}}X_l \wh \beta_l(z)} - \|X_j \|^{-2} X_j ^T\br{Y -\sum_{l\in[p]\backslash\{j\}} X_l\wh \beta_l(z^*) } } \\
 =&-\|X_j \| \br{\wt \beta_j (z) - \wt \beta_j (z^*) }.
\end{align*}
Hence, 
we have
\begin{align*}
&\max_{b\in\cbr{-1,0,1}\setminus \{ z^*_j\}}\frac{G_j(z_j^*,b;z)^2\|\mu_j(B^*,b)-\mu_j(B^*,z_j^*)\|^2}{\Delta_j(z_j^*,b)^4\ell(z,z^*)} \\
& \leq \max\cbr{4\br{\frac{\abs{\beta^*_j}}{\abs{\beta^*_j} - t(X_j)}}^2,\br{\frac{\lambda}{t(X_j)}}^2} \frac{\norm{X_j}^2 \br{\wt \beta_j (z) - \wt \beta_j (z^*)}^2}{\ell(z,z^*)}\\
& \leq\max\cbr{4\br{\frac{\abs{\beta^*_j}}{\abs{\beta^*_j} - t(X_j)}}^2,\br{\frac{\lambda}{t(X_j)}}^2} \frac{2n \br{\wt \beta_j (z) - \wt \beta_j (z^*)}^2}{\ell(z,z^*)},
\end{align*}
where the last inequality is by   (\ref{eqn:re0}). Therefore, 
\begin{align}
&\max_{T\subset [p]} \frac{1}{s+\abs{T}}  \sum_{j\in T} \max_{b\in\cbr{-1,0,1}\setminus \{ z^*_j\}} \frac{G_j(z_j^*,b;z)^2\|\mu_j(B^*,b)-\mu_j(B^*,z_j^*)\|^2}{\Delta_j(z_j^*,b)^4\ell(z,z^*)} \nonumber \\
& \leq  \max_{j\in[p]}\max\cbr{4\br{\frac{\abs{\beta^*_j}}{\abs{\beta^*_j} - t(X_j)}}^2,\br{\frac{\lambda}{t(X_j)}}^2} \frac{2n}{\ell(z,z^*)}  \max_{T\subset [p]} \frac{1}{s+\abs{T}}   \sum_{j\in T} \br{\wt \beta_j (z) - \wt \beta_j (z^*)}^2,\label{eqn:reg_proof_main}
\end{align}
which is all about studying $ \max_{T\subset [p]} \frac{1}{s+\abs{T}}   \sum_{j\in T} (\wt \beta_j (z) - \wt \beta_j (z^*))^2$.

In the following, we will focus on understanding $\wt \beta_j (z) - \wt \beta_j (z^*)$ for different $j\in[p]$.
Define
\begin{align*}
S^* = \cbr{j\in[p]:z^*_j  \neq 0},\quad \text{and}\quad S(z) = \cbr{j\in[p]:z_j  \neq 0}.
\end{align*}
For simplicity of notation, we just write $S$ instead of $S(z)$ from now on. Recall that $\wh \beta(z)$ is the least square estimator on the support $S$. That is,
\begin{align*}
\wh \beta_S(z)= \br{X_S^TX_S}^{T}X_S^Ty,\quad \text{and}\quad\wh \beta_{S^\complement}(z)=0.
\end{align*}
Thus, the explicit expression of   $\wt \beta_j (z)$ is given by
\begin{align*}
\wt \beta_j (z) =
\begin{cases}
\beta_j^*  + \left[\br{X_S^TX_S}^{-1}X_S^TX_{S_1} \beta^*_{S_1}\right]_{\phi_S(j)}   +\left[\br{X_S^TX_S}^{-1}X_S^Tw\right]_{\phi_S(j)} & j\in S\\
  \frac{1}{\norm{X_j}^2}X_j^T \sbr{(I -P_S)X_{S_1} \beta_{S_1}^*   + (I-P_S)w }  & j\notin S.
\end{cases}
\end{align*}
Similarly, we also have
\begin{align*}
\wt \beta_j(z^*) = \begin{cases}
\beta_j^*   +\left[\br{X_{S^*}^TX_{S^*}}^{-1}X_{S^*}^T w\right]_{\phi_{S^*}(j)}  & j\in S^*\\
\frac{1}{\norm{X_j}^2}X_j^T (I-P_{S^*}) w & j\notin S^*.
\end{cases}
\end{align*}
The analysis of $\wt \beta_j (z^*) - \wt \beta_j (z)$ will be studied in four different regimes.
We divide $[p]$ into four disjoint sets,
\begin{align*}
S_1 = S^{*} \cap S^\complement,\quad S_2 = S^{*} \cap S,\quad S_3 = S^{*\complement} \cap S,\quad \text{and}\quad S_4 = S^{*\complement} \cap S^\complement.
\end{align*}
Note that by   (\ref{eqn:h_z_2s}), we have
\begin{align}
\abs{S_1} + \abs{S_3} = h(z,z^*)\leq 2C_0s.\label{eqn:reg_S_1_S_3}
\end{align}
We denote $\mathsf{X}_l = X_{S_l},l=1,2,3,4$ for simplicity. We also denote $\mathsf{P}_l = \mathsf{X}_l \br{\mathsf{X}_l^T \mathsf{X}_l}^{-1}\mathsf{X}_l^T$ to be the projection matrix onto the subspace $\text{span}(\mathsf{X}_l),$ for $l=1,2,3,4$. 

~\\
\emph{(1) Regime $j\in S_1$.} In this case, we have $\wt \beta_j(z^*) = \beta_j^*   +\left[\br{X_{S^*}^TX_{S^*}}^{-1}X_{S^*}^T w \right]_{\phi_{S^*}(j)}$. We can also write
\begin{align*}
\wt \beta_j(z) & = \frac{1}{\norm{X_j}^2}X_j^T \sbr{(I -P_S)X_{S_1} \beta_{S_1}^*   + (I-P_S) w  }\\
& =  \frac{1}{\norm{X_j}^2}X_j^T (I -P_S) X_j \beta_j^* + \sum_{l\in S_1,l\neq j} \frac{1}{\norm{X_j}^2}X_j^T (I -P_S) X_l \beta_l ^* +  \frac{1}{\norm{X_j}^2}X_j^T(I-P_S) w \\
&= \beta_j^* -   \frac{1}{\norm{X_j}^2}X_j^TP_S X_j \beta_j^* + \sum_{l\in S_1,l\neq j} \frac{1}{\norm{X_j}^2}X_j^T (I -P_S) X_l \beta_l ^* +  \frac{1}{\norm{X_j}^2}X_j^T(I-P_S) w .
\end{align*}
This leads to the decomposition
\begin{align*}
\wt \beta_j(z) -\wt \beta_j(z^*)   & = -   \frac{1}{\norm{X_j}^2}X_j^TP_S X_j \beta_j^* + \sum_{l\in S_1,l\neq j} \frac{1}{\norm{X_j}^2}X_j^T (I -P_S) X_l \beta_l ^* -  \frac{1}{\norm{X_j}^2}X_j^TP_S w  \\
&\quad  + \br{\frac{1}{\norm{X_j}^2}X_j^T w- \left[\br{X_{S^*}^TX_{S^*}}^{-1}X_{S^*}^T w \right]_{\phi_{S^*}(j)}}.
\end{align*}
We will bound each term on the right hand side of the above equation. 

\emph{(1.1).}
First, we have
\begin{align*}
 \abs{ -   \frac{1}{\norm{X_j}^2}X_j^TP_S X_j \beta_j^*} & \leq \frac{2}{\min_{j} \norm{X_j}^2} \max_{j\in S_1} \abs{X_j^TP_S X_j } \abs{\beta_j^*}\\
 & \leq \frac{2}{n} Cs\abs{\beta_j^*}\log p ,
\end{align*}
where the last inequality is by  (\ref{eqn:re0}) and (\ref{eqn:re3_0}). Then
\begin{align*}
\sum_{j\in S_1}\br{ -   \frac{1}{\norm{X_j}^2}X_j^TP_S X_j \beta_j^*}^2 \leq \frac{4C^2s^2 \log^2 p}{n^2} \norm{\beta^*_{S_1}}^2.
\end{align*}


\emph{(1.2).} For the second term, we have a matrix representation,
\begin{align*}
& \sum_{j\in S_1} \br{\sum_{l\in S_1,l\neq j} \frac{1}{\norm{X_j}^2}X_j^T (I -P_S) X_l \beta_l ^*}^2 \\
& \leq \frac{1}{\min_{j} \norm{X_j}^4} \norm{ \br{X_{S_1}^T (I-P_S) X_{S_1}  - \text{diag}\cbr{X_{S_1}^T (I-P_S) X_{S_1}  }}\beta^*_{S_1}  }^2\\
&\leq \frac{1}{\min_{j} \norm{X_j}^4}  \norm{  X_{S_1}^T (I-P_S) X_{S_1}  - \text{diag}\cbr{X_{S_1}^T (I-P_S) X_{S_1}  }  }^2 \norm{\beta^*_{S_1}}^2\\
& \leq \frac{2}{n^2} C s\log p \norm{\beta^*_{S_1}}^2.
\end{align*}
where the last inequality is by  (\ref{eqn:re0}) and (\ref{eqn:re3}).

\emph{(1.3).} For the third term, we have
\begin{align*}
\sum_{j\in S_1} \br{- \frac{1}{\norm{X_j}^2}X_j^TP_S w }^2 &\leq   \frac{1}{\min_{j} \norm{X_j}^4}  \sum_{j\in S_1} \br{X_j^T P_S  w }^2 \\
& \leq  \frac{2}{n^2} Cs\log^2 p \br{\abs{S_1} + \abs{S_3}}\\
& =  \frac{2}{n^2} Csh(z,z^*)\log^2 p  ,
\end{align*}
where the second to the  last inequality is by  (\ref{eqn:re0}) and (\ref{eqn:re4}).

\emph{(1.4).} For the last  term, using (\ref{eqn:re0}) and (\ref{eqn:eqn_54}), we have
\begin{align*}
&\max_{j\in S^*} \abs{\frac{1}{\norm{X_j}^2}X_j^T w- \left[\br{X_{S^*}^TX_{S^*}}^{-1}X_{S^*}^T w \right]_{\phi_{S^*}(j)}}  \\
&\quad  \leq \max_{j\in S} \frac{1}{\norm{X_j}} \max_{j\in S^*} \abs{\frac{1}{\norm{X_j}}X_j^T w-  \norm{X_j}\left[\br{X_{S^*}^TX_{S^*}}^{-1}X_{S^*}^T w \right]_{\phi_{S^*}(j)}} \\
& \leq \sqrt{\frac{2}{n}}\sqrt{\frac{Cs\log^2 p}{n}}.
\end{align*}
Hence,
\begin{align*}
&\sum_{j\in S_1} \br{\frac{1}{\norm{X_j}^2}X_j^T w - \left[\br{X_{S^*}^TX_{S^*}}^{-1}X_{S^*}^T w \right]_{\phi_{S^*}(j)}}^2 \\
&\leq \abs{S_1}\max_{j\in S^*} \br{\frac{1}{\norm{X_j}^2}X_j^T w- \left[\br{X_{S^*}^TX_{S^*}}^{-1}X_{S^*}^T w \right]_{\phi_{S^*}(j)}}^2 \\
& \leq \frac{2Cs\abs{S_1} \log^2 p}{n^2}\\
&\leq \frac{2Csh(z,z^*) \log^2 p}{n^2},
\end{align*}
where the last inequality is due to (\ref{eqn:reg_S_1_S_3}).

\emph{(1.5).}
Combining the above results, we have
\begin{align*}
&\sum_{j\in S_2}\br{\wt \beta_j(z) -\wt \beta_j(z^*) }^2 \\
&\leq  4\br{\frac{4C^2s^2 \log^2 p}{n^2} \norm{\beta^*_{S_1}}^2  + \frac{2}{n^2} C s\log p \norm{\beta^*_{S_1}}^2    +   \frac{2}{n^2} Csh(z,z^*)\log^2 p     +\frac{2Csh(z,z^*) \log^2 p}{n^2}}\\
& \leq 4\br{\frac{16C^2 s^2\log^2p}{n^2} + \frac{16s\log p}{\lambda^2 n^2}} \frac{\ell(z,z^*)}{n}
\end{align*}
where the last inequality is by  (\ref{eqn:reg_h_explicit}) and (\ref{eqn:reg_beta_s}).

~\\
\emph{(2) Regime $j\in S_2$.} In this case, $\wt \beta_j(z) -\wt \beta_j(z^*)$ can be written as
$$\left[\br{X_S^TX_S}^{-1}X_S^TX_{S_1} \beta^*_{S_1}\right]_{\phi_S(j)}   +\left[\br{X_S^TX_S}^{-1}X_S^Tw\right]_{\phi_S(j)} - \left[\br{X_{S^*}^TX_{S^*}}^{-1}X_{S^*}^Tw\right]_{\phi_{S^*}(j)}.$$
We will bound the first term, and then the second and the third term will be analyzed together.

\emph{(2.1).}
For the first term, we have
\begin{align*}
\sum_{j\in S_2}\left[\br{X_S^TX_S}^{-1}X_S^TX_{S_1} \beta^*_{S_1}\right]_{\phi_S(j)}^2 & \leq \sum_{j\in S}\left[\br{X_S^TX_S}^{-1}X_S^TX_{S_1} \beta^*_{S_1}\right]_{\phi_S(j)}^2 \\
& = \norm{\br{X_S^TX_S}^{-1}X_S^TX_{S_1} \beta^*_{S_1}}^2\\
& \leq  \norm{\br{X_S^TX_S}^{-1}X_S^TX_{S_1}}^2 \norm{\beta^*_{S_1}}^2\\
&\leq C\frac{s\log p}{n}\norm{\beta^*_{S_1}}^2,
\end{align*}
where the last inequality is due to   (\ref{eqn:re1}) in Lemma \ref{lem:random-events-reg}.

\emph{(2.2).}
Note that
\begin{align*}
&\sum_{j\in S_2} \br{\left[\br{X_S^TX_S}^{-1}X_S^Tw\right]_{\phi_S(j)} - \left[\br{X_{S^*}^TX_{S^*}}^{-1}X_{S^*}^Tw\right]_{\phi_{S^*}(j)}}^2 \\
& = \norm{\left[\br{X_S^TX_S}^{-1}X_S^Tw\right]_{\phi_S(S_2)} - \left[\br{X_{S^*}^TX_{S^*}}^{-1}X_{S^*}^Tw\right]_{\phi_{S^*}(S_2)}}^2.
\end{align*}
Since $S$ is close to $S^*$, the two length-$\abs{S_2}$ vectors on the right hand side of the above equation should also be close to each other. Applying block matrix inverse formula, we have
\begin{align}\label{eqn:block}
\br{X_{S}X_{S}^T}^{-1}
=
 \begin{pmatrix}
\mathsf{X}_2^T\mathsf{X}_2 & \mathsf{X}_2^T \mathsf{X}_3\\
 \mathsf{X}_3^T \mathsf{X}_2 & \mathsf{X}_3^T\mathsf{X}_3
\end{pmatrix}^{-1} \triangleq \begin{pmatrix}
A_{11} & A_{12}\\
A_{21} & A_{22}
\end{pmatrix},
\end{align}
where 
\begin{align*}
A_{11} &= \br{\mathsf{X}_2^T \mathsf{X}_2}^{-1} + \br{\mathsf{X}_2^T \mathsf{X}_2}^{-1}\mathsf{X}_2^T \mathsf{X}_3 \br{\mathsf{X}_3^T\br{I-\mathsf{P}_2}\mathsf{X}_3}^{-1}  \mathsf{X}_3^T \mathsf{X}_2\br{\mathsf{X}_2^T\mathsf{X}_2}^{-1}, \\
A_{12} &=  - \br{\mathsf{X}_2^T\mathsf{X}_2}^{-1}\mathsf{X}_2^T \mathsf{X}_3 \br{\mathsf{X}_3^T\br{I-\mathsf{P}_2}\mathsf{X}_3 }^{-1}, \\
A_{21} &=  -  \br{\mathsf{X}_3^T(I-\mathsf{P}_2)\mathsf{X}_3}^{-1} \mathsf{X}_3^T \mathsf{X}_2 \br{\mathsf{X}_2^T\mathsf{X}_2}^{-1}, \\
A_{22} &=  \br{\mathsf{X}_3^T(I-\mathsf{P}_2)\mathsf{X}_3}^{-1}.
\end{align*}
With these notation, we have
\begin{align*}
\sbr{\br{X_S^TX_S}^{-1}X_S^T w }_{\phi_{S}(S_2)} & = A_{11}\mathsf{X}_2^T w  + A_{12}\mathsf{X}_3^T w  \\
 & =\br{\mathsf{X}_2^T\mathsf{X}_2}^{-1} \mathsf{X}_2^T w  -  \br{\mathsf{X}_2^T\mathsf{X}_2}^{-1}\mathsf{X}_2^T \mathsf{X}_3 \br{\mathsf{X}_3^T(I-\mathsf{P}_2)\mathsf{X}_3 }^{-1}  \mathsf{X}_3^T (I-\mathsf{P}_2) w,
 \end{align*}
and
 \begin{align*}
\sbr{\br{X_{S^*}^TX_{S^*}}^{-1}X_{S^*}^T w }_{\phi_{S^*}(S_2)}=\br{\mathsf{X}_2^T\mathsf{X}_2}^{-1} \mathsf{X}_2^T w  -  \br{\mathsf{X}_2^T\mathsf{X}_2}^{-1}\mathsf{X}_2^T \mathsf{X}_1 \br{\mathsf{X}_1^T(I-\mathsf{P}_2)\mathsf{X}_1 }^{-1}  \mathsf{X}_1^T (I-\mathsf{P}_2) w .
\end{align*}
Thus
\begin{align*}
&\sum_{j\in S_2} \br{\left[\br{X_S^TX_S}^{-1}X_S^Tw\right]_{\phi_S(j)} - \left[\br{X_{S^*}^TX_{S^*}}^{-1}X_{S^*}^Tw\right]_{\phi_{S^*}(j)}}^2 \\
& = \norm{\br{\mathsf{X}_2^T\mathsf{X}_2}^{-1}\mathsf{X}_2^T \mathsf{X}_1 \br{\mathsf{X}_1^T(I-\mathsf{P}_2)\mathsf{X}_1 }^{-1}  \mathsf{X}_1^T (I-\mathsf{P}_2) w  -  \br{\mathsf{X}_2^T\mathsf{X}_2}^{-1}\mathsf{X}_2^T \mathsf{X}_3 \br{\mathsf{X}_3^T(I-\mathsf{P}_2)\mathsf{X}_3 }^{-1}  \mathsf{X}_3^T (I-\mathsf{P}_2) w }^2\\
&\leq \norm{\br{\mathsf{X}_2^T\mathsf{X}_2}^{-1}\mathsf{X}_2^T \mathsf{X}_1}^2 \norm{\br{\mathsf{X}_1^T(I-\mathsf{P}_2)\mathsf{X}_1 }^{-1}  \mathsf{X}_1^T (I-\mathsf{P}_2) w }^2\\
&\quad + \norm{\br{\mathsf{X}_2^T\mathsf{X}_2}^{-1}\mathsf{X}_2^T \mathsf{X}_3}^2\norm{ \br{\mathsf{X}_3^T(I-\mathsf{P}_2)\mathsf{X}_3 }^{-1}  \mathsf{X}_3^T (I-\mathsf{P}_2) w }^2\\
& \leq \br{C\frac{s\log p}{n}} \br{C\frac{\abs{S_1}\log p}{n}  +   C\frac{\abs{S_3}\log p}{n}}\\
& = C^2 \frac{s h(z,z^*)\log^2 p}{n^2},
\end{align*}
where the second to the last inequality is due to  (\ref{eqn:re1}) and (\ref{eqn:re2}).

\emph{(2.3).}
Combining the above results, we have
\begin{align*}
\sum_{j\in S_2}\br{\wt \beta_j(z) -\wt \beta_j(z^*) }^2 &\leq  2\br{C\frac{s\log p}{n}\norm{\beta^*_{S_1}}^2 + C^2 \frac{s h(z,z^*)\log^2 p}{n^2}}\\
& \leq  \br{\frac{4Cs\log p}{n} + \frac{4C^2 s\log^2 p}{\lambda^2 n^2}} \frac{\ell(z,z^*)}{n},
\end{align*}
where the last inequality is by  (\ref{eqn:reg_h_explicit}) and (\ref{eqn:reg_beta_s}).

~\\
\emph{(3) Regime $j\in S_3$.} Since
\begin{align*}
\wt \beta_j(z) &= \beta_j^*  + \left[\br{X_S^TX_S}^{-1}X_S^TX_{S_1} \beta^*_{S_1}\right]_{\phi_S(j)}   +\left[\br{X_S^TX_S}^{-1}X_S^T w \right]_{\phi_S(j)}, \\
\wt \beta_j(z^*) &=\frac{1}{\norm{X_j}^2}X_j^T (I-P_{S^*})  w,
\end{align*}
and $\beta^*_j =0$, we can write $\wt \beta_j(z) -\wt \beta_j(z^*)$ as
\begin{align*}
&   \left[\br{X_S^TX_S}^{-1}X_S^TX_{S_1} \beta^*_{S_1}\right]_{\phi_S(j)}   +\left[\br{X_S^TX_S}^{-1}X_S^T w \right]_{\phi_S(j)}-\frac{1}{\norm{X_j}^2}X_j^T (I-P_{S^*})  w  \\
&=  \left[\br{X_S^TX_S}^{-1}X_S^TX_{S_1} \beta^*_{S_1}\right]_{\phi_S(j)}   + \frac{1}{\norm{X_j}^2} X_j^T P_{S^*} w  +   \left[\br{X_S^TX_S}^{-1}X_S^T w \right]_{\phi_S(j)} - \frac{1}{\norm{X_j}^2}X_j^T w  \\
&=  \left[\br{X_S^TX_S}^{-1}X_S^TX_{S_1} \beta^*_{S_1}\right]_{\phi_S(j)}   + \frac{1}{\norm{X_j}^2} X_j^T P_{S^*} w  +   \left[\br{X_S^TX_S}^{-1}X_S^T w \right]_{\phi_S(j)} - \frac{1}{\norm{X_j}^2}X_j^T w  .
\end{align*}
We are going to bound each term separately. The last two terms will be analyzed together.

\emph{(3.1).} Note that the first term here is identical to the first term in the regime $j\in S_2$. By the same argument, we have
\begin{align*}
\sum_{j\in S_3} \left[\br{X_S^TX_S}^{-1}X_S^TX_{S_1} \beta^*_{S_1}\right]_{\phi_S(j)}^2 \leq C\frac{s\log p}{n}\norm{\beta^*_{S_1}}^2.
\end{align*}

\emph{(3.2).} For the second term, we have
\begin{align*}
\sum_{j\in S_3}\br{\frac{1}{\norm{X_j}^2} X_j^T P_{S^*} w  }^2 & \leq \frac{1}{\min_j \norm{X_j}^4} \norm{X_3^T P_{S^*} w }^2 \leq \frac{4}{n^2} Cs \abs{S_3}\log p \leq  \frac{4}{n^2} Cs h(z,z^*)\log p,
\end{align*} 
where the second to last inequality is due to  (\ref{eqn:re0}) and (\ref{eqn:re7}).

\emph{(3.3).} For the last two terms, we can again apply block matrix inverse formula to simplify them. Using (\ref{eqn:block}), we have
\begin{align*}
 \left[\br{X_S^TX_S}^{-1}X_S^T w \right]_{\phi_S(j)}  & =  \sbr{\begin{pmatrix}
A_{11} & A_{12} \\
A_{21} & A_{22}
\end{pmatrix} \begin{pmatrix}
\mathsf{X}_2^T  \\
\mathsf{X}_3^T 
\end{pmatrix}  w }_{\phi_S(j)} \\
& = \sbr{A_{21}\mathsf{X}_2^T w   + A_{22} \mathsf{X}_3^T  w }_{\phi_{S_3}(j)}\\
& =  - \sbr{ \br{\mathsf{X}_3^T(I-\mathsf{P}_2)\mathsf{X}_3}^{-1} \mathsf{X}_3^T\mathsf{P}_2  w   }_{\phi_S(j)} + \sbr{ \br{\mathsf{X}_3^T(I-\mathsf{P}_2)\mathsf{X}_3}^{-1} \mathsf{X}_3^T  w   }_{\phi_S(j)}.
\end{align*}
Then
\begin{align*}
\left[\br{X_S^TX_S}^{-1}X_S^T w \right]_{\phi_S(j)} - \frac{1}{\norm{X_j}^2}X_j^T w   & =- \sbr{ \br{\mathsf{X}_3^T(I-\mathsf{P}_2)\mathsf{X}_3}^{-1} \mathsf{X}_3^T\mathsf{P}_2  w   }_{\phi_S(j)} \\
&  \quad +  \br{\sbr{ \br{\mathsf{X}_3^T(I-\mathsf{P}_2)\mathsf{X}_3}^{-1} \mathsf{X}_3^T  w   }_{\phi_S(j)} -  \frac{1}{\norm{X_j}^2}\mathsf{X}_j^T w }.
\end{align*}
Consequently,
\begin{align*}
&\sum_{j\in S_3} \br{\left[\br{X_S^TX_S}^{-1}X_S^T w \right]_{\phi_S(j)} - \frac{1}{\norm{X_j}^2}X_j^T w  }^2 \\
& = \norm{- \br{\mathsf{X}_3^T(I-\mathsf{P}_2)\mathsf{X}_3}^{-1} \mathsf{X}_3^T\mathsf{P}_2  w  + \br{\mathsf{X}_3^T(I-\mathsf{P}_2)\mathsf{X}_3}^{-1} \mathsf{X}_3^T  w  -  D^{-1}\mathsf{X}_3^T}^2 \\
& \leq 2 \norm{ \br{\mathsf{X}_3^T(I-\mathsf{P}_2)\mathsf{X}_3}^{-1} \mathsf{X}_3^T\mathsf{P}_2  w   }^2  + 2\norm{ \br{\mathsf{X}_3^T(I-\mathsf{P}_2)\mathsf{X}_3}^{-1} \mathsf{X}_3^T  w   - D^{-1}\mathsf{X}_3^T  w }^2 \\
& \leq 2 \norm{ \br{\mathsf{X}_3^T(I-\mathsf{P}_2)\mathsf{X}_3}^{-1} \mathsf{X}_3^T\mathsf{P}_2  w   }^2  + 2\norm{ \br{\mathsf{X}_3^T(I-\mathsf{P}_2)\mathsf{X}_3}^{-1}  - D^{-1}}^2 \norm{\mathsf{X}_3^T  w }^2\\
& \leq 2 \norm{ \br{\mathsf{X}_3^T(I-\mathsf{P}_2)\mathsf{X}_3}^{-1} \mathsf{X}_3^T\mathsf{P}_2  w   }^2  + 4\norm{ \br{\mathsf{X}_3^T(I-\mathsf{P}_2)\mathsf{X}_3}^{-1}  - (n-\abs{S_2})^{-1}I_{\abs{S_3}}}^2 \norm{\mathsf{X}_3^T  w }^2\\
&\quad + 4\norm{ (n-\abs{S_2})^{-1}I_{\abs{S_3}} - D^{-1}}^2 \norm{\mathsf{X}_3^T  w }^2
\end{align*}
where $D \in\mathr^{\abs{S_3} \times \abs{S_3}}$ is a diagonal matrix with diagonal entries $\{1/\norm{X_j}^2\}_{j\in S_3}$ and off-diagonal entries being 0. 
By   (\ref{eqn:re8_0}), we have
\begin{align*}
\norm{ \br{\mathsf{X}_3^T(I-\mathsf{P}_2)\mathsf{X}_3}^{-1} \mathsf{X}_3^T\mathsf{P}_2  w   }^2 \leq Cn^{-2}s\abs{S_3} \log p.
\end{align*}
By   (\ref{eqn:re8}), we have
\begin{align*}
\norm{\mathsf{X}_3^T(I-\mathsf{P}_2)\mathsf{X}_3 - \br{n-\abs{S_2}}I_{\abs{S_3}}  }^2 \leq Cns\log p.
\end{align*}
Together with   (\ref{eqn:re9}), we have 
\begin{align*}
& \norm{\br{\mathsf{X}_3^T(I-\mathsf{P}_2)\mathsf{X}_3}^{-1} - (n-\abs{S_2})^{-1}I_{\abs{S_3}}}^2 \\
& \leq \norm{\br{\mathsf{X}_3^T(I-\mathsf{P}_2)\mathsf{X}_3}^{-1} }^2 \norm{I_{\abs{S_3}}  - \frac{\mathsf{X}_3^T(I-\mathsf{P}_2)\mathsf{X}_3}{n-\abs{S_2}}}^2\\
&\leq  \frac{2C^3s\log p}{n^3},
\end{align*}
where we have used the assumption $\abs{S_3}\leq 2C_0s$.
By   (\ref{eqn:re0}), we have
\begin{align*}
\norm{(n-\abs{S_2})^{-1}I_{\abs{S_3}} - D^{-1}}^2 & \leq \max_{j\in [p]} \abs{\frac{1}{n- \abs{S_2}}  -\frac{1}{\norm{X_j}^2} }^2\\
&\leq  \max\cbr{\abs{\frac{1}{n}  - \frac{1}{n -C\sqrt{n\log p}}}^2, \abs{\frac{1}{n-2C_0s} - \frac{1}{n}}^2}\\
&\leq \frac{2C\log p}{n^3} + \frac{8s^2}{n^4}.
\end{align*}
By   (\ref{eqn:re10}), we have
\begin{align*}
\norm{X^T_3 w }^2& \leq Cn\abs{S_3}\log p.
\end{align*}
As a consequence, we have
\begin{align*}
&\sum_{j\in S_3} \br{\left[\br{X_S^TX_S}^{-1}X_S^T w \right]_{\phi_S(j)} - \frac{1}{\norm{X_j}^2}X_j^T w  }^2 \\
& \leq 2Cn^{-2}s\abs{S_3} \log p +   \br{\frac{4C^3s\log p}{n^3} + \frac{8s^2}{n^4}} Cn\abs{S_3}\log p\\
& \leq 8C^3\frac{s\log p}{n^2}\abs{S_3} \log p\\
& \leq 8C^3\frac{s\log p}{n^2}h(z,z^*) \log p .
\end{align*}


\emph{(3.4).} Combining the above results, we have
\begin{align*}
&\sum_{j\in S_3}\br{\wt \beta_j(z) -\wt \beta_j(z^*) }^2 \\
&\leq  3\br{C\frac{s\log p}{n}\norm{\beta^*_{S_1}}^2 +  \frac{4}{n^2} Cs h(z,z^*)\log p  +8C^3\frac{s\log p}{n^2}h(z,z^*) \log p}\\
&\leq 3\br{\frac{2Cs\log p}{n} + \frac{32C^3 s\log^2 p}{\lambda^2n^2}} \frac{\ell(z,z^*)}{n},
\end{align*}
where the last inequality is by  (\ref{eqn:reg_h_explicit}) and (\ref{eqn:reg_beta_s}).

~\\
\emph{(4) Regime $j\in S_4$.}
In this case, we have
\begin{align*}
\wt \beta_j(z) -\wt \beta_j(z^*)   & =  \frac{1}{\norm{X_j}^2}X_j^T (I -P_S)X_{S_1} \beta_{S_1}^*   +  \frac{1}{\norm{X_j}^2}X_j^T(P_{S^*}-P_S) w  .
\end{align*}
Then,
\begin{align*}
&\max_{T\subset  S_4} \frac{1}{\abs{T} \vee s} \sum_{j \in T} \br{\wt \beta_j(z) -\wt \beta_j(z^*) }^2 \\
&\leq  \frac{1}{\min_j \norm{X_j}^4} \br{\max_{T\subset  S_4} \frac{1}{\abs{T} \vee s}\sum_{j \in T} \br{X_j^T (I -P_S)X_{S_1} \beta_{S_1}^* }^2 +  \max_{T\subset  S_4} \frac{1}{\abs{T} \vee s}\sum_{j \in T} \br{X_j^T(P_{S^*}-P_S) w }^2} \\
& =  \frac{1}{\min_j \norm{X_j}^4} \br{\max_{T\subset  S_4}\frac{1}{\abs{T} \vee s} \norm{X_T^T(I -P_S)X_{S_1} \beta_{S_1}^*}^2  + \max_{T\subset  S_4} \frac{1}{\abs{T} \vee s} \norm{ X_T^T(P_{S^*}-P_S) w }^2   }\\
& \leq \frac{4}{n^2} \br{Cn \log p \norm{\beta^*_{S_1}}^2 + C \br{\abs{S_1} + \abs{S_3}}\log^2 p} \\
& =  \frac{4}{n^2} \br{Cn \log p \norm{\beta^*_{S_1}}^2 + C h(z,z^*)\log^2 p}.
\end{align*}
where the last inequality is by  (\ref{eqn:re0}), (\ref{eqn:re11}) and (\ref{eqn:re12}). Then by  (\ref{eqn:reg_h_explicit}) and (\ref{eqn:reg_beta_s}), we have 
\begin{align*}
&\max_{T\subset  S_4} \frac{1}{\abs{T} \vee s} \sum_{j \in T} \br{\wt \beta_j(z) -\wt \beta_j(z^*) }^2 \leq \br{\frac{8C\log p }{n} + \frac{8C \log^2 p}{\lambda^2 n^2}} \frac{\ell(z,z^*)}{n}.
\end{align*}

~\\
\emph{(5) Combining the bounds.}
Now we are ready to combine the bounds obtained in the four regimes. Let $T\subset [p]$ be any set. We have 
\begin{align*}
& \sum_{j \in T} \br{\wt \beta_j(z) -\wt \beta_j(z^*) }^2\\
& \leq   \sum_{j \in S_2 } \br{\wt \beta_j(z) -\wt \beta_j(z^*) }^2 +  \sum_{j \in S_1 } \br{\wt \beta_j(z) -\wt \beta_j(z^*) }^2 +  \sum_{j \in S_3 } \br{\wt \beta_j(z) -\wt \beta_j(z^*) }^2 +  \sum_{j \in S_4\cap T } \br{\wt \beta_j(z) -\wt \beta_j(z^*) }^2\\
& \leq \br{\br{\frac{4Cs\log p}{n} + \frac{4C^2 s\log^2 p}{\lambda^2 n^2}  }+ 4\br{\frac{16C^2 s^2\log^2p}{n^2} + \frac{16s\log p}{\lambda^2 n^2}} + 3\br{\frac{2Cs\log p}{n} + \frac{32C^3 s\log^2 p}{\lambda^2n^2}}} \frac{\ell(z,z^*)}{n}\\
&\quad +  \br{\frac{8C\log p }{n} + \frac{8C \log^2 p}{\lambda^2 n^2}}\br{\abs{T}   \vee s}\frac{\ell(z,z^*)}{n} \\
& \leq \br{\frac{128C^2 \log p}{n} + \frac{256 C^3 \log^2 p}{\lambda^2 n^2}} \br{s + \abs{T}} \frac{\ell(z,z^*)}{n}.
\end{align*}
Thus
\begin{align*}
\max_{T\subset [p]} \frac{1}{s+\abs{T}} \sum_{j\in T} \br{\wt \beta_j(z) -\wt \beta_j(z^*) }^2 \leq\br{\frac{128C^2 \log p}{n} + \frac{256 C^3 \log^2 p}{\lambda^2 n^2}} \frac{\ell(z,z^*)}{n}.
\end{align*}
Together with (\ref{eqn:reg_proof_main}), we have
\begin{align*}
&\max_{T\subset [p]} \frac{1}{s+\abs{T}}  \sum_{j\in T}  \max_{b\in\{-1,1,0\}\setminus \{z^*_j\}} \frac{G_j(z_j^*,b;z)^2\|\mu_j(B^*,b)-\mu_j(B^*,z_j^*)\|^2}{\Delta_j(z_j^*,b)^4\ell(z,z^*)} \\
& \leq  2 \br{\frac{128C^2 \log p}{n} + \frac{256 C^3 \log^2 p}{\lambda^2 n^2}}\max_{j\in[p]}\max\cbr{4\br{\frac{\abs{\beta^*_j}}{\abs{\beta^*_j} - t(X_j)}}^2,\br{\frac{\lambda}{t(X_j)}}^2}.
\end{align*}
Recall that $\Delta^2_{\min} = \lambda^2 \min_{j\in[p]}\norm{X_j}^2 \geq n\lambda^2/2$. For any $T\subset[p]$, we have $\tau/(\tau + 4\Delta_{\min}^2 \abs{T}) \leq \tau /(\tau + 2n\lambda^2\abs{T}) \leq C_0/(C_0s + \abs{T})$, since $\tau \leq C_0sn\lambda^2$.
This gives us
\begin{align*}
&\max_{T\subset [p]} \frac{\tau}{\tau +4\Delta_{\min}^2 \abs{T}}  \sum_{j\in T} \max_{b\in\{-1,1,0\}\setminus \{z^*_j\}} \frac{G_j(z_j^*,b;z)^2\|\mu_j(B^*,b)-\mu_j(B^*,z_j^*)\|^2}{\Delta_j(z_j^*,b)^4\ell(z,z^*)}  \\
&\leq C' s\br{\frac{\log^2 p}{n} + \frac{\log^2 p}{\lambda^2n^2}}\max_{j\in[p]}\max\cbr{4\br{\frac{\abs{\beta^*_j}}{\abs{\beta^*_j} - t(X_j)}}^2,\br{\frac{\lambda}{t(X_j)}}^2},
\end{align*}
for some constant $C'$. The proof is complete.
\end{proof}

\begin{proof}[Proof of Lemma \ref{lem:regression-ideal}]
Recall for any $j\in[p]$, $T_j$ the local test to recover $z^*_j$ is  defined in (\ref{eqn:local_statistics_reg}). We have the decomposition $T_j  = \mu_j(B^*,z^*_j) + \epsilon_j$, where $\epsilon_j = \|X_j\|^{-1} X_j^Tw \sim \mathn(0,1)$. 
Since $\nu_j(\wh{B}(z^*),z_j^*)-\nu_j(\wh{B}(z^*),b) = 2(z^*_j - b)\norm{X_j}t(X_j)$, by (\ref{eqn:reg_delta}), for  any $0< \delta <1$, we have
\begin{align*}
&\indc{\iprod{\epsilon_j}{\nu_j(\wh{B}(z^*),z_j^*)-\nu_j(\wh{B}(z^*),b)} \leq -\frac{1-\delta}{2}\Delta_j(z_j^*,b)^2} \\
& \leq \begin{cases}
\indc{z_j^* \epsilon_j \leq -\br{1-\delta} \norm{X_j}\br{\abs{\beta_j^*} - t(X_j)}}   & z_j^* \neq  0 \text{ and } b\neq z_j^*  \\
\indc{-b\epsilon_j \leq -\br{1-\delta} \norm{X_j} t(X_j)}  & z_j^* = 0 \text{ and } b\neq 0.
\end{cases}
\end{align*}
Together with (\ref{eqn:reg_mu_diff}), for $b\neq z_j^*$, we have
\begin{align*}
&\|\mu_j(B^*,b)-\mu_j(B^*,z_j^*)\|^2\indc{\iprod{\epsilon_j}{\nu_j(\wh{B}(z^*),z_j^*)-\nu_j(\wh{B}(z^*),b)} \leq -\frac{1-\delta}{2}\Delta_j(z_j^*,b)^2}\\
& \leq \begin{cases}
4\abs{\beta^*_j}^2 \norm{X_j}^2 \indc{z_j^* \epsilon_j \leq -\br{1-\delta} \norm{X_j}\br{\abs{\beta_j^*} - t(X_j)}} & z_j^* \neq 0\\
4\lambda^2\norm{X_j}^2 \indc{-b\epsilon_j \leq -\br{1-\delta} \norm{X_j} t(X_j)} & z_j^* = 0.
\end{cases}
\end{align*}
As a consequence
\begin{align*}
\xi_{\rm ideal}(\delta) & \leq 8 \sum_{j\in S^*}\abs{\beta^*_j}^2\norm{X_j}^2 \indc{z_j^*\epsilon_j \leq -\br{1-\delta} \norm{X_j}\br{\abs{\beta_j^*}  - t(X_j)}} + 4 \sum_{j\notin S^*}\lambda^2\norm{X_j}^2 \indc{\abs{\epsilon_j }\geq \br{1-\delta} \norm{X_j} t(X_j)} .
\end{align*}
Define $\mathcal{F}$ to be the event that   (\ref{eqn:re0}) holds. Then by Lemma \ref{lem:random-events-reg}, we know that $\pbr{\mathcal{F}} \geq 1-p^{-C'}$. Under the event $\mathcal{F}$ and the condition that $\snr\rightarrow\infty$, we have 
\begin{align*}
\xi_{\rm ideal}(\delta) \indc{\mathcal{F}}& \leq   8 \sum_{j\in S^*}\abs{\beta^*_j}^2\norm{X_j}^2 \indc{\frac{-z^*_j\epsilon_j}{ {1-\delta}} \geq \norm{X_j}\br{\abs{\beta_j^*}  - t(X_j)}}\indc{\abs{\norm{X_j}^2 -n} \leq C\sqrt{n\log p}} \\
& \quad+ 4 \sum_{j\notin S^*}\lambda^2 \norm{X_j}^2\indc{\frac{\abs{\epsilon_j}}{ {1-\delta}} \geq  \norm{X_j} t(X_j)}\indc{\abs{\norm{X_j}^2 -n} \leq C\sqrt{n\log p}},
\end{align*}
which implies
\begin{align*}
\E\xi_{\rm ideal}(\delta) \indc{\mathcal{F}} &\leq   16n \sum_{j\in S^*}\abs{\beta^*_j}^2 \pbr{\frac{\epsilon_j}{ {1-\delta}} \geq \norm{X_j}\br{\abs{\beta_j^*}  - t(X_j)} \; \&\; \abs{\norm{X_j}^2 -n} \leq C\sqrt{n\log p}}  \\
&\quad + 16n \sum_{j\notin S^*}\lambda^2 \pbr{\frac{\epsilon_j}{ {1-\delta}} \geq  \norm{X_j} t(X_j)  \; \&\; \abs{\norm{X_j}^2 -n} \leq C\sqrt{n\log p}}.
\end{align*}
We are going to upper bound the above quantity by $\wt \psi(n,p,s,\lambda,\delta,C)$, defined in (\ref{eqn:tilde_psi_def}).
To this end, we will first show the function $f(y) = y^2\pbr{\frac{\epsilon}{1-\delta} \geq \norm{\zeta} \br{y - t(\zeta)}, \abs{\norm{\zeta}^2 - n} \leq C\sqrt{n\log p}}$ is a decreasing function of $y$ when $y\geq \lambda$ and $\lambda >0$. 
Since the function $t(\zeta)$ only depends on $\|\zeta\|$, we can also write $t(\zeta)$ as $t(\|\zeta\|)$ with a slight abuse of notation in (\ref{eqn:t_u_def}).
Define $u_{\min} = \sqrt{n- C\sqrt{n\log p}}$ and $u_{\max} = \sqrt{n+C\sqrt{n\log p}}$. Then, we have 
\begin{align*}
f(y) & = y^2\pbr{\frac{\epsilon}{1-\delta} \geq \norm{\zeta} \br{y - t(\zeta) } , \abs{\norm{\zeta}^2 - n} \leq C\sqrt{n\log p}}\\
& = y^2 \int_{u_{\min}}^{u_{\max}} p(u) G\br{\br{1-\delta} u \br{y -t(u)}} du,
\end{align*}
where $p(\cdot)$ is the density of $\norm{\zeta}$. According to the same argument used in the proof of Lemma \ref{prop:useful-reg}, it can be shown that $\min_{u\in[u_{\min},u_{\max}]}u(\lambda-t(u))\rightarrow\infty$.
Thus, $u(y-t(u))\geq u(\lambda-t(u))>0$ for $y\geq \lambda$ and $u\in[u_{\min},u_{\max}]$. Moreover, we also have $\br{1-\delta}^2 u^2 y\br{y -t(u) }\geq\br{1-\delta}^2 u^2\lambda \br{\lambda-t(u)} >2$ for $y\geq \lambda$ and $u\in[u_{\min},u_{\max}]$. Therefore,
\begin{align*}
\frac{2}{\br{1-\delta} u \br{y - t(u)} } -y\br{1-\delta}u & = \frac{2 -\br{1-\delta}^2u^2 y(y-t(u))}{\br{1-\delta} u \br{y - t(u)} }\\
& \leq  \frac{2 -\br{1-\delta}^2u^2 \lambda(\lambda-t(u))}{\br{1-\delta} u \br{y - t(u)} }\\
& < 0.
\end{align*}
This gives
\begin{align*}
f'(y) &=\int_{u_{\min}}^{u_{\max}}p(u) \br{2y  G\br{\br{1-\delta} u \br{y - t(u)}}    - y^2\br{1-\delta}u g\br{\br{1-\delta} u \br{y - t(u)}}   }  du\\
& \leq\int_{u_{\min}}^{u_{\max}} p(u)y\br{\frac{2}{\br{1-\delta} u \br{y - t(u)} } -y\br{1-\delta}u} g\br{\br{1-\delta} u \br{y - t(u)}}    du\\
&\leq 0,
\end{align*}
where we have used   (\ref{eqn:gaussian_tail_bound}). 
As a result, $f(y)$ is a decreasing function  for all $y\geq \lambda$, which implies
\begin{align*}
\E\xi_{\rm ideal}(\delta) \indc{\mathcal{F}} &\leq  16n  \sum_{j\in S^*}\lambda^2 \pbr{\frac{\epsilon_j}{ {1-\delta}} \geq \norm{X_j}\br{\lambda - t(X_j)} \; \&\; \abs{\norm{X_j}^2 -n} \leq C\sqrt{n\log p}}  \\
&\quad + 16n  \sum_{j\notin S^*}\lambda^2 \pbr{\frac{\epsilon_j}{ {1-\delta}} \geq  \norm{X_j} t(X_j)  \; \&\; \abs{\norm{X_j}^2 -n} \leq C\sqrt{n\log p}}\\
& =16n    \lambda^2  \wt \psi(n,p,s,\lambda,\delta,C).
\end{align*}
By applying Markov inequality, we have with probability at least $1-w^{-1}$,
\begin{align*}
\xi_{\rm ideal}(\delta) \indc{\mathcal{F}} \leq 16wn  \lambda^2  \wt \psi(n,p,s,\lambda,\delta,C),
\end{align*}
where $w$ is any sequence that goes to infinity.
A union bound implies
\begin{align}\label{eqn:xi_ideal_upper}
\xi_{\rm ideal}(\delta)\leq 16wn  \lambda^2  \wt \psi(n,p,s,\lambda,\delta,C)
\end{align}
holds with probability at least $1-w^{-1}-p^{-C'}$. Taking $\delta = \delta_p = o((\log p)^{-1})$ and $w=\exp(\snr)$, the desired conclusion follows an application of Lemma \ref{prop:reg_rate_simplify}. Thus, the proof is complete.
\end{proof}

\begin{proof}[Proof of Proposition \ref{prop:reg_init}]
By Proposition 5.1 of \citep{ndaoud2018optimal}, we have 
\begin{align*}
\norm{\wt \beta - \beta^*}^2 \leq \frac{C_1 s\log \frac{ep}{s}}{n},
\end{align*}
with probability at least $1-  2^{-C_2 s}$ for some constants $C_1, C_2 >0$, as long as $A$ is chosen to be sufficiently large. In the rest part of the proof, we assume   (\ref{eqn:re0}) holds. We divide the calculation of $\ell(\wt z,z^*)$ into three parts. First we have 
\begin{eqnarray*}
 \sum_{j=1}^p\lambda^2\|X_j\|^2\mathbb{I}\{\wt {z}_j\neq 0,z_j^*=0\}  &\leq & \sum_{j=1}^p\lambda^2\|X_j\|^2\mathbb{I}\left\{|\wt{\beta}_j |>\frac{\lambda}{2},\beta_j^*=0\right\} \\
&\leq & 4\sum_{j=1}^p\|X_j\|^2(\wt{\beta}_j -\beta_j^*)^2\mathbb{I}\left\{|\wt{\beta}_j |>\frac{\lambda}{2},\beta_j^*=0\right\} \\
&\leq & 8n\sum_{j=1}^p(\wt{\beta}_j -\beta_j^*)^2\mathbb{I}\left\{|\wt{\beta}_j |>\frac{\lambda}{2},\beta_j^*=0\right\}.
\end{eqnarray*}
Similarly, we have 
\begin{eqnarray*}
 \sum_{j=1}^p|\beta_j^*|^2\|X_j\|^2\mathbb{I}\{z_j^*\neq 0, \wt{z}_j=0\} &\leq&  \sum_{j=1}^p|\beta_j^*|^2\|X_j\|^2\mathbb{I}\left\{|\beta^*_j|\geq \lambda, |\wt{\beta}_j  |\leq\frac{\lambda}{2}\right\} \\
&\leq& 8n\sum_{j=1}^p(\wt{\beta}_j  -\beta_j^*)^2\mathbb{I}\left\{|\beta^*_j|\geq \lambda, |\wt{\beta}_j  |\leq\frac{\lambda}{2}\right\},
\end{eqnarray*}
where in the last inequality, since $|\beta^*_j|\geq \lambda$ and $|\wt{\beta}_j  |\leq\frac{\lambda}{2}$, we have
$$|\beta_j^*-\wt{\beta}_j  |\geq |\beta_j^*|-|\wt{\beta}_j  |\geq \frac{|\wt{\beta}_j  |}{2}+\frac{\lambda}{2}-\frac{\lambda}{2}=\frac{|\wt{\beta}_j  |}{2}.$$
Finally,
\begin{eqnarray*}
&& 4\sum_{j=1}^p|\beta_j^*|^2\|X_j\|^2\mathbb{I}\{\wt{z}_jz_j^*=-1\} \\
&\leq& 4\sum_{j=1}^p|\beta_j^*|^2\|X_j\|^2\mathbb{I}\left\{\beta_j^*\leq -\lambda,\wt{\beta}_j  >\frac{\lambda}{2}\right\}+ 4\sum_{j=1}^p|\beta_j^*|^2\|X_j\|^2\mathbb{I}\left\{\beta_j^*\geq \lambda,\wt{\beta}_j  <-\frac{\lambda}{2}\right\} \\
&\leq& 8n\sum_{j=1}^p(\wh{\beta}_j  -\beta_j^*)^2\mathbb{I}\left\{\beta_j^*\leq -\lambda,\wt{\beta}_j  >\frac{\lambda}{2}\right\} \\
&& + 8n\sum_{j=1}^p(\wt{\beta}_j  -\beta_j^*)^2\mathbb{I}\left\{\beta_j^*\geq \lambda,\wh{\beta}_j  <-\frac{\lambda}{2}\right\},
\end{eqnarray*}
because
when $\beta_j^*\leq -\lambda$ and $\wt{\beta}_j  >\frac{\lambda}{2}$, we have
$$|\beta_j^*-\wt{\beta}_j  |=-\beta_j^*+\wt{\beta}_j  \geq |\beta_j^*|,$$
and when $\beta_j^*\geq \lambda$ and $\wt{\beta}_j  <-\frac{\lambda}{2}$, we have
$$|\beta_j^*-\wt{\beta}_j  |=\beta_j^*-\wt{\beta}_j  \geq |\beta_j^*|.$$
Combining all of the above results together, we have 
\begin{align*}
\ell(\wt z, z^*) \leq 8n\norm{\wt \beta -\beta^*}^2 \leq 8C_1s\log \frac{ep}{s}.
\end{align*}
Under the assumption $\limsup s/p <1/2$ and $\snr\rightarrow\infty$, we have $n\lambda^2 \geq 2 \log \frac{p-s}{s} >C_3 \log \frac{ep}{s}$. Thus, there exists some constant $C_0 >0$ such that $\ell(\wt z, z^*) \leq C_0 sn\lambda^2$. A union bound with the probability that the event  (\ref{eqn:re0}) holds leads to the desired result.
\end{proof}

\subsection{Proofs in Section \ref{sec:MRA}}

In this section, we present the proofs of Theorem \ref{thm:mra-lower}, Lemma \ref{lem:error-MRA}, Lemma \ref{lem:ideal-MRA} and Proposition \ref{prop:ini-MRA}. The conclusions of Theorem \ref{thm:final-MRA} and Corollary \ref{cor:final-MRA} are direct consequences of Theorem \ref{thm:main}, and thus their proofs are omitted.

\begin{proof}[Proof of Theorem \ref{thm:mra-lower}]
Recall the definition $\Delta_{\min}=\min_{U\in\mathcal{C}_d}\|(I_d-U)\theta^*\|$. There must exists some $\bar{U}\in\mathcal{C}_d$ such that $\Delta_{\min}=\|(I_d-\bar{U})\theta^*\|$. We define the parameter space
$$\mathcal{Z}=\left\{Z: Z_j=I_d\text{ for all }1\leq j\leq p/2\text{ and }Z_j\in\{I_d,\bar{U}\}\text{ for all }p/2<j\leq p\right\}.$$
Then, we have
\begin{eqnarray}
\nonumber && \inf_{\wh{Z}}\sup_{Z^*}\mathbb{E}\min_{U\in\mathcal{C}_d}\frac{1}{p}\sum_{j=1}^p\indc{\wh{Z}_jU\neq Z_j^*} \\
\nonumber &\geq& \inf_{\wh{Z}}\sup_{Z^*\in\mathcal{Z}}\frac{1}{p}\sum_{j>p/2}\mathbb{P}(\wh{Z}_j\neq Z_j^*) \\
\nonumber &\geq& \inf_{\wh{Z}}\frac{1}{p}\sum_{j>p/2}\text{ave}_{Z_{-j}^*}\left(\frac{1}{2}\mathbb{P}_{(Z_j^*=I_d,Z_{-j}^*)}(\wh{Z}_j\neq I_d) + \frac{1}{2}\mathbb{P}_{(Z_j^*=\bar{U},Z_{-j}^*)}(\wh{Z}_j\neq \bar{U})\right) \\
\label{eq:metroidvania} &\geq& \frac{1}{p}\sum_{j>p/2}\text{ave}_{Z_{-j}^*}\inf_{\wh{Z}_j}\left(\frac{1}{2}\mathbb{P}_{(Z_j^*=I_d,Z_{-j}^*)}(\wh{Z}_j\neq I_d) + \frac{1}{2}\mathbb{P}_{(Z_j^*=\bar{U},Z_{-j}^*)}(\wh{Z}_j\neq \bar{U})\right),
\end{eqnarray}
where the operator $\text{ave}_{Z_{-j}^*}$ is with respect to the uniform measure of $Z_{-j}^*$ in the space $\mathcal{Z}$. We use the notation $Z_{-j}^*$ for $Z^*$ with $Z_j^*$ excluded. Note that the quantity
$$\inf_{\wh{Z}_j}\left(\frac{1}{2}\mathbb{P}_{(Z_j^*=I_d,Z_{-j}^*)}(\wh{Z}_j\neq I_d) + \frac{1}{2}\mathbb{P}_{(Z_j^*=\bar{U},Z_{-j}^*)}(\wh{Z}_j\neq \bar{U})\right)$$
is the optimal testing error for the hypothesis testing problem $H_0: Y_j\sim \mathn(\theta^*,I_d)$ vs $H_1:Y_j\sim \mathn(\bar{U}\theta^*,I_d)$. By Neyman-Pearson lemma, this quantity equals $\mathbb{P}(\mathn(0,1)>\Delta_{\min}/2)$. Therefore,
$$\inf_{\wh{Z}}\sup_{Z^*}\mathbb{E}\min_{U\in\mathcal{C}_d}\frac{1}{p}\sum_{j=1}^p\indc{\wh{Z}_jU\neq Z_j^*}\geq \frac{1+o(1)}{2}\mathbb{P}(\mathn(0,1)>\Delta_{\min}/2)=\exp\left(-(1+o(1))\frac{\Delta_{\min}^2}{8}\right),$$
under the condition that $\Delta_{\min}\rightarrow\infty$. Moreover, $\inf_{\wh{Z}}\sup_{Z^*}\mathbb{E}\min_{U\in\mathcal{C}_d}\frac{1}{p}\sum_{j=1}^p\indc{\wh{Z}_j\neq UZ_j^*}>0$ for some constant $c>0$ when $\Delta_{\min}=O(1)$.
\end{proof}

\begin{proof}[Proof of Lemma \ref{lem:error-MRA}]
Let us write $\wh{\theta}(Z)=\theta(Z)+\bar{\epsilon}(Z)$, where $\theta(Z)=\frac{1}{p}\sum_{j=1}^pZ_j^TZ_j^*\theta^*$ and $\bar{\epsilon}(Z)=\frac{1}{p}\sum_{j=1}^pZ_j^T\epsilon_j$. We have
\begin{eqnarray}
\nonumber \|\theta(Z)-\theta(Z^*)\| &=& \left\|\left(\frac{1}{p}\sum_{j=1}^pZ_j^TZ_j^*-I_d\right)\theta^*\right\| \\
\nonumber &\leq& \frac{1}{p}\sum_{j=1}^p\|(Z_j^TZ_j^*-I_d)\theta^*\| \\
\nonumber &\leq& \frac{1}{p\Delta_{\min}}\sum_{j=1}^p\|(Z_j^TZ_j^*-I_d)\theta^*\|^2 \\
\label{eq:haken-virus} &=& \frac{1}{p\Delta_{\min}}\ell(Z,Z^*). 
\end{eqnarray}
For $\|\bar{\epsilon}(Z)-\bar{\epsilon}(Z^*)\| = \left\|\frac{1}{p}\sum_{j=1}^p(Z_j-Z_j^*)^T\epsilon_j\right\|$, we notice that
$$\frac{1}{p}\sum_{j=1}^p(Z_j-Z_j^*)^T\epsilon_j\sim \mathn\left(0,\frac{1}{p^2}\sum_{j=1}^p(Z_j-Z_j^*)^T(Z_j-Z_j^*)\right).$$
Therefore, by Lemma \ref{prop:chisq}, we have for each fixed $Z\in\mathcal{C}_d^p$,
$$\mathbb{P}\left(\|\bar{\epsilon}(Z)-\bar{\epsilon}(Z^*)\|^2 \geq \left\|\frac{1}{p^2}\sum_{j=1}^p(Z_j-Z_j^*)^T(Z_j-Z_j^*)\right\|(d+2\sqrt{dx}+2x)\right)\leq e^{-x}.$$
With a union bound argument, we have
$$\|\bar{\epsilon}(Z)-\bar{\epsilon}(Z^*)\|\lesssim \left\|\frac{1}{p^2}\sum_{j=1}^p(Z_j-Z_j^*)^T(Z_j-Z_j^*)\right\|^{1/2}\sqrt{p\log d},$$
uniformly over all $Z\in\mathcal{C}_d^p$
with probability at least $1-e^{-C'p \log d}$.
Since
\begin{eqnarray*}
\left\|\frac{1}{p^2}\sum_{j=1}^p(Z_j-Z_j^*)^T(Z_j-Z_j^*)\right\| &\leq& \frac{1}{p^2}\sum_{j=1}^p\|Z_j-Z_j^*\|^2 \\
&\lesssim& \frac{1}{p^2}\sum_{j=1}^p\indc{Z_j\neq Z_j^*} \\
&\leq& \frac{1}{p^2\Delta_{\min}^2}\ell(Z,Z^*),
\end{eqnarray*}
we have
\begin{equation}
\|\bar{\epsilon}(Z)-\bar{\epsilon}(Z^*)\| \lesssim  \sqrt{\frac{\log d}{p\Delta^2_{\min}}}\sqrt{\ell(Z,Z^*)}. \label{eq:haken-vector} 
\end{equation}
Combine (\ref{eq:haken-virus}) and (\ref{eq:haken-vector}), and we obtain
\begin{equation}
\|\wh{\theta}(Z)-\wh{\theta}(Z^*)\| \lesssim \frac{1}{p\Delta_{\min}}\ell(Z,Z^*) +  \sqrt{\frac{\log d}{p\Delta^2_{\min}}}\sqrt{\ell(Z,Z^*)}. \label{eq:super-mario-maker}
\end{equation}

With (\ref{eq:super-mario-maker}), we are prepared to derive the bounds (\ref{eq:error-MRA1})-(\ref{eq:error-MRA3}). For (\ref{eq:error-MRA1}), we have
\begin{align*}
 & \sum_{j=1}^p\max_{U\neq Z_j^*} \frac{F_j(Z_j^*,U;Z)^2\|\mu_j(B^*,U) - \mu_j(B^*,Z_j^*)\|^2}{\Delta_j(Z_j^*,U)^4\ell(Z,Z^*)}\\
 &= \sum_{j=1}^p\max_{U\neq Z_j^*}\frac{\left|\iprod{\epsilon_j}{(Z_j^*-U)(\wh{\theta}(Z^*)-\wh{\theta}(Z))}\right|^2}{\|(Z_j^*-U)\theta^*\|^2\ell(Z,Z^*)} \\
 & \leq  \frac{1}{\Delta_{\min}^2\ell(Z,Z^*)}  \sum_{j=1}^p \max_{U\neq Z_j^*} (\wh{\theta}(Z^*)-\wh{\theta}(Z))^T  (Z_j^*-U)^T \epsilon_j \epsilon_j^T (Z_j^*-U) (\wh{\theta}(Z^*)-\wh{\theta}(Z)) \\
 & =  \frac{1}{\Delta_{\min}^2\ell(Z,Z^*)} \max_{\tilde U \in \mathcal{C}_d^p: \tilde U_j\neq Z_j^*,\forall j}  (\wh{\theta}(Z^*)-\wh{\theta}(Z))^T   \br{ \sum_{j=1}^p(Z_j^*-\tilde U_j)^T \epsilon_j \epsilon_j^T (Z_j^*-\tilde U_j) } (\wh{\theta}(Z^*)-\wh{\theta}(Z))  \\
 & \leq    \frac{\|\wh{\theta}(Z^*)-\wh{\theta}(Z)\|^2}{\Delta_{\min}^2\ell(Z,Z^*)} \max_{\tilde U \in \mathcal{C}_d^p: \tilde U_j\neq Z_j^*,\forall j}  \norm{  \sum_{j=1}^p(Z_j^*-\tilde U_j)^T \epsilon_j \epsilon_j^T (Z_j^*-\tilde U_j) }.
\end{align*}
Note that 
\begin{align*}
& \max_{\tilde U \in \mathcal{C}_d^p: \tilde U_j\neq Z_j^*,\forall j}   \norm{  \sum_{j=1}^p(Z_j^*-\tilde U_j)^T \epsilon_j \epsilon_j^T (Z_j^*-\tilde U_j) }^2  \\
 &=  \max_{\tilde U \in \mathcal{C}_d^p: \tilde U_j\neq Z_j^*,\forall j}  \norm{\br{(Z_1^*-\tilde U_1)^T \epsilon_1,\ldots, (Z_p^*-\tilde U_p)^T \epsilon_p} }^2\\
 & \leq 2\norm{\br{Z_1^T\epsilon_1,\ldots, Z_p^T \epsilon_p}}^2 + \max_{\tilde U \in \mathcal{C}_d^p: \tilde U_j\neq Z_j^*,\forall j}  \norm{\br{\tilde U_1^T \epsilon_1,\ldots, \tilde U_p^T \epsilon_p} }^2\\
 & \leq  4 \max_{\tilde U \in \mathcal{C}_d^p}  \norm{\br{\tilde U_1^T \epsilon_1,\ldots, \tilde U_p^T \epsilon_p} }^2.
\end{align*}
Note that for any fixed $\tilde U \in \mathcal{C}_d^p$, $(\tilde U_1^T \epsilon_1,\ldots, \tilde U_p^T \epsilon_p)$ each entries independently distributed from a standard normal distribution. By Lemma B.1 of \cite{loffler2019optimality}, we have $\p( \|(\tilde U_1^T \epsilon_1,\ldots, \tilde U_p^T \epsilon_p) \| \geq \sqrt{p} + \sqrt{d} + t) \leq \ebr{-t^2/2}$. By a union bound, we have 
\begin{align*}
 \max_{\tilde U \in \mathcal{C}_d^p: \tilde U_j\neq Z_j^*,\forall j}   \norm{  \sum_{j=1}^p(Z_j^*-\tilde U_j)^T \epsilon_j \epsilon_j^T (Z_j^*-\tilde U_j) }^2 \leq 64\br{d + p\log d}
\end{align*}
with probability at least $1-\ebr{-p\log d}$. Hence,
\begin{align*}
 \sum_{j=1}^p\max_{U\neq Z_j^*} \frac{F_j(Z_j^*,U;Z)^2\|\mu_j(B^*,U) - \mu_j(B^*,Z_j^*)\|^2}{\Delta_j(Z_j^*,U)^4\ell(Z,Z^*)} &\leq  64\br{d + p\log d}  \frac{\|\wh{\theta}(Z^*)-\wh{\theta}(Z)\|^2}{\Delta_{\min}^2\ell(Z,Z^*)} \\
 &\lesssim  \frac{(\log d+d/p)\log d}{\Delta_{\min}^4} +  \frac{\ell(Z,Z^*)(\log d + d/p) }{p\Delta_{\min}^4},
\end{align*}
where the last inequality is by Lemma \ref{prop:chisq} using the fact that $\sum_{j=1}^p\|\epsilon_j\|^2\sim\chi_{pd}^2$ and (\ref{eq:super-mario-maker}). For (\ref{eq:error-MRA2}), we have
\begin{eqnarray*}
&& \frac{\tau}{4\Delta_{\min}^2|T|}\sum_{j\in T}\max_{U\neq Z_j^*}  \frac{G_j(Z_j^*,U;Z)^2\|\mu_j(B^*,U) - \mu_j(B^*,Z_j^*)\|^2}{\Delta_j(Z_j^*,U)^4\ell(Z,Z^*)}  \\
& = & \frac{\tau}{4\Delta_{\min}^2|T|}\sum_{j\in T}\max_{U\neq Z_j^*}\frac{\left|\iprod{\wh{\theta}(Z)-\wh{\theta}(Z^*)}{(U^TZ_j^*-I_d)\theta^*}\right|^2}{\|(Z_j^*-U)\theta^*\|^2\ell(Z,Z^*)} \\
&\leq& \frac{\tau}{4\Delta_{\min}^2}\frac{\|\wh{\theta}(Z)-\wh{\theta}(Z^*)\|^2}{\ell(Z,Z^*)} \\
&\lesssim& \frac{\tau \log d}{p\Delta_{\min}^4} + \frac{\tau \ell(Z,Z^*)}{p^2\Delta_{\min}^4},
\end{eqnarray*}
where the last inequality is by (\ref{eq:super-mario-maker}). Finally, for (\ref{eq:error-MRA3}), since $\wh{\theta}(Z^*)-\theta^*=\frac{1}{p}\sum_{j=1}^pZ_j^T\epsilon_j\sim \mathn(0,p^{-1}I_d)$, we have $\|\wh{\theta}(Z^*)-\theta^*\|^2\lesssim \frac{d}{p}$ with probability at least $1-e^{-C'd}$ by Lemma \ref{prop:chisq}. Then,
\begin{eqnarray*}
\frac{|H_j(Z_j^*,U)|}{\Delta_j(Z_j^*,U)^2} = \frac{\left|\iprod{\wh{\theta}(Z^*)-\theta^*}{(U^TZ_j^*-I_d)\theta^*}\right|}{\|(Z_j^*-U)\theta^*\|^2} \leq \frac{\|\wh{\theta}(Z^*)-\theta^*\|}{\Delta_{\min}} \lesssim& \sqrt{\frac{d}{\Delta_{\min}^2p}},
\end{eqnarray*}
for any $j\in[p]$ and any $U\neq Z_j^*$.
The proof is complete.
\end{proof}

\begin{proof}[Proof of Lemma \ref{lem:ideal-MRA}]
For any $U\in\mathcal{C}_d\backslash\{Z_1^*\}$, we have
\begin{eqnarray}
\nonumber && \mathbb{P}\left(\iprod{\epsilon_1}{(Z_1^*-U)\wh{\theta}(Z^*)}\leq -\frac{1-\delta}{2}\|(Z_1^*-U)\theta^*\|^2\right) \\
\label{eq:hollow-knight} &\leq& \mathbb{P}\left(\iprod{\epsilon_1}{(Z_1^*-U)\theta^*}\leq -\frac{1-\delta-\bar{\delta}}{2}\|(Z_1^*-U)\theta^*\|^2\right) \\
\label{eq:silksong} && + \mathbb{P}\left(\iprod{\epsilon_1}{(Z_1^*-U)(\wh{\theta}(Z^*)-\theta^*)}\leq -\frac{\bar{\delta}}{2}\|(Z_1^*-U)\theta^*\|^2\right),
\end{eqnarray}
where the sequence $\bar{\delta}=\bar{\delta}_p$ is to be chosen later.
For (\ref{eq:hollow-knight}), a standard Gaussian tail bound gives
$$\mathbb{P}\left(\iprod{\epsilon_1}{(Z_1^*-U)\theta^*}\leq -\frac{1-\delta-\bar{\delta}}{2}\|(Z_1^*-U)\theta^*\|^2\right)\leq \exp\left(-\frac{(1-\delta-\bar{\delta})^2}{8}\|(Z_1^*-U)\theta^*\|^2\right).$$
Since
$$\iprod{\epsilon_1}{(Z_1^*-U)(\wh{\theta}(Z^*)-\theta^*)}=\frac{1}{p}\sum_{j=1}^p\epsilon_1^T(Z_1^*-U)Z_j^{*T}\epsilon_j,$$
we can bound (\ref{eq:silksong}) by
\begin{eqnarray}
\label{eq:silksong1} && \mathbb{P}\left(\frac{1}{p}\epsilon_1^T(I_d-UZ_1^{*T})\epsilon_1 \leq -\frac{\bar{\delta}}{4}\|(Z_1^*-U)\theta^*\|^2\right) \\
\label{eq:silksong2} && + \mathbb{P}\left(\frac{1}{p}\sum_{j=2}^p\epsilon_1^T(Z_1^*-U)Z_j^{*T}\epsilon_j\leq -\frac{\bar{\delta}}{4}\|(Z_1^*-U)\theta^*\|^2\right).
\end{eqnarray}
We first bound (\ref{eq:silksong1}) by
\begin{eqnarray*}
&& \mathbb{P}\left(\frac{1}{p}\epsilon_1^T(I_d-UZ_1^{*T})\epsilon_1 \leq -\frac{\bar{\delta}}{4}\|(Z_1^*-U)\theta^*\|^2\right) \\
&\leq& \mathbb{P}\left(2\|\epsilon_1\|^2 > \frac{\bar{\delta}p}{4}\|(Z_1^*-U)\theta^*\|^2\right) \\
&\leq& \exp\left(-C\bar{\delta}p\|(Z_1^*-U)\theta^*\|^2\right),
\end{eqnarray*}
where we use $\|I_d-UZ_1^{*T}\|\leq 2$, under the condition that $p\Delta_{\min}^2/d\rightarrow\infty$ and $\bar{\delta}$ tends to zero at a sufficiently slow rate. Note that conditional on $\epsilon_1$, $\frac{1}{p}\sum_{j=2}^p\epsilon_1^T(Z_1^*-U)Z_j^{*T}\epsilon_j \sim\mathn(0,p^{-2}(p-1)\|(Z_1^* - U)^T \epsilon_1\|^2)$, where the variance is upper bounded by $4p^{-1} \norm{\epsilon_1}^2$.
We then bound (\ref{eq:silksong2}) by
\begin{eqnarray*}
&& \mathbb{P}\left(\frac{1}{p}\sum_{j=2}^p\epsilon_1^T(Z_1^*-U)Z_j^{*T}\epsilon_j\leq -\frac{\bar{\delta}}{4}\|(Z_1^*-U)\theta^*\|^2\Big| \|\epsilon_1\|^2<d+2\sqrt{xd}+2x\right) \\
&& + \mathbb{P}\left(\|\epsilon_1\|^2\geq d+2\sqrt{xd}+2x\right) \\
&\leq& \mathbb{P}\left(\mathn(0,1)> \frac{\sqrt{p}\frac{\bar{\delta}}{4}\|(Z_1^*-U)\theta^*\|^2}{2\|\epsilon_1\|}\Big| \|\epsilon_1\|^2<d+2\sqrt{xd}+2x\right) + e^{-x} \\
&\leq& \exp\left(-\frac{p\bar{\delta}^2\|(Z_1^*-U)\theta^*\|^4}{128\left(d+2\sqrt{xd}+2x\right)}\right) + e^{-x}.
\end{eqnarray*}
Choosing $x=\bar{\delta}\|(Z_1^*-U)\theta^*\|^2\sqrt{p}$, we have
\begin{eqnarray*}
&& \mathbb{P}\left(\frac{1}{p}\sum_{j=2}^p\epsilon_1^T(Z_1^*-U)Z_j^{*T}\epsilon_j\leq -\frac{\bar{\delta}}{4}\|(Z_1^*-U)\theta^*\|^2\right) \\
&\leq& \exp\left(-C\frac{\bar{\delta}^2\|(Z_1^*-U)\theta^*\|^4p}{d + \bar{\delta}\|(Z_1^*-U)\theta^*\|^2\sqrt{p}}\right) + \exp\left(-C\bar{\delta}\|(Z_1^*-U)\theta^*\|^2\sqrt{p}\right) \\
&\leq& 2\exp\left(-\frac{(1-\delta-\bar{\delta})^2}{8}\|(Z_1^*-U)\theta^*\|^2\right),
\end{eqnarray*}
under the condition that $\sqrt{p}\Delta_{\min}^2/d\rightarrow\infty$ and $\bar{\delta}$ tends to zero at a sufficiently slow rate. Combining the above bounds, we have
$$\mathbb{P}\left(\iprod{\epsilon_1}{(Z_1^*-U)\wh{\theta}(Z^*)}\leq -\frac{1-\delta}{2}\|(Z_1^*-U)\theta^*\|^2\right)\leq 4\exp\left(-\frac{(1-\delta-\bar{\delta})^2}{8}\|(Z_1^*-U)\theta^*\|^2\right).$$
A similar bound holds for $\mathbb{P}\left(\iprod{\epsilon_j}{(Z_j^*-U)\wh{\theta}(Z^*)}\leq -\frac{1-\delta}{2}\|(Z_j^*-U)\theta^*\|^2\right)$ for each $j\in[p]$.

Now we are ready to bound $\xi_{\rm ideal}(\delta)$. We first bound its expectation. We have
\begin{eqnarray*}
\mathbb{E}\xi_{\rm ideal}(\delta) &=& \sum_{j=1}^p\sum_{U\in\mathcal{C}_d\backslash\{Z_j^*\}}\|(Z_j^*-U)\theta^*\|^2\mathbb{P}\left(\iprod{\epsilon_j}{(Z_j^*-U)\wh{\theta}(Z^*)}\leq -\frac{1-\delta}{2}\|(Z_j^*-U)\theta^*\|^2\right) \\
&\leq& 4\sum_{j=1}^p\sum_{U\in\mathcal{C}_d\backslash\{Z_j^*\}}\|(Z_j^*-U)\theta^*\|^2\exp\left(-\frac{(1-\delta-\bar{\delta})^2}{8}\|(Z_j^*-U)\theta^*\|^2\right) \\
&\leq& p\exp\left(-(1+o(1))\frac{\Delta_{\min}^2}{8}\right),
\end{eqnarray*}
where the last inequality uses the condition that $\Delta_{\min}^2/\log d\rightarrow\infty$. The desired conclusion is implied by Markov inequality.
\end{proof}

\begin{proof}[Proof of Proposition \ref{prop:ini-MRA}]
We adopt the notation $\wh{Z}_j=Z_j^{(0)}$ in the proof.
For any $j\geq 2$,
$$\|Y_1-\wh{Z}_j^TY_j\|^2\leq \|Y_1-Z_1^*Z_j^{*T}Y_j\|^2.$$
After rearrangement, we get
\begin{align*}
\|(Z_1^*-\wh{Z}_j^TZ_j^*)\theta^*\|^2 &\leq 2\left|\iprod{\wh{Z}_j^T\epsilon_j-\epsilon_1}{(Z_1^*-\wh{Z}_j^TZ_j^*)\theta^*}\right| + 2\epsilon_1^T \br{\wh{Z}_j^T - Z_1^*Z_j^{*T}}\epsilon_j\\
& \leq 2 \|(Z_1^*-\wh{Z}_j^TZ_j^*)\theta^*\| \frac{\left|\iprod{\wh{Z}_j^T\epsilon_j-\epsilon_1}{(Z_1^*-\wh{Z}_j^TZ_j^*)\theta^*}\right|}{\|(Z_1^*-\wh{Z}_j^TZ_j^*)\theta^*\|} + 2\abs{\epsilon_1^T \br{\wh{Z}_j^T - Z_1^*Z_j^{*T}}\epsilon_j}.
\end{align*}
This implies
\begin{align*}
\|(Z_1^*-\wh{Z}_j^TZ_j^*)\theta^*\| & \leq 2 \frac{\left|\iprod{\wh{Z}_j^T\epsilon_j-\epsilon_1}{(Z_1^*-\wh{Z}_j^TZ_j^*)\theta^*}\right|}{\|(Z_1^*-\wh{Z}_j^TZ_j^*)\theta^*\|}  + \sqrt{2} \sqrt{\abs{\epsilon_1^T \br{\wh{Z}_j^T - Z_1^*Z_j^{*T}}\epsilon_j}},
\end{align*}
and consequently,
\begin{align*}
\|(Z_1^*-\wh{Z}_j^TZ_j^*)\theta^*\| ^2 & \leq 8  \frac{\left|\iprod{\wh{Z}_j^T\epsilon_j-\epsilon_1}{(Z_1^*-\wh{Z}_j^TZ_j^*)\theta^*}\right|^2}{\|(Z_1^*-\wh{Z}_j^TZ_j^*)\theta^*\|^2}   + 8 \abs{\epsilon_1^T \br{\wh{Z}_j^T - Z_1^*Z_j^{*T}}\epsilon_j} \\
& \leq 16 \max_{U\in \mathcal{C}_d}   \frac{\left|\iprod{U^T\epsilon_j-\epsilon_1}{(Z_1^*-U^T Z_j^*)\theta^*}\right|^2}{2\|(Z_1^*-U^T Z_j^*)\theta^*\|^2}   + 8\max_{U\in \mathcal{C}_d} \abs{\epsilon_1^T \br{U^T - Z_1^*Z_j^{*T}}\epsilon_j}
\end{align*}

We are going to calculate its expectation. By a standard union bound argument, we have
$$\mathbb{P}\left(\max_{U\in\mathcal{C}_d}\frac{\left|\iprod{U^T\epsilon_j-\epsilon_1}{(Z_1^*-U^TZ_j^*)\theta^*}\right|^2}{2\|(Z_1^*-U^TZ_j^*)\theta^*\|^2} > t\right)\leq 2d\exp\left(-\frac{t}{2}\right).$$
Integrating up the tail bound, we obtain
\begin{align*}
\E \max_{U\in\mathcal{C}_d}\frac{\left|\iprod{U^T\epsilon_j-\epsilon_1}{(Z_1^*-U^TZ_j^*)\theta^*}\right|^2}{2\|(Z_1^*-U^TZ_j^*)\theta^*\|^2} \leq  4\log d
\end{align*}
Note that $\abs{\epsilon_1^T U^T \epsilon_j }= \norm{\epsilon_1} \abs{ \norm{\epsilon_1} ^{-1}\epsilon_1^T U^T \epsilon_j }$ where $\norm{\epsilon_1} ^{-1}\epsilon_1^T U^T \epsilon_j  \sim \mathn(0,1)$ is independent of $ \norm{\epsilon_1}$, we have
\begin{align*}
\E \max_{U\in \mathcal{C}_d} \abs{\epsilon_1^T \br{U^T - Z_1^*Z_j^{*T}}\epsilon_j} & \leq  2 \E \max_{U\in \mathcal{C}_d} \abs{\epsilon_1^T U^T \epsilon_j}  \leq  \E \norm{\epsilon_1} \E \max_{U\in \mathcal{C}_d}  \frac{\abs{\epsilon_1^T U^T \epsilon_j}}{ \norm{\epsilon_1} } \lesssim \sqrt{d\log d},
\end{align*}
where in the last inequality we integrate up the tail bound of $\p(\max_{U\in \mathcal{C}_d}  \abs{\epsilon_1^T U^T \epsilon_j}/ \norm{\epsilon_1}   >t )\leq 2d\ebr{-t^2/2}$. Hence,
\begin{align*}
\E \|(Z_1^*-\wh{Z}_j^TZ_j^*)\theta^*\| ^2 \lesssim \E \max_{U\in\mathcal{C}_d}\frac{\left|\iprod{U^T\epsilon_j-\epsilon_1}{(Z_1^*-U^TZ_j^*)\theta^*}\right|^2}{2\|(Z_1^*-U^TZ_j^*)\theta^*\|^2}  +  \E \max_{U\in \mathcal{C}_d} \abs{\epsilon_1^T \br{U^T - Z_1^*Z_j^{*T}}\epsilon_j}  & \lesssim \sqrt{d\log d}.
\end{align*}
Then,
$$\sum_{j=1}^p\mathbb{E}\|(\wh{Z}_jZ_1^*-Z_j^*)\theta^*\|^2=\sum_{j=1}^p\mathbb{E}\|(Z_1^*-\wh{Z}_j^TZ_j^*)\theta^*\|^2\lesssim p \sqrt{d\log d}.$$
The desired conclusion is implied by Markov inequality.
\end{proof}

\subsection{Proofs in Section \ref{sec:z2}}

In this section, we present the proofs of Lemma \ref{lem:error-z2}, Lemma \ref{lem:ideal-z2} and Proposition \ref{prop:ini-z2}. The conclusions of Theorem \ref{thm:z2-alg} and Corollary \ref{cor:alg-z2} are direct consequences of Theorem \ref{thm:main}, and their proofs are omitted. We first need a technical lemma that bound the operator nomr of a Gaussian random matrix. The following lemma is a standard result that can be found in \cite{vershynin2010introduction}.

\begin{lemma}\label{lem:gw-r-m}
Consider a matrix $W=(w_{ij})\in\mathbb{R}^{p\times p}$ with $W_{ij}=W_{ji}\sim\mathn(0,1)$ independently for all $i<j$ and $W_{ii}=0$ for all $i$. There exists some constant $C>0$, such that
$$\mathbb{P}\left(\|W\| > C\sqrt{p +x}\right)\leq e^{-x},$$
for any $x>0$.
\end{lemma}

\begin{proof}[Proof of Lemma \ref{lem:error-z2}]
We can write
$$\wh{\beta}(z)=\frac{z^T(\lambda^* z^*(z^*)^T+W)z}{p^2}z=\frac{\lambda^*|z^Tz^*|^2}{p^2}z+\frac{z^TWz}{p^2}z,$$
where $W$ is a symmetric matrix with $W_{ij}\sim\mathn(0,1)$ on the off-diagonal and $W_{ii}=0$ on the diagonal. Then, we have
$$\wh{\beta}(z^*)=\lambda^*z^*+\frac{z^{*T}Wz^*}{p^2}z^*=\beta^*+\frac{z^{*T}Wz^*}{p^2}z^*.$$
Using triangle inequality, we have
\begin{equation}
\|\wh{\beta}(z)-\wh{\beta}(z^*)\| \leq \left\|\frac{\lambda^*|z^Tz^*|^2}{p^2}z-\beta^*\right\| + \left\|\frac{z^TWz}{p^2}z-\frac{z^{*T}Wz^*}{p^2}z^*\right\|. \label{eq:triangle-joycon}
\end{equation}
We will bound the two terms on the right hand side of (\ref{eq:triangle-joycon}) separately. The first term can be bounded as
\begin{eqnarray}
\nonumber \left\|\frac{\lambda^*|z^Tz^*|^2}{p^2}z-\beta^*\right\| &\leq& \lambda^*\left|\frac{|z^Tz^*|^2}{p^2}-1\right|\|z\| + \lambda^*\|z-z^*\| \\
\label{eq:splatoon} &\leq& \lambda^*\frac{\|z-z^*\|^2}{\sqrt{p}} + \lambda^*\|z-z^*\| \\
&\leq& 3\lambda^*\|z-z^*\|,
\nonumber \end{eqnarray}
where the inequality (\ref{eq:splatoon}) is by
\begin{equation}
|p^2-|z^Tz^*|^2|=(p-|z^Tz^*|)(p+|z^Tz^*|)\leq 2p(p-|z^Tz^*|)\leq p\|z-z^*\|^2. \label{eq:pro-con111}
\end{equation}
For the second term of (\ref{eq:triangle-joycon}), we have
\begin{eqnarray}
\nonumber && \left\|\frac{z^TWz}{p^2}z-\frac{z^{*T}Wz^*}{p^2}z^*\right\| \\
\nonumber &\leq& \frac{|\iprod{W}{zz^T-z^*(z^*)^T}|}{p^2}\|z\| + \frac{|z^{*T}Wz^*|}{p^2}\|z-z^*\| \\
\label{eq:ringcon1} &\leq& \sqrt{2}\|W\|\fnorm{zz^T-z^*z^{*T}}\frac{\|z\|}{p^2} + \frac{|z^{*T}Wz^*|}{p^2}\|z-z^*\| \\
\label{eq:ringcon2} &\leq& \frac{3\|W\|}{p}\|z-z^*\| \\
\label{eq:ringcon3} &\lesssim& \frac{\|z-z^*\|}{\sqrt{p}}.
\end{eqnarray}
The inequality (\ref{eq:ringcon1}) is by applying SVD to the rank-$2$ matrix $zz^T-z^*z^{*T}=\sum_{l=1}^2d_lu_lu_l^T$ so that $|\iprod{W}{zz^T-z^*(z^*)^T}|$ is bounded by
$$\sum_{l=1}^2|d_l||u_l^TWu_l| \leq \|W\|(|d_1|+|d_2|)\leq \sqrt{2}\|W\|\sqrt{d_1^2+d_2^2}=\sqrt{2}\|W\|\fnorm{zz^T-z^*(z^*)^T}.$$
To obtain (\ref{eq:ringcon2}), we have used $\fnorm{zz^T-z^*(z^*)^T}^2=2|p^2-|z^Tz^*|^2|\leq p\|z-z^*\|^2$ according to (\ref{eq:pro-con111}) and $|z^{*T}Wz^*|\leq p\|W\|$. Finally, (\ref{eq:ringcon3}) is by Lemma \ref{lem:gw-r-m}.
Combining the two bounds for the two terms on the right hand side of (\ref{eq:triangle-joycon}), we have
\begin{equation}
\|\wh{\beta}(z)-\wh{\beta}(z^*)\| \lesssim \left(\lambda^*+p^{-1/2}\right)\|z-z^*\|. \label{eq:steady-body}
\end{equation}

Now we are ready to bound (\ref{eq:error-z21})-(\ref{eq:error-z23}). For (\ref{eq:error-z21}), we have
\begin{align*}
&\sum_{j=1}^p  \max_{b\in\{-1,1\}\setminus\{z_j^*\}} \frac{F_j(z_j^*,b;z)^2 \norm{\mu_j(B^*,b)-\mu_j(B^*,z_j^*)}^2}{\Delta_j(z_j^*,b)^4 \ell(z,z^*)} \\
& = \sum_{j=1}^p \frac{F_j(z_j^*,-z_j^*;z)^2 \norm{\mu_j(B^*,-z_j^*)-\mu_j(B^*,z_j^*)}^2}{\Delta_j(z_j^*,-z_j^*)^4 \ell(z,z^*)}\\
&\lesssim  \sum_{j=1}^p\frac{\left|\iprod{\epsilon_j}{\wh{\beta}(z)-\wh{\beta}(z^*)}\right|^2}{\|\beta^*\|^2\ell(z,z^*)}\\
& \leq  \frac{\|\wh{\beta}(z)-\wh{\beta}(z^*)\|^2}{\lambda^{*2}p\ell(z,z^*)}\left\|\sum_{j=1}^p\epsilon_j\epsilon_j^T\right\| \\
&\lesssim \frac{1}{p\lambda^{*2}} + \frac{1}{p^2\lambda^{*4}},
\end{align*}
where the last inequality is by (\ref{eq:z2_loss_identity}), (\ref{eq:steady-body}) and (\ref{eq:esp-km3}). For (\ref{eq:error-z22}), for any $T\subset [p]$, we have
\begin{align*}
& \frac{\tau}{4\Delta_{\min}^2|T|} \sum_{j\in T}  \max_{b\in\{-1,1\}\setminus\{z_j^*\}} \frac{G_j(z_j^*,b;z)^2 \norm{\mu_j(B^*,b)-\mu_j(B^*,z_j^*)}^2}{\Delta_j(z_j^*,b)^4 \ell(z,z^*)} \\
& = \frac{\tau}{4\Delta_{\min}^2|T|}  \sum_{j\in T} \frac{G_j(z_j^*,-z_j^*;z)^2 \norm{\mu_j(B^*,-z_j^*)-\mu_j(B^*,z_j^*)}^2}{\Delta_j(z_j^*,-z_j^*)^4 \ell(z,z^*)}\\
& \lesssim \frac{\tau}{4\Delta_{\min}^2|T|}\sum_{j\in T}\frac{\left|\iprod{\beta^*}{\wh{\beta}(z)-\wh{\beta}(z^*)}\right|^2}{\|\beta^*\|^2\ell(z,z^*)} \\
&\lesssim \frac{\tau\|\wh{\beta}(z)-\wh{\beta}(z^*)\|^2}{\|\beta^*\|^2\ell(z,z^*)} \\
&\lesssim \frac{\tau}{\lambda^{*2}p^2} + \frac{\tau}{\lambda^{*4}p^3},
\end{align*}
where the last inequality is by (\ref{eq:steady-body}). Finally, for (\ref{eq:error-z23}), since $\|\wh{\beta}(z^*)-\beta^*\|=\frac{|z^{*T}Wz^*|}{p^2}\|z^*\|\lesssim 1$ by Lemma \ref{lem:gw-r-m}, we have
\begin{eqnarray*}
\max_{b\neq z_j^*} \frac{\abs{H_j(z_j^*,b)}}{\Delta_j(z_j^*,b)^2} = \frac{\abs{H_j(z_j^*,-z_j^*)}}{\Delta_j(z_j^*,-z_j^*)^2} = \frac{\left|\iprod{\beta^*}{\wh{\beta}(z^*)-\beta^*}\right|}{2\|\beta^*\|^2} &\leq& \frac{1}{2}\frac{\|\wh{\beta}(z^*)-\beta^*\|}{\|\beta^*\|} \lesssim \frac{1}{\sqrt{p\lambda^{*2}}},
\end{eqnarray*}
for any $j\in[p]$.
The proof is complete.
\end{proof}

\begin{proof}[Proof of Lemma \ref{lem:ideal-z2}]
Without loss of generality, assume $z_1^*=1$, and then for $b\neq z_1^*$, we have
\begin{eqnarray}
\nonumber && \mathbb{P}\left(\iprod{W_1}{\wh{\beta}(z^*)(z_1^*-b)}\leq -\frac{1-\delta}{2}4\|\beta^*\|^2\right) \\
\nonumber &=& \mathbb{P}\left(\iprod{W_1}{\wh{\beta}(z^*)}\leq (1-\delta)p\lambda^{*2}\right) \\
\label{eq:cuphead1} &\leq& \mathbb{P}\left(\iprod{W_1}{\beta^*}\leq -(1-\delta-\delta')p\lambda^{*2}\right) \\
\label{eq:cuphead2}&& + \mathbb{P}\left(\iprod{W_1}{\wh{\beta}(z^*)-\beta^*}\leq -\delta'p\lambda^{*2}\right),
\end{eqnarray}
for some sequence $\delta'=\delta'_p$ to be determined later. For (\ref{eq:cuphead1}), a standard Gaussian tail bound gives
\begin{eqnarray*}
&& \mathbb{P}\left(\iprod{W_1}{\beta^*}\leq -(1-\delta-\delta')p\lambda^{*2}\right) \\
&\leq& \mathbb{P}\left(\mathn(0,p\lambda^{*2})>(1-\delta-\delta')p\lambda^{*2}\right) \\
&\leq& \exp\left(-\frac{(1-\delta-\bar{\delta})^2p\lambda^{*2}}{2}\right).
\end{eqnarray*}
For (\ref{eq:cuphead2}), note that we have
$$\wh{\beta}(z^*)-\beta^*=\frac{z^{*T}Wz^*}{p^2}z^*,$$
and thus
\begin{eqnarray*}
&& \mathbb{P}\left(\iprod{W_1}{\wh{\beta}(z^*)-\beta^*}\leq -\bar{\delta}p\lambda^{*2}\right) \\
&=& \mathbb{P}\left(\frac{z^{*T}Wz^*}{p^2}\iprod{W_1}{z^*}\leq-\bar{\delta}p\lambda^{*2}\right) \\
&\leq& \mathbb{P}\left(\|W\||\iprod{W_1}{z^*}|\geq \bar{\delta} p^2\lambda^{*2}\right) \\
&\leq& \mathbb{P}\left(\|W\||\iprod{W_1}{z^*}|\geq \bar{\delta} p^2\lambda^{*2}, \|W\|\leq C\sqrt{p+x}\right) + \mathbb{P}(\|W\|>C\sqrt{p+x}) \\
&\leq& \mathbb{P}\left(C\sqrt{p+x}|\iprod{W_1}{z^*}|\geq \bar{\delta}p^2\lambda^{*2}\right) + e^{-x} \\
&\leq& 2\exp\left(-C'\frac{\bar{\delta}^2p^3\lambda^{*4}}{p+x}\right) + e^{-x}.
\end{eqnarray*}
Take $x=\bar{\delta}p^{3/2}\lambda^{*2}$, and we have
\begin{eqnarray*}
&& \mathbb{P}\left(\iprod{W_1}{\wh{\beta}(z^*)-\beta^*}\leq -\bar{\delta}p\lambda^{*2}\right) \\
&\leq& 2\exp\left(-C'\bar{\delta}^2p^2\lambda^{*4}\right) + 3\exp\left(-C'\bar{\delta}p^{3/2}\lambda^{*2}\right) \\
&\leq& 5\exp\left(-\frac{(1-\delta-\bar{\delta})^2p\lambda^{*2}}{2}\right),
\end{eqnarray*}
where the last inequality uses the condition that $p\lambda^{*2}\rightarrow\infty$ and the fact that $\bar{\delta}$ tends to zero at a sufficiently slow rate. Combining the above bounds, we obtain
$$\mathbb{P}\left(\iprod{W_1}{\wh{\beta}(z^*)(z_1^*-b)}\leq -\frac{1-\delta}{2}4\|\beta^*\|^2\right)\leq 6\exp\left(-\frac{(1-\delta-\bar{\delta})^2p\lambda^{*2}}{2}\right).$$
A similar bound holds for $\mathbb{P}\left(\iprod{W_j}{\wh{\beta}(z^*)(z_j^*-b)}\leq -\frac{1-\delta}{2}4\|\beta^*\|^2\right)$ for each $j\in[p]$. This implies
\begin{eqnarray*}
\mathbb{E}\xi_{\rm ideal}(\delta) &=& 4p\lambda^{*2}\sum_{j=1}^p\mathbb{P}\left(2\iprod{W_j}{\wh{\beta}(z^*)z_j^*}\leq -\frac{1-\delta}{2}4\|\beta^*\|^2\right) \\
&\leq& 24p^2\lambda^{*2}\exp\left(-\frac{(1-\delta-\bar{\delta})^2p\lambda^{*2}}{2}\right) \\
&\leq& p\exp\left(-(1+o(1))\frac{p\lambda^{*2}}{2}\right),
\end{eqnarray*}
where the last inequality uses $p\lambda^{*2}\rightarrow\infty$. The desired conclusion is implied by Markov inequality.
\end{proof}

\begin{proof}[Proof of Proposition \ref{prop:ini-z2}]
We can write $z^{(0)}=\argmin_{z\in\{-1,1\}^p}\|z-\sqrt{p}\wt{u}\|^2$. This implies
$$\|z^{(0)}-z^*\|^2\leq 2\|\sqrt{p}\wt{u}-z^*\|^2 + 2\|z^{(0)}-\sqrt{p}\wt{u}\|^2\leq 4\|\sqrt{p}\wt{u}-z^*\|^2.$$
Similarly, we also have $\|z^{(0)}+z^*\|^2\leq 4\|\sqrt{p}\wt{u}+z^*\|^2$. Thus,
\begin{equation}
\|z^{(0)}-z^*\|^2\wedge \|z^{(0)}+z^*\|^2\leq 4\left(\|\sqrt{p}\wt{u}-z^*\|^2\wedge \|\sqrt{p}\wt{u}+z^*\|^2\right).\label{eq:dirtmouth}
\end{equation}
Since $\wt{u}$ is the leading eigenvector of $Y$ and $z^*/\sqrt{p}$ is the leading eigenvector of $\lambda^*z^*z^{*T}$, by Davis-Kahan theorem, we have
$$\left\|\wt{u}-\frac{z^*}{\sqrt{p}}\right\|\wedge \left\|\wt{u}+\frac{z^*}{\sqrt{p}}\right\|\lesssim \frac{\|W\|}{p|\lambda^*|}\lesssim \frac{1}{\sqrt{p\lambda^{*2}}},$$
where the last inequality above is by Lemma \ref{lem:gw-r-m}. Finally, by (\ref{eq:dirtmouth}), we have
$$\ell(z^{(0)},z^*)\wedge\ell(z^{(0)},-z^*)\lesssim p,$$
with high probability.
\end{proof}

\subsection{Proofs in Section \ref{sec:zk}}\label{sec:proof_zk}

In this section, we present the proofs of Theorem \ref{thm:lower-zk}, Lemma \ref{lem:error-zk}, Lemma \ref{lem:ideal-zk} and Proposition \ref{prop:ini}. The conclusion of Theorem \ref{thm:zk-con-alg} and Corollary \ref{cor:zk-con-alg} are direct consequences of Theorem \ref{thm:main}, and thus their proofs are omitted. We first present a technical lemma.

\begin{lemma}\label{lem:tech-lem-zk}
Consider i.i.d. random variables $w_{ij}\sim\mathn(0,1)$ for $1\leq i\neq j\leq p$ and some $z^*\in(\Zk)^p$. For any constant $C'>0$, there exists some constant $C>0$ only depending on $C'$ such that
\begin{equation}
\max_{z\in(\Zk)^p}\left|\frac{\sum_{i\neq j}w_{ij}\left[(z_i\circ z_j^{-1})-(z_i^*\circ z_j^{*-1})\right]}{\sqrt{\sum_{i\neq j}\left[(z_i\circ z_j^{-1})-(z_i^*\circ z_j^{*-1})\right]^2}}\right| \leq C\sqrt{p\log k},
\end{equation}
with probability at least $1-e^{-C'p\log k}$.
\end{lemma}
\begin{proof}
The result is implied by a standard Gaussian tail bound and a union bound argument.
\end{proof}

\begin{proof}[Proof of Theorem \ref{thm:lower-zk}]
Given the similarity between $\mathbb{Z}/2\mathbb{Z}$ and $\mathbb{Z}_2$, the proof for $k=2$ is very similar to the proof of Theorem \ref{thm:lower-z2} and thus is omitted here.
We only need to consider the case $k\geq 3$. Let $\rho=o(1)$ tends to zero at a sufficiently slow rate.
Consider the parameter space
$$\mathcal{Z}=\{z:z_j=2\text{ for all }1\leq j\leq (1-\rho) p\text{ and }z_j\in\{0,1\}\text{ for all }(1-\rho)p<j\leq p\}.$$
Then, by the same argument that leads to (\ref{eq:metroidvania}), we have
\begin{eqnarray*}
&& \inf_{\wh{z}}\sup_{z^*}\mathbb{E}\min_{a\in \Zk}\left(\frac{1}{p}\sum_{j=1}^p\indc{\wh{z}_j\neq z_j^*\circ a^{-1}}\right) \\
&\geq& \frac{1}{p}\sum_{j>(1-\rho )p}\text{ave}_{z_{-j}^*}\inf_{\wh{z}_j}\left(\frac{1}{2}\mathbb{P}_{(z_j^*=1,z_{-j}^*)}(\wh{z}_j\neq 1)+\frac{1}{2}\mathbb{P}_{(z_j^*=0,z_{-j}^*)}(\wh{z}_j\neq 0)\right).
\end{eqnarray*}
Note that the quantity
$$\inf_{\wh{z}_j}\left(\frac{1}{2}\mathbb{P}_{(z_j^*=1,z_{-j}^*)}(\wh{z}_j\neq 1)+\frac{1}{2}\mathbb{P}_{(z_j^*=0,z_{-j}^*)}(\wh{z}_j\neq 0)\right)$$
is the optimal error for the hypothesis testing problem,
\begin{eqnarray*}
H_0: && Y_{ij}\sim \mathn(\lambda^*[(z_i^*-1)\text{ mod }k],1)\text{ and }Y_{ji}\sim \mathn(\lambda^*[(1-z_i^*)\text{ mod }k],1)\text{ for }i\in[p]\backslash\{j\}, \\
H_1: && Y_{ij}\sim \mathn(\lambda^*[(z_i^*-0)\text{ mod }k],1)\text{ and }Y_{ji}\sim \mathn(\lambda^*[(0-z_i^*)\text{ mod }k],1)\text{ for }i\in[p]\backslash\{j\}.
\end{eqnarray*}
Note that
\begin{align*}
\lambda^{*-1}|[(z_i^*-1)\text{ mod }k]-[(z_i^*-0)\text{ mod }k]| \begin{cases}
 = 1, \text{ if }i\leq (1-\rho)p\\
 \leq k-1,\text{ o.w.}
\end{cases}
\end{align*}
and same result holds for $|[(1-z_i^*)\text{ mod }k]-[(0-z_i^*)\text{ mod }k]|$. We have
\begin{align*}
\sum_{j\neq i} \br{\abs{\E_{H_0} Y_{ij} - \E_{H_1} Y_{ij}} + \abs{\E_{H_0} Y_{ji} - \E_{H_1} Y_{ji}}} \leq 2\lambda^*\br{(1-\rho) p + (k-1) \rho p} \leq 2\lambda^*(1+\rho k)p.
\end{align*}
By Neyman-Pearson lemma, 
the testing error can be lower bounded by $\mathbb{P}(\mathn(0,1)\geq \lambda^*(1+\rho k)p/\sqrt{2(p-1)}$. Therefore, we have
\begin{eqnarray*}
\inf_{\wh{z}}\sup_{z^*}\mathbb{E}\min_{a\in \Zk}\left(\frac{1}{p}\sum_{j=1}^p\indc{\wh{z}_j\neq z_j^*\circ a^{-1}}\right) &\geq&  \rho\mathbb{P}(\mathn(0,1)\geq \lambda^*(1+\rho k)p/\sqrt{2(p-1)}) \\
&=& \exp\left(-(1+o(1))\frac{p\lambda^{*2}}{4}\right),
\end{eqnarray*}
under the condition that $p\lambda^{*2}\rightarrow\infty$ and $\rho$  tends to zero at a sufficiently slow rate. When $p\lambda^{*2}=O(1)$, we can take $\rho =1/2$ instead and obtain a constant lower bound.
\end{proof}

Now we are ready to state the proofs of Lemma \ref{lem:error-zk} and Lemma \ref{lem:ideal-zk}. Note that under the setting of $\Zk$ synchronization, the error terms are
\begin{eqnarray*}
F_j(a,b,z) &=& \iprod{\epsilon_j}{\wh{\lambda}(z^*)\left(z^*\circ a^{-1}-z^*\circ b^{-1}\right)-\wh{\lambda}(z)\left(z\circ a^{-1}-z\circ b^{-1}\right)}, \\
G_j(a,b,z) &=& \frac{1}{2}\left(\|\lambda^*(z^*\circ a^{-1})-\wh{\lambda}(z)(z\circ a^{-1})\|^2-\|\lambda^*(z^*\circ a^{-1})-\wh{\lambda}(z^*)(z^*\circ a^{-1})\|^2\right) \\
&& -\frac{1}{2}\left(\|\lambda^*(z^*\circ a^{-1})-\wh{\lambda}(z)(z\circ b^{-1})\|^2-\|\lambda^*(z^*\circ a^{-1})-\wh{\lambda}(z^*)(z^*\circ b^{-1})\|^2\right), \\
H_j(a,b) &=& \frac{1}{2}\|\lambda^*(z^*\circ a^{-1})-\wh{\lambda}(z^*)(z^*\circ a^{-1})\|^2 \\
&& -\frac{1}{2}\left(\|\lambda^*(z^*\circ a^{-1})-\wh{\lambda}(z^*)(z^*\circ b^{-1})\|^2-\|\lambda^*(z^*\circ a^{-1})-\lambda^*(z^*\circ b^{-1})\|^2\right),
\end{eqnarray*} 
where we have used the notation $\epsilon_j=W_j$ for the $j$th column of the error matrix, and $z\circ a^{-1}$ stands for the vector $\{z_i\circ a^{-1}\}_{i\in[p]}$.

\begin{proof}[Proof of Lemma \ref{lem:error-zk}]
Throughout the proof, we define $h(z,z^*) = \sum_{j\in[p]}\indc{z_j \neq z^*_j},\forall z$. Then we have
\begin{align}\label{eqn:zk_h_square_connection}
  \sum_{i\neq j}\left[(z_i^*\circ z_j^{*-1})-(z_i\circ z_j^{-1})\right]^2 \leq k^2 \sum_{i,j} \br{\indc{z^*_i \neq z_i} + \indc{z^*_j\neq z_j}}\leq  2pk^2h(z,z^*).
\end{align}
In addition, (\ref{eqn:zk_h_l_connection}) is equivalent to $h(z,z^*)\leq \frac{1}{p\lambda^{*2}} \ell(z,z^*)$.

Under the assumption that $\max_{a\in\Zk}\sum_{j=1}^p\indc{z_j^*=a}\leq (1-\alpha)p$ where $\alpha>0$ is a constant, we have
\begin{equation}
\sum_{i\neq j}(z_i^*\circ z_j^{*-1})^2\geq \sum_{i\neq j}\indc{z_i^*\neq z_j^*}\geq \alpha p^2.
\end{equation}
Similarly, since for any $z$ such that $h(z,z^*)\leq \alpha p/2$ we have $\max_{a\in\Zk}\sum_{j=1}^p\indc{z_j=a}\leq (1-\alpha/2)p$, we also obtain
\begin{equation}
\sum_{i\neq j}(z_i\circ z_j^{-1})^2\geq \sum_{i\neq j}\indc{z_i\neq z_j}\geq \alpha p^2/2. \label{eq:link-zelda}
\end{equation}
Now we will derive a bound for $|\wh{\lambda}(z)-\wh{\lambda}(z^*)|$, where $\wh{\lambda}(z)=\frac{\sum_{i\neq j}Y_{ij}(z_i\circ z_j^{-1})}{\sum_{i\neq j}(z_i\circ z_j^{-1})^2}$. We can write $\wh{\lambda}(z)$ as $\wh{\lambda}(z)=\lambda(z)+\bar{w}(z)$, where
$$\lambda(z)=\lambda^*\frac{\sum_{i\neq j}(z_i^*\circ z_j^{*-1})(z_i\circ z_j^{-1})}{\sum_{i\neq j}(z_i\circ z_j^{-1})^2},$$
and
$$\bar{w}(z)=\frac{\sum_{i\neq j}w_{ij}(z_i\circ z_j^{-1})}{\sum_{i\neq j}(z_i\circ z_j^{-1})^2}.$$
Then,
\begin{equation}
|\wh{\lambda}(z)-\wh{\lambda}(z^*)|\leq |\lambda(z)-\lambda(z^*)| + |\bar{w}(z)-\bar{w}(z^*)|,\label{eq:zelda}
\end{equation}
and we will bound the two terms on the right hand side of the above inequality separately. For the first term of (\ref{eq:zelda}), we have
\begin{eqnarray}
\nonumber |\lambda(z)-\lambda(z^*)| &=& |\lambda^*|\frac{\left|\sum_{i\neq j}\left[(z_i^*\circ z_j^{*-1})-(z_i\circ z_j^{-1})\right](z_i\circ z_j^{-1})\right|}{\sum_{i\neq j}(z_i\circ z_j^{-1})^2} \\
\nonumber &\leq& |\lambda^*|\frac{\sqrt{\sum_{i\neq j}\left[(z_i^*\circ z_j^{*-1})-(z_i\circ z_j^{-1})\right]^2}}{\sqrt{\sum_{i\neq j}(z_i\circ z_j^{-1})^2}} \\
\nonumber &\leq& \frac{\sqrt{2}|\lambda^*|}{\sqrt{\alpha} p}\sqrt{\sum_{i\neq j}\left[(z_i^*\circ z_j^{*-1})-(z_i\circ z_j^{-1})\right]^2} \\
\label{eq:mipha} &\lesssim& \frac{|\lambda^*|k}{p}\sqrt{p h(z,z^*)},
\end{eqnarray}
where we have used (\ref{eqn:zk_h_square_connection}) and (\ref{eq:link-zelda}).
For the second term of (\ref{eq:zelda}), we have
\begin{eqnarray}
\label{eq:moblin} |\bar{w}(z)-\bar{w}(z^*)| &\leq& \left|\frac{\sum_{i\neq j}w_{ij}\left[(z_i\circ z_j^{-1})-(z_i^*\circ z_j^{*-1})\right]}{\sum_{i\neq j}(z_i\circ z_j^{-1})^2}\right| \\
\label{eq:hoplin} && + \left|\frac{1}{\sum_{i\neq j}(z_i\circ z_j^{-1})^2}-\frac{1}{\sum_{i\neq j}(z_i^*\circ z_j^{*-1})^2}\right|\left|\sum_{i\neq j}w_{ij}(z_i^*\circ z_j^{*-1})\right|.
\end{eqnarray}
We can bound (\ref{eq:moblin}) by
\begin{eqnarray*}
&& \frac{\sqrt{\sum_{i\neq j}\left[(z_i\circ z_j^{-1})-(z_i^*\circ z_j^{*-1})\right]^2}}{\sum_{i\neq j}(z_i\circ z_j^{-1})^2}\left|\frac{\sum_{i\neq j}w_{ij}\left[(z_i\circ z_j^{-1})-(z_i^*\circ z_j^{*-1})\right]}{\sqrt{\sum_{i\neq j}\left[(z_i\circ z_j^{-1})-(z_i^*\circ z_j^{*-1})\right]^2}}\right| \\
&\lesssim& \frac{\sqrt{k^2(\log k) h(z,z^*)}}{p},
\end{eqnarray*}
where we have used (\ref{eqn:zk_h_square_connection}), (\ref{eq:link-zelda}) and Lemma \ref{lem:tech-lem-zk}. For (\ref{eq:hoplin}), we have
\begin{eqnarray*}
&& \left|\frac{1}{\sum_{i\neq j}(z_i\circ z_j^{-1})^2}-\frac{1}{\sum_{i\neq j}(z_i^*\circ z_j^{*-1})^2}\right| \\
&=& \left|\frac{\sum_{i\neq j}\left[(z_i\circ z_j^{-1})^2-(z_i^*\circ z_j^{*-1})^2\right]}{\left(\sum_{i\neq j}(z_i\circ z_j^{-1})^2\right)\left(\sum_{i\neq j}(z_i^*\circ z_j^{*-1})^2\right)}\right| \\
&\lesssim& \frac{k^2h(z,z^*)}{p^3},
\end{eqnarray*}
and thus
$$\left|\frac{1}{\sum_{i\neq j}(z_i\circ z_j^{-1})^2}-\frac{1}{\sum_{i\neq j}(z_i^*\circ z_j^{*-1})^2}\right|\left|\sum_{i\neq j}w_{ij}(z_i^*\circ z_j^{*-1})\right|\lesssim \frac{k^3\sqrt{\log p}h(z,z^*)}{p^2},$$
with probability at least $1-p^{-C'}$
by $\sum_{i\neq j}(z_i^*\circ z_j^{*-1})^2\lesssim k^2p^2$ and a standard Gaussian tail bound. Therefore, we can bound (\ref{eq:moblin}) by $\frac{\sqrt{k^2(\log k) h(z,z^*)}}{p}+\frac{k^3\sqrt{\log p}h(z,z^*)}{p^2}$ up to a constant. Together with (\ref{eq:mipha}), we have
\begin{equation}
|\wh{\lambda}(z)-\wh{\lambda}(z^*)| \lesssim \frac{|\lambda^*|k}{p}\sqrt{p h(z,z^*)} + \frac{\sqrt{k^2(\log k) h(z,z^*)}}{p}+\frac{k^3\sqrt{\log p}h(z,z^*)}{p^2}. \label{eq:deepnest}
\end{equation}
Moreover, we also have $\wh{\lambda}(z^*)-\lambda^*=\bar{w}(z^*) \sim \mathn(0, 1/\sum_{i\neq j} (z_i^* \circ z_j^{*-1})^2)$ where the variance is smaller than $1/\alpha p^2$ according to (\ref{eq:link-zelda}).  Thus,
\begin{equation}
|\wh{\lambda}(z^*)-\lambda^*| \lesssim \frac{1}{\sqrt{p}}, \label{eq:theabyss}
\end{equation}
with probability $1-e^{-C'p}$. This implies $\abs{\wh{\lambda}(z^*)}\lesssim \abs{\lambda^*} + \frac{1}{\sqrt{p}}$ and
\begin{equation}
\|\lambda^*z^*\circ a^{-1} - \wh{\lambda}(z^*)z^*\circ a^{-1}\|^2 \lesssim k^2. \label{eq:radiance}
\end{equation}
Using (\ref{eq:deepnest}) and (\ref{eq:theabyss}), we have
\begin{eqnarray}
\nonumber && \|\wh{\lambda}(z)z\circ a^{-1} - \wh{\lambda}(z^*)z^*\circ a^{-1}\| \\
\nonumber &\leq& |\wh{\lambda}(z)-\wh{\lambda}(z^*)|\|z\circ a^{-1}\| + |\wh{\lambda}(z^*)|\|z\circ a^{-1} - z^*\circ a^{-1}\| \\
\nonumber &\lesssim& \br{ |\lambda^*|k\sqrt{k}\sqrt{h(z,z^*)} + \frac{\sqrt{k^3\log k h(z,z^*)}}{\sqrt{p}}+\frac{k^3\sqrt{k}h(z,z^*)}{p\sqrt{p}} } + \br{\abs{\lambda^*} + \frac{1}{\sqrt{p}}} \sqrt{k h(z,z^*)}\\
\label{eq:aspid} &\lesssim& |\lambda^*|k\sqrt{k}\sqrt{h(z,z^*)},
\end{eqnarray}
under the condition that $\frac{p\lambda^{*2}}{k^4}\rightarrow\infty$.

Now we are ready to bound (\ref{eq:error-zk1})-(\ref{eq:error-zk3}). 
For (\ref{eq:error-zk1}), we have
\begin{eqnarray*}
&& \sum_{j=1}^p\max_{b\in\Zk\backslash\{z_j^*\}}\frac{F_j(z_j^*,b;z)^2 \norm{\mu_j(B^*,b) - \mu_j(B^*,z_j^*)}^2}{\Delta_j(z_j^*,b)^4\ell(z,z^*)} \\
&=& \sum_{j=1}^p\max_{b\in\Zk\backslash\{z_j^*\}}\frac{F_j(z_j^*,b;z)^2}{\Delta_j(z_j^*,b)^2\ell(z,z^*)} \\
&\leq& \sum_{a\in\Zk}\sum_{b\in\Zk\backslash\{a\}}\sum_{j=1}^p\indc{z_j^*=a}\frac{F_j(a,b;z)^2}{(p\lambda^{*2})^2h(z,z^*)} \\
&\leq& \frac{\max_{a\in\Zk}\|\wh{\lambda}(z)z\circ a^{-1} - \wh{\lambda}(z^*)z^*\circ a^{-1}\|^2}{(p\lambda^{*2})^2h(z,z^*)}\sum_{a\in\Zk}\left\|\sum_{j=1}^p\indc{z_j^*=a}\epsilon_j\epsilon_j^T\right\| \\
&\lesssim& \frac{k^4}{p\lambda^{*2}},
\end{eqnarray*}
by (\ref{eq:aspid}) and (\ref{eq:esp-km3}).
For (\ref{eq:error-zk2}), we have
\begin{eqnarray*}
&& \frac{\tau}{4\Delta_{\min}^2|T|}\sum_{j\in T}\max_{b\in\Zk\backslash\{z_j^*\}}\frac{G_j(z_j^*,b;z)^2 \norm{\mu_j(B^*,b) - \mu_j(B^*,z_j^*)}^2}{\Delta_j(z_j^*,b)^2\ell(z,z^*)} \\
&=& \frac{\tau}{4\Delta_{\min}^2|T|}\sum_{j\in T}\max_{b\in\Zk\backslash\{z_j^*\}}\frac{G_j(z_j^*,b;z)^2}{\Delta_j(z_j^*,b)^2\ell(z,z^*)} \\
&\lesssim& \frac{\tau}{(p\lambda^{*2})^3h(z,z^*)}\max_{a\in\Zk}\|\wh{\lambda}(z)z\circ a^{-1} - \wh{\lambda}(z^*)z^*\circ a^{-1}\|^4 \\
&& + \frac{\tau}{(p\lambda^{*2})^3h(z,z^*)}\left(\max_{a\in\Zk}\|\wh{\lambda}(z)z\circ a^{-1} - \wh{\lambda}(z^*)z^*\circ a^{-1}\|^2\right)\left(\max_{a\in\Zk}\|\lambda^*z^*\circ a^{-1} - \wh{\lambda}(z^*)z^*\circ a^{-1}\|^2\right) \\
&& +  \frac{\tau}{(p\lambda^{*2})^2h(z,z^*)}\max_{a\in\Zk}\|\wh{\lambda}(z)z\circ a^{-1} - \wh{\lambda}(z^*)z^*\circ a^{-1}\|^2 \\
&\lesssim& \frac{\tau k^6}{p^2\lambda^{*2}},
\end{eqnarray*}
by (\ref{eq:aspid}) and (\ref{eq:radiance}).
Finally, for (\ref{eq:error-zk3}), we have
\begin{eqnarray*}
\frac{|H_j(z_j^*,a)|}{\Delta_j(z_j^*,a)^2} &\lesssim& \frac{\max_{a\in\Zk}\|\lambda^*z^*\circ a^{-1} - \wh{\lambda}(z^*)z^*\circ a^{-1}\|^2}{p\lambda^{*2}} \\
&& + \sqrt{\frac{\max_{a\in\Zk}\|\lambda^*z^*\circ a^{-1} - \wh{\lambda}(z^*)z^*\circ a^{-1}\|^2}{p\lambda^{*2}}} \\
&\lesssim& \frac{k^2}{p\lambda^{*2}} + \sqrt{\frac{k^2}{p\lambda^{*2}}},
\end{eqnarray*}
by (\ref{eq:radiance}).
The proof is complete.
\end{proof}

\begin{proof}[Proof of Lemma \ref{lem:ideal-zk}]
For any $z_1^*$ and $b\neq z_1^*$, we have
\begin{eqnarray*}
&& \mathbb{P}\left(\sum_{j=1}^pw_{j1}(z_j^*\circ z_1^{*-1}-z_j^*\circ b^{-1})\wh{\lambda}(z^*)\leq -\frac{1-\delta}{2}\lambda^{*2}\sum_{j=1}^p(z_j^*\circ z_1^{*-1}-z_j^*\circ b^{-1})^2\right) \\
&\leq& \mathbb{P}\left(\sum_{j=1}^pw_{j1}(z_j^*\circ z_1^{*-1}-z_j^*\circ b^{-1})\lambda^*\leq -\frac{1-\delta-\bar{\delta}}{2}\lambda^{*2}\sum_{j=1}^p(z_j^*\circ z_1^{*-1}-z_j^*\circ b^{-1})^2\right) \\
&& + \mathbb{P}\left(\sum_{j=1}^pw_{j1}(z_j^*\circ z_1^{*-1}-z_j^*\circ b^{-1})(\wh{\lambda}(z^*)-\lambda^*)\leq -\frac{\bar{\delta}}{2}\lambda^{*2}\sum_{j=1}^p(z_j^*\circ z_1^{*-1}-z_j^*\circ b^{-1})^2\right),
\end{eqnarray*}
where the sequence $\bar{\delta}=\bar{\delta}_p$ is to be determined later. For the first term, by a standard Gaussian tail bound, we have
\begin{eqnarray*}
&& \mathbb{P}\left(\sum_{j=1}^pw_{j1}(z_j^*\circ z_1^{*-1}-z_j^*\circ b^{-1})\lambda^*\leq -\frac{1-\delta-\bar{\delta}}{2}\lambda^{*2}\sum_{j=1}^p(z_j^*\circ z_1^{*-1}-z_j^*\circ b^{-1})^2\right) \\
&\leq& \exp\left(-\frac{(1-\delta-\bar{\delta})^2}{8}\lambda^{*2}\sum_{j=1}^p(z_j^*\circ z_1^{*-1}-z_j^*\circ b^{-1})^2\right).
\end{eqnarray*}
For the second term, as we have established in the proof of Lemma \ref{lem:error-zk}, $\wh{\lambda}(z^*)-\lambda^* \sim \mathn(0, 1/\sum_{i\neq j} (z_i^* \circ z_j^{*-1})^2)$ where the variance is smaller than $1/\alpha p^2$. Then by a standard Gaussian tail bound, we have
\begin{eqnarray*}
&& \mathbb{P}\left(\sum_{j=1}^pw_{j1}(z_j^*\circ z_1^{*-1}-z_j^*\circ b^{-1})(\wh{\lambda}(z^*)-\lambda^*)\leq -\frac{\bar{\delta}}{2}\lambda^{*2}\sum_{j=1}^p(z_j^*\circ z_1^{*-1}-z_j^*\circ b^{-1})^2\right) \\
&\leq& \mathbb{P}\left(\frac{C\sqrt{x}}{p}\left|\sum_{j=1}^pw_{j1}(z_j^*\circ z_1^{*-1}-z_j^*\circ b^{-1})\right| \geq \frac{\bar{\delta}}{2}\lambda^{*2}\sum_{j=1}^p(z_j^*\circ z_1^{*-1}-z_j^*\circ b^{-1})^2\right) \\
&& + \mathbb{P}\left(|\wh{\lambda}(z^*)-\lambda^*|>\frac{C\sqrt{x}}{p}\right) \\
&\leq& 2\exp\left(-C'\frac{\bar{\delta}^2p^2\lambda^{*4}\sum_{j=1}^p(z_j^*\circ z_1^{*-1}-z_j^*\circ b^{-1})^2}{x}\right) + e^{-x}.
\end{eqnarray*}
Take $x=\bar{\delta}p|\lambda^*|^2\sqrt{\sum_{j=1}^p(z_j^*\circ z_1^{*-1}-z_j^*\circ b^{-1})^2}$, and we obtain the bound
\begin{eqnarray*}
&& 3\exp\left(-C'\bar{\delta}p|\lambda^*|^2\sqrt{\sum_{j=1}^p(z_j^*\circ z_1^{*-1}-z_j^*\circ b^{-1})^2}\right) \\
&\leq& 3\exp\left(-\frac{(1-\delta-\bar{\delta})^2}{8}\lambda^{*2}\sum_{j=1}^p(z_j^*\circ z_1^{*-1}-z_j^*\circ b^{-1})^2\right),
\end{eqnarray*}
where the above inequality uses the condition that $p/k^2\rightarrow\infty$ and the fact that $\bar{\delta}$ tends to zero at a sufficiently slow rate. Combining the above bounds, we obtain
\begin{eqnarray*}
&& \mathbb{P}\left(\sum_{j=1}^pw_{j1}(z_j^*\circ z_1^{*-1}-z_j^*\circ b^{-1})\wh{\lambda}(z^*)\leq -\frac{1-\delta}{2}\lambda^{*2}\sum_{j=1}^p(z_j^*\circ z_1^{*-1}-z_j^*\circ b^{-1})^2\right) \\
&\leq& 4\exp\left(-\frac{(1-\delta-\bar{\delta})^2}{8}\lambda^{*2}\sum_{j=1}^p(z_j^*\circ z_1^{*-1}-z_j^*\circ b^{-1})^2\right).
\end{eqnarray*}
A similar bound holds for
$$\mathbb{P}\left(\sum_{j=1}^pw_{jl}(z_j^*\circ z_l^{*-1}-z_j^*\circ b^{-1})\wh{\lambda}(z^*)\leq -\frac{1-\delta}{2}\lambda^{*2}\sum_{j=1}^p(z_j^*\circ z_l^{*-1}-z_j^*\circ b^{-1})^2\right)$$
for each $l\in[p]$. Since $\max_l \max_{b\neq z_l^*} \sum_{j=1}^p(z_j^*\circ z_l^{*-1}-z_j^*\circ b^{-1})^2 \geq  \lambda^{*2}p$, this implies
$$\mathbb{E}\xi_{\rm ideal}(\delta)\leq p\exp\left(-\frac{1+o(1)}{8}\lambda^{*2}p\right),$$
under the condition that $p\lambda^{*2}\rightarrow\infty$. The conclusion is implied by Markov inequality.
\end{proof}

\begin{proof}[Proof of Proposition \ref{prop:ini}]
Let $P=\mathbb{E}Y\in\mathbb{R}^{p\times p}$. Then, $P_{ij}=\lambda^*(z_i^*\circ z_j^{*-1})$ and thus $P$ has $k$ different columns. We denote the $k$ different columns by $\theta_0,\cdots,\theta_{k-1}\in\mathbb{R}^p$.  Namely, $\theta_l$ is the vector $\lambda^*(z^*\circ l^{-1})$. The same analysis in the proof of Proposition \ref{prop:ini-clust} leads to the bound
\begin{equation}
\sum_{j=1}^p\|\theta_{\pi(\bar{z}_j)}-\theta_{z_j^*}\|^2\lesssim (M+1)kp,\label{eq:use-spec-bound}
\end{equation}
for some permutation $\pi$ acting on the set $\{0,1,\cdots,k-1\}$ under the condition that $\min_{a\in\Zk}\sum_{j=1}^p\indc{z_j^*=a}\geq\frac{\alpha p}{k}$ for some constant $\alpha>0$. Since $\min_{a\neq b}\|\theta_a-\theta_b\|^2\gtrsim p\lambda^{*2}$, we have
\begin{equation}
\sum_{j=1}^p\indc{\pi(\bar{z}_j)\neq z_j^*}\lesssim \frac{(M+1)k}{\lambda^{*2}}.\label{eq:overcook}
\end{equation}
Therefore, when $\frac{(M+1)k^2}{p\lambda^{*2}}=o(1)$, we must have
\begin{equation}
\min_{a\in\Zk}\sum_{j=1}^p\indc{\bar{z}_j=a}\geq \frac{\alpha p}{2k}.\label{eq:sd-40}
\end{equation}
For any $a,b\in\Zk$, recall the notation $\bar{\Z}_{ab}=\{(i,j):\bar{z}_i=a,\bar{z}_j=b\}$. We will bound the difference between $\bar{Y}_l=\frac{1}{|\bar{\Z}_{l0}|}\sum_{(i,j)\in\bar{\Z}_{l0}}Y_{ij}$ and $\lambda^*(\pi(l)\circ \pi(0)^{-1})$. By triangle inequality, we have
\begin{equation}
|\bar{Y}_l-\lambda^*(\pi(l)\circ \pi(0)^{-1})| \leq \left|\frac{1}{|\bar{\Z}_{l0}|}\sum_{(i,j)\in\bar{\Z}_{l0}}(\mathbb{E}Y_{ij}-\lambda^*(\pi(l)\circ \pi(0)^{-1}))\right| + \left|\frac{1}{|\bar{\Z}_{l0}|}\sum_{(i,j)\in\bar{\Z}_{l0}}w_{ij}\right|. \label{eq:shift-joycon}
\end{equation}
For the second term of (\ref{eq:shift-joycon}), we have
$$\left|\frac{1}{|\bar{\Z}_{l0}|}\sum_{(i,j)\in\bar{\Z}_{l0}}w_{ij}\right|\lesssim \frac{k}{p}\max_{\Z_{l0}}\left|\frac{1}{\sqrt{|\Z_{l0}|}}\sum_{(i,j)\in\Z_{l0}}w_{ij}\right|,$$
where the above inequality is by (\ref{eq:sd-40}), and the maximization is over all $k^p$ possible clustering configurations. A simple union bound argument implies $\max_{\Z_{l0}}\left|\frac{1}{\sqrt{|\Z_{l0}|}}\sum_{(i,j)\in\Z_{l0}}w_{ij}\right|\lesssim \sqrt{p\log k}$ with probability at least $1-e^{-C'p\log k}$. Therefore,
$$\max_{l}\left|\frac{1}{|\bar{\Z}_{l0}|}\sum_{(i,j)\in\bar{\Z}_{l0}}w_{ij}\right|\lesssim \frac{k\sqrt{\log k}}{\sqrt{p}},$$
with high probability. For the first term of (\ref{eq:shift-joycon}), we have
\begin{eqnarray*}
&& \left|\frac{1}{|\bar{\Z}_{l0}|}\sum_{(i,j)\in\bar{\Z}_{l0}}(\mathbb{E}Y_{ij}-\lambda^*(\pi(l)\circ \pi(1)^{-1}))\right| \\
&\leq& |\lambda^*|k \frac{1}{|\bar{\Z}_{l0}|}\sum_{(i,j)\in\bar{\Z}_{l0}}\left(\indc{z_i^*\neq \pi(l)}+\indc{z_j^*\neq\pi(1)}\right) \\
&\leq& k|\lambda^*|\left(\frac{\sum_{i=1}^p\indc{\bar{z}_i=l,z_i^*\neq \pi(l)}}{\sum_{i=1}^p\indc{\bar{z}_i=l}}+\frac{\sum_{j=1}^p\indc{\bar{z}_j=1,z_j^*\neq \pi(1)}}{\sum_{j=1}^p\indc{\bar{z}_j=1}}\right) \\
&\lesssim& \frac{k^2|\lambda^*|}{p}\frac{(M+1)k}{\lambda^{*2}},
\end{eqnarray*}
where the last inequality is by (\ref{eq:overcook}) and (\ref{eq:sd-40}). Combining the two bounds, we can then bound (\ref{eq:shift-joycon}) by
$$\max_l|\bar{Y}_l-\lambda^*(\pi(l)\circ \pi(0)^{-1})|\lesssim \frac{k^3(M+1)}{p|\lambda^*|}+\frac{k\sqrt{\log k}}{\sqrt{p}}=o(|\lambda^*|),$$
under the condition that $\frac{(M+1)k^3}{p\lambda^{*2}}=o(1)$. We thus have
$$
\max_l||\bar{Y}_l|-|\lambda^*|(\pi(l)\circ \pi(0)^{-1})| = o(|\lambda^*|),
$$
with high probability. Since the difference among $|\lambda^*|(\pi(l)\circ \pi(0)^{-1})$ with different $l$'s is at least $|\lambda^*|$, the order of $\{|\bar{Y}_l|\}$ perfectly recovers the order of $\{|\lambda^*|(\pi(l)\circ \pi(0)^{-1})\}$. Since $|\bar{Y}_{l}|=|\bar{Y}|_{(\wh{\pi}(l))}$, we have
\begin{eqnarray*}
\ell(z^{(0)},z^*\circ \pi(0)^{-1}) &=& \sum_{j=1}^p\|\theta_{z_j^{(0)}}-\theta_{z_j^*\circ \pi(0)^{-1}}\|^2 \\
&=& \sum_{j=1}^p\|\theta_{\pi(\bar{z}_j)}-\theta_{z_j^*}\|^2 \\
&\lesssim& (M+1)kp,
\end{eqnarray*}
where the last inequality is by (\ref{eq:use-spec-bound}). The proof is complete.
\end{proof}

\subsection{Proofs in Section \ref{sec:perm}}

In this section, we present the proofs of Theorem \ref{thm:lower-per}, Lemma \ref{lem:error-per}, Lemma \ref{lem:ideal-per} and Proposition \ref{prop:ini-per}. The conclusions of Theorem \ref{thm:per-con-alg} and Corollary \ref{cor:per-con-alg} are direct consequences of Theorem \ref{thm:main}, and thus their proofs are omitted. We first state a technical lemma.

\begin{lemma}\label{lem:large-w}
Consider the error matrix $W=Y-\mathbb{E}Y\in\mathbb{R}^{pd\times pd}$ in the problem of permutation synchronization. There exists some constant $C>0$, such that
$$\mathbb{P}\left(\|W\|>C\sqrt{pd+x}\right)\leq e^{-x},$$
for any $x>0$.
\end{lemma}
\begin{proof}
This lemma is an extension of Lemma \ref{lem:gw-r-m}, and can be proved using the same technique in \cite{vershynin2010introduction}. We omit the details.
\end{proof}

\begin{proof}[Proof of Theorem \ref{thm:lower-per}]
Let $\bar{U}$ to be a matrix obtained by switching the first and the second rows of $I_d$. In other words, we have $\fnorm{\bar{U}-I_d}^2=4$.
Consider the parameter space
$$\mathcal{Z}=\left\{Z: Z_j=I_d\text{ for all }1\leq j\leq p/2\text{ and }Z_j\in\{I_d,\bar{U}\}\text{ for all }p/2<j\leq p\right\}.$$
Then, by the same argument that leads to (\ref{eq:metroidvania}), we have
\begin{eqnarray}
\nonumber && \inf_{\wh{Z}}\sup_{Z^*}\mathbb{E}\min_{U}\frac{1}{p}\sum_{j=1}^p\indc{\wh{Z}_j\neq UZ_j^*} \\
\nonumber &\geq& \frac{1}{p}\sum_{j>p/2}\text{ave}_{Z_{-j}^*}\inf_{\wh{Z}_j}\left(\frac{1}{2}\mathbb{P}_{(Z_j^*=I_d,Z_{-j}^*)}(\wh{Z}_j\neq I_d) + \frac{1}{2}\mathbb{P}_{(Z_j^*=\bar{U},Z_{-j}^*)}(\wh{Z}_j\neq \bar{U})\right),
\end{eqnarray}
Note that the quantity
$$\inf_{\wh{Z}_j}\left(\frac{1}{2}\mathbb{P}_{(Z_j^*=I_d,Z_{-j}^*)}(\wh{Z}_j\neq I_d) + \frac{1}{2}\mathbb{P}_{(Z_j^*=\bar{U},Z_{-j}^*)}(\wh{Z}_j\neq \bar{U})\right)$$
is the optimal testing error for a hypothesis testing problem where under $H_0$ we have $\mathbb{E}Y_{ij}=\lambda^*Z_i^*$ for all $i\in[p]\backslash\{j\}$ and under $H_1$ we have $\mathbb{E}Y_{ij}=\lambda^*Z_i^*\bar{U}^T$ for all $i\in[p]\backslash\{j\}$. The distributions are both Gaussian under the two hypotheses. By Neyman-Pearson lemma, the optimal testing error is $\mathbb{P}(\mathn(0,1)>\lambda^*\sqrt{p-1})$. Therefore,
$$\inf_{\wh{Z}}\sup_{Z^*}\mathbb{E}\min_{U}\frac{1}{p}\sum_{j=1}^p\indc{\wh{Z}_j\neq UZ_j^*}\geq \frac{1+o(1)}{2}\mathbb{P}(\mathn(0,1)>\lambda^*\sqrt{p-1})=\exp\left(-\frac{1+o(1)}{2}p\lambda^{*2}\right),$$
under the condition that $p\lambda^{*2}\rightarrow\infty$. When $p\lambda^{*2}=O(1)$, we obtain a constant lower bound.
\end{proof}

\begin{proof}[Proof of Lemma \ref{lem:error-per}]
We can write
$$\wh{B}(z)=\frac{\iprod{\lambda^*Z^*Z^{*T}+W}{ZZ^T}}{p^2d^2}Z=\lambda^*\frac{\iprod{Z^*Z^{*T}}{ZZ^T}}{p^2d^2}Z + \frac{\iprod{W}{ZZ^T}}{p^2d^2}Z,$$
where $W=Y-\lambda^*Z^*Z^{*T}\in\mathbb{R}^{pd\times pd}$ is a Gaussian error matrix. Then, we have $\wh{B}(Z^*)=B^*+\frac{\iprod{W}{Z^*Z^{*T}}}{p^2d^2}Z^*$. By triangle inequality, we have
\begin{equation}
\fnorm{\wh{B}(Z)-\wh{B}(Z^*)} \leq \fnorm{\lambda^*\frac{\iprod{Z^*Z^{*T}}{ZZ^T}}{p^2d^2}Z-B^*} + \fnorm{\frac{\iprod{W}{ZZ^T}}{p^2d^2}Z-\frac{\iprod{W}{Z^*Z^{*T}}}{p^2d^2}Z^*}. \label{eq:mario-odyssey}
\end{equation}
We will bound the two terms on the right hand side of (\ref{eq:mario-odyssey}) separately. The first term can be bounded as
\begin{eqnarray}
\nonumber\fnorm{\lambda^*\frac{\iprod{Z^*Z^{*T}}{ZZ^T}}{p^2d^2}Z-B^*} &\leq& \lambda^*\fnorm{Z}\left|\frac{\iprod{Z^*Z^{*T}}{ZZ^T}}{p^2d^2}-1\right| + \lambda^*\fnorm{Z-Z^*} \\
\label{eq:korok}&\leq& \frac{2\lambda^*}{p^{1/2}d^{3/2}}\fnorm{Z-Z^*}^2 + \lambda^*\fnorm{Z-Z^*} \\
\nonumber&\lesssim& \lambda^*\fnorm{Z-Z^*},
\end{eqnarray}
where the inequality (\ref{eq:korok}) is due to
\begin{eqnarray*}
|p^2d^2-\iprod{Z^*Z^{*T}}{ZZ^T}| &=& \frac{1}{2}\fnorm{ZZ^T-Z^*Z^{*T}}^2 \\
&\leq& \fnorm{Z^*(Z-Z^*)^T}^2+\fnorm{(Z-Z^*)Z^T}^2 \\
&\leq& (\norm{Z}^2+\norm{Z^*}^2)\fnorm{Z-Z^*}^2 \\
&\leq& 2p\fnorm{Z-Z^*}^2.
\end{eqnarray*}
For the second term of (\ref{eq:mario-odyssey}), we have
\begin{eqnarray}
\nonumber && \fnorm{\frac{\iprod{W}{ZZ^T}}{p^2d^2}Z-\frac{\iprod{W}{Z^*Z^{*T}}}{p^2d^2}Z^*} \\
\nonumber &\leq& \frac{|\iprod{W}{ZZ^T-Z^*Z^{*T}}|}{p^2d^2}\fnorm{Z} + \fnorm{Z-Z^*}\frac{|\iprod{W}{Z^*Z^{*T}}|}{p^2d^2} \\
\label{eq:akinfenwa} &\leq& \frac{\sqrt{2d}\sqrt{pd}\|W\|\fnorm{ZZ^T-Z^*Z^{*T}}}{p^2d^2} + \fnorm{Z-Z^*}\frac{\|W\|\sqrt{d}\fnorm{Z^*Z^{*T}}}{p^2d^2} \\
\nonumber &\lesssim& \frac{\|W\|\fnorm{Z-Z^*}}{pd} \\
\label{eq:adebayor} &\lesssim& \frac{\fnorm{Z-Z^*}}{\sqrt{pd}}.
\end{eqnarray}
The inequality (\ref{eq:akinfenwa}) is by applying SVD to the rank-$(2d)$ matrix $ZZ^T-Z^*Z^{*T}=\sum_{l=1}^{2d}d_lu_lu_l^T$ so that $|\iprod{W}{ZZ^T-Z^*Z^{*T}}|$ is bounded by
$$\sum_{l=1}^{2d}|d_l||u_l^TWu_l| \leq \|W\|\sum_{l=1}^{2d}|d_l|\leq \sqrt{2d}\|W\|\sqrt{\sum_{l=1}^{2d}d_l^2}=\sqrt{2d}\|W\|\fnorm{ZZ^T-Z^*Z^{*T}}.$$
Similarly, we also have $|\iprod{W}{Z^*Z^{*T}}|\leq \sqrt{d}\|W\|\fnorm{Z^*Z^{*T}}$. The last inequality (\ref{eq:adebayor}) is by Lemma \ref{lem:large-w}. Combining the above bounds, we have
\begin{equation}
\fnorm{\wh{B}(Z)-\wh{B}(Z^*)} \leq \left(\lambda^*+\frac{1}{\sqrt{pd}}\right)\fnorm{Z-Z^*}. \label{eq:crystal-peak}
\end{equation}
Now we are ready to bound (\ref{eq:error-per1})-(\ref{eq:error-per3}). For (\ref{eq:error-per1}), we have
\begin{align*}
&\sum_{j=1}^p\max_{U\neq Z_j^*} \frac{F_j(Z_j^*,U;Z)^2\norm{\mu_j(B^*,U) - \mu_j(B^*,Z_j^*)}^2}{\Delta_j(Z_j^*,U)^4\ell(Z,Z^*)} \\
& =\sum_{j=1}^p\max_{U\neq Z_j^*}\frac{\left|\iprod{\epsilon_j}{(\wh{B}(Z^*)-\wh{B}(Z))(Z_j^*-U)^T}\right|^2}{\fnorm{B^*(Z_j^*-U)^T}^2\ell(Z,Z^*)}  \\
 & \lesssim  \sum_{j=1}^p\max_{U\neq Z_j^*}\frac{\left|\iprod{Z_j^*-U}{\epsilon_j^T(\wh{B}(Z^*)-\wh{B}(Z))}\right|^2}{p\lambda^{*2} \fnorm{Z_j^* - U}^2\ell(Z,Z^*)}  \\
 &  \lesssim  \sum_{j=1}^p\max_{U\neq Z_j^*}\frac{\fnorm{Z_j^*-U}^2\fnorm{\epsilon_j^T(\wh{B}(Z^*)-\wh{B}(Z))}^2}{p\lambda^{*2} \fnorm{Z_j^* - U}^2\ell(Z,Z^*)}   \\
 & =  \sum_{j=1}^p\frac{\fnorm{\epsilon_j^T(\wh{B}(Z^*)-\wh{B}(Z))}^2}{p\lambda^{*2} \ell(Z,Z^*)}   \\
 &  = \frac{\iprod{ (\wh{B}(Z^*)-\wh{B}(Z))(\wh{B}(Z^*)-\wh{B}(Z))^T}{\sum_{j=1}^p \epsilon_j\epsilon_j^T}}{p\lambda^{*2} \ell(Z,Z^*)}  \\
 & \leq  \frac{ \norm{\sum_{j=1}^p \epsilon_j\epsilon_j^T}\norm{(\wh{B}(Z^*)-\wh{B}(Z))(\wh{B}(Z^*)-\wh{B}(Z))^T}_*}{p\lambda^{*2} \ell(Z,Z^*)},
\end{align*}
where $\norm{\cdot}_*$ is the matrix nuclear norm. 
Note that $\sum_{j=1}^p \epsilon_j\epsilon_j^T = WW^T$ and $\|(\wh{B}(Z^*)-\wh{B}(Z))(\wh{B}(Z^*)-\wh{B}(Z))^T\|_* = \fnorm{\wh{B}(Z^*)-\wh{B}(Z)}^2$. We have
\begin{align*}
\sum_{j=1}^p\max_{U\neq Z_j^*}\frac{F_j(Z_j^*,U;Z)^2\norm{\mu_j(B^*,U) - \mu_j(B^*,Z_j^*)}^2}{\Delta_j(Z_j^*,U)^4\ell(Z,Z^*)}  &\lesssim \frac{\fnorm{\wh{B}(Z^*)-\wh{B}(Z)}^2 \norm{W}^2}{p\lambda^{*2} \ell(Z,Z^*)}\\
& \lesssim  \frac{d}{p\lambda^{*2}} + \frac{1}{p^2\lambda^{*4}},
\end{align*}
where we have used Lemma \ref{lem:large-w} and (\ref{eq:crystal-peak}). 
For (\ref{eq:error-per2}), we have
\begin{eqnarray*}
&& \frac{\tau}{\Delta_{\min}^2|T|}\sum_{j\in T}\max_{U\neq Z_j} \frac{G_j(Z_j^*,U;Z)^2\norm{\mu_j(B^*,U) - \mu_j(B^*,Z_j^*)}^2}{\Delta_j(Z_j^*,U)^4\ell(Z,Z^*)}  \\
&=& \frac{\tau}{\Delta_{\min}^2|T|}\sum_{j\in T}\max_{U\neq Z_j}\frac{\left|\iprod{\wh{B}(Z^*)-\wh{B}(Z)}{B^*(I_d-Z_j^{*T}U)}\right|^2}{\fnorm{B^*(Z_j^*-U)^T}^2\ell(Z,Z^*)} \\
&\lesssim& \frac{\tau}{p\lambda^{*2}}\frac{\fnorm{\wh{B}(Z)-\wh{B}(Z^*)}^2\fnorm{B^*}^2}{p\lambda^{*2}\ell(Z,Z^*)} \\
&\lesssim& \frac{\tau\left(\lambda^2+\frac{1}{pd}\right)}{(p\lambda^{*2})^2}.
\end{eqnarray*}
Finally, for (\ref{eq:error-per3}), we have
$$\fnorm{\wh{B}(Z^*)-B^*}=\left|\frac{\iprod{W}{ZZ^T}}{p^2d^2}\right|\fnorm{Z}\leq \frac{\sqrt{d}\|W\|\fnorm{Z^*Z^{*T}}}{p^2d^2}\fnorm{Z^*}\lesssim 1,$$
by Lemma \ref{lem:large-w}, and thus for all $j\in[p]$ and $U\neq Z_j^*$,
$$\frac{\abs{H_j(Z_j^*,U)}}{\Delta_j(Z_j^*,U)^2} = \frac{\left|\iprod{B^*-\wh{B}(Z^*)}{B^*(I_d-Z_j^{*T}U)}\right|}{\fnorm{B^*(Z_j^*-U)^T}^2}\lesssim \fnorm{\wh{B}(Z^*)-B^*}\sqrt{\frac{d}{p\lambda^{*2}}}\lesssim \sqrt{\frac{d}{p\lambda^{*2}}}.$$
The proof is complete.
\end{proof}

\begin{proof}[Proof of Lemma \ref{lem:ideal-per}]
We use the notation $\wh{\lambda}(Z)=\frac{\iprod{Y}{ZZ^T}}{p^2d^2}$. Then, for $U\neq Z_1^*$, we have
\begin{eqnarray*}
&& \mathbb{P}\left(\wh{\lambda}(Z^*)\sum_{j=1}^p\iprod{W_{j1}}{Z_j^*(Z_1^*-U)^T}\leq -\frac{1-\delta}{2}\lambda^{*2}p\fnorm{Z_1^*-U}^2\right) \\
&\leq& \mathbb{P}\left(\lambda^*\sum_{j=1}^p\iprod{W_{j1}}{Z_j^*(Z_1^*-U)^T}\leq -\frac{1-\delta-\bar{\delta}}{2}\lambda^{*2}p\fnorm{Z_1^*-U}^2\right) \\
&& + \mathbb{P}\left((\wh{\lambda}(Z^*)-\lambda^*)\sum_{j=1}^p\iprod{W_{j1}}{Z_j^*(Z_1^*-U)^T}\leq -\frac{\bar{\delta}}{2}\lambda^{*2}p\fnorm{Z_1^*-U}^2\right).
\end{eqnarray*}
The first term can be bounded by a standard Gaussian tail bound,
\begin{eqnarray*}
&& \mathbb{P}\left(\lambda^*\sum_{j=1}^p\iprod{W_{j1}}{Z_j^*(Z_1^*-U)^T}\leq -\frac{1-\delta-\bar{\delta}}{2}\lambda^{*2}p\fnorm{Z_1^*-U}^2\right) \\
&\leq& \exp\left(-\frac{(1-\delta-\bar{\delta})^2}{8}\lambda^{*2}p\fnorm{Z_1^*-U}^2\right).
\end{eqnarray*}
For the second term, note that we have $\wh{\lambda}(Z^*) -\lambda^* = (p^2d^2)^{-1} \iprod{W}{Z^*Z^{*T}}$ which is normally distribution. We have
\begin{eqnarray*}
&& \mathbb{P}\left((\wh{\lambda}(Z^*)-\lambda^*)\sum_{j=1}^p\iprod{W_{j1}}{Z_j^*(Z_1^*-U)^T}\leq -\frac{\bar{\delta}}{2}\lambda^{*2}p\fnorm{Z_1^*-U}^2\right) \\
&\leq& \mathbb{P}\left(\frac{C\sqrt{pd+x}}{pd}\left|\sum_{j=1}^p\iprod{W_{j1}}{Z_j^*(Z_1^*-U)^T}\right|\geq \frac{\bar{\delta}}{2}\lambda^{*2}p\fnorm{Z_1^*-U}^2\right) \\
&& + \mathbb{P}\left(|\wh{\lambda}(Z^*)-\lambda^*|>\frac{C\sqrt{pd+x}}{pd}\right) \\
&\leq& 2\exp\left(-C'\frac{p^2d^2\bar{\delta^2}\lambda^{*4}p\fnorm{Z_1-U}^2}{pd+x}\right) + e^{-x}.
\end{eqnarray*}
Take $x=pd\bar{\delta}\lambda^{*2}\sqrt{p}\fnorm{Z_1-U}$, and we obtain the bound
\begin{eqnarray*}
&& 2\exp\left(-C_1p^2d\bar{\delta}^2\lambda^{*4}\fnorm{Z_1-U}^2\right) + \exp\left(-pd\bar{\delta}\lambda^{*2}\sqrt{p}\fnorm{Z_1-U}\right) \\
&\leq& 3\exp\left(-\frac{(1-\delta-\bar{\delta})^2}{8}\lambda^{*2}p\fnorm{Z_1^*-U}^2\right),
\end{eqnarray*}
under the condition that $p\lambda^{*2}\rightarrow\infty$ and $\bar{\delta}$ tends to zero at a sufficiently slow rate. Combining the above bounds, we have
\begin{eqnarray*}
&& \mathbb{P}\left(\wh{\lambda}(Z^*)\sum_{j=1}^p\iprod{W_{j1}}{Z_j^*(Z_1^*-U)^T}\leq -\frac{1-\delta}{2}\lambda^{*2}p\fnorm{Z_1^*-U}^2\right)\\
&\leq& 4\exp\left(-\frac{(1-\delta-\bar{\delta})^2}{8}\lambda^{*2}p\fnorm{Z_1^*-U}^2\right).
\end{eqnarray*}
A similar bound holds for $\mathbb{P}\left(\wh{\lambda}(Z^*)\sum_{j=1}^p\iprod{W_{jl}}{Z_j^*(Z_l^*-U)^T}\leq -\frac{1-\delta}{2}\lambda^{*2}p\fnorm{Z_l^*-U}^2\right)$ for each $l\in[p]$.

Now we are ready to bound $\xi_{\rm ideal}(\xi)$. We have
\begin{eqnarray*}
\mathbb{E}\xi_{\rm ideal}(\xi) &=& \sum_{l=1}^p\sum_{U\in\mathcal{P}_d}\lambda^{*2}p\fnorm{Z_l^*-U}^2\mathbb{P}\left(\wh{\lambda}(Z^*)\sum_{j=1}^p\iprod{W_{jl}}{Z_j^*(Z_l^*-U)^T}\leq -\frac{1-\delta}{2}\lambda^{*2}p\fnorm{Z_l^*-U}^2\right) \\
&\leq& 4\sum_{l=1}^p\sum_{U\in\mathcal{P}_d}\lambda^{*2}p\fnorm{Z_l^*-U}^2\exp\left(-\frac{(1-\delta-\bar{\delta})^2}{8}\lambda^{*2}p\fnorm{Z_l^*-U}^2\right) \\
&=& p\exp\left(-\frac{1+o(1)}{2}p\lambda^{*2}\right),
\end{eqnarray*}
under the condition that $\frac{p\lambda^{*2}}{d\log d}\rightarrow\infty$. The desired conclusion is implied by Markov inequality.
\end{proof}

\begin{proof}[Proof of Proposition \ref{prop:ini-per}]
It is direct to check that the matrix $U^*$ such that $U^{*T}=(Z_1^*,\cdots,Z_p^*)/\sqrt{p}$ satisfies $U^*\in\mathcal{O}(pd,d)$ and collects the eigenvectors of $\mathbb{E}Y$. By Davis-Kahan theorem, there exists some $\mathcal{O}(d,d)$, such that
$$\|\wh{U}-U^*O\|\lesssim \frac{\|W\|}{p\lambda^*}\lesssim \sqrt{\frac{d}{p\lambda^{*2}}},$$
where the last inequality is by Lemma \ref{lem:large-w}.
According to the definition of $Z_j^{(0)}$, we have
\begin{eqnarray*}
\fnorm{Z_j^{(0)}-Z_j^*O} &\leq& \fnorm{Z_j^{(0)}-\sqrt{p}\wh{U}_j} + \fnorm{\sqrt{p}\wh{U}_j-\sqrt{p}U_j^*O} \\
&\leq& 2\fnorm{\sqrt{p}\wh{U}_j-\sqrt{p}U_j^*O}.
\end{eqnarray*}
Therefore,
$$\sum_{j=1}^p\fnorm{Z_j^{(0)}-Z_j^*O}^2 \leq 4\sum_{j=1}^p\fnorm{\sqrt{p}\wh{U}_j-\sqrt{p}U_j^*O}^2 = 4p\fnorm{\wh{U}-U^*O}^2\lesssim pd\|\wh{U}-U^*O\|^2\lesssim \frac{d^2}{\lambda^{*2}}.$$
Then,
\begin{eqnarray*}
\sum_{i=1}^p\sum_{j=1}^p\fnorm{Z_i^{(0)}-Z_i^*Z_j^{*T}Z_j^{(0)}}^2 &=& \sum_{i=1}^p\sum_{j=1}^p\fnorm{Z_i^{(0)}Z_j^{(0)T}-Z_i^*Z_j^{*T}}^2 \\
&=& \fnorm{Z^{(0)}Z^{(0)T}-Z^*Z^{*T}}^2 \\
&\leq& 2\fnorm{Z^{(0)}(Z^{(0)}-Z^*O)^T}^2 + 2\fnorm{(Z^{(0)}-Z^*O)O^TZ^{*T}}^2 \\
&\leq& 4p\fnorm{Z^{(0)}-Z^*O}^2 \\
&=& 4p\sum_{j=1}^p\fnorm{Z_j^{(0)}-Z_j^*O}^2 \\
&\lesssim& \frac{pd^2}{\lambda^{*2}}.
\end{eqnarray*}
Let $\bar{j}=\argmin_{j\in[p]}\sum_{i=1}^p\fnorm{Z_i^{(0)}-Z_i^*Z_j^{*T}Z_j^{(0)}}^2$. Then, we have
$$\sum_{i=1}^p\fnorm{Z_i^{(0)}-Z_i^*Z_{\bar{j}}^{*T}Z_{\bar{j}}^{(0)}}^2\leq \frac{1}{p}\sum_{i=1}^p\sum_{j=1}^p\fnorm{Z_i^{(0)}-Z_i^*Z_j^{*T}Z_j^{(0)}}^2\lesssim \frac{d^2}{\lambda^{*2}}.$$
The proof is complete.
\end{proof}

\bibliographystyle{dcu}
\bibliography{reference}

\end{document}